\documentclass[11pt]{article} 
\usepackage{amsmath,amssymb,latexsym,amsthm,tikz, enumerate,amscd,hyperref,enumitem}
\usepackage{graphicx,tikz}

\usepackage{blindtext}
\usepackage{geometry}

\usetikzlibrary{positioning,arrows}
\usepackage[utf8]{inputenc} 
\usepackage[T1]{fontenc}
\usepackage{hyperref}
\usepackage{theoremref}
\usepackage{float}

\numberwithin{equation}{section}
\theoremstyle{plain}
\newtheorem{lem}[equation]{Lemma}

\newtheorem{prop}[equation]{Proposition}
\newtheorem{thm}[equation]{Theorem}
\newtheorem{cor}[equation]{Corollary}

\theoremstyle{definition}
\newtheorem{definition}[equation]{Definition}
\newtheorem{conj}[equation]{Conjecture}

\newtheorem{cons}[equation]{Construction}
\newtheorem{remark}[equation]{Remark}

\newtheorem{claim}{Claim}

\newtheorem*{que*}{Question}
\newtheorem*{remark*}{Remark}
\newtheorem*{definition*}{Definition}

\newcommand{\type}{\operatorname{Type}}
\newcommand{\cq}{\mathcal{Q}}
\newcommand{\od}{\widehat{\Sigma}}
\newcommand{\Si}{\Sigma}
\newcommand{\p}{\Pi}
\newcommand{\bD}{\mathbb D}

\newcommand{\bC}{\mathfrak C}

\newcommand{\ca}{\mathcal {A}} 
\newcommand{\prj}{\operatorname{Proj}} 
\newcommand{\s}{\operatorname{Sal}} 

\newcommand{\vertex}{\operatorname{Vert}} 
\newcommand{\lk}{\operatorname{lk}} 
\newcommand{\st}{\operatorname{st}}

\newcommand{\cp}{\mathcal {P}} 
\newcommand{\ce}{\mathcal {E}}

\newcommand{\whC}{\widehat C}

\newcommand{\act}{\curvearrowright}
\newcommand{\eps}{\epsilon} 
\newcommand{\minus}{{-1}} 
\newcommand{\Ga}{\Gamma} 
\newcommand{\g}{\textbf{G}} 
\newcommand{\eq}{\stackrel{*}{=}}

\newcommand{\G}[0]{\mathbf{G}}
\newcommand{\ant}{\operatorname{Ant}}
\newcommand{\supp}{\operatorname{Supp}}
\newcommand{\Sii}{\Sigma^1}
\newcommand{\w}{\operatorname{wd}}
\newcommand{\sfs}{\mathsf s}
\newcommand{\sft}{\mathsf t}
\newcommand{\cw}{\mathcal{W}}
\newcommand{\Up}{\Upsilon}

\newcommand{\CAT}{\operatorname{CAT}}

	{ 
	\end{enumerate}
}

\begin{document}

\title{Labeled four cycles and the $K(\pi,1)$-conjecture for Artin groups}
\author{Jingyin Huang}
\maketitle

\begin{abstract}
We show that for a large class of Artin groups with Dynkin diagrams being a tree, the $K(\pi,1)$-conjecture holds. 
We also establish the $K(\pi,1)$-conjecture for another class of Artin groups whose Dynkin diagrams contain a cycle, which applies to some Artin groups whose Dynkin diagrams are of hyperbolic type. This is based on a new approach to the $K(\pi,1)$-conjecture for Artin groups.
\end{abstract}

\section{Introduction}
\label{sec:intro}

In this article, we establish a close connection between a simple property in metric graph theory about 4-cycles and the long-standing $K(\pi,1)$-conjecture for hyperplane complements associated to infinite reflection groups, via elements of non-positively curvature geometry. We also propose a new approach for studying the $K(\pi,1)$-conjecture.
As a consequence, we find a large number of new cases of Artin groups which satisfy the $K(\pi,1)$-conjecture.

\subsection{Background and summary}
\label{subsec:background}
Let $W_S$ be a Coxeter group with generating set $S$. A \emph{reflection} of $W_S$ is a conjugate an element in $S$. Let $R$ be the collection of all reflections in $W_S$.
 Recall that $W_S$ admits a \emph{canonical representation} $\rho:W\to GL(V)$, such that each element in $R$ acts as linear reflection on $V$, and the action $W_S\act V$ stabilizes the interior $I$ of a convex cone in $V$, called the \emph{Tits cone}. The $W_S$-action on $I$ is properly discontinuous.
For each reflection $r\in W_S$, let $H_r$ be the set of fix points of $\rho(r)$ on $I$. The collection of all such $H_r$ forms an arrangement of hyperplanes in $I$.
We consider the manifold
$$
M(W_S)= (I\times I)\setminus(\cup_{r\in R} (H_r\times H_r)). 
$$
The \emph{$K(\pi,1)$-conjecture for reflection arrangement complements}, due to Arnold, Brieskorn, Pham and Thom, predicts that the space $M(W_S)$ is aspherical for any Coxeter group $W_S$. 

There is an induced action of $W_S\act M(W_S)$ which is free and properly discontinuous. The fundamental group of the quotient manifold is defined to be the \emph{Artin group}, or the \emph{Artin-Tits group} $A_S$ associated with $W_S$. We will also say the $K(\pi,1)$-conjecture holds for the Artin group $A_S$, if $M(W_S)$ is aspherical.

As an example, if $W_S$ is the symmetry group on $n$-elements, then the canonical presentation permutes coordinates in $\mathbb R^n$. The Tits cone is all of $\mathbb R^n$. Then $$M(W_S)=\mathbb C^n\setminus \cup_{i\neq j}\{(z_1,\cdots,z_n)\in \mathbb C^n\mid z_i=z_j\}$$ is the configuration space of ordered $n$-points in $\mathbb C^n$. The associated Artin group is a braid group. 

The study of $K(\pi,1)$-conjecture connects to the study of Artin groups. These are also poorly understood, see \cite{godelle2012basic}.  
Artin groups have simple presentations \cite{lek}.
Let $\Gamma$ be a finite simplicial graph with each edge labeled by an integer $\ge 2$. The \emph{Artin group with presentation graph $\Gamma$}, denoted $A_\Gamma$, is a group whose generators are in one to one correspondence with vertices of $\Gamma$, and there is a relation of the form $aba\cdots=bab\cdots$ with both sides being alternating words of length $m$
whenever two vertices $a$ and $b$ are connected by an edge labeled by $m$. 

To date, $K(\pi,1)$-conjecture is widely open and the list of known cases is short. 
The $K(\pi,1)$-conjecture was proved for all spherical type Artin groups (i.e. the associated Coxeter group is finite) by Deligne \cite{deligne}, after being proved in some subclasses in \cite{fox1962braid,brieskorn2006groupes}. The $K(\pi,1)$-conjecture was proved for all affine type Artin groups (i.e. the associated Coxeter group is affine) recently by Paolini and Salvetti \cite{paolini2021proof}, based on work of McCammond and Sulway \cite{mccammond2017artin}, before some subclasses were treated in \cite{okonek1979k,CharneyDavis,charney2003cal,callegaro2010k}. Very recently Haettel \cite{haettel2021lattices} recovered the $K(\pi,1)$-conjecture of some affine type Artin groups by an alternative geometric method, which partly inspired this article.

Charney and Davis proved the $K(\pi,1)$-conjecture for all FC type Artin groups \cite{CharneyDavis}. This leads to a general principle that the study of $K(\pi,1)$-conjecture reduces to $K(\pi,1)$-conjecture for Artin groups whose presentation graphs are complete \cite{ellis2010k,godelle2012k}. Charney and Davis also proved the $K(\pi,1)$-conjecture for all 2-dimensional Artin groups \cite{CharneyDavis}. There are also results on classes which have some ``lower-dimensional'' feature but not exactly 2-dimensional, see \cite{charney2004deligne,juhasz2018relatively,goldman2022k,juhasz2023class}.

The $K(\pi,1)$-conjecture for reflection arrangement complements fits into the more general theme of understanding the topology of complex hyperplane arrangement complements, and is closely related to the study of the topology of discriminant complements of singularities.  See e.g. \cite{hendriks1985hyperplane,terao1986modular,falk1995k,paris1995topology,bessis2015finite,Allcock} for some notable results on the $K(\pi,1)$ problem for other types of complements. 

\smallskip

Given an Artin group with presentation graph $\Gamma$, its \emph{Dynkin diagram} $\Lambda$ is obtained from $\Gamma$ by removing all open edges of $\Gamma$ labeled by $2$, and add extra edges labeled by $\infty$ between vertices of $\Gamma$ which are not adjacent. 
A Dynkin diagram is spherical or affine, if the associated Coxeter group is spherical or affine. See Figure~\ref{fig:sa} for a complete list of spherical Dynkin diagrams (left two columns) made of four infinite families and six exceptional cases, and affine Dynkin diagrams (right two columns) made of four infinite families and six exceptional cases.
It is a convention that we suppress the label of all edges that are labeled by $3$ in a Dynkin diagram. 
We write $A_\Lambda$ for Artin group with Dynkin diagram $\Lambda$. We will refer to vertices in the presentation graphs and Dynkin diagrams as \emph{nodes}.

\begin{figure}[H]
	\centering
	\includegraphics[scale=0.86]{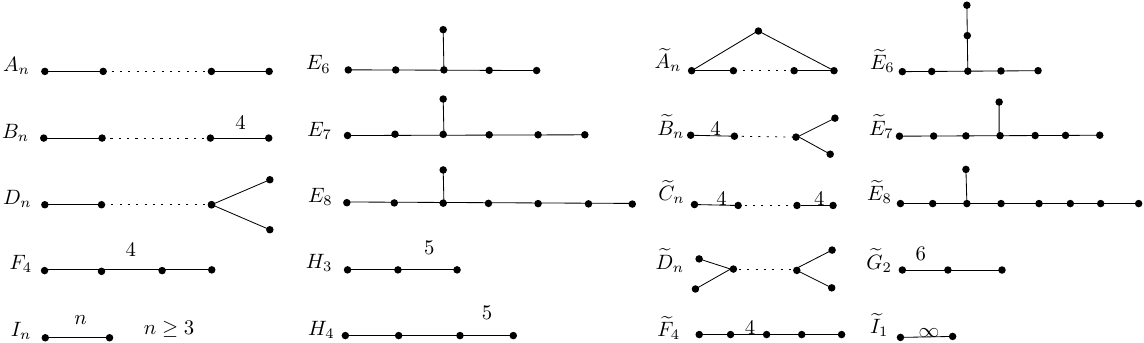}
	\caption{Spherical and affine Dynkin diagrams}
	\label{fig:sa}
\end{figure}

\textbf{By the principle discussed above, from now on we will only consider Artin groups whose Dynkin diagram does not have any edges labeled by $\infty$ - equivalently the presentation graph is complete.}
We vastly extend the known results on diagrams in Figure~\ref{fig:sa} in two different directions. The first is on classes of Artin groups with tree Dynkin diagrams (in light of that most diagrams in Figure~\ref{fig:1} are trees). This is discussed in Section~\ref{subsec:main results1}. The second is on some Artin groups associated with hyperbolic reflection groups, and a family of variations, which are discussed in Section~\ref{subsec:main results2}.

These results are obtained via a new approach to the $K(\pi,1)$-conjecture, explained in Section~\ref{subsec:proof1} and a connection to a simple condition in metric graph theory, explained in Section~\ref{subsec:proof2}. Some of the methods will be further developed in an upcoming article.

\subsection{Artin groups with tree Dynkin diagrams}
\label{subsec:main results1}
The $K(\pi,1)$-conjecture for Artin groups with Dynkin diagrams being trees is widely open. Existing methods for treating $K(\pi,1)$-conjecture for diagrams in Figure~\ref{fig:sa} are highly sensitive to the shape or label of the diagram, e.g. the $K(\pi,1)$-conjecture already becomes unknown if one takes the $B_n$-type diagram, and change the label of the last edge from $4$ to any number $>4$.

Let $\mathcal C_{I_n,A_3}$ (resp. $\mathcal C_{I_n}$) denotes the collection of Dynkin diagrams such that the only allowed connected spherical subdiagrams are either of type $I_n$ or $A_3$ (resp. $I_n$). The class $\mathcal C_{I_n}$ is more commonly studied, and it is often referred as \emph{locally reducible} diagrams.

The only known cases of Artin groups with tree Dynkin diagrams satisfying $K(\pi,1)$-conjecture are either spherical, or affine, or those belong to $\mathcal C_{I_n,A_3}$ \cite{charney2004deligne}. 
Our first theorem treats the $K(\pi,1)$-conjecture for much more general Dynkin trees.

\begin{thm}(=Theorem~\ref{thm:combine1})
	\label{thm:main1}
	Suppose $\Lambda$ is a tree Dynkin diagram. Suppose there exists a collection $E$ of open edges with labels $\ge 6$ such that each component of $\Lambda\setminus E$ is either spherical, or locally reducible. Then $A_\Lambda$ satisfies the $K(\pi,1)$-conjecture.
\end{thm}

Note that we do not requires edges in $\Lambda\setminus E$ has label $<6$. If an edge of $\Lambda\setminus E$ lies in a spherical component, then the largest possible label is $5$, see Figure~\ref{fig:1} (this is what motivates the number $6$ in the theorem), however, if the edge is in a local irreducible component, it could possibly have label $\ge 6$. The theorem also holds true if we replace ``locally reducible'' by ``belongs to $\mathcal C_{I_n,A_3}$'', as each diagram in $\mathcal C_{I_n,A_3}$ which is a tree satisfies the assumption of Theorem~\ref{thm:main1}.

%can be cut along open edges of label $\ge 6$ into $\mathcal C_{I_n}$ diagrams and $A_3$ diagrams. The number $6$ is related to the fact that the largest possible edge label in any spherical Dynkin diagram is $5$.

Examples of Dynkin trees satisfying Theorem~\ref{thm:main1} can be obtained assembling spherical or locally reducible diagrams together by connecting them by edges of label $\ge 6$. See Figure~\ref{fig:tree} for a simple example. We are using Dynkin diagram rather than presentation graph, so the group is not an iterated amalgamation along parabolic subgroups starting from spherical Artin groups.
\begin{figure}[H]
	\centering
	\includegraphics[scale=1.2]{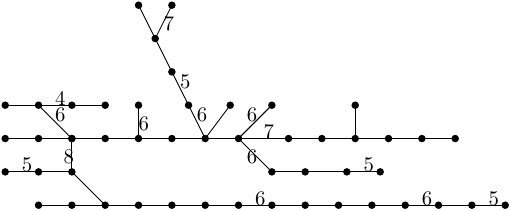}
	\caption{A Dynkin diagram satisfying Theorem~\ref{thm:main1}.}
	\label{fig:tree}
\end{figure}

%The format of Theorem~\ref{thm:main1} is similar to combination type theorems for amalgamated products, though the Artin groups in Theorem~\ref{thm:main1} are not amalgamations of spherical Artin groups and locally reducible Artin groups ($\Lambda$ is \emph{not} the presentation graph). They are rather obtained by taking a direct sum of spherical Artin groups and locally reducible Artin groups, then changing some of the commutation relations to generalized braid relations of length $\ge 6$. 

%The condition on $\ge 6$ leads to a fairly muted form of non-positively-curvature or hyperbolicity.

We can also combine affine type Artin groups in the above theorem as follows, modulo two group theoretic statements on certain affine type Artin groups, one on parabolic closure of elements, and another on commutation of a pair of very specific elements. Note that the study of parabolic closure of elements in Artin groups is very active in recent years \cite{cumplido2019parabolic,cumplido2020parabolic,morris2021parabolic,blufstein2022parabolic,godelle2022parabolic,gonzalez2022parabolic}.
\begin{thm}(=Theorem~\ref{thm:combine2})
	\label{thm:main2}
	Suppose $\Lambda$ is a tree Dynkin diagram. Suppose there exists a collection $E$ of open edges with label $\ge 6$ such that each component of $\Lambda\setminus E$ is either spherical, or affine, or locally reducible. Suppose  the two group theoretic assumptions of Corollary~\ref{cor:algebraic} hold for affine Artin groups of type $\widetilde B$, $\widetilde D$, $\widetilde E$ and $\widetilde F$.
	Then $A_\Lambda$ satisfies the $K(\pi,1)$-conjecture.
\end{thm}

Theorem~\ref{thm:main1} and Theorem~\ref{thm:main2} are special cases of Proposition~\ref{prop:tree contractible intro} allowing components of $\Lambda\setminus E$ to be any diagrams satisfying $K(\pi,1)$-conjecture and an additional assumption.

\subsection{Some hyperbolic type cases and their relatives}
\label{subsec:main results2}
Recall that a Dynkin diagram is \emph{Lann\'er} if the associated Coxeter group is a geometric reflection group generated by reflections across the
codimension 1 faces of a compact $n$-simplex in either $\mathbb S^n$, $\mathbb E^n$ or $\mathbb H^n$. Such diagrams are classified in \cite{lanner1950complexes}.
A Dynkin diagram is \emph{quasi-Lann\'er} if the associated Coxeter group is generated by reflections across the
codimension 1 faces of a non-compact finite volume $n$-simplex in $\mathbb H^n$. Such diagrams are classified in \cite{chein1969recherche}.
As the cases of $\mathbb S^n$ and $\mathbb E^n$ correspond to diagrams in Figure~\ref{fig:sa}, the next natural step is to understand the case of $\mathbb H^n$, where $K(\pi,1)$-conjecture is widely open for $n\ge 3$. 
Unlike the spherical or affine case, there are only finitely many hyperbolic Lann\'er or quasi-Lann\'er groups acting on $\mathbb H^n$ for $n\ge 3$. However, we are able to treat some of these groups which fit into more general infinite families containing members of arbitrarily high dimension. 

A key geometric feature of all quasi-Lann\'er diagrams, which motivates some of our infinite families, is that $\Lambda$ has a \emph{core} (coresponding to the cusps), which is an affine subdiagram $\Lambda'$, such that for any node $s\in \Lambda'$, each component of $\Lambda\setminus\{s\}$ is either spherical or affine. Axiomatizing this leads to a much larger class of Artin groups. Our method is based on the type of cores. In this article we treat the case when the core is a cycle - though we allow a slightly more general situation that the core might not be affine. 

\begin{thm}
	\label{thm:single cycle}
	Suppose $\Lambda$ is a connected Dynkin diagram with an induced cyclic subgraph $C$. Suppose that for each node $s\in C$, the component $\Lambda_s$ of $\Lambda\setminus\{s\}$ that contains remaining nodes of $C$ is either spherical or a locally reducible tree. Then $A_\Lambda$ satisfies the $K(\pi,1)$-conjecture.
	
More generally, assume, in addition, that the two group theoretic assumptions of Corollary~\ref{cor:algebraic} holds for affine Artin groups of type $\widetilde B$, $\widetilde D$, $\widetilde E$ and $\widetilde F$. If $\Lambda_s$ is a tree which is either spherical or affine or locally reducible for any $s\in C$, then $A_\Lambda$ satisfies the $K(\pi,1)$-conjecture.
\end{thm}
This theorem is a special case of more general (and more technical) results in the text, see Theorem~\ref{thm:folded},  Theorem~\ref{thm:folded1} and Theorem~\ref{thm:folded2}, and they cover all Lann\'er or quasi-Lann\'er diagrams in Figure~\ref{fig:known}, which contains examples up to dimension $7$.

To demonstrate Theorem~\ref{thm:single cycle} and its companions in the text, we list a collection of Dynkin diagrams in Figure~\ref{fig:9} below, which contains 11 infinite families of Artin groups, and 20 exceptional examples.  
These diagrams are produced by taking each of the quasi-Lann\'er diagrams or Lann\'er diagrams we can treat from Theorem~\ref{thm:folded2}, and modifying them in one or several ways, so to produce higher dimensional members with similar geometric features (which might not be (quasi)-Lann\'er).
For example, the two ``horseshoe crab'' families are motivated from the quasi-Lann\'er diagram $\Lambda_0$ made of two triangles glued along an edge. The $K(\pi,1)$-conjecture for $A_{\Lambda_0}$ is a consequence of a highly-nontrivial computation of Charney \cite{charney2004deligne}, though the $K(\pi,1)$-conjecture for other members in the horseshoe crab families were unknown. Figure~\ref{fig:9} does not exhaust all the possibilities where our method applies, as we mainly want to list out those families which are closest to (quasi)-Lann\'er diagrams.

\begin{cor}(=Corollary~\ref{cor:extended lanner})
	\label{thm:main4}
	All Artin groups whose Dynkin diagrams belongs to Figure~\ref{fig:9} satisfy the $K(\pi,1)$-conjecture.
\end{cor}

Note that 9 of the 11 infinite families in Figure~\ref{fig:9}  contain members of arbitrarily high dimension, and 19 of the 20 exceptional examples contains spherical subdiagrams of type $F_4,H_3,H_4,E_6$, $E_7,E_8$. The $K(\pi,1)$-conjecture was previously open for all of them, except finitely many examples of dimension $\le 3$ in some of the infinite classes.

We also obtain the following corollary regarding torsion and center of certain Artin groups, using \cite{jankiewicz2023k}, which belongs to some of the fundamental unsolved problems of Artin groups.

\begin{cor}
	Let $\Lambda$ be a connected Dynkin diagram satisfying the assumptions in one of the previous theorems. Then $A_\Lambda$ is torsion free and it has trivial center as long as $\Lambda$ is not spherical.
\end{cor}

\begin{figure}[H]
	\centering
	\includegraphics[scale=0.8]{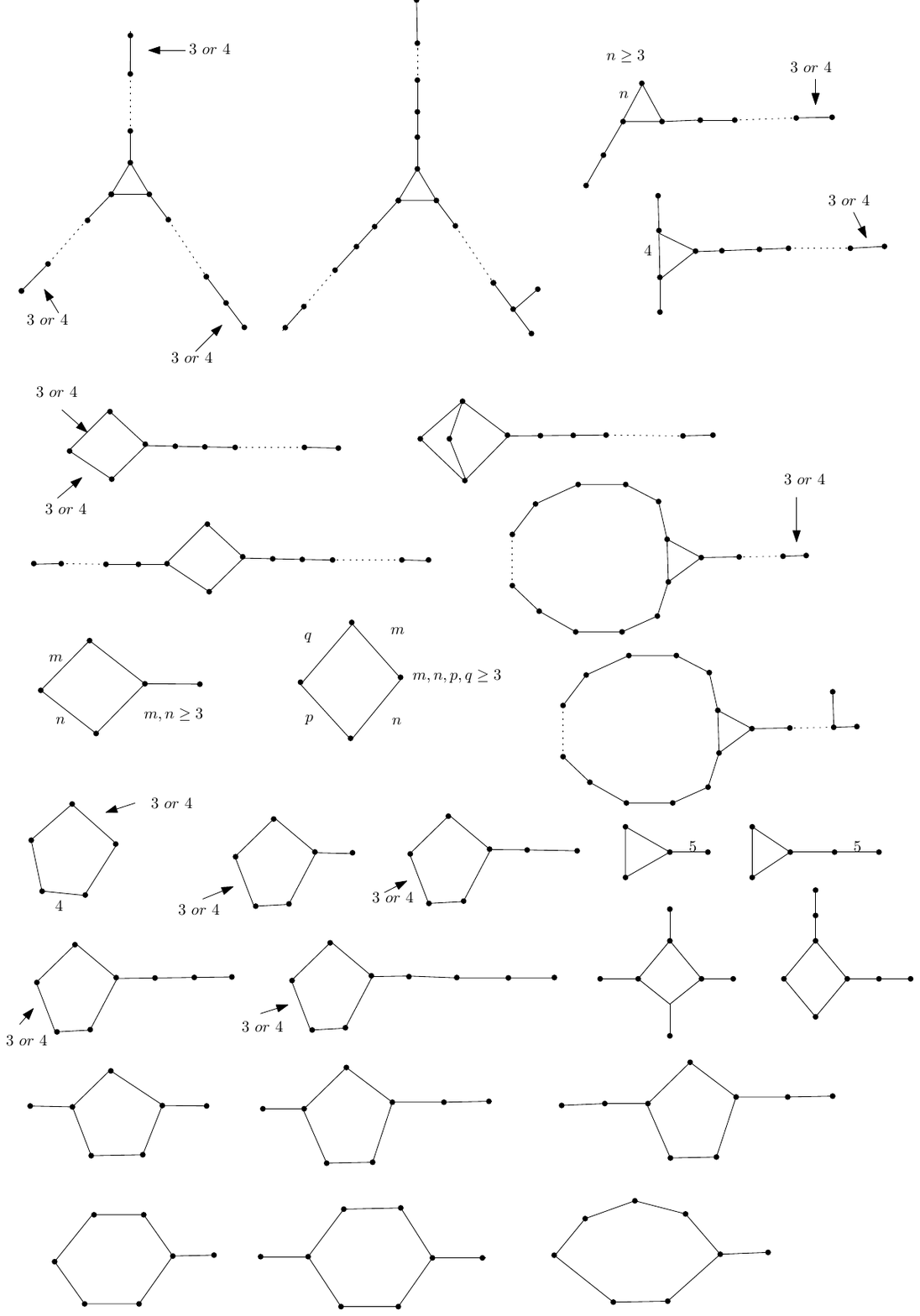}
	\caption{Some diagrams.}
	\label{fig:9}
\end{figure}

Finally, we summarize the collection of \emph{hyperbolic} Lann\'er or quasi-Lann\'er diagrams with known $K(\pi,1)$-conjecture as in Figure~\ref{fig:known}. The $K(\pi,1)$-conjecture for diagrams in the first three columns were known before from \cite{CharneyDavis} and \cite{charney2004deligne}. 
The remaining diagrams are new and they are consequences of Corollary~\ref{thm:main4}. 

%Besides the diagrams in Figure~\ref{fig:known}, there are four remaining hyperbolic Lann\'er or quasi-Lann\'er diagrams with unknown $K(\pi,1)$-conjecture that contains a cycle, see Figure~\ref{f:lanner}. Corollary~\ref{cor:lanner} gives a potential approach to these remaining cases. 

\begin{figure}[H]
	\centering
	\includegraphics[scale=0.8]{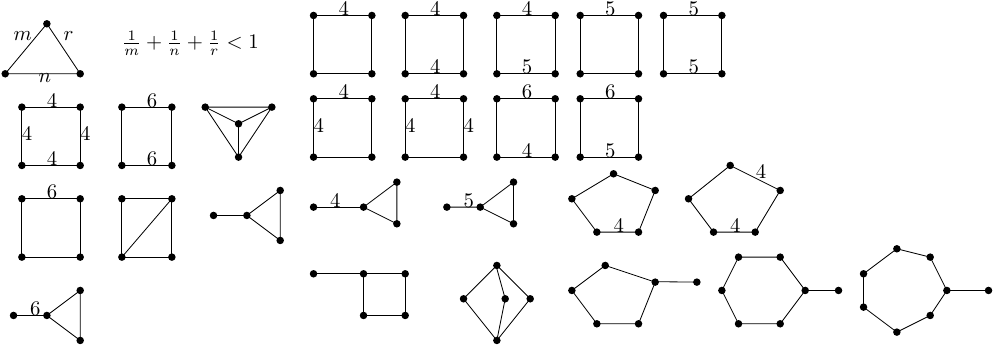}
	\caption{(quasi)-Lann\'er Diagrams with known $K(\pi,1)$-conjecture.}
	\label{fig:known}
\end{figure}

\paragraph{Remark on the hyperbolic cyclic case}
An Artin group is of \emph{hyperbolic cyclic} type if its Dynkin diagram is a circle, and the associated reflection group acts \emph{cocompactly} on $\mathbb H^n$. Their Dynkin diagrams are either $(m,n,r)$-triangles (whose $K(\pi,1)$-conjecture were proved in \cite{CharneyDavis}) or belong to the $6$ examples in Figure~\ref{fig:known} with edge labeling $(3,4,3,3),(3,4,3,4),(3,4,3,5)$, $(3,5,3,3)$, $(3,5,3,5),(3,3,3,3,4)$.
Theorem~\ref{thm:main4} implies that all hyperbolic cyclic type Artin groups satisfy the $K(\pi,1)$-conjecture. 
In another joint article with Thomas Haettel \cite{garside}, 
we provided an alternative proof to the $K(\pi,1)$-conjecture of hyperbolic cyclic type Artin groups. However, the objectives of \cite{garside} and this article are different: the article \cite{garside} gives a list of highly non-trivial geometric, algorithmic and topological properties for hyperbolic cyclic type Artin groups which are not  consequences of this article. On the other hand, this article treats many cases of $K(\pi,1)$-conjecture which are not consequences of \cite{garside}. This article also proves several group theoretic properties for hyperbolic cyclic type Artin groups, including the existence of parabolic closure, see Corollary~\ref{cor:group1}, which are not covered by \cite{garside}.
%
%We also obtain several non-trivial group theoretic properties of hyperbolic cyclic Artin groups, all of them are new for these six extra examples.
%
%\begin{cor}(=Corollary~\ref{cor:group1})
%	\label{cor:group}
%Suppose $A_S$ is an Artin group of hyperbolic cyclic type. Then
%\begin{enumerate}
%	\item the intersection of any parabolic subgroups of $A$ is a parabolic subgroup;
%	\item $A_S$ satisfies Properties $(\ast),(\ast\ast)$ and $(\ast\ast\ast)$ from \cite{godelle2007artin}, in particular, for any $X\subset S$, the commensurator of $A_X$ in $A_S$ coincides with the normalizer of $A_X$ in $A_S$, which is equal to $A_X$ times its quasi-center $QZ_{A_S}(X)$ (recall that $QZ_{A_S}(X)=\{g\in A_S\mid g\cdot X=X\}$);
%	\item for any group $G$ of symmetries of the Artin systems $A_S$, the fix point subgroup $A^G_S$ is isomorphic to an Artin group.
%\end{enumerate}
%\end{cor}

\subsection{Structure of the paper}
In Section~\ref{sec:proof} we discuss the proofs of some of the main theorems, via a new approach to the $K(\pi,1)$-conjecture.
In Sections~\ref{sec:pre1} and \ref{sec:pre2} we collect some preliminaries and in Section~\ref{sec:Davis} we prove some properties of the Davis complex and oriented Davis complex for later use. In Section~\ref{sec:relative} we introduce the notion of relative Artin complexes, and consider the two of the key properties on relative Artin complexes, namely the labeled 4-cycle condition and the bowtie free condition. In Section~\ref{subsec:homotopy}, we discuss homotopy types of Artin complexes.
In Section~\ref{sec:4-cycle} we prove Proposition~\ref{prop:labeled 4 wheel intro}. In Section~\ref{sec:tree} we prove Proposition~\ref{prop:labeled 4 wheel intro} and deduce Theorem~\ref{thm:main1} and Theorem~\ref{thm:main2}. In Section~\ref{sec:cycle} we introduce the notion of folded Artin complexes and prove Theorem~\ref{thm:single cycle} and Corollary~\ref{thm:main4}. A reader already with some experience with Artin groups and Coxeter groups might wish to start with Section~\ref{sec:relative}, and refer to previous sections when necessary.

\subsection{Acknowledgment}
We thank Mike Davis warmly for several long conversations on Coxeter groups acting on hyperbolic spaces and their Artin groups.
We thank Thomas Haettel warmly for several long discussions on finding an alternative algebraic proof of the bowtie free properties for the spherical Deligne complex of 4 strand braid group, without referring to curves in surfaces. We also thank Thomas Haettle for providing interesting comments on relative Artin complexes.
We thank Piotr Przytycki warmly for several inspiring  conversations on contractibility of Coxeter complexes and Artin complexes, and Helly graphs. We also thank Srivatsav Kunnawalkam Elayavalli, Camille Horbez, Damian Osajda for related discussions and comments. 

The author thanks the referee for a careful reading and very helpful comments.

The author thanks the Centre de Recherches Mathématiques at Montreal for hospitality. The author is partially supported by a Sloan fellowship.

\setcounter{tocdepth}{1}
\tableofcontents

\section{Discussion of proofs}
\label{sec:proof}
\subsection{An approach to the $K(\pi,1)$-conjecture based on relative Artin complexes}
\label{subsec:proof1}
\begin{definition}
	\label{def:Artin complex}
	The \emph{Artin complex} associated with an Artin group, introduced in \cite{CharneyDavis} and further studied in \cite{godelle2012k,cumplido2020parabolic}, is defined as follows. Let $\Gamma$ be a presentation graph with its collection of nodes $S$. For each $s\in S$, let $A_{\hat s}$ be the standard parabolic subgroup generated by $S\setminus\{s\}$. The \emph{Artin complex}, denoted by $\Delta_\Gamma$ (or $\Delta_S$), is the simplicial complex whose vertex set corresponds to left cosets of $\{A_{\hat s}\}_{s\in S}$. A collection of vertices span a simplex if the associated cosets have nonempty common intersection. By \cite[Proposition 4.5]{godelle2012k},  $\Delta_\Gamma$ is a flag complex. When $A_\Gamma$ is spherical, then $\Delta_\Gamma$ is also called the \emph{spherical Deligne complex}.
	
A vertex of $\Delta_S$ has \emph{type $\hat s=S\setminus \{s\}$}, if it corresponds to a coset of form $g A_{S\setminus \{s\}}$. The \emph{type} of each face of $\Delta_S$ is defined to be the subset of $S$ which is the intersection of the types of the vertices of the face. In particular, the type of each top-dimensional simplex is the empty set.
\end{definition}

%Artin complexes for Artin groups are analogue of a more classical object, called \emph{Coxeter complexes} for Coxeter groups. See Section~\ref{subsec:Coxeter}.  

When $A_S$ is spherical, the Artin complex $\Delta_S$ is a union of spheres called \emph{apartments}, such that each apartment is tiled by the fundamental domains of the action of the associated finite Coxeter group $W_S$ on the sphere. Hence $\Delta_S$ has a natural piecewise spherical metric, and the question of proving the $K(\pi,1)$-conjecture can be reduced to proving that, for each spherical Artin group $A_S$, $\Delta_S$ with such metric is $\CAT(1)$ \cite{CharneyDavis}. The latter question is rather challenging and widely open. Even the case when $A_S$ being the 4-strand braid group is already highly involved \cite{charney2004deligne}.

Instead, we explore the possibility that new cases of $K(\pi,1)$-conjecture can be proved by showing $\Delta_S$ satisfies certain combinatorial condition (i.e. labeled four cycle condition in Definition~\ref{def:labeled 4-cycle intro}) for each spherical $A_S$. The condition we consider will be
 reminiscence of the $\CAT(1)$ property, though they are easier to verify than the $\CAT(1)$ property. We will develop the framework of a new approach to $K(\pi,1)$-conjecture by exploiting combinatorial conditions on $\Delta_S$ for spherical $A_S$. The starting point is the following. 

%some combinatorial conditions do not fit immediately into the framework of \cite{CharneyDavis} to imply contractibility, they do fit into the framework as in Section~\ref{subsec:proof1}.

%Charney and Davis \cite{CharneyDavis} showed that the $K(\pi,1)$-conjecture reduces to the conjecture that for each spherical Artin group $A_S$, the 

\begin{thm}\cite[Theorem 3.1]{godelle2012k}
	\label{thm:combine}
	Suppose $S$ is not spherical. If the associated Artin complex $\Delta_\Gamma$ is contractible and each group in $\{A_{\hat s}\}_{s\in S}$ satisfies the $K(\pi,1)$-conjecture, then $A_S$ satisfies the $K(\pi,1)$-conjecture.
\end{thm}

\begin{proof}
	Let $\mathcal S$ be the collection of all proper standard parabolic subgroup of $A_\Gamma$. Then the assumption of the theorem and \cite[Corollary 2.4]{godelle2012k} imply that $\mathcal S$ is complete and $K(\pi,1)$ in the sense of explained before \cite[Theorem 3.1]{godelle2012k}. In this case the complex $\Phi(\Gamma,\mathcal S)$ in \cite[Theorem 3.1]{godelle2012k} is isomorphic to the barycentric subdivision of $\Delta_\Gamma$. Thus we are done. 
\end{proof}

If a vertex $v\in \Delta_\Gamma$ of type $\hat s$ satisfies that $\lk(v,\Delta_\Gamma)$ is contractible, then the link of each vertex of type $\hat s$ is contractible. So $\Delta_\Gamma$ can deformation retracts onto its subcomplex spanned by vertices whose type is not $\hat s$ (cf. Lemma~\ref{lem:dr}). The first step in our strategy is to keep preforming such kind of deformation retraction to reduce the dimension of the complex, until one reaches a ``core'' where such deformation retraction is not possible, which motivates the following:

\begin{definition}
	\label{def:rel}
	Let $A_S$ be an Artin group with presentation graph $\Gamma$ and Dynkin diagram $\Lambda$. Let $S'\subset S$. The \emph{$(S,S')$-relative Artin complex $\Delta_{S,S'}$} is defined to be the induced subcomplex of the Artin complex $\Delta_S$ of $A_S$ spanned by vertices of type $\hat s$ with $s\in S'$. In other words, vertices of $\Delta_{S,S'}$ correspond to left cosets of $\{A_{\hat s}\}_{s\in S'}$, and a collection of vertices span a simplex if the associated cosets have nonempty common intersection.
	
	Let $\Gamma'$ and $\Lambda'$ be the induced subgraphs of $\Gamma$ and $\Lambda$ spanned by $S'$. Then we will also refer a $(S,S')$-relative Artin complex as $(\Gamma,\Gamma')$-relative Artin complex or $(\Lambda,\Lambda')$-relative Artin complex, and denote it by $\Delta_{\Gamma,\Gamma'}$ or $\Delta_{\Lambda,\Lambda'}$.
\end{definition}

As we will see later (Lemma~\ref{lem:link}), if a vertex $x\in \Delta_{S,S'}$ is of type $\hat s$ for a node $s\in S'$, then $\lk(x,\Delta_{S,S'})$ is a copy of $\Delta_{S\setminus\{s\},S'\setminus\{s\}}$. So the discussion before Definition~\ref{def:rel} implies that if $\Delta_{S\setminus\{s\},S'\setminus\{s\}}$ is contractible, then $\Delta_{S,S'}$ deformation retracts onto $\Delta_{S,S'\setminus\{s\}}$. In other words, we know $\Delta_{S,S'}$ is contractible, if $\Delta_{S\setminus\{s\},S'\setminus\{s\}}$ and $\Delta_{S,S'\setminus\{s\}}$ are contractible. This implies the contractibility of $\Delta_{S,S'}$ reduces to contractibility of relative Artin complexes whose $S$ or $S'$ has smaller cardinality. Thus, in order to show $\Delta_S=\Delta_{S,S}$ is contractible, we could repeat this procedure to show $\Delta_{S,S}$ is homotopic to $\Delta_{S,S'}$ for $S'$ being as small as possible. At some point, we might reach a $\Delta_{S,S'}$ such that no vertex links are contractible, this is called a \emph{core} of $\Delta_S$. 

Recall that
an Artin group $A_S$ is \emph{almost spherical} if its presentation graph is complete and not spherical, but for each $s\in S$, $A_{\hat s}$ is spherical. A subset $T\subset S$ is \emph{almost spherical} if $A_T$ is almost spherical. Almost spherical Artin groups are classified: they are either affine, or one of the hyperbolic Lann\'er types. If $A_S$ is almost spherical, then the link of each vertex in $\Delta_S$ is a spherical Deligne complex, which is not contractible \cite{deligne}. More generally, we expect the links of vertices in $\Delta_{S,S'}$ to be not contractible whenever $S'$ is almost spherical. These relative Artin complexes serve as our ``cores''. 
We conjecture that the cores are contractible.

\begin{conj}
	\label{conj:contractible}
	Suppose $S'\subset S$ and $S'$ is almost spherical (but not spherical). Suppose the Dynkin diagrams for $S'$ and $S$ are connected and do not have edges labeled by $\infty$.
	Then $\Delta_{S,S'}$ is contractible.
\end{conj}

In fact, we conjecture that $\Delta_{S,S'}$ is contractible whenever $S'$ is not spherical.

Using the iterated deformation retraction argument as above, we show in Corollary~\ref{cor:contractible1} that 
if Conjecture~\ref{conj:contractible} is true, then the $K(\pi,1)$ holds true for any Artin groups. This strategy can also be adjusted if one only aims at showing that a subclass of Artin groups satisfy the $K(\pi,1)$-conjecture, see Proposition~\ref{prop:contractible} and Corollary~\ref{cor:contractible}.

\subsection{Contractibility of core}
\label{subsec:core}
The main challenge is to show $\Delta_{S,S'}$ is contractible. 
We believe $\Delta_{S,S'}$ in Conjecture~\ref{conj:contractible} can be metrized as non-positively curvature metric space in an appropriate sense, which would imply contractibility. There is a natural metric on $\Delta_{S,S'}$ defined as follows. The Coxeter group $W_{S'}$ associated with $A_{S'}$ has a natural action on either $\mathbb E^n$ or $\mathbb H^n$. 
We metrize $\Delta_{S,S'}$ such that each maximal simplex has the shape of the fundamental domain of this action of $W_{S'}$ and conjecture the resulting space is either $\CAT(0)$ or $\CAT(-1)$. However, the feasibility of this plan is not clear as verifying $\CAT(0)$ in high dimension situation is notoriously challenging. Instead we explore alternative ways of metrizing $\Delta_{S,S'}$, and alternative notions of non-positive curvature.

We say that an Artin group $A_\Gamma$ \emph{dominates} another Artin group $A_{\Gamma'}$ if there exists an isomorphism $f:\Gamma\to\Gamma'$ such that the label of $e$ is $\ge$ the label of $f(e)$ for each edge $e$ of $\Gamma$. An almost spherical Artin group is \emph{tight}, if it cannot dominate other almost spherical Artin groups. The list of tight almost spherical Artin groups are much smaller - they are either of affine type, or they belong to two extra examples, associated with reflection groups acting on $\mathbb H^3$ and $\mathbb H^4$. Now for the complex $\Delta_{S,S'}$ as above with $A_{S'}$ almost spherical, we find a tight almost spherical Artin group $A_{T'}$ which is dominated by $A_{S'}$. Then we metrize $\Delta_{S,S'}$ such that each maximal simplex has the ``shape'' of the fundamental domain $F$ of the action of $W_{T'}$ rather than $W_{S'}$. 

There is a natural ``shape'' of $F$ -- it is a convex simplex in $\mathbb E^n$, $\mathbb H^3$ or $\mathbb H^4$, hence inherits an metric from the ambient metric space. However, instead of using this canonical shape, there are other types of metric or combinatorial structure on $F$, which allows us to arranging different notions of non-positive curvature condition on $\Delta_{S,S'}$, more precisely metric with convex geodesic bicombing \cite{descombes2015convex} and weakly modular graphs \cite{chepoi1989classification,bandelt1996helly}. Some of these notions have the advantage of being easier to verify than $\CAT(0)$, especially in high dimensions. To show $\Delta_{S,S'}$ has some form of non-positive curvature property, one usually has to check certain conditions on the link of vertices. This is the heart of the matter. 

We now discuss a main obstacle of checking such link conditions through an example. Let $\Lambda'$ be a triangle subgraph in a bigger Dynkin diagram $\Lambda$. Then $\Delta_{\Lambda,\Lambda'}$ is a 2-dimensional complex made of triangles, and the above discussion leads us to metrize each triangle as flat equilateral triangles (despite that the label of edges of $\Lambda'$ could be $>3$) and prove $\Delta_{\Lambda,\Lambda'}$ is $\CAT(0)$\footnote{For contractibility of 2-dimensional core, we still use $\CAT(0)$ metric, as the difficulty of checking $\CAT(0)$ lies in dimension $\ge 3$.}, which amounts to prove that girth of the link of each vertex is $\ge 6$. 
 This deceptively sounds like a problem concerning 2-dimensional complexes, but it is not of a 2-dimensional nature. Showing such link condition is equivalent to proving a property regarding commutation of a pair of elements of certain form in the Artin group $A_\Lambda$ (Lemma~\ref{lem:triming}), which could be of arbitrarily high dimension. However, we do not even know how to solve the word problem of $A_\Lambda$.
\medskip

\noindent
\underline{An induction scheme for checking link conditions:}
To resolve this issue, note that the link of a vertex in the relative Artin complex $\Delta_{\Lambda,\Lambda'}$ is again a relative Artin complex of form $\Delta_{\Lambda_1,\Lambda'_1}$, where $\Lambda_1\subsetneq \Lambda$. If $\Lambda_1$ is spherical, then we are reduced to checking a particular property of spherical Artin groups. If $\Lambda_1$ is not spherical, then we are forced to check a property of a non-spherical Artin group $A_{\Lambda_1}$, which a priori we know nothing about.
The plan is to find a core subgraph $\Lambda''_1$ such that $\Delta_{\Lambda_1,\Lambda''_1}$ is non-positively curved, and use the non-positive curvature geometry of $\Delta_{\Lambda_1,\Lambda''_1}$ to study $\Delta_{\Lambda_1,\Lambda'_1}$, finally justifying the desired link condition. Note that a priori $\Lambda''_1$ and $\Lambda'_1$ might not be related, however, $A_\Gamma$ acts on both $\Delta_{\Lambda_1,\Lambda''_1}$ and $\Delta_{\Lambda_1,\Lambda'_1}$. In principle, checking a link condition on $\Delta_{\Lambda_1,\Lambda'_1}$ can be converted to a group theoretical property of $A_{\Lambda_1}$. This property can be understood by the action of $A_{\Lambda_1}$ on the non-positively curved space $\Delta_{\Lambda,\Lambda''_1}$, using geometry.
However, to justify $\Delta_{\Lambda_1,\Lambda''_1}$ is non-positively curved, we need to justify an appropriate link condition again, and we will repeat the same procedure as before. Note that $\Lambda_1$ has fewer vertices compared to $\Lambda$, and the complexity of the problem goes down after each iteration. This allows us to use an induction process to eventually reduce to a combinatorial condition on the Artin complexes $\Delta_S$ for $S$ spherical. 

Reversing the procedure, if we are able to identifying combinatorial conditions on $\Delta_S$ (with $S$ spherical) that are compatible with both the induction procedure and the link conditions of the form of non-positive curvature on certain types of core we chose to work with, then we will be able to prove new cases of $K(\pi,1)$-conjecture for Artin groups with certain types of cores. In this article, we show a simple combinatorial property about 4-cycles studied in metric graph theory, discussed in the next subsection, fits well into this procedure. More generally, we speculate that proving the full $K(\pi,1)$-conjecture might be reduced identifying and proving a strong enough combinatorial condition on $\Delta_S$ (with $S$ spherical) that is compatible with the induction and the non-positive geometry of all kinds of cores.

If the Dynkin diagram has extra symmetry, then there is a symmetrized version of relative Artin complexes, which is called folded Artin complexes, discussed in Section~\ref{sec:cycle}, and the above strategy can be symmetrized as well so to allow simplification.

%In this article, we show a very simple combinatorial property about 4-cycles studied in metric graph theory is compatible with the procedure described above. Moreover, we prove this property holds true for any $\Delta_S$ with $S$ spherical.

\subsection{Connection to metric graph theory}
\label{subsec:proof2}
Given a simplicial graph $\Gamma$, an \emph{induced} 4-cycle in $\Gamma$ is a $4$-cycle such that opposite vertices in the $4$-cycle are not adjacent in $\Gamma$. We recall that the following condition studied in \cite{chalopin2020weakly}, where it was called $(C_4,W_4)$-condition.

\begin{definition}
	\label{def:4wheel}
A simplicial graph $X$ satisfies the \emph{4-wheel} condition if for any induced 4-cycle in $X$, there exists a vertex $x\in X$ which is adjacent to each of the vertices in this 4-cycle. 
\end{definition}

This definition is connected to the $\CAT(1)$ property in the following way. Recall that a metric space is $\CAT(1)$ if and only if any loop of length $<2\pi$ admits a length non-increasing homotopy to the trivial loop \cite{bowditch1995notes}. If $X$ is the 1-skeleton of an Artin complex $\Delta_S$ such that $A_S$ is spherical and $\Delta_S$ is endowed with its canonical piecewise spherical metric \cite{CharneyDavis}, then $X$ satisfying $4$-wheel condition implies that any 4-cycle in $\Delta_S$ admits a length non-increasing homotopy. Indeed, the 4-wheel condition implies that any embedded 4-cycle in $\Delta_S$ can be filled by a combinatorial disk in the 2-skeleton, such that this disk either is made of two triangles (when the 4-cycle has two opposite vertices being adjacent in $\Delta_S$), or is made of four triangles with a central vertex in the interior of the disk. This combinatorial disk gives the desired length non-increasing homotopy. Thus Definition~\ref{def:4wheel} can be viewed a combinatorial version of $\CAT(1)$, where we only place conditions on loops in the 1-skeleton, and instead of requiring  filling disks for these loops to trace out a length decreasing homotopy, we only need  combinatorial filling disks in the 2-skeleton with a particular combinatorial feature.

In the metric setting of piecewise spherical complexes, we assign the simplices with certain spherical shapes and the edge lengths are usually different. However, this information is lost if we consider the graph which is the $1$-skeleton of $\Delta_S$. To maintain some information on the ``shape'' of simplices in a combinatorial fashion, we make use of the vertex labeling and consider the following variation of Definition~\ref{def:4wheel}.

\begin{definition}
	\label{def:labeled 4-cycle intro}
Let $\Lambda$ be a Dynkin diagram which is a tree, with its collection of nodes $S$. Let $X$ be the 1-skeleton of $\Delta_S$, such that (as explained above) each vertex is labeled by its type, $\hat s$, with  $s$ being a node of $\Lambda$.
We say $\Delta_S$ satisfies the \emph{labeled 4-cycle condition} if for any induced 4-cycle in $X$ with consecutive vertices being $\{x_i\}_{i=1}^4$, the following conditions are satisfied:
\begin{enumerate}
	\item there exists a vertex $x\in X$ adjacent to each of $x_i$;
	\item if $x_i$ is of type $\hat s_i$ for node $s_i\in \Lambda$ for $1\le i\le 4$, then the vertex $x\in X$ in the previous item can be chosen so it satisfies, in addition, that $x$ is of type $\hat s$ for node $s\in \Lambda'$ such that the node $s$ is contained the smallest subtree of $\Lambda'$ containing all the nodes $\{s_i\}_{i=1}^4$.
\end{enumerate}	
\end{definition}
We refer to Figure~\ref{fig:l4cycle} after Definition~\ref{def:labeled 4-cycle} for a pictorial illustration of this definition.

In the text we discuss two ways to connect Definition~\ref{def:labeled 4-cycle intro} to the $K(\pi,1)$-conjecture.  The first method, leading to Theorem~\ref{thm:main1} and Theorem~\ref{thm:main2}, relying on the induction procedure in Section~\ref{subsec:core}, which gives the following result.

%The induction procedure replies on an interesting connection between labeled 4-cycle condition and the intersection pattern of certain convex subcomplexes in non-positively curved spaces

\begin{prop}(=Proposition~\ref{prop:tree contractible})
	\label{prop:tree contractible intro}
	Suppose $\Lambda$ is a tree Dynkin diagram. Suppose there exists a nonempty collection $E$ of open edges with label $\ge 6$ such that for each component $\Lambda'$ of $\Lambda\setminus E$ the Artin complex $\Delta_{\Lambda'}$ satisfies the labeled 4-cycle condition. If each component of $\Lambda\setminus E$ satisfies the $K(\pi,1)$-conjecture, then $A_\Lambda$ satisfies the $K(\pi,1)$-conjecture.
\end{prop}

The proof of Proposition~\ref{prop:tree contractible intro} uses the induction scheme in Section~\ref{subsec:core}. We refer to Section~\ref{subsec:examples} for an informal discussion on how the induction works on two concrete examples.

Now we explain how Proposition~\ref{prop:tree contractible intro} can be used to prove Theorem~\ref{thm:main1}. Let $\Lambda$ be as Theorem~\ref{thm:main1}. Since $\Lambda$ is a tree, $\Lambda\setminus E$ decomposes into a union of components each of which is a tree. As each such component $\Lambda'$ is either spherical or locally reducible, the $K(\pi,1)$-conjecture of the associated Artin group is already known (\cite{deligne}, \cite{CharneyDavis}). By Proposition~\ref{prop:tree contractible intro}, to prove the $K(\pi,1)$-conjecture for $A_\Lambda$, it remains to show that whenever $\Lambda'$ is a tree Dynkin diagram such that the corresponding Artin group is spherical or locally reducible, then $\Delta_{\Lambda'}$ satisfies the labeled 4-cycle condition. This is exactly the content of the next result.

\begin{prop}(=Corollary~\ref{cor:wheel} and Corollary~\ref{cor:locally reducible})
	\label{prop:labeled 4 wheel intro}
Suppose $\Lambda$ is a tree which is either spherical or locally reducible. Then $\Delta_S$ satisfies the labeled 4-cycle condition.
\end{prop}

In the special case $\Lambda$ is a spherical Dynkin diagram of type $A_n$ or $B_n$, we can reformulate the labeled 4-cycle condition in several different ways, then the proposition can be deduced from \cite[Lemma 4.1]{charney2004deligne}, \cite{haettel2021lattices} and an unpublished work of Crisp and McCammond as explained in \cite{haettel2021lattices}. The previous proofs on type $A_n$ or $B_n$ are already highly non-trivial, and rely crucially on the combinatorics of curves on surfaces. Our method is different -- it relies on a connection between the labeled 4-cycle condition and the study of parabolic closure in \cite{cumplido2019parabolic}. We also give a more general algebraic criterion for labeled 4-cycle condition in Corollary~\ref{cor:algebraic}.

The second method, leading to Theorem~\ref{thm:single cycle} and Corollary~\ref{thm:main4}, 
comes from on the compatibility of the labeled 4-cycle condition and some link conditions leading to metrics with convex geodesic bicombing and/or weakly modular condition on cores that dominate an $\widetilde A_n$-type diagram.  This relies on work of Haettel \cite{haettel2022injective}, and Haettel and the author \cite{haettel2022lattices}.

\section{Preliminaries I}
\label{sec:pre1}
\subsection{Some terminology on complexes}

Let $X$ be a simplicial complex. We use $X^{(k)}$ to denote the $k$-skeleton of $X$.
A subcomplex $Y$ of $X$ is \emph{induced} if for any simplex of $X$ with its vertex set inside $Y$ is entirely contained in $Y$.  Given a vertex $x\in X$, the \emph{link} of $x$ in $X$, denoted by $\lk(x,X)$, is the induced subcomplex spanned by all vertices of $X$ that are adjacent to $x$. The link can also be thought as an $\eps$-sphere around $x$, with the induced simplicial structure from the ambient space. The \emph{star} of $x$ in $X$, denoted by $\st(x,X)$, is the union of simplices in $X$ containing $x$. The \emph{open star} of $x$ in $X$, denoted by $\st^o(x,X)$, is the union of the interior of all simplices in $X$ which contain $x$.

We will be working with piecewise Euclidean simplicial complexes, namely each simplex has the metric of a convex simplex in some Euclidean space, and the simplices are glued together isometrically along their faces. In this case, for any $x\in X$ (not necessarily a vertex), we can define $\lk(x,X)$ as a metric space obtained by gluing together the unit spheres in the tangent spaces of $x$ in $\Delta$, with $\Delta$ ranging over all simplices that contains $x$. For example, when $x$ is a vertex and $X$ is 2-dimensional, $\lk(x,X)$ is a metric graph, whose edges corresponding to 2-simplices containing $x$, and the length of the edges are angles of these 2-simplices at $x$. We refer to \cite[Chapter I.7]{BridsonHaefliger1999} for more background on piecewise Euclidean simplicial complexes.

We refer to \cite{BridsonHaefliger1999} for background on $\CAT(0)$ spaces. Let $X$ be a $\CAT(0)$ piecewise Euclidean simplicial complex with $x\in X$. Note that any $y\neq x$ gives rise to a point in $\lk(x,X)$, by considering the geodesic from $x$ to $y$. This gives a map $\log_x:X\setminus\{x\}\to \lk(x,X)$. For any $y_1,y_2\neq x$, we define the angle $$\angle_x(y_1,y_2)=\min\{d(\log_x(y_1),\log_x(y_2)),\pi\}.$$

\subsection{Artin groups and Coxeter groups}
\label{subsec:Artin}
Let $\Gamma$ be a finite simplicial graph with each edge labeled by an integer $\ge 2$. The \emph{Artin group with presentation graph $\Gamma$}, denoted $A_\Gamma$, is a group whose generators are in one to one correspondence with vertices of $\Gamma$, and there is a relation of the form $aba\cdots=bab\cdots$ with both sides being alternating words of length $m$
whenever two vertices $a$ and $b$ are connected by an edge labeled by $m$.  The \emph{Coxeter group with presentation graph $\Gamma$}, denoted $W_\Gamma$, has the same generating sets and the same relators as the Artin group, with extra relations $v^2=1$ for each vertex $v\in\Gamma$. There is a homomorphism $A_\Gamma\to W_\Gamma$, whose kernel is called the \emph{pure Artin group}, and is denoted by $PA_\Gamma$. Recall that we will refer to vertices of $\Gamma$ as \emph{nodes}.

There is a set theoretic section $\mathfrak{s}:W_\Gamma\to A_\Gamma$ to the quotient map $A_\Gamma\to W_\Gamma$ defined as follows. Take $g\in W_\Gamma$ and we write $g$ as a word $w$ in the free monoid $S^*$ generated by $S$. We require $w$ is \emph{reduced}, i.e. it is a minimal length word in $S^*$ representing $g$. Then $\mathfrak{s}(g)$ is defined to be $w$, viewed as an element in $A_\Gamma$. This map is well-defined, since by Tits' solution to the word problem for Coxeter groups, two different reduced words in $S^*$ representing the same element of $W_\Gamma$ differ by a finite sequence of relators in $A_\Gamma$ (see \cite{matsumoto1964generateurs}).

Let $S$ be the set of nodes of $\Gamma$. We will also write $A_\Gamma$ as $A_S$. Recall that for any $S'\subset S$ generates a subgroup of $A_\Gamma$ which is also an Artin group, whose presentation graph is the induced subgraph $\Gamma_{S'}$ of $\Gamma$ spanned by $S'$ \cite{lek}. This subgroup is called a \emph{standard parabolic subgroup} of type $S'$. A \emph{parabolic subgroup} of $A_\Gamma$ of type $S'$ is a conjugate of a standard parabolic subgroup of type $S'$. A parabolic subgroup of $A_S$ is \emph{reducible} if its type $S'$ admits a disjoint non-trivial decomposition $S'_1\sqcup S'_2$ such that each element in $S'_1$ commutes with every element in $S'_2$. If such decomposition does not exist, then the parabolic subgroup is \emph{irreducible}.

\begin{lem}
	\label{lem:flag}
	(\cite{godelle2012k})
Suppose $\{g_iA_{S_i}\}_{i=1}^n$ be a collection of left cosets of standard parabolic subgroups in $A_S$ (with $S_i\subset S$). If these left cosets have pairwise non-empty intersection, then they have nonempty common intersection.
\end{lem}

\begin{proof}
The case $n=3$ is \cite[Lemma 4.7]{godelle2012k}. The arbitrary $n$ case follows from the $n=3$ case, and the same induction argument in the proof of \cite[Proposition 4.5]{godelle2012k}
\end{proof}	

\begin{thm}
	\label{thm:parabolic}
	(\cite{blufstein2023parabolic})
Let $S_1,S_2\subset S$ and $g\in A_S$ such that $gA_{S_1}g^{-1}\subset A_{S_2}$. Then there exists $h\in A_{S_2}$ and $S_1'\subset S_2$ such that $gA_{S_1}g^{-1}=hA_{S'_1}h^{-1}$.
\end{thm}

Now we assume $A_S$ is spherical.
For each standard parabolic subgroup $A_{S'}$ of $A_S$, let $\delta_{S'}$ be its Garside element, and let $c_{S'}$ be the smallest positive power of $\delta_{S'}$ that is contained in the center of $A_{S'}$ (actually either $c_{S'}=\delta_{S'}$, or $c_{S'}=\delta^2_{S'}$). If $ \mathsf{P}=gA_{S'}g^{-1}$ is a parabolic subgroup of $A_S$, we define $c_\mathsf{P}=gc_{S'}g^{-1}$. 
The following is a consequence of \cite[Lemma 33]{cumplido2019minimal} and \cite[Proposition 2.2]{godelle2003normalisateur}.
\begin{lem}
	\label{lem:contain}
	Suppose $A_S$ is spherical and suppose $\mathsf{P}$ is a parabolic subgroup of $A_S$. Then
	\begin{enumerate}
		\item for $X,Y\subset S$, $g^{-1}A_Xg=A_Y$ if and only if $g^{-1}c_Xg=c_Y$, in particular, if $gA_{S'}g^{-1}=g_1A_{S'}g^{-1}_1$, then $gc_{S'}g^{-1}=g_1c_{S'}g^{-1}_1$, hence $c_\mathsf{P}$ is well-defined;
		\item $gc_\mathsf{P}g^{-1}\in A_{S'}$ for some $S'\subset S$ if and only if $g\mathsf{P}g^{-1}\subset A_{S'}$.
	\end{enumerate}		
\end{lem}

\subsection{Davis complexes}
\label{subsec:complex}
By a cell, we always mean a closed cell unless otherwise specified.

\begin{definition}[Davis complex]
	Given a Coxeter group $W_\Gamma$, let $\mathcal{P}$ be the poset of left cosets of spherical standard parabolic subgroups in $W_\Gamma$ (with respect to inclusion) and let $b\Si_\Gamma$ be the geometric realization of this poset (i.e.\ $b\Si_\Gamma$ is a simplicial complex whose simplices correspond to chains in $\mathcal{P}$). Now we modify the cell structure on $b\Si_\Gamma$ to define a new complex $\Sigma_\Gamma$, called the \emph{Davis complex}. The cells in $\Sigma_\Gamma$ are induced subcomplexes of $b\Si_\Gamma$ spanned by a given vertex $v$ and all other vertices which are $\le v$ (note that vertices of $b\Si_\Gamma$ correspond to elements in $\mathcal{P}$, hence inherit the partial order).
\end{definition}

Suppose $W_\Gamma$ is finite with $n$ generators. Then there is a canonical faithful orthogonal action of $W_\Gamma$ on the Euclidean space $\mathbb E^n$. Take a point in $\mathbb E^n$ with trivial stabilizer, then the convex hull of the orbit of this point under the $W_\Gamma$ action (with its natural cell structure) is isomorphic to $\Sigma_\Gamma$. In such case, we call $\Sigma_\Gamma$ a \emph{Coxeter cell}. In general Davis complex is a union of Coxeter cells.

The 1-skeleton $\Sigma^1_\Gamma$ of $\Sigma_\Gamma$ is the unoriented Cayley graph of $W_\Gamma$ (i.e.\ we start with the usual Cayley graph and identify the double edges arising from $s^2_i$ as single edges), and $\Sigma_\Gamma$ can be constructed from the unoriented Cayley graph by filling in Coxeter cells. Each edge of $\Sigma_\Gamma$ is labeled by a generator of $W_\Gamma$. We endow $\Sigma^1_\Gamma$ with the path metric with edge length $1$. 

A \emph{reflection} of $W_\Gamma$ is a conjugate of one of its generators. There is a natural action of $W_\Gamma$ on $b\Si_\Ga$ by simplicial automorphisms. The fix point set $H$ of a reflection $r$ is a subcomplex of $b\Si_\Ga$ (also viewed as a subset of $\Si_\Ga$), which is called a \emph{wall}. Then $\Si_\Ga\setminus H$ has exactly two connected components, exchanged by the action of $r$. Two vertices of $\Sigma_\Ga$ are \emph{separated} by a wall $H$ if they are in different connected components of $\Sigma_\Gamma\setminus H$. The distance between any two vertices with respect to the path metric on $\Sigma^1_\Gamma$ is the number of walls separating these two vertices. A wall is \emph{dual} to an edge if the wall and the edge have nonempty intersection. Two edges are \emph{parallel} if they are dual to the same wall.

If $\Gamma'\subset \Gamma$ is an induced subgraph, then $W_{\Gamma'}\to W_{\Gamma}$ induces an embedding $\Sigma_{\Gamma'}\to \Sigma_{\Gamma}$. The image of this embedding and its left translations are \emph{standard subcomplexes} of type $\Gamma'$. There is a correspondence between standard subcomplexes of type $\Gamma'$ in $\Sigma_{\Gamma}$ and left cosets of $W_{\Gamma'}$ in $W_{\Gamma}$.

Lemma~\ref{lem:gate}, Lemma~\ref{lem:more gate} and Lemma~\ref{lem:pair gate} below are standard, see e.g.\ \cite{bourbaki2002lie} or \cite{davis2012geometry}, also \cite{dress1987gated}. Let $d$ be the path metric on the 1-skeleton of $\Si_\Ga$, with each edge having length 1.
Lemma~\ref{lem:gate} describes nearest point projection into the vertex set of a standard subcomplex.
\begin{lem}
	\label{lem:gate}
	Let $F$ be a standard subcomplex of $\Si_\Ga$ and let $x\in \Si_\Gamma$ be a vertex. Then the following hold.
	\begin{enumerate}
		\item For two vertices $x_1,x_2\in\mathcal R$, the vertex set of any geodesic in $\Si^1_\Gamma$ joining $v_1$ and $v_2$ is inside $F$. Moreover, there exists a unique vertex $x_F\in F$ such that $d(x,x_F)\le d(x,y)$ for any vertex $y\in F$, where $d$ denotes the path metric on the 1-skeleton of $\Sigma_\Gamma$. The vertex $x_F$ is called the \emph{projection} of $x$ to $F$, and is denoted $\prj_F(x)$. 
			\item For any vertex $y\in F$, there exists a shortest edge path $\omega$ in $\Si^1_\Ga$ from $x$ to $y$ so that $\omega$ passes through $x_F$ and so that the segment of $\omega$ between $x_F$ and $y$ is contained in $F$.
		\item Let $\mathcal W(F)$ be the collection of walls in $\Si_\Ga$ dual to an edge of $F$. Then $x_F$ can be characterized as the unique vertex in $F$ such that	no element in $\mathcal W(F)$ separates $x$ from $x_F$.
		\end{enumerate}
\end{lem}

The next lemma describes two more properties of nearest point projection.
\begin{lem}
	\label{lem:more gate}
The projection map from the vertex set of $\Si_\Ga$ to the vertex set of a standard subcomplex of $\Si_\Ga$ in Lemma~\ref{lem:gate} (1) satisfies the following properties.
	\begin{enumerate}
		\item 	If two standard subcomplexes $F_1\cap F_2\neq\emptyset$, then for any $x\in F_2$, $\prj_{F_1}(x)\in F_1\cap F_2$.
		\item Let $e$ be an edge with its endpoints $x$ and $y$. Then $e$ is parallel to an edge in $F$ iff $d(x,F)=d(y,F)$ iff the wall dual to $e$ has nonempty intersection with $F$ iff $\prj_F(x)\neq \prj_F(y)$. In this case, $\prj_F(x)$ and $\prj_F(y)$ are two distinct vertices in an edge. If $d(x,F)=d(y,F)$, then $\prj_F(x)=\prj_F(y)$.
	\end{enumerate}
\end{lem}

Now we consider the properties of nearest point sets between two standard subcomplexes.
\begin{lem}
	\label{lem:pair gate}
	Let $E$ and $F$ be standard subcomplexes of $\Si_\Ga$. Define $$X=\{x\in \vertex E\mid d(x,\vertex F)=d(\vertex E,\vertex F)\}$$ and $$Y=\{y\in \vertex F\mid d(y,\vertex E)=d(\vertex E,\vertex F)\}.$$ Then 
	\begin{enumerate}
		\item there are standard subcomplexes $E'\subset E$, $F'\subset F$ such that $X=\vertex E'$ and $Y=\vertex F'$;
		\item $\prj_E(\vertex F)=X$  and $\prj_F(\vertex E)=Y$;
		\item $\prj_E|_{\vertex F'}$ and $\prj_F|_{\vertex E'}$ gives a bijection and its inverse between $\vertex E'$ and $\vertex F'$;
		\item if $\mathcal W(E')$ is the collection of all walls dual to an edge in $E'$, then $\mathcal W(E')=\mathcal W(F')=\mathcal W(E)\cap \mathcal W(F)$;
		\item if $\mathcal W(E)=\mathcal W(F)$, then $E=E'$ and $F=F'$.
	\end{enumerate}
	In the situation of this lemma we will write $E'=\prj_E(F)$.
\end{lem}

In the situation of Lemma~\ref{lem:pair gate} (5), we will say $E$ and $F$ are \emph{parallel}. In this case, the bijection between $\vertex E$ and $\vertex F$ given by $\prj_E|_{\vertex F}$ and $\prj_F|_{\vertex E}$ are called \emph{parallel translation} between $E$ and $F$.

\begin{definition}
	\label{def:intersection poset}
Let $\cq_\Ga$ be the collection of walls in $\Si_\Ga$ and their non-empty intersections. For each element $B\in \cq_\Ga$, there is a parallel class of faces in $\Si_\Ga$ such that each face in the class intersects $B$ in exactly one point. These faces are \emph{dual} to $B$. There is a one to one correspondence between elements in $\cq_\Ga$ and parallel classes of faces of dimension $\ge 1$ in $\Si_\Ga$. Not that when $\Gamma$ is spherical, then $\ca_Q$ is the set of subspaces in the associated reflection hyperplane arrangement.
\end{definition}

\begin{definition}
	\label{def:projection1}
	Let $F$ be a standard subcomplex of $\Si_\Ga$. Lemma~\ref{lem:gate} gives a map $\Pi_F:\vertex\Si_\Ga\to\vertex F$ which extends to a retraction $\Pi_F:\Si_\Ga\to F$ as follows. Note that for each face $E$ of $\Si_\Ga$, $\pi(\vertex E)$ is the vertex set of a face $E'\subset F$. Then we extends $\pi$ to a map $\pi'$ from the vertex set of $b\Si_\Ga$ to the vertex set of $bF$, by sending the barycenter of $E$ to the barycenter of $E'$. As $\pi'$ map vertices in a simplex to vertices in a simplex, it extends linearly to a map $\Pi_F:b\Si_\Ga\cong \Si_\Ga\to bF\cong F$.
\end{definition}

\subsection{Oriented Davis complexes and Salvetti complexes}
\label{subsec:Sal}
Let $\mathcal P$ be the poset of faces of $\Si_\Ga$ (under containment), and let $V$ be the vertex set of $\Si_\Ga$. We now define the \emph{oriented Davis complex} $\widehat\Si_\Ga$ as follows.
Consider the set of pairs $(F,v)\in \cp \times V$.  Define  an equivalence relation $\sim$ on this set by $$(F,v)\sim (F,v') \iff F=F' \text{\ and\ } \prj_F(v') = \prj_F(v).$$
Denote the equivalence class of $(F,v')$ by $[F,v']$ and let $\ce(\ca)$ be  the set of equivalence classes.   Note  that each equivalence class $[F,v']$ contains a unique representative of the form $(F,v)$, with $v\in \vertex F$.  The \emph{oriented Davis complex} $\widehat\Si_\Ga$ is defined as the regular CW complex given by taking  $\Si_\Ga\times V$ (i.e., a disjoint union of copies of $\Si_\Ga$) and then identifying faces $F\times v$ and $F\times v'$ whenever $[F,v]=[F,v']$, i.e.,
\begin{equation}
\widehat\Si_\Ga=( \Si_\Ga\times V)/ \sim \ .
\end{equation}
For example, for each edge $F$ of $\Si_\Ga$ with endpoints $v_0$ and $v_1$, we get two $1$-cells $[F,v_0]$ and $[F,v_1]$ of $\widehat\Si_\Ga$ glued together along their endpoints $[v_0,v_0]$ and $[v_1,v_1]$.  So, the $0$-skeleton of $\widehat\Si_\Ga$ is equal to the $0$-skeleton of $\widehat\Si_\Ga$ while its $1$-skeleton is formed from the $1$-skeleton of $\widehat\Si_\Ga$ by doubling each edge.  
There is a natural map $\pi:\widehat\Si_\Ga\to\Si_\Ga$ defined by ignoring the second coordinate. 

The definition of oriented Davis complex traced back to work of Salvetti \cite{s87}, so it is also called Salvetti complex by many other authors. The naming ``oriented Davis complex'' comes from an article of J. McCammond \cite{mccammond2017mysterious}, clarifying the relation between Salvetti's work and Davis complex, which suits better for our latter discussion. We will reserve the term ``Salvetti complex'' for a quotient of the oriented Davis complex.

	Each edge of $\od_\Ga$ has a natural orientation, namely, if $F=\{v_0,v_1\}$ is an edge of $\Si_\Ga$, then $[F,v_0]$ is oriented so that $[v_0,v_0]$ is its initial vertex and $[v_1,v_1]$ is its terminal vertex.  An edge path in the $\od_\Ga$ is \emph{positive} if each of its edges is positively oriented. (Positive paths are related to the Deligne groupoid defined in Section~\ref{subsec:Deligne} below.)

\begin{definition}
	\label{def:label}
As the 1-skeleton of $\Si_\Ga$ can be identified with the unoriented Cayley graph of $W_\Ga$, each edge of $\Si_\Ga$ is labeled by an element in the generating set $S$.
	We pull back the edge labeling from $\Si_\Ga$ to $\od_\Ga$ via the map $\pi:\od_\Ga\to \Si_\Ga$. For a subset $E$ in $\Sigma$ or $\Gamma(\ca)$, we define $\supp(E)$ to be the collection of labels of edges in $E$. Let $u$ be an edge path in $\Sii$ or a positive path in $\od_\Ga$. Then reading off labels of edges of $u$ gives a word in the free monoid generated by $S$, which we denote by $\w(u)$. If $u$ is an arbitrary edge path in $\od_\Ga$, then when an edge travels opposite to its orientation, we read off the inverse of the associated label. Then $\w(u)$ gives a word in the free group on $S$. 
\end{definition}

For each subcomplex $Y$ of $\Si_\Ga$, we write $\widehat Y=p^{-1}(Y)$ and call $\widehat Y$ the subcomplex of $\od_\Ga$ associated with $Y$.
A \emph{standard subcomplex} of $\widehat\Si_\Ga$ is a subcomplex of $\od_\Ga$ associated with a standard subcomplex of $\Si_\Ga$. In other words, if $F\subset \Si_\Ga$ is a standard subcomplex, then $\widehat F$ is the union of faces of form $E\times v$ in $\widehat \Si_\Ga$ with $E\subset F$ and $v$ ranging over vertices in $\Si_\Ga$.

\begin{lem}
	\label{lem:compactible}
	Let $E$ be a face of $\Si_\Ga$ and let $F$ be a standard subcomplex of $\Si_\Ga$.
	If $[E,v_1]=[E,v_2]$, then $[\prj_F(E),v_1]=[\prj_F(E),v_2]$.
\end{lem}

\begin{proof}
	Note that $[E,v_1]=[E,v_2]$ if and only if for each wall $H$ with $H\cap E\neq\emptyset$, $v_1$ and $v_2$ are on the same side of $H$. Thus for each wall $H$ dual to $\prj_F(E)$, $v_1$ and $v_2$ are on the same side of $H$. Now the lemma follows.
\end{proof}

We will be make use of the following important construction of Godelle and Pairs in \cite{godelle2012k}.
\begin{definition}
	\label{def:retraction}
	Let $F$ be a face in $\Si_\Ga$. Then there is a retraction map $\Pi_{\widehat F}:\widehat\Si_\Ga\to \widehat F$ defined as follows. Recall that $\widehat\Si_\Ga=( \Si_\Ga\times V)/ \sim$. For each $v\in V$, let $(\Si_\Ga)_v$ be the union of all faces in $\widehat\Si_\Ga$ of form $E\times v$ with $E$ ranging over faces of $\Si_\Ga$. By Definition~\ref{def:projection1}, there is a retraction $(\Pi_F)_v:(\Si_\Ga)_v\to F\times v$ for each $v\in V$. It follows from Lemma~\ref{lem:compactible} that these maps $\{(\Pi_F)_v\}_{v\in V}$ are compatible in the intersection of their domains. Thus they fit together to define a retraction $\Pi_{\widehat F}:\widehat\Si_\Ga\to \widehat F$.
\end{definition}
The following lemma is a direct consequence of the definition.
\begin{lem}
	\label{lem:retraction property}
	Take standard subcomplexes $E,F\subset \Si_\Ga$. Then $\Pi_{\widehat F}(\widehat E)=\widehat{\Pi_F(E)}$.
\end{lem}

The action of $W_\Gamma$ on $\Si_\Ga$ gives a free action of $W_\Ga$ on $\widehat\Si_\Ga$, whose quotient complex is denoted by $\s_\Ga$. The fundamental group of $\s_\Ga$ is $A_\Ga$ and its 2-skeleton is the presentation complex of $A_\Ga$, see e.g. \cite{paris2012k}. The fundamental group of $\od_\Ga$ is the \emph{pure Artin subgroup} of $A_\Ga$, as $\od_\Ga$ is a regular cover of $\s_\Ga$ corresponding to the kernel of $A_\Ga\to W_\Ga$.

\subsection{Coxeter complexes and Artin complexes}
\label{subsec:Coxeter}
Let $\Gamma$ be a presentation graph.
The definition of a Coxeter complex $\bC_\Gamma$ of a Coxeter group $W_\Gamma$ is almost identical to Artin complex (cf. Section~\ref{def:Artin complex}), except one replaces $A_{\hat s}$ by $W_{\hat s}$, which is the standard parabolic subgroup of $W_\Gamma$ generated by $S\setminus \{s\}$. We define \emph{types} of vertices and faces of $\bC_\Gamma$ in a similar way as in Definition~\ref{def:Artin complex}.

\begin{definition}
	\label{def:barpi}
	Let $\mathfrak{s}:W_\Gamma\to A_\Gamma$ be the section of $A_\Gamma\to W_\Gamma$ described before. Then $\mathfrak{s}$ maps a coset $gW_{\hat s}$ inside $g A_{\hat s}$. This induces a simplicial embedding $\mathfrak{s}':\bC_\Gamma\to \Delta_\Gamma$. The image of $\mathfrak{s}'$ and their translations under the group action $A_\Gamma\act \Delta_\Gamma$ are called \emph{apartments} of $\Delta_\Gamma$. The quotient complex of $\Delta_\Gamma$ under the action of the pure Artin group is isomorphic to $\bC_\Gamma$. This quotient map $\bar\pi:\Delta_\Gamma\to \bC_\Gamma$ restricts to an isomorphism on each apartment of $\Delta_\Gamma$. 
\end{definition}

\begin{definition}
	\label{def:associated subcomplex}
For each vertex $x\in \bC_\Gamma$, let $gW_{S\setminus\{s\}}$ be the associated left coset in $W_\Gamma$. Recall that vertices of the Davis complex $\Si_\Gamma$ can be identified with $W_\Gamma$. The \emph{standard subcomplex of $\Si_\Gamma$ associated with the vertex $x\in \bC_\Gamma$} is defined to be the standard subcomplex of $\Si_\Gamma$ spanned by vertices in $gW_{S\setminus\{s\}}$.
\end{definition}

When $W_\Gamma$ is a finite group, $\bC_\Gamma$ is homeomorphic to a sphere. More precisely, consider the canonical representation $\rho: W_\Gamma\to GL(n,\mathbb R)$ and let $\ca$ be the collection of all reflection walls in $\mathbb R^n$. Then elements in $\ca$ cuts the unit sphere of $\mathbb R^n$ into a simplicial complex, which is isomorphic to $\bC_\Gamma$. From this, we know that $\bC_\Gamma$ and $\Si_\Ga$ are dual complexes of each other.

%(DO WE NEED THIS?)Now we can also describe the Artin complex $\Delta_\Gamma$ in terms of $\od_\Ga$ as follows. Let $X$ be the universal cover of $\od_\Ga$. A \emph{lift} of a standard subcomplex in $\od_\Ga$ is a connected component of the inverse image of this subcomplex under the map $X\to\od_\Ga$. Vertices of $\Delta_\Gamma$ are in 1-1 correspondence with lifts standard subcomplexes of $\od_\Ga$ of type $\hat s$ for some $s\in S$. A collection of vertices span a simplex if their associated lifts have non-trivial common intersection. 

\subsection{Singular disk diagrams and combinatorial Gauss-Bonnet}
\label{ss:disk}
We quickly review disk diagrams and the Gauss-Bonnet formula. We will only need the Gauss-Bonnet formula for piecewise flat simplicial complexes, though it is more natural to set it up in the more general context of combinatorial CW complexes. We refer to \cite[Definition 2.1]{mccammond2002fans} for the definition of combinatorial CW complexes and combinatorial maps between them. We recall from \cite[Definition~2.6]{mccammond2002fans} that a \emph{(singular) disk diagram} $D$ is a finite contractible 2-dimensional combinatorial CW complex with a fixed embedding in the plane $\mathbb R^2$. A \emph{boundary cycle} of $D$ is a combinatorial map from a polygon $P$ to $D$ whose image is an edge-path in the graph $D^1$ (i.e. the 1-skeleton of $D$) corresponding to going around $D$ once in the clockwise direction along the boundary of the unbounded complementary region $\mathbb R^2\setminus D$ (see also \cite[p.~150]{LSbook}).

Let $P \to X$ be a closed null-homotopic edge path in a 2-dimensional combinatorial CW complex $X$. A \emph{singular disk diagram in $X$ for $P$}
is a singular disk diagram $D$ together with a map $D\to X$ such that the closed path
$P\to X$ factors as $P\to D\to X$ where $P\to D$ is a boundary cycle of $D$.
It is a theorem of Van Kampen that every null-homotopic closed edge path $P\to X$ is the boundary cycle of a singular disk diagram $D\to X$; moreover, we can assume this singular disk diagram is \emph{reduced}, i.e. $D-D^{(0)}\to X$ is an immersion, see \cite[V.2.1]{LSbook} or \cite[Lemma~2.17]{mccammond2002fans}. We caution the reader that $D$ is usually not homeomorphic to a 2-dimensional disk (thus the name ``singular'' disk diagram), e.g.\ it is not if $P\to X$ is not an embedding. 
Also even if $P\to X$ is an embedding, there might not exist a singular disk diagram for $P$ such that $D\to X$ is an embedding. If $X$ has a piecewise flat metric, then we equip the singular disk diagram $D$ with the natural piecewise flat metric induced by $D\to X$.

We will use the following version of the Gauss-Bonnet formula for a singular disk diagram $D$ which is a special case of the Gauss-Bonnet formula in \cite[Section 2]{ballmann1996nonpositively}. We assume that $D$ has a piecewise flat structure. For a vertex $v\in D^{(0)}$, let $\chi(v)$ be the Euler characteristic of $\lk(v,D)$. Recall that the length of an edge of $\lk(v,D)$ is the interior angle at $v$ of the 2-cell of $D$ corresponding to this edge. Let $\alpha(v)$ be the sum of the lengths of all edges in $\lk(v,D)$.
Define $\kappa(v)=(2-\chi(v))\pi-\alpha(v)$. Then
\begin{equation}
	\label{eq:GB}
	\sum_{v\in D^{(0)}}\kappa(v)=2\pi.
\end{equation}

Now suppose $X$ is a piecewise hyperbolic complex (e.g. each 2-cell of $X$ has the metric of a convex polygon in the hyperbolic plane). 
Given a $2$-cell $C\in D^{(2)}$, we denote by $\mathrm{Area}(C)$ the area of $C$. Then 
\begin{equation}
	\label{eq:GB1}
	\sum_{v\in D^{(0)}}\kappa(v)-\sum_{C\in D^{(2)}}\mathrm{Area}(C)=2\pi.
\end{equation}

\section{Preliminaries II}
\label{sec:pre2}
\subsection{Posets and lattices}
Let $P$ be a poset, i.e. a partially ordered set.
Let $S\subset P$. An \emph{upper bound} (resp. lower bound) for $S$ is an element $x\in P$ such that $s\le x$ (resp. $s\ge x$) for any $s\in S$. The \emph{join} of $S$ is an upper bound $x$ of $S$ such that $x\le y$ for any other upper bound $y$ of $S$. The \emph{meet} of $S$ is a lower bound of $x$ of $S$ such that $x\ge y$ for any other lower bound $y$ of $S$. We will write $x\vee y$ for the join of two elements $x$ and $y$, and $x\wedge y$ for the meet of two elements (if the join or the meet exists). A poset is \emph{bounded} if it has a maximal element and a minimal element. We say $P$ is \emph{lattice} if $P$ is a poset and any two elements in $P$ have a join and have a meet. 

A chain in $P$ is any totally ordered
subset, subsets of chains are subchains and a maximal chain is one that is not a proper subchain of any other chain. A poset has rank $n$ if it is bounded, every chain is
a subchain of a maximal chain and all maximal chains have length $n$. For $a,b\in P$ with $a\le b$, the \emph{interval}  between $a$ and $b$, denoted by $[a,b]$, is the collection of all elements $x$ of $P$ such that $a\le x$ and $x\le b$. The poset $P$ is \emph{graded} if every interval in $P$ has a rank. For $a,b\in P$, we say $b$ \emph{covers} $a$, if $b>a$ and $[a,b]=\{a,b\}$.

Recall that a \emph{bowtie} in a poset $P$ is a subset $\{x_1,x_2,y_1,y_2\}$ such that $y_i$ covers $x_j$ for $1\le i,j\le 2$. The name comes from that if we draw $y_1,y_2$ above $x_1,x_2$ in the Hasse diagram, then we obtain a bowtie shaped configuration.
\begin{definition}
	\label{def:bowtie free0}
	Let $P$ be a poset. We say $P$ is \emph{bowtie free} if any subset $\{x_1,x_2,y_1,y_2\}\subset P$ made of mutually distinct elements with $x_i<y_j$ for $i,j\in\{1,2\}$, there exists $z\in P$ such that $x_i\le z\le y_j$ for any $i,j\in\{1,2\}$.
\end{definition}

The interest of the bowtie free condition lies in the following observation.

\begin{lem}(\cite[Proposition 1.5]{brady2010braids})
	\label{lem:posets}
	If $P$ is a bowtie free graded poset, then any pair of elements in $P$ with a lower bound has a join, and any pair of elements in $P$ with a upper bound has a meet.
	
	Let $P$ be a bounded graded poset. Then $P$ is lattice if it is bowtie free. 
\end{lem}

\subsection{Deligne groupoids}
\label{subsec:Deligne}
Let $\od_\Ga$ and $\Si_\Ga$ be the oriented Davis complex and the Davis complex introduced before, with the map $\pi:\od_\Ga\to \Si_\Ga$. Given an edge $e$ of $\od_\Ga$, write $\bar e=\pi(e)$. We will say an edge $e$ of $\od_\Ga$ is dual to a wall $H$ of $\Si_\Ga$ if $\bar e$ is dual to $H$.

Since the edges of $\od_\Ga$ are directed, each edge $a$ has a \emph{source}, denoted $\sfs(a)$, and a \emph{target}, denoted $\sft(a)$. Introduce a formal inverse of $a$, denoted $a^{-1}$. It can be thought as traveling  the same edge but in the opposite direction. Thus, $\sft(a^{-1})=\sfs(a)$ and $\sfs(a^{-1})=\sft(a)$. 

A \emph{path} of $\od_\Ga$ is an expression $g=a^{\eps_1}_1a^{\eps_2}_2\cdots a^{\eps_n}_n$ where $a_i$ being an edge of $\od_\Ga$, $\eps_i\in\{\pm 1\}$ and $\sft(a^{\eps_i}_i)=\sfs(a^{\eps_{i+1}}_{i+1})$. Define $\sfs(g)=\sfs(a^{\eps_1}_1)$ and $\sft(g)=\sft(a^{\eps_n}_n)$. The \emph{length} of $g$ is $n$. A vertex is a path of length $0$. The path $g$ is \emph{positive} if $\eps_i=1$ for each $1\le i\le n$. The path $g$ is \emph{minimal} if any path from $\sfs(g)$ to $\sft(g)$ has length $\ge n$. 

Let $\sim$ be the smallest equivalence relation on the set of paths such that
\begin{enumerate}
	\item $ff^\minus \sim \sfs(f)$ for any path $f$;
	\item if $f\sim g$, then $f^{-1}\sim g^{-1}$;
	\item if $f\sim g$, and $h_1$ is a path with $\sft(h_1)=\sfs(f)=\sfs(g)$, and $h_2$ is a path with $\sfs(h_2)=\sft(f)=\sft(g)$, then $h_1fh_2\sim h_1gh_2$;
	\item if $f$ and $g$ are both minimal positive paths with $\sfs(f)=\sfs(g)$ and $\sft(f)=\sft(g)$, then $f\sim g$.
\end{enumerate} 

Let $[f]$ be the collection of all paths equivalent to $f$. Define another equivalence relation on the set of positive paths, called \emph{positive equivalence} and denoted $\sim_+$: it is the smallest equivalence relation generated by conditions (3) and (4) above. If $f$ is positive, let $[f]_+$ be the set of all positive paths that are positively equivalent to $f$. Note that the elements of $[f]_+$ all have the same length; so, $[f]_+$ has a well-defined length.

Define $\g(\Ga)$ (resp. $\g^+(\Ga)$) to be the collection of all equivalence (resp. positive equivalence) classes of paths (resp. positive paths). Then $\g(\Ga)$ (resp. $\g^+(\Ga)$) has the structure of a groupoid (resp. a category), whose objects are vertices of $\od_\Ga$, whose morphisms are equivalence classes of paths (resp. positive paths) with compositions given by concatenation of paths.  Then $\g(\Ga)$ is called the \emph{Deligne groupoid}.  Put $\g=\g(\Ga)$ and $\g^+=\g^+(\Ga)$. For objects $x,y\in\g^+$, let $\g^+_{x\to}$ denote the collection of morphisms whose source  is $x$. Define $\g^+_{\to y}$ and $\g^+_{x\to y}$ similarly. These notions  can be defined for $\g$ (that is, without using the superscript ${}^+$) in the same way. Note that the collection of morphisms of $\g$ with source and target both equal to $x$ is a group, called the \emph{isotropy group} at $x$  denoted by $\g_x$.

For two morphisms (i.e., for two positive equivalence classes of positive paths) $f$ and $g$ in $\g^+$, 
define the \emph{prefix order} in such a way that $[f]\preccurlyeq [g]$ if there is a morphism $h$ such that $[g]=[fh]$. Similarly, define the \emph{suffix order} so that $[g]\succcurlyeq [f]$ if there is a morphism $[h]$ such that $[g]=[hf]$. Then $(\g^+_{x\to},\preccurlyeq)$ and $(\g^+_{\to y},\succcurlyeq)$ are posets. 

\begin{thm}\label{thm:deligne} 
	The following statements are true.
	\begin{enumerate}
		\item For $[g],[h_1],[h_2]\in\g^+$, if, in addition, $[g][h_1]=[g][h_2]$, then $[h_1]=[h_2]$; if, in addition, $[h_1][g]=[h_2][g]$, then $[h_1]=[h_2]$.
		\item The natural map $\g^+\to\g$ is injective.  In other words, two positive path are positively equivalent if and only if they are equivalent.  
		\item For any vertex $x\in\od_\Ga$, any two elements in  $(\g^+_{x\to},\preccurlyeq)$ have a meet, and any two elements in $(\g^+_{x\to},\preccurlyeq)$ with a common upper bound have a join. A similar statements hold for  $(\g^+_{\to x},\succcurlyeq)$.
		\item Suppose $A_\Gamma$ is spherical. Then for any vertex $x\in\od_\Ga$, the posets $(\g^+_{x\to},\preccurlyeq)$ and $(\g^+_{\to x},\succcurlyeq)$ are lattices.
		\item Suppose $A_\Gamma$ is spherical. Then for any vertex $x\in\od_\Ga$, the posets $(\g_{x\to},\preccurlyeq)$ and $(\g_{\to x},\succcurlyeq)$ are lattices.
	\end{enumerate}
\end{thm}

Assertions (1) and (3) are proved by \cite{brieskorn1972artin}, Assertion (2) is proved in \cite{paris2002artin}, Assertions (4) and (5) are proved by \cite{brieskorn1972artin} and \cite{deligne}.

Given vertex $x\in \od_\Ga$, there is a 1-1 correspondence between elements in $\g_{x\to}$ and elements in $A_\Gamma$ by sending $[f]$ to $\w(f)$ (cf. Definition~\ref{def:label}), as $[f]=[g]$ if and only if $\w(f)=\w(g)$ in $A_\Gamma$. Similarly, there is a 1-1 correspondence between elements in $\g^+_{x\to}$ and elements in the positive monoid of $A_\Gamma$.

For two morphisms $f$ and $g$ with $\sfs(f)=\sfs(g)$ (resp. $\sft(f)=\sft(g)$), we write $[f]\vee_p [g]$ and $[f]\wedge_p [g]$ (resp. $[f]\vee_s [g]$ and $[f]\wedge_s [g]$) for the join and meet of $[f]$ and $[g]$ with respect to the prefix order (resp. suffix order).  
For paths $f$ and $g$, write $f\eq g$  if $[f]=[g]$.

%\begin{lem}
%	\label{lem:pn form}
%	Given any path $f$ on $\od_\Ga$, there exist positive paths $a$ and $b$ such that $f\eq ab^{-1}$ and $a\wedge_s b=\sft(a)$. Moreover, if $f\eq cd^{-1}$ where $c$ and $d$ are positive paths with $c\wedge_s d=\sft(c)$, then $a\eq c$ and $b\eq d$. Thus if $f\eq a_1b^{-1}_1$ for $a_1$ and $b_1$ positive, then $a\pr a_1$ and $b\pr b_1$.
%\end{lem}
%In the above lemma $\sft(a)$  denotes the identity morphism at the vertex $\sft(a)$. The proof of Lemma~\ref{lem:pn form} is identical to that of \cite[Theorem 2.6]{charney1995geodesic}. The decomposition $f\eq ab^{-1}$ is called the \emph{$pn$-normal form} of $f$.

For a given wall $H$ and a path $f=a^{\eps_1}_1a^{\eps_2}_2\cdots a^{\eps_n}_n$,  define the \emph{signed intersection number}, denoted $i(f,H)$, to be the sum of all the $\eps_i$'s such that $a_i$ is dual to $H$. In the special case when $f$ is positive, $i(f,H)$ is the number of times the edge path $\pi(f)$ crosses $H$. Two positive minimal paths with the same end points cross the same collection of walls (and each wall is crossed exactly once).  This gives the following lemma.

\begin{lem}(\cite[Proposition 1.11]{deligne})
	\label{lem:intersection number}
	Let $f$ and $g$ be paths on $\od_\Ga$ such that $f\sim g$. Then $i(f,H)=i(g,H)$ for any wall $H$. In particular, if $f$ and $g$ are positive paths with $f\sim g$, then they cross the same collection of walls.
\end{lem}

\subsection{Centralizer of parabolic Garside elements}
\label{subsec:centralizer}
Throughout this subsection, we assume $A_\Gamma$ is a spherical Artin group.
We recall from \cite[Section 5]{paris1997parabolic} the following definition. Let $X,X'\subset S$. A \emph{$(X',X)$-conjugator} is an element $\alpha\in A_\Gamma$ such that $\alpha X\alpha^{-1}=X'$. If $t\in S\setminus X$, $Y=X\cup\{t\}$ and $\alpha=\delta_Y\delta^{-1}_X$, then $\delta_Y\delta^{-1}_X$ is a $(X',X)$-conjugator, where $X'=\alpha X\alpha^{-1}$. Such a $(X',X)$-conjugator is called
an \emph{elementary $(X',X)$-conjugator}.

The following is \cite[Lemma 5.6]{paris1997parabolic} and \cite[Lemma 2.2]{godelle2003normalisateur}. 
\begin{lem}
	\label{lem:conjugator1}
	Suppose $A_\Gamma$ is a spherical Artin group and let $A_+$ be its positive monoid. Let $X\subset S$. Suppose there exists $p>0$ and $g\in A_+$ such that $\delta^p_X$ is central in $A_X$ and $g\delta^p_Xg^{-1}\in A_+$. Then $g=u_1\cdots u_nu_{n+1}$ where $u_{n+1}\in A_X$ and  there exists a sequence $X_n=X,X_{n-1},\cdots,X_0$ of subsets of $S$ such that $u_i$ an elementary $(X_{i-1},X_{i})$-conjugator for all $1\le i\le n$. Moreover, $g\delta^p_Xg^{-1}=\delta^p_{X_0}$ and $\delta^p_{X_0}$ is central in $A_{X_0}$.
\end{lem}

\begin{proof}
	As the statement of the lemma is slightly different from the exact statement of \cite[Lemma 5.6]{paris1997parabolic}, we give an explanation. \cite[Lemma 5.6]{paris1997parabolic} requires that we cannot write $g=g's$ with $g'$ positive and $s\in X$. However, it is always possible to reduce to this case as we are assuming $\delta^p_X$ is central in $A_X$. For the last statement, as the Dynkin diagram on $X$ and the Dynkin diagram on $X_0$ are isomorphic, $\delta^p_{X_0}$ is central in $A_{X_0}$.
\end{proof}

\begin{cor}
	\label{cor:conjugator2}
	Suppose $A_\Gamma$ is a spherical Artin group and let $A_+$ be its positive monoid. Let $X\subset S$. Suppose there exists $p>0$ and $g\in A_+$ such that $\delta^p_X$ is central in $A_X$ and $g\delta^p_Xg^{-1}=\delta^p_Y$. Then $g=v_{n+1}v_{n}\cdots v_1$ where $v_{n+1}\in A_Y$ and there exists a sequence $X_n=Y,X_{n-1},\cdots,X_0=X$ of subsets of $S$ such that $v_i$ an elementary $(X_{i},X_{i-1})$-conjugator for all $1\le i\le n$. 
\end{cor}

\begin{proof}
	We consider $A_\Gamma$ with a generating set $\bar S$ whose elements are inverses of elements in $S$. Then $g^{-1}$ is a positive element in $(A_\Gamma,\bar S)$ and $g^{-1}\delta^{-p}_Yg$ is a positive element in $(A_\Gamma,\bar S)$. By Lemma~\ref{lem:conjugator1}, $g^{-1}=u^{-1}_1\cdots u^{-1}_{n+1}$ with $u_i$ positive in $(A_\Gamma,S)$ (hence $u^{-1}_i$ positive in $(A_\Gamma,\bar S)$). Thus $g=u_{n+1}\cdots u_1$ and $u_i$ satisfies the required conditions.
\end{proof}

\subsection{A contractibility criterion for simplicial complexes}
Let $S=\{s_1,s_2,\ldots,s_n\}$. A simplicial complex $X$ is of \emph{type $S$} if all the maximal simplices of $X$ has dimension $n-1$ and $X$ admits a labeling of each vertex of $X$ by an element in $S$ such that adjacent vertices have different labels. This labeling induces a bijection between $S$ and the vertex set of each maximal simplex of $X$. 

We put a cyclic order $s_1<s_2<\cdots s_n<s_1$ on $S$. 
For each vertex $x$ in $X$ of type $s_i$, we consider a relation $<_x$ in $\lk(x,X)$ as follows. We identify vertices in $\lk(x,X)$ as vertices in $X$ which are adjacent to $x$. For each $s_i\in S$, this cyclic order induces an order on $S\setminus\{s_i\}$ by declaring $s_{i+1}<\cdots<s_n<s_1<\cdots s_{i-1}$. 
For $y,z\in \lk(x,X)$, define $y<_x z$ if $y$ is adjacent to $z$ and $\type(y)<\type(z)$ in $S\setminus \{s_i\}$.

The following contractility criterion is due to Haettel \cite{haettel2021lattices,haettel2022link}. This criterion is also closely related to another contractibility criterion due to Bestvina \cite{bestvina1999}, see Remark~\ref{rmk:connection}.

%The former can alternatively be deduced from the latter, 

\begin{thm}\cite{haettel2021lattices,haettel2022link}
	\label{thm:contractible}
	Let $X$ be a simplicial complex of type $S$, with $S$ being a cyclically ordered set as above. Suppose the following are true.
	\begin{enumerate}
		\item $X$ is simply-connected.
		\item For each vertex $x\in X$, the relation $<_x$ on the vertex set of $\lk(x,X)$ is a partial order.
		\item For each vertex $x\in X$, the set of vertex in $\lk(x,X)$ with this partial order is bowtie free.
	\end{enumerate}
Then $X$ is contractible. 

Moreover, if a group $G$ acts on $X$ by type-preserving automorphisms, then $X$ can be equipped with a metric with an $G$-equivariant consistent convex geodesic bicombing $\sigma$ such that each simplex of $X$ is $\sigma$-convex.
\end{thm}

The contractibility of $X$ will be used later in the proof of $K(\pi,1)$-conjecture. The last part can be used to prove a number of other group theoretic properties other than the $K(\pi,1)$-conjecture.
We refer to \cite{descombes2015convex} for background on convex geodesic bicombing. 

We will also need an extra combinatorial property of $X$. To prove this combinatorial property, we need the following relaxation of the partial order.

\begin{definition}
	\label{defi:weak_order}
	A \emph{weakly ordered} set $P$ is a set with a binary relation $\le$ over $P$ which is  reflexive and antisymmetric. Moreover, while transitivity may fail, we do require the following associativity law for transitivity. Define $(a,b,c)\in P^3$ to be a \emph{transitive triple} if $a\le b$, $b\le c$ and $a\le c$. We require $\le$ satisfies the following condition:
	
	\noindent
	$(\ast)$: for any quadruple $a,b,c,d\in P$ with $a\le b,b\le c$ and $c\le d$, we have $(b,c,d)$ and $(a,b,d)$ are transitive triples if and only if $(a,b,c)$ and $(a,c,d)$ are transitive triples.
\end{definition}

The notions of upper bound, join, lower bound, meet, interval and homogeneity can be defined for a weakly ordered set in the same way. 
Let $(P,\le)$ be a weakly ordered set. For $x\in P$, let $P_{\ge x}$ be the collection of all elements which are $\ge x$. Similarly we define $P_{\le x}$. Note that $P_{\ge x}$ and $P_{\le x}$ are actually posets by $(\ast)$. 
A \emph{weak chain} in a weakly ordered set $P$ is a sequence of elements $c_1<c_2<c_3<\cdots<c_n$ such that any two adjacent elements in the sequence are comparable. A \emph{chain} is a weak chain such that any two elements in the weak chain are comparable.
Note that if $P$ does not contain non-trivial weak chains which start and end at the same element, then the weak order $\le$ on $P$ actually generates a partial order $\le_t$, where $a\le_tb$ if $a$ and $b$ are the first and the last member of a weak chain in $P$ with respect to $\le$.

\begin{lem}
	\label{lem:4-cycle}
	Let $X$ be as in Theorem~\ref{thm:contractible}, with all the three assumptions there satisfied. Then for any induced 4-cycle in the 1-skeleton of $X$, there is a vertex $x\in X$ such that $x$ is adjacent to each vertex of this 4-cycle.
\end{lem}

\begin{proof}
We will follow the proof of \cite[Theorem 5.6]{haettel2022lattices}.
	Given $X$ as in the theorem, we can apply a construction in the work of Hirai and Haettel \cite{hirai2020uniform,haettel2022injective,haettel2021lattices} to obtain a new complex $Y$ as follows. 
	Let $P=X^{(0)}\times \mathbb Z$. 
	We put a weak order on $P$ as follows. Suppose $x_1,x_2\in X$ are vertices of type $s_{i_1}$ and $s_{i_2}$ respectively.
	Define $(x_1,n)\le (x_2,m)$ if one of the following holds:
	\begin{enumerate}
		\item $n=m$, $x_1$ and $x_2$ are adjacent and $i_1\le i_2$;
		\item $m=n+1$, $x_1$ and $x_2$ are adjacent and $i_2\le i_1$.
	\end{enumerate}
	Note that $(P,\le)$ is a weakly ordered set. If $(x,n)\le (y,m)$, then $x$ and $y$ are adjacent in $X$. Conversely, if $x$ and $y$ are adjacent, then either $(x,n)\ge (y,n)$ or $(x,n)\ge (y,n-1)$. As the 1-skeleton $X^1$ is connected, we deduce that for any $(x,n)\in P$ and any $y\in X$, there exists $m\in\mathbb Z$ such that we can find a weak chain from $(y,m)$ to $(x,n)$.
	The weakly ordered set $P$ is homogeneous, indeed, for $a\le b$ in $P$, we can define the length function $\ell(a\le b)$ to be the maximal length of chains in $P$ from $a$ to $b$. 
	
	Let $\varphi: P\to P$ be the map sending $(x,n)$ to $(x,n+1)$. Then $\varphi$ is an automorphism of weakly ordered set. Then $\varphi$ generates the weak order $\le$. Let $X_\varphi$ be the simplicial complex whose vertex set is $P$ and whose edges correspond
	to morphisms of form $a\le b$ such that $a,b\in [x,\varphi(x)]$ for $x\in P$. As $X_{\varphi}$ is homeomorphic to $X\times \mathbb R$, by Assumption 1 of Theorem~\ref{thm:contractible}, $X_{\varphi}$ is simply-connected. By Assumption 3 of Theorem~\ref{thm:contractible} and Lemma~\ref{lem:poset}, $[x,\varphi(x)]$ is a lattice for $x\in P$.
	Thus \cite[Theorem 1.3]{haettel2022lattices} implies $\le$ generates a partial order $\le_t$ on $P$. The previous paragraph implies any two elements in $(P,\le_t)$ have lower bound, thus $(P,\le_t)$ is a lattice by \cite[Theorem 1.3]{haettel2022lattices}. As $(P,\le_t)$ is homogeneous, we know any lower bounded subset in $(P,\le_t)$ has a meet and any upper bounded subset in $(P,\le_t)$ has a join. Now the automorphism $\varphi$ gives an action $\mathbb Z\act (P,\le_t)$. 
	The previous paragraph also implies for any $a,b\in (P,\le_t)$, there exists $k\in \mathbb Z$ with $$-k\cdot a\le b\le k\cdot a.$$
	Note that $X^1$ coincides with the quotient graph of $(P,\le_t)$ by $\mathbb Z$ as defined in \cite[Section 2]{haettel2022lattices}. Thus we are done by the remark at the end of \cite[Section 2]{haettel2022lattices}. 
\end{proof}

\begin{remark}
	\label{rmk:connection}
Note that in the last paragraph of the proof of \cite[Theorem 1.3]{haettel2022lattices}, it is shown that $X_\varphi$ is contractible by \cite[Corollary 7.6]{bessis2006garside}, which relies on a Morse theoretic argument of Bestvina \cite{bestvina1999}. In particular, this implies the contractibility part of Theorem~\ref{thm:contractible} immediately. For comparison, Haettel's proof of Theorem~\ref{thm:contractible} relies on putting an injective metric on $X_\varphi$. Both proofs rely crucially on the lattice structure on the interval $[x,\varphi(x)]$.
\end{remark}

\section{Some properties of (oriented) Davis complexes}
\label{sec:Davis}

In this section we collect and prove several properties of Davis complexes and oriented Davis complexes for later use. In Section~\ref{subsec:elementary segments}, we reformulate the discussion of centralizers in Section~\ref{subsec:centralizer} into a geometric property of oriented Davis complexes, see Corollary~\ref{cor:decomposition}. In Section~\ref{subsec:convex hull}, we prove two properties of convex hulls of two standard subcomplexes in a Davis complex. In Section~\ref{subsec:escaping} we prove some properties of certain types of paths in the Davis complexes.

Let $\od_\Ga$, $\g^+=\g^+(\Ga)$, $\g=\g(\Ga)$, and $\pi:\od_\Ga\to \Si_\Ga$ be as in Section~\ref{subsec:Deligne}.
Recall that edges of $\Si_\Ga$ and $\od_\Ga$ are labeled by elements in the generating set $S$ (cf. Definition~\ref{def:label}). We will write $\Si=\Si_\Ga$ and $\od=\od_\Ga$.
We endow the 1-skeleta of these two complexes with path metrics such that each edge has length one. We use $d$ to denote the path metric. By a \emph{geodesic} between two vertices of $\Si$ or $\od$, we always mean geodesic in the 1-skeleton with respect to this path metric.

We will be repeatedly using the following basic observation which is a direct consequence of \cite{matsumoto1964generateurs}. Suppose $u_1$ and $u_2$ are two geodesics in the 1-skeleton $\Sii$ with the same endpoints. Then $\supp(u_1)=\supp(u_2)$. Likewise, if $u_1$ and $u_2$ are minimal positive paths on $\od$ with $[u_1]=[u_2]$, then $\supp(u_1)=\supp(u_2)$. As a consequence, by Theorem~\ref{thm:deligne} (2), the same statement holds true if $u_1$ and $u_2$ are positive paths.

\subsection{Elementary segments}\label{subsec:face}
\label{subsec:elementary segments}
\begin{definition}
	Suppose $F$ is a standard subcomplex of $\Si_\Ga$.  An \emph{$F$-path} of $\od_\Ga$ is a path whose image under $\pi$ is contained in $F$.
\end{definition}

Let $\G(F)$ denote the Deligne groupoid  over the 1-skeleton of $\widehat F$. Then $\G(F)$ also satisfies Theorem~\ref{thm:deligne}.
For a vertex $x\in\od_\Ga$ and a standard subcomplex $F$ of $\Si_\Ga$, denote by $\Sigma_{x\to}(F)$ (resp. $\Sigma_{\to x}(F)$)  the collection of edges in $\od_\Ga$ whose image under $\pi$ is contained in $F$ and whose source (resp. target) is $x$. 

If $F$ is a face of $\Si_\Ga$, 
let $\ant(x,F)$ denote the antipodal vertex to $x$ in $F$.  For $y=\ant(x,F)$ define $\delta^{x}(F)$ (resp. $\delta_{x}(F)$) to be the morphism represented by a minimal positive path from $y$ to $x$ (resp. $x$ to $y$). The equivalence classes of such positive paths of this form are called \emph{Garside elements}. For an integer $k\ge 1$, define 
$$(\delta_x(F))^k:=\underbrace{\delta_x(F)\delta_y(F)\delta_x(F)\cdots}_{k \text{ times}}.$$
In the special case when $A_\Ga$ is spherical, to simplify notation, put 
\begin{align*}
	\delta^x=\delta^x(&\Si_\Ga), \quad \Sigma_{x\to}=\Sigma_{x\to}(\Si_\Ga),\\
	\delta_x=\delta_x(&\Si_\Ga), \quad\Sigma_{\to x}=\Sigma_{\to x}(\Si_\Ga).
\end{align*}
Note that if we read off the labels of consecutive edges in $\delta_x(F)$, then we obtain a word which is the Garside element in the spherical standard parabolic subgroup $A_{\Gamma'}$ of $A_{\Gamma}$, where $\Gamma'$ is the full subgraph spanned by $\supp(F)$.

Recall that the notion of two faces of $\Si_\Ga$ being parallel is defined before Definition~\ref{def:intersection poset}.
\begin{definition}
	\label{def:adj}
Two	Parallel faces $F$ and $F'$ of $\Si_\Ga$ are \emph{adjacent} if $F\neq F'$ and if they are contained in a face $F_0$ with $\dim(F_0)=\dim(F)+1$.
\end{definition}

\begin{definition}
	\label{def:elementary segment}
	Let $B\in\cq_\Ga$ (cf. Definition~\ref{def:intersection poset}). Let $F$ and $F'$ be two adjacent parallel faces of $\Si_\Ga$ that are dual to $B$. An \emph{elementary $B$-segment}, or an \emph{$(F,F')$-elementary $B$-segment} is a minimal positive path from a vertex $x\in F$ to $x'=p(x)\in F'$, where $p:F\to F'$ is parallel translation.
\end{definition}

\begin{lem}
	\label{lem:geodesic in carrier}
	Suppose $p:F\to F'$ be the parallel translation between two parallel faces of $\Si_\Ga$ that are dual to an element $B\in \cq_\ca$. Take a vertex $x\in F$. Then there exists an edge path $u$ of shortest length between $x$ and $p(x)$ in the 1-skeleton of $\Si_\Ga$ such that $u$ is made of elementary $B$-segments.
\end{lem}

\begin{proof}
	Let $d$ denotes the length metric on the 1-skeleton of $\Si_\Ga$ and let $\cw$ denotes the collection of walls in $\Si_\Ga$. Then $d(x,p(x))$ equals to the number of walls in $\Si_\Ga$ separating $x$ from $p(x)$. By the definition of $p$, this equals to the number of walls in $\cw(F,F')$, which is defined to be the collection of walls separating $F$ from $F'$. Thus it suffices find a concatenation of elementary $B$-segments from $x$ to $p(x)$ such that each element in $\cw(F,F')$ are crossed exactly once, and no walls in $\cw\setminus\cw(F,F')$ are crossed.

Let $\cw^B=\{H\cap B\mid H\cap B\neq \emptyset\ \mathrm{and}\ H\in \cw\}$. Each element in $\cw^B$ separates $B$ into exactly two connected components. We endow $B$ with the induced cell structure from $\Si_\Ga$. 
Vertices of $B$ are in 1-1 correspondence with faces of $\Si_\Ga$ that are dual to $B$, and the distance between two vertices of $B$ in the 1-skeleton is the number of walls in $\cw^B$ that separate these two vertices.
Moreover, $B\cap b\Si_\Ga$ is a subcomplex of $b\Si_\Ga$ which can be viewed as the barycentric subdivision $bB$ of $B$.

Let $bF$ be the barycenter of $F$, viewed as a vertex of $bB$. Take a geodesic edge path $u_B$ in the 1-skeleton of $bB$ from $bF$ to $bF'$. 
	Let $\cw^B(bF,bF')$ be the collection of walls in $\cw^B$ that separate $bF$ from $bF'$ in $B$. Then $$\cw^B(bF,bF')=\{H\cap B\mid H\in \cw(F,F')\}.$$

	Let $\{v_i\}_{i=1}^n$ be consecutive vertices in $u_B$ with $v_1=bF$ and $v_n=bF'$. Let $H^B_i\in \cw^B$ be the wall of $B$ separating $v_i$ from $v_{i+1}$. Then $$\cw^B(bF,bF')=\{H^B_1,\ldots, H^B_{n-1}\}.$$ 
	Suppose $v_i$ is the barycenter of $F_i$. Then $F_i$ and $F_{i+1}$ are adjacent parallel faces. Let $p_i:F\to F_i$ be the parallel translation and let $u_i$ be an elementary $B$-segment from $p_{i}(x)$ to $p_{i+1}(x)$. Let $u$ be the concatenation of $\{u_i\}_{i=1}^{n-1}$. Note that the collection of walls of $\cw$ crossed by $u_i$ equals to the collection of walls separating $F_i$ and $F_{i+1}$, which equals to $$\cw_i=\{H\in \cw\mid H\cap B=H^B_i\}.$$ Note that $$\cw(F,F')=\sqcup_{i=1}^{n-1}\cw_i.$$ Thus $u$ is the desired geodesic.
\end{proof}

\begin{lem}
	\label{lem:equal support}
	 Let $e_1=\overline{x_1y_1}$ and $e_2=\overline{x_2y_2}$ be parallel edges in $\Si_\Ga$ such that the parallel translation between them sends $x_1$ to $y_1$. Let $u_x$ be a geodesic from $x_1$ to $x_2$, and let $u_y$ be a geodesic from $y_1$ to $y_2$. Then $\supp(u_x)=\supp(u_y)$ (see Definition~\ref{def:label} for the definition of support).
\end{lem}

\begin{proof}
By Lemma~\ref{lem:geodesic in carrier}, we assume $u_x$ is made of elementary $H$-segments where $H$ is the wall dual to $e_1$. For each elementary $H$-segment $u_{x,i}$ of $u_x$, we can find another elementary $H$-segment $u_{y,i}$ on the other side of $H$ such that $u_{x,i}$ and $u_{y,i}$ are contained in a 2-dimensional face which intersects $H$. Note that $\supp(u_{x,i})=\supp(u_{y,i})$ and the concatenation of $u_{y,i}$'s is a geodesic from $y_1$ and $y_2$. Thus the lemma is proved.
\end{proof}

Take $B\in\cq_\Ga$. Let $F$ and $F'$ be two adjacent parallel faces of $\Si_\Ga$ dual to $B$ such $F$ and $F'$ are contained in a face $F_0$ with $\dim(F_0)=\dim(F)+1$.
Any elementary $B$-segment can be written as $\delta_x(F_0)(\delta^y(F'))^{-1}$ for some vertex $x\in F$ and $y=\ant(x,F_0)$ (cf. \cite[Lemma 3.8]{davis2021bordifications}). Thus for any $(F,F')$-elementary $B$-segment $h$, $\w(h)$ is an elementary $(\supp(F),\supp(F'))$-conjugator in the sense of \cite[Section 5]{paris1997parabolic}. Recall that $\w(h)$ and $\supp(F)$ are defined in Definition~\ref{def:label}. Conversely, any elementary conjugator arises as a word associated with some elementary $B$-segment for some $B\in\cq_\Ga$. We also record the following consequence of this discussion.

%or one type of positive elementary $\supp(F)$-ribbon-$\supp(F')$ in the sense of \cite[Definition 0.2]{godelle2003parabolic}. 

Let $S$ be the vertex set of $\Gamma$. A subset $T\subset S$ is \emph{irreducible}, if the sub-Dynkin diagram spanned by element in $T$ is connected. Given $T'\subset S$, an \emph{irreducible component} of $T'$ is the vertex set of a connected component of the sub-Dynkin diagram spanned by $T'$. For $T_1,T_2\subset S$, we write $T_1\perp T_2$ if $T_1\cap T_2=\emptyset$, and each element of $T_1$ commutes with every element of $T_2$.

\begin{lem}
	\label{lem:irreducible segment}
	Let $B,F,F'$ and $F_0$ be as in Definition~\ref{def:elementary segment}, and
	let $u$ be an $(F,F')$-elementary $B$-segment from $x$ to $x'$. Then
	\begin{enumerate}
		\item  $\supp(u)$ equals to the irreducible component of $\supp(F_0)$ that contains $\supp(F_0)\setminus\supp(F)$, which is also the irreducible component of $\supp(F_0)$ that contains $\supp(F_0)\setminus\supp(F')$;
		\item for any $p>1$, there exists an $(F,F')$-elementary $B$-segment $v$ such that $$[u](\delta_y(F'))^p=(\delta_x(F))^p[v]\ \textmd{and}\ \w(u)=\w(v);$$
		\item if $v$ is an $(F',F'')$-elementary $B$-segment from $x'$ to $x''$ with $\supp(v)\perp \supp(u)$, then $[uv]=[v'u']$ where $v'$ and $u'$ are elementary $B$-segments with $\w(u)=\w(u')$ and $\w(v)=\w(v')$.
	\end{enumerate} 
\end{lem}

\begin{proof}
	For (1), note that if $\supp(F_0)$ is irreducible, then $\supp(u)=\supp(F_0)$. If $F_0$ is not irreducible, let $I_1$ be the irreducible component of $\supp(F_0)$ that contains $\supp(F_0)\setminus\supp(F)$, and let $I_2=\supp(F_0)\setminus I_1$.
	For $i=1,2$, let $F_i$ be the face of $F_0$ containing $u\cap F$ such that $\supp(F_i)=I_i$. Then $F_0\cong F_1\times F_2$. 
	Note that $F\cap F_1$ and $F'\cap F_1$ are parallel adjacent faces, and $u\subset F_1$. Now (1) follows. For (2), let $z$ (resp. $z'$) be the antipodal point of $x$ (resp. $x'$) in $F$ (resp. $F'$). Then $z'=p(z)$. Let $v$ be an elementary $B$-segment from $z$ to $z'$. By Lemma~\ref{lem:gate}, both $u\delta_{x'}(F')$ and $\delta_x(F)v$ are minimal positive paths with same endpoints. Thus, $u\delta_{x'}(F')\eq\delta_x(F)v$. Now we show $\w(u)=\w(v)$. As $x$ and $z$ are antipodal points in $F_0$, we know $u\delta_{x'}(F')\eq\delta_x(F)v\eq\delta_x(F_0)$. Similarly, $v\delta_{z'}(F')\eq\delta_z(F)u\eq\delta_{z'}(F_0)$. As $\w(\delta_{x'}(F'))=\w(\delta_{z'}(F'))$ and $\w(\delta_x(F_0))=\w(\delta_{z'}(F_0))$, we deduce that $\w(u)=\w(v)$. (2) now follows by repeatedly applying this argument. (3) is clear. 
\end{proof}

We also record the following statement, which is a direct translation of
Lemma~\ref{lem:conjugator1} and Corollary~\ref{cor:conjugator2}.

\begin{cor}
	\label{cor:decomposition}
	Suppose $A_\Gamma$ is spherical.
	Let $F$, $F'$ of $\Si_\Ga$ be two parallel faces. Take $k>0$ such that $\w((\delta_{x'}(F'))^k)$ is central in $A_{\supp(F')}$.  Let $f,f'$ be a positive paths in $\od_\Ga$ from $x\in F$ to $x'\in F'$ satisfying the both of the following assumptions:
	\begin{itemize}
		\item $f(\delta_{x'}(F'))^k\eq hf'$, for a positive path $h$;
		\item $\w(f)=\w(f').$
	\end{itemize}
	Then $f$ and $f'$ admit decompositions, let us call them of type I, satisfying all of the following conditions:
	\begin{itemize}
		\item $f\eq u_1u_2\cdots u_nu_{n+1}$ such that $u_{n+1}$ is a positive $F'$-path and each $u_i$ with $1\le i\le n$ is an elementary $B$-segment and $B\in\cq_\Ga$ is dual to $F$;
		\item $f'\eq u'_1u'_2\cdots u'_nu'_{n+1}$ such that $u'_{n+1}$ is a positive $F'$-path and each $u'_i$ with $1\le i\le n$ is an elementary $B$-segment, moreover $\w(u'_i)=\w(u_i)$ for $1\le i\le n+1$;
		\item for $2\le i\le n+1$, there is a face $F_i$ parallel to $F$ such that for $1\le i\le n+1$,  $$[u_i](\delta_{\sft(u_i)}(F_{i+1}))^k=(\delta_{\sfs(u_i)}(F_i))^k[u'_i]$$ (here $F_1=F$ and $F_{n+2}=F'$).
	\end{itemize}
	Alternatively, $f$ and $f'$ admit different decompositions, let us call them of type II, satisfying all of the following three conditions:
	\begin{itemize}
		\item $f\eq v_{n+1}v_{n}\cdots v_1$ such that $v_{n+1}$ is a positive $F'$-path and each $u_i$ with $1\le i\le n$ is an elementary $B$-segment and $B\in\cq_\Ga$ is dual to $F$;
		\item $f'\eq v'_{n+1}v'_{n}\cdots v'_1$ such that $v'_{n+1}$ is a positive $F'$-path and each $v'_i$ with $1\le i\le n$ is an elementary $B$-segment, moreover $\w(v'_i)=\w(v_i)$ for $1\le i\le n+1$;
		\item for $2\le i\le n+1$, there is a face $E_i$ parallel to $F$ such that for $1\le i\le n+1$,  $$[v_i](\delta_{\sft(v_i)}(E_{i}))^k=(\delta_{\sfs(v_i)}(E_{i+1}))^k[v'_i]$$ (here $E_1=F'$ and $E_{n+2}=F$).
	\end{itemize}
\end{cor}

\subsection{Convex hulls of two standard subcomplexes}
\label{subsec:convex hull}
 Let $E_1$ and $E_n$ be two standard subcomplexes of $\Si$. A subset $K$ of $\vertex \Si$ is \emph{convex} if for any $x,y\in K$ and any geodesic $u\subset \Si^1$ (or $u\subset \od^1$), all the vertices of $u$ are contained in $K$. In particular, $\vertex E_i$ is convex for $i=1,n$ by Lemma~\ref{lem:gate}. We define an \emph{halfspace} of $\vertex \Si$ to be the collection of all points in one side of a wall $H\in \ca$. Each halfspace is convex. Recall that we write $T_1\perp T_2$ for subsets $T_1,T_2$ of vertex set $S$ of $\Gamma$, if $T_1\cap T_2=\emptyset$ and each element in $T_1$ commutes with every element of $T_2$.

\begin{lem}
	\label{lem:hull}
	Let $K$ be the convex hull of $\vertex E_1\cup \vertex E_n$, i.e. the smallest convex subset of $\vertex\Si$ containing $\vertex E_1\cup \vertex E_n$. Take $x\in K$, let $x_i=\prj_{E_i}(x)$ and let $u_i$ be a geodesic from $x$ to $x_i$ ($i=1,n$). Then $\supp(u_1)\perp\supp(u_n)$.
\end{lem}

A more natural notation to use might be $E_1$ and $E_2$, instead of $E_1$ and $E_n$. However, we hope to align the notation in later sections when this lemma is used, so we use $E_1$ and $E_n$.

\begin{proof}
	First note that for any edge connecting two vertices in $K$, the wall $H$ dual to this edge must intersect at least one of $E_1$ or $E_n$, otherwise $H$ would contain $E_1$ and $E_n$ on the same side as $E_1\cap E_n\neq\emptyset$, implying $K$ is contained on one side of $H$.

	We prove the lemma by induction on the distance $d(x,E_1)$. The case $d(x,E_1)=0$ is trivial as $\supp(u_1)=\emptyset$. Now we suppose the lemma is true if $d(x,E_1)=k$. Take $x\in K$ with $d(x,E_1)=k+1$. Let $x'$ be the vertex in $u_1$ such that $d(x',E_1)=k$. Let $u'_1$ be the subsegment of $u_1$ between $x'$ and $\prj_{E_1}(x)=\prj_{E_1}(x')$, and let $u'_n$ be a geodesic between $x'$ and $\prj_{E_n}(x')$. By induction $\supp(u'_1)\perp\supp(u'_n)$. Let $F$ be the standard subcomplex with $\supp(F)=\supp(u'_n)$ that contains $u'_n$. 
	
	Let $H$ be the wall dual to the edge $e$ between $x$ and $x'$. Then $H$ intersects at least one of $E_1$ and $E_n$, otherwise $E_1\cup E_n$ is contained on one side of $H$, hence $K$ is contained on one side of $H$ as each halfspace is convex, contradicting that $x,x'\in K$. As $e$ is an edge in the geodesic from $\prj_{E_1}(x)$ to $x$, $H\cap E_1=\emptyset$ by Lemma~\ref{lem:gate}. Thus $H\cap E_n\neq\emptyset$. It follows that $\prj_{E_n}(x)$ and $\prj_{E_n}(x')$ span an edge $e'$ which is parallel to $e$. By Lemma~\ref{lem:geodesic in carrier}, we can find a geodesic $u$ from $x'$ to $\prj_{E_n}(x')$ that is made of elementary $H$-segments. 
	
	We claim each elementary $H$-segment in $u$ is made of a single edge. Suppose $u$ contains an elementary $H$-segment $v$ which is not a single edge. Then $\supp(v)$ consists of two elements of $S$ which do not commute. Let $F_v$ the 2-dimensional face with  $\supp(F_v)=\supp(v)$ that contains $v$. Note that $v\subset F\cap F_v$. Indeed, as $F$ is convex (Lemma~\ref{lem:gate} (1)), we know $u\subset F$, hence $v\subset F$. As $F_v\cap F$ is face of $\Sigma$ and $F_v\cap F$ contains at least two consecutive edges, we must have $$F_v\subset F.$$ As $F_v$ contains an edge parallel to $e$, we know $H\cap F\neq\emptyset$. As $$\supp(F)=\supp(u'_n)\ \mathrm{and}\ \supp(u'_n)\perp \supp(u'_1),$$ we know $F$ is parallel to $F'$, where $F'$ is defined to be the standard subcomplex with $\supp(F')=\supp(F)$ containing $\prj_{E_1}(x')$. As each wall dual to an edge of $u'_n$ must have non-empty intersection with $E_1$ by a similar argument as before, and $\supp(u'_n)\perp \supp(u'_1)$, we know that $\supp(u'_n)\subset \supp(E_1)$. Hence $\supp(F')\subset\supp(E_1)$, and $F'\subset E_1$. Thus $H\cap E_1\neq\emptyset$.
	This contradicts $H\cap E_1=\emptyset$ in the previous paragraph. Thus the claim is proved.
	
	By the claim, $\supp(e)\perp\supp(u'_n)$, thus $\supp(u'_n)\perp\supp(u_1)$. Moreover, the geodesic path $u'_n$ and $e$ together span a ``ladder-like'' subcomplexes of $\Si$ made of a sequence of squares with $u'_n$ being one side of the ladder. Let $w_n$ be the other side of the ladder. Then $w_n$ is a geodesic from $x$ to a vertex in $E_n$. As $H\cap E_n\neq\emptyset$, we know $d(x,E_n)=d(x',E_n)$. Note that $w_n$ and $u'_n$ have the same length, thus Lemma~\ref{lem:gate} implies that $w_n$ goes from $x$ to $\prj_{E_n}(x)$. Thus $\supp(u'_n)=\supp(w_n)=\supp(u_n)$, which implies that $\supp(u_1)\perp\supp(u_n)$.
\end{proof}

Given a standard subcomplex $F$ of $\Sigma$ and a geodesic segment $u$ from $x$ to $y$ in $\Si^1$, we say $u$ is \emph{$F$-escaping} if $x\in F$ and $d(y,F^1)=\ell(u)$, where $\ell(u)$ means the length of $u$. It follows from Lemma~\ref{lem:gate} that a geodesic $u$ is $F$-escaping if and only if each wall dual to an edge of $u$ has empty intersection with $F$.
Similarly, we define a minimal positive path $u\subset\od$ from vertex $x$ to vertex $y$ to be $F$-escaping if $x\in \widehat F$ and $d(y,\widehat F)=\ell(u)$. 

\begin{lem}
	\label{lem:extra edge}
	Let $F$ be a standard subcomplex of $\Si$.
	Let $u$ be an $F$-escaping geodesic in $\Sii$ (or $\od^1$) ending at $y\notin F$. Let $e$ be an edge containing $y$. Then
	\begin{enumerate}
		\item if the two endpoints of $e$ have the same distance from $F$, then either $\supp(e)\perp\supp(u)$ or $\supp(e)\subset\supp(u)$;
		\item if $\supp(e)\notin \supp(u)$ and $\supp(e)\cup \supp(u)$ is irreducible, then $ue$ is an $F$-escaping geodesic.
	\end{enumerate}
\end{lem}

\begin{proof}
	We prove (1). By Lemma~\ref{lem:more gate}, the projection of two endpoints of $e$ onto $F$ span an edge $e'\subset F$ which is parallel to $e$. Let $H$ be the wall dual to $e$. Then Lemma~\ref{lem:geodesic in carrier} implies that there is a geodesic $u'$ from $y$ to $\prj_F(y)$ such that $u'$ is made of elementary $H$-segments. Note that $\supp(u')=\supp(u)$. If $\supp(e)$ and $\supp(u')$ do not commute, then we take the first edge $f_1$ of $u'$ (starting from $y$) such that $\supp(f_1)$ does not commute with $\supp(e)$. Let $u_1$ be the elementary $H$-segment of $u'$ that contains $f_1$, and let $y_1$ be the endpoint of $u_1$ that is closer to $y$. Let $e_1$ be the edge parallel to $e$ that contains $y_1$.
	As the subsegment of $u'$ from $y$ to $y_1$ has support commuting with $\supp(e)$, we know the support $\supp(e)=\supp(e_1)$.  Note that $e_1$ and $u_1$ span a 2-dimensional face of $\Si$. As $\supp(e_1)$ and $\supp(u_1)$ do not commute, $u_1$ contains at least two consecutive edges of the boundary of this 2-face. Thus $\supp(e_1)\subset \supp(u_1)$. Hence $\supp(e)\subset\supp(u)$.
	
	Now we prove (2).	
	As $\supp(e)\notin \supp(u)$, it follows from the ``Exchange Condition'' in Coxeter groups (see e.g. \cite[Chapter 3]{davis2012geometry}) that $ue$ is a geodesic. It remains to show the wall $H$ dual to $e$ satisfies $H\cap F=\emptyset$. If $H\cap F\neq\emptyset$, then Lemma~\ref{lem:more gate} and part (1) imply that either $\supp(e)\perp\supp(u)$ or $\supp(e)\subset\supp(u)$, which is a contradiction.
\end{proof}

\begin{lem}
	\label{lem:two cells}
	Let $F_n,F'_n\subset \Si$ be two standard subcomplexes in $\Sigma$. Suppose $F_n\neq F'_n$ and both $F_n$ and $F'_n$ have nonempty intersection with a standard subcomplex $F_1$. Let $\Pi_{F_n}$ be the map in Definition~\ref{def:retraction}. Let $\bar F_n=\Pi_{F_n}(F'_n)$ and $\bar F'_n=\Pi_{F'_n}(F_n)$ with the parallel translation $p:\bar F_n\to \bar F'_n$ between them. Then
	\begin{enumerate}
		\item  there is a vertex $x\in \bar F_n$ such that both $x$ and $p(x)$ are contained in $F_1$; moreover, for any vertex $y\in \bar F_n\cap F_1$, $p(y)\in F_1$;
		\item  define $I_n$ (resp. $I'_n$) to be the union of the irreducible components of $\supp(\bar F_n)$ (resp. $\supp(\bar F'_n)$) that are not contained in $\supp(F_1)$, then
		$I_n=I'_n$; moreover, for any pair of vertices $z_1,z_2\in \bar F_n$ and any choice of geodesics $u_i$ between $z_i$ and $p(z_i)$ with $i=1,2$, we have $\supp(u_1)=\supp(u_2)\perp I_n$;
		\item for any edge $e'\subset \bar F'_n$ with $\supp(e')\in I'_n$, $\supp(e')=\supp(\Pi_{F_n}(e'))$.
	\end{enumerate}
\end{lem}

\begin{proof}
	Let $x'\in F_n\cap F_1$ and $y'\in F'_n\cap F_1$. By Lemma~\ref{lem:more gate} (2), there exists a geodesic from $x'$ to $y'$ passing through $y=\prj_{F'_n}(x')$. By Lemma~\ref{lem:more gate} (1), $y\in F_1$. Applying Lemma~\ref{lem:more gate} (2) again, there is a geodesic from $x'$ to $y$ passing through $x=\prj_{F_n}(y)$. By Lemma~\ref{lem:pair gate}, $y\in \bar F'_n$ and $x=p^{-1}(y)\in \bar F_n$. As there is a geodesic from $x'$ to $y'$ passing through $x$ and $y$ and both $x'$ and $y'$ are in $F_1$, the convexity of $F_1$ implies that $x,y\in F_1$. This finishes the first part Assertion 1. The last part of Assertion 1 can be proved similarly.
	
	For Assertion 2, let $L'_n\subset I'_n$ be an irreducible component. Let $F_{L'_n}$ be the standard subcomplex containing $p(x)$ from the previous paragraph with $\supp(F_{L'_n})=L'_n$. As $L'_n$ is irreducible and contains an element outside $\supp(F_1)$, we can repeatedly apply Lemma~\ref{lem:extra edge} (2) to find an $F_1$-escaping geodesic $v$ starting from $p(x)$ and ending at $x''$ such that $$\supp(v)=\supp(L'_n).$$ 
	As $v\subset \bar F'_n$, and $\bar F'_n$ and $\bar F_n$ are parallel, we know any wall dual to an edge of $v$ intersects $\bar F_n$, hence $F_n$. Thus $v$ is in the convex hull of $\vertex F_n$ and $\vertex F_1$. Now Lemma~\ref{lem:hull} implies that $\supp(v)\perp\supp(u')$ where $u'$ is an $F_n$-escaping geodesic ending at $x''$. Thus $F_{L'_n}$ and $\prj_{F_n}(F_{L'_n})\subset \bar F_n$ are contained in a common standard subcomplex $F$ with $\supp(F)=\supp(v)\cup \supp(u')$. As $\supp(v)\perp\supp(u')$, $F$ is a product of $F_{L'_n}$ and another standard subcomplex whose support is $\supp(u')$. Thus $\supp(F_{L'_n})=\supp(\prj_{F_n}(F_{L'_n}))$ and the first part of Assertion 2 follows. This discussion and Lemma~\ref{lem:equal support} also imply the last part of Assertion 2. As a consequence, for any $z'\in \bar F'_n$ and any geodesic $u_{z'}$ from $z'$ to $p(z')$, $\supp(u_{z'})=\supp(u')\perp I'_n$. Assertion 3 follows by taking $z'$ to be a vertex of $e'$.
\end{proof}

\subsection{Some properties of escaping geodesics and irreducible paths}
\label{subsec:escaping}

\begin{lem}
	\label{lem:support contain}
	Let $F$ be a standard subcomplex of $\Si$.
	Let $u$ be an edge path (not necessarily geodesic) in $\Sii$ (or $\od^1$) from $x\in F$ to $y\notin F$. Let $u'$ be an $F$-escaping geodesic ending at $y$. Then $\supp(u')\subset\supp(u)$.
\end{lem}

\begin{proof}
	Let $u_1\subset\Sii$ be a geodesic path from $x$ to $y$. By Tits' solution to the word problem of Coxeter groups, we know $\supp(u_1)\subset\supp(u)$. By Lemma~\ref{lem:gate}, there is a geodesic $u_2$ from $x$ to $y$ that goes through $\prj_F(y)$. Let $u'_2$ be the subsegment of $u_2$  between $y$ and $\prj_F(y)$. Then $\supp(u')=\supp(u'_2)\subset\supp(u_2)=\supp(u_1)\subset \supp(u)$.
\end{proof}

We say a proper standard subcomplex $F$ of $\Si$ is \emph{maximal} if $\supp(F)=S\setminus\{s\}$ for some $s\in S$, where $S$ is the vertex set of $\Gamma$. In this case we define the \emph{type} of $F$ to be $\hat s$.

\begin{lem}
	\label{lem:irreducible}
	Let $u$ be an $E$-escaping geodesic in $\Sii$ (or $\od^1$) ending at $y\notin E$, where $E$ is a proper maximal standard subcomplex $\Si$ of type $\hat s$. Then $\supp(u)$ is irreducible.
\end{lem}

\begin{proof}
	Suppose $\supp(u)$ is not irreducible. Let $I$ be the irreducible component of $\supp(u)$ that contains $s$ (note that $s\in \supp(u)$ as $u$ is $E$-escaping), and let $I'=\supp(u)\setminus I$. As $I\perp I'$, we know there are words $w_1$ and $w_2$ on $S$ such that 
	$$\w(u)=w_1w_2\ \mathrm{in}\ W_\Gamma, \supp(w_1)=I'\ \mathrm{and}\ \supp(w_2)=I.$$ The word $w_1w_2$ gives a geodesic path $u_1u_2$ with the same endpoints as $u$. Note that $\ell(u)=\ell(u_1u_2)$. Thus $u_1u_2$ is a geodesic. As $\supp(u_1)=I'$, we know $s\notin I'$. Thus $\supp(u_1)\subset \supp(E)$. Hence $u_1\subset E$. Then $d(y,\vertex E)\le \ell(u_2)<\ell(u)$, contradicting that $u$ is $E$-escaping.
\end{proof}

We say a positive path $u=e_1e_2\cdots e_n$ is \emph{left irreducible} if $\supp(e_1\cdots e_i)$ is irreducible for $1\le i\le n$. We say $u$ is \emph{right irreducible} if $\supp(e_i\cdots e_n)$ is irreducible for $1\le i\le n$.

\begin{lem}
	\label{lem:left irreducible}
	Take a face $F\subset \Si$ dual to $B\in\cq_\Ga$. Suppose $u=u_1u_2\cdots u_n$ is a concatenation of elementary $B$-segments from $x$ to $y$. 
	\begin{enumerate}
		\item  Let $F_x$ be the face parallel to $F$ and containing $x$. We define $F_y$ similarly. Then each irreducible component of $\supp(F_x)$ is either contained in $\supp(u)$, or commutes with $\supp(u)$. A similar statement holds for $F_y$. Moreover,  $(\supp(F_x)\setminus \supp(u))=(\supp(F_y)\setminus \supp(u))$.
		\item  If $\supp(u_1\cdots u_i)$ is irreducible for any $1\le i\le n$, then $u$ is left irreducible. Similarly, if $\supp(u_i\cdots u_n)$ is irreducible for any $1\le i\le n$, then $u$ is right irreducible. 
	\end{enumerate}
\end{lem}

\begin{proof}
	For (1), we induct on $n$. The case $n=1$ follows from Lemma~\ref{lem:irreducible segment} (1). Suppose (1) is true for $u=u_1\cdots u_{n-1}$. Let $F_{n-1}$ and $F_{n}$ be faces parallel to $F$ that contain $\sft(u_{n-1})$ and $\sft(u_n)$ respectively. Take an irreducible component $I$ of $\supp(F_{n})$ which is not contained in $\supp(u)$. Then $I\nsubseteq \supp(u_n)$. By the $n=1$ case, $I\perp\supp(u_n)$ and $I$ is also an irreducible component of $\supp(F_{n-1})$. As $I\nsubseteq \supp(u)$, we know $I\nsubseteq \supp(u_1\cdots u_{n-1})$. By the $n-1$ case, we know $I\perp  \supp(u_1\cdots u_{n-1})$ and $I$ is also an irreducible component of $\supp(F_x)$. Thus $I\perp \supp(u_1\cdots u_n)$. The last part of Assertion 1 follows immediately.
	
	For (2), note that any elementary $(F_1,F_2)$-segment is both $F_1$-escaping and $F_2$-escaping. As $F_1$ and $F_2$ are proper maximal standard subcomplexes of the face containing them with one dimensional higher, by Lemma~\ref{lem:irreducible}, each $u_i$ is both left irreducible and right irreducible. Suppose $u_i=e_{i1}\cdots e_{ik}$. Thus it suffices to show $\supp(u_1\cdots u_{i-1}e_{i1})$ is irreducible for each $i\ge 2$.
	Let $F_i$ be the face parallel to $F$ containing $\sfs(u_i)$. Suppose $\supp(F_i)=I_1\sqcup I_2$ where $I_1$ is the union of irreducible components of $\supp(F_i)$ which are not contained in $\supp(u_1\cdots u_{i-1})$. Then $\supp(u_1\cdots u_{i-1})\perp I_1$ by (1). By Lemma~\ref{lem:irreducible segment} (1), $\supp(u_i)$ is the irreducible component of $I_1\cup I_2\cup \supp(e_{i1})$ that contains $\supp(e_{i1})$. 
	If $$\supp(e_{i1})\perp\supp(u_1\cdots u_{i-1}),$$ then $\supp(e_{i1})\perp I_2$ (as $I_2\subset \supp(u_1\cdots u_{i-1})$). As $I_1\perp I_2$ by definition,  $(\supp(e_{i1})\cup I_1)\perp I_2$.
	Thus $\supp(u_i)$ is the irreducible component of $I_1\cup \supp(e_{i1})$ that contains $\supp(e_{i1})$. As $$(I_1\cup \supp(e_{i1}))\perp \supp(u_1\cdots u_{i-1}),$$ we know $\supp(u_i)\perp \supp(u_1\cdots u_{i-1})$, and this is a contradiction to our assumption. The other part of (2) can be proved similarly.
\end{proof}
We record the following immediate consequence of Lemma~\ref{lem:left irreducible} (1) and Lemma~\ref{lem:conjugator1}.
\begin{lem}
	\label{cor:commute}
	Suppose $A_S$ is spherical.	Let $S_1,S_2\subset S$ such that $c_{S_1}$ and $c_{S_2}$ commute (see Section~\ref{subsec:Artin} for the definition of $c_{S_i}$). Then for any irreducible component $I_1\subset S_1$ and any irreducible component $I_2\subset S_2$, we know either $I_1\subset I_2$, or $I_2\subset I_1$, or $I_1\perp I_2$.
\end{lem}

\section{Relative Artin complexes, the labeled 4-cycle condition and the bowtie free condition}
\label{sec:relative}

In this section we introduce the notion of relative Artin complexes and discuss some of its basic properties. We introduce relative Artin complexes in Section~\ref{subsec:relative Artin}. Then we discuss two important properties on certain classes of relative Artin complexes, namely the bowtie free property, discussed in Section~\ref{subsec:bowtie free} and the labeled 4-cycle property, discussed in Section~\ref{subsec:labeled 4-cycle}. The bowtie free property will depend on a choice of order on suitable subsets of the generating sets of the Artin groups, and the definition of labeled 4-cycle property restricts ourselves to the case when the Dynkin diagram is a tree. 
The relation between these two properties are also discussed in Section~\ref{subsec:labeled 4-cycle}, where we prove they are equivalent when the Dynkin diagram is a linear graph.
We also collect some additional properties of relative Artin complexes in Section~\ref{subsec:labeled 4-cycle property} for later use.

\subsection{Relative Artin complexes}
\label{subsec:relative Artin}
\begin{definition}
Let $A_S$ be an Artin group with presentation graph $\Gamma$ and Dynkin diagram $\Lambda$. Let $S'\subset S$ such that $S'\neq\emptyset$. The \emph{$(S,S')$-relative Artin complex $\Delta_{S,S'}$} is defined to be the induced subcomplex of the Artin complex $\Delta_S$ of $A_S$ spanned by vertices of type $\hat s$ with the node $s\in S'$. In other words, vertices of $\Delta_{S,S'}$ correspond to left cosets of $\{A_{\hat s}\}_{s\in S'}$, and a collection of vertices span a simplex if the associated cosets have nonempty common intersection.

Let $\Gamma'$ and $\Lambda'$ be the induced subgraphs of $\Gamma$ and $\Lambda$ spanned by $S'$. Then we will also refer a $(S,S')$-relative Artin complex as $(\Gamma,\Gamma')$-relative Artin complex or $(\Lambda,\Lambda')$-relative Artin complex, and denote it by $\Delta_{\Gamma,\Gamma'}$ or $\Delta_{\Lambda,\Lambda'}$.
\end{definition}

The proof of the following lemma is identical to \cite[Lemma 4]{cumplido2020parabolic}.
\begin{lem}
	\label{lem:relative sc}
	Suppose $S'$ has at least three elements. Then the $(S,S')$-relative Artin complex is simply-connected. 
\end{lem}

%Note that this lemma can alternatively be proved by using the main result of \cite{bjorner2003nerves}, as the Cayley complex $X$ of $A_\Gamma$ is simply-connected, each vertex of $\Delta_\Gamma$ gives a subcomplex of $X$ spanned by vertices in the associated coset, these subcomplexes cover $X$ when $|S|\ge 3$ with nerve of the covering being $\Delta_\Gamma$, and any non-empty intersection of members from these collection is simply-connected \cite{lek}. 

\begin{lem}
	\label{lem:embedding}
Suppose $S_1\subset S$ and $S'_1\subset S'\cap S_1$. Then each left coset of $gA_{S_1}$ in $A_S$ and choice of base point $h\in gA_{S_1}$ induces a canonical embedding $\Delta_{S_1,S'_1}\to \Delta_{S,S'}$.
\end{lem}

\begin{proof}
First we look at the case of the identity coset $A_{S_1}$ in $A_S$ and we choose the base point to be identity.
Note that each left coset of form $g A_{S_1\setminus\{s\}}$ for $g\in A_{S_1}$ and $s\in S'_1$ gives a left coset of form $g A_{S\setminus\{s\}}$, which gives a map $\vertex \Delta_{S_1,S'_1}\to \vertex \Delta_{S,S'}$. Recall from \cite{lek} that $A_{S\setminus\{s\}}\cap A_{S_1}=A_{S_1\setminus\{s\}}$, thus $g A_{S\setminus\{s\}}\cap A_{S_1}=g A_{S_1\setminus\{s\}}$ for any $s\in S'_1$ and $g\in A_{S_1}$. Hence the map $\vertex \Delta_{S_1,S'_1}\to \Delta_{S,S'}$ is injective, which clearly extends to a simplicial embedding.
The general case of $gA_{S_1}$ with a choice of base point $h\in gA_{S_1}$ follows immediately as there is unique left translation of $A_S$ sending $A_{S_1}$ to $gA_{S_1}$ and identity to $h$.
\end{proof}

\begin{lem}
	\label{lem:link}
Let $\Delta$ be the $(\Lambda,\Lambda')$-relative Artin complex, and let $v\in \Delta$ be a vertex of type $\hat s$ with node $s\in \Lambda'$. Let $\Lambda_s$ and $\Lambda'_s$ be the induced subgraphs of $\Lambda$ and $\Lambda'$ respectively spanned all the nodes which are not equal to $s$. 
Then the following assertions are true.
\begin{enumerate}
	\item There is a type-preserving isomorphism between $\lk(v,\Delta)$ and the $(\Lambda_s,\Lambda'_s)$-relative Artin complex.
	\item Let $I_s$ be the union of connected components of $\Lambda_s$ that contain at least one component of $\Lambda'_s$. Then $\Lambda'_s\subset I_s$ and there is a type-preserving isomorphism between $\lk(v,\Delta)$ and the $(I_s,\Lambda'_s)$-relative Artin complex.
	\item Let $\{I_i\}_{i=1}^k$ be the connected components of $I_s$. Then $\lk(v,\Delta)=K_1*\cdots*K_k$ where $K_i$ is the induced subcomplex of $\lk(v,\Delta)$ spanned by vertices of type $\hat t$ with $t\in I_i$.
\end{enumerate}
\end{lem}

\begin{proof}
For (1), let $gA_{\hat s}$ be the coset associated with $v$. Then vertices in $\lk(v,\Delta)$ are in 1-1 correspondence with cosets of form $hA_{\hat t}$ with $t\in \Lambda'$ such that $hA_{\hat t}\cap gA_{\hat s}\neq\emptyset$ and $hA_{\hat t}\neq gA_{\hat s}$. Then $t\neq s$, hence $t\in \Lambda'_s$. As $A_{\hat t}\cap A_{\hat s}=A_{S\setminus\{s,t\}}$ (cf. \cite{lek}), we know $hA_{\hat t}\cap gA_{\hat s}$ is a left coset of $A_{S\setminus\{s,t\}}$ in $gA_{\hat s}$. This associated each vertex in $\lk(v,\Delta)$ a vertex in the $(\Lambda_s,\Lambda'_s)$-relative Artin complex, which extends to a type-preserving isomorphism. For (2), $\Lambda'_s\subset I_s$ is clear. Suppose $\Lambda_s=I_s\sqcup J_s$. Then $A_{\Lambda_s}=A_{I_s}\times A_{J_s}$. Thus for any $t\in \Lambda'_s$, there is a 1-1 correspondence between left $A_{\Lambda_s\setminus \{t\}}$ cosets in $A_{\Lambda_s}$ and left $A_{I_s\setminus\{t\}}$ cosets in $A_{I_s}$. Assertion 2 follows. The last assertion is a consequence of the fact that $A_{I_s}=\prod_{i=1}^k A_{I_i}$, and $gA_{I_s\setminus \{t\}}\cap hA_{I_s\setminus\{r\}}\neq\emptyset$ for any $g,h\in A_{I_s}$ and pair $t,r$ which are in different components of $I_s$. 
\end{proof}

The following is an immediate consequence of Lemma~\ref{lem:link}.
\begin{cor}
	\label{cor:adj}
Let $\Delta$ be the $(\Lambda,\Lambda')$-relative Artin complex. For $i=1,2,3$, let $x_i\in \Delta$ be a vertex of type $\hat s_i$ with node $s_i\in \Lambda'$. Suppose $s_1$ and $s_3$ are in different components of $\Lambda\setminus\{s_2\}$.

If $x_1$ is adjacent to $x_2$, and $x_2$ is adjacent to $x_3$, then $x_1$ is adjacent to $x_3$.
\end{cor}

\subsection{Bowtie free relative Artin complexes}
\label{subsec:bowtie free}

We will be interested in a special kind of relative Artin complex.

\begin{definition}
	An induced subgraph $\Lambda'$ of $\Lambda$ is \emph{admissible} if for any node $s\in \Lambda'$, if $s_1,s_2\in \Lambda'$ are in different connected components of $\Lambda'\setminus\{s\}$, then they are in different components of $\Lambda\setminus\{s\}$.
\end{definition}

\begin{lem}
	\label{lem:poset}
Suppose $\Lambda'$ is an admissible linear subgraph of the Dynkin diagram $\Lambda$ of $A_S$. We choose an orientation on $\Lambda'$ and let $\{s_i\}_{i=1}^n$ be nodes of $\Lambda'$ from left to right.  Let $\Delta$ be the $(\Lambda,\Lambda')$-relative Artin complex. Then the following are true.
\begin{enumerate}
	\item If $v\in\Delta$ is a vertex of type $\hat s_i$ with $2\le i\le n-1$, then $\lk(v,\Delta)=K_1*K_2$ where $K_1$ is the induced subcomplex of $\lk(v,\Delta)$ spanned by vertices of type $\hat s_j$ with $j<i$, and $K_2$ is the induced subcomplex of $\lk(v,\Delta)$ spanned by vertices of type $\hat s_j$ with $j>i$.
	\item Let $V$ be the vertex set of $\Delta$. For a vertex $x\in V$, we write $\type(x)=i$ if $x$ has type $\hat s_i$. Define a relation $<$ on  $V$ such that $x<y$ if $x$ and $y$ are adjacent in $\Delta$ and $\type(x)<\type(y)$. Then $(V,\le)$ is a graded poset.
\end{enumerate}
\end{lem}

\begin{proof}
Assertion 1 is a consequence of Assertion 3 of Lemma~\ref{lem:link}. The part of Assertion 2 concerning $(V,\le)$ being a poset follows from Assertion 1. To see $(V,\le)$ is graded, it suffices to show for $i<j<k$ and $g_i A_{\hat s_i}\cap g_k A_{\hat s_k}\neq\emptyset$, there exists $g_j\in A_\Lambda$ such that $g_i A_{\hat s_i}\cap g_j A_{\hat s_j}\neq\emptyset$ and $g_j A_{\hat s_j}\cap g_k A_{\hat s_k}\neq\emptyset$. However, $g_i A_{\hat s_i}\cap g_k A_{\hat s_k}=g_{ik} A_{\Lambda\setminus\{s_i,s_k\}}$ for some $g_{ik}\in A_\Lambda$ by \cite{lek}.
\end{proof}

%So we only prove Assertion 2. Take $\{v_1,v_2,v_3\}\subset V$ with $v_1<v_2$ and $v_2<v_3$. Let $g_i A_i,g_jA_j,g_kA_k$ be the left cosets corresponding to $\{v_1,v_2,v_3\}$. Then $g_iA_i\cap g_jA_j\neq\emptyset$ and $g_jA_j\cap g_kA_k\neq\emptyset$. Note that $g_iA_i\cap g_jA_j$ is a left coset of form $g_{ij}A_{ij}$ where $A_{ij}=A_{S\setminus\{s_i,s_j\}}$. Similarly, $g_kA_k\cap g_jA_j$ is a left coset of form $g_{jk}A_{jk}$ where $A_{jk}=A_{S\setminus\{s_j,s_k\}}$. Let $\{S_q\}_{q=1}^m$ be the collection of connected components of $\Lambda\setminus\{s_j\}$. Then $A_j$ is a direct sum of all the $A_{S_q}$'s. By our assumption, $s_i$ and $s_k$ are in different components of $\Lambda\setminus\{s_j\}$. We assume without loss of generality that $s_i\in S_1$ and $s_j\in S_2$. Then $A_{S_q}\subset A_{ij}$ for all $q\neq 1$ and $A_{S_q}\subset A_{jk}$ for all $q\neq 2$. Thus an arbitrary $A_{ij}$ left coset in $A_j$ has nonempty intersection with all the $A_{jk}$ left cosets. Thus $g_{ij}A_{ij}\cap g_{jk}A_{jk}\neq\emptyset$. Then $g_iA_i\cap g_kA_k\neq\emptyset$ and $v_i<v_k$. 

The following is motivated by the work of Haettel \cite{haettel2021lattices,haettel2022link}. We refer to Definition~\ref{def:bowtie free0} and the discussion before for the definition of bowtie in a poset, as well as bowtie free posets.
\begin{definition}
	\label{def:bowtie free}
Suppose $\Lambda'$ is an admissible linear subgraph of the Dynkin diagram $\Lambda$ of $A_S$ with consecutive nodes of $\Lambda'$ being $\{s_i\}_{i=1}^n$. We say the $(\Lambda,\Lambda')$-relative Artin complex is \emph{bowtie free} if the poset in Lemma~\ref{lem:poset} (2) is bowtie free. Note that the property of being bowtie free does not depend on the choice of one of the two orientation of $\Lambda'$.
\end{definition}

Note that the property of being bowtie free is defined for the relative Artin complex $\Delta_{\Lambda,\Lambda'}$ only when $\Lambda'$ is a line.
We record the following direct consequence of the definition.
%
%We record an immediate consequence of the definition.
%\begin{lem}
%	\label{lem:inherit}
%Let $\Lambda,\Lambda'$ be as in Definition~\ref{def:bowtie free}. Let $\Lambda''$ be a linear subgraph of $\Lambda'$. Let $\Lambda_0$ be an induced subgraph of $\Lambda$ containing $\Lambda'$.
%If the $(\Lambda,\Lambda')$-relative Artin complex is bowtie free, so is the $(\Lambda,\Lambda'')$-relative Artin complex and the $(\Lambda_0,\Lambda')$-relative Artin complex.
%\end{lem}

\begin{lem}
	\label{lem:inherit}
Suppose $\Lambda'$ is an admissible linear subgraph of the Dynkin diagram $\Lambda$. Suppose $\Delta_{\Lambda,\Lambda'}$ is bowtie free. Then
\begin{enumerate}
	\item if $s$ is an endpoint of $\Lambda'$, then the link of any vertex of type $\hat s$ in $\Delta_{\Lambda,\Lambda'}$ is bowtie free;
	\item if $\Lambda''$ is a linear subgraph of $\Lambda'$.  then $\Delta_{\Lambda,\Lambda''}$ is bowtie free.
\end{enumerate}
\end{lem}

We give several more criterion for relative Artin complexes to be bowtie free.
\begin{lem}
	\label{lem:bowtie free criterion}
Suppose $\Lambda'$ is an admissible linear subgraph of the Dynkin diagram $\Lambda$ of $A_S$ with consecutive nodes of $\Lambda'$ being $\{s_i\}_{i=1}^n$. Let $\Delta$ be the $(\Lambda,\Lambda')$-relative Artin complex with its vertex set $V$ endowed with the order in Lemma~\ref{lem:poset} (2). 
We assume that
\begin{enumerate}
	\item for each $v\in V$ of type $\hat s_1$ or $\hat s_n$, the vertex set of $\lk(v,\Delta)$ with the induced order from $(V,\le)$ is a bowtie free poset; 
	\item for any embedded 4-cycle $x_1y_1x_2y_2$ in $\Delta$ such that $x_1,x_2$ have type $\hat s_1$ and $y_1,y_2$ have type $\hat s_n$, there is a vertex $z\in V$ such that $x_i\le z\le y_j$ for $1\le i,j\le 2$.
\end{enumerate}
Then $(V,\le)$ is bowtie free. 
\end{lem}

\begin{proof}
Take embedded 4-cycle $x_1y_1x_2y_2$ in $\Delta$ with $x_i\le y_j$ for $1\le i,j\le 2$. We need to find $z\in V$ such that $x_i\le z\le y_j$ for $1\le i,j\le 2$. For $z\in V$, we write $\type(z)=i$ if $z$ is of type $\hat s_i$.
	
First we consider the case $\type(y_1)=\type(y_2)=n$.
If $\type(x_1)=\type(x_2)=1$, we are done by assumption.  If $\type(x_1)=1$ and $\type(x_2)>1$, then we take $x'_2\le x_2$ such that $\type(x'_2)=1$. By Assumption 2, there is $z'$ such that $\{x_1,x'_2\}< z'< \{y_1,y_2\}$. As $z',y_1,x_2,y_2$ form a 4-cycle in $\lk(x'_2,\Delta)$ and $\{z',x_2\}<\{y_1,y_2\}$, by Assumption 1, there is $z$ such that $\{z',x_2\}<z<\{y_1,y_2\}$. Note that $x_1<z'<z$, thus $x_1<z$ by Lemma~\ref{lem:poset}, as desired. If $\type(x_1)>1$ and $\type(x_2)>1$, then we take $x'_2\le x_2$ such that $\type(x'_2)=1$ and repeat the previous argument. 

Now we assume $\type(y_1)=\type(y_2)=n$ is not true. By a similar discussion as before, we only need to look at the case $\type(y_1)=n$ and $\type(y_2)<n$. Take $y'_2$ such that $y_2<y'_2$ and $\type(y'_2)=n$. Then the previous paragraph implies that there is $z'$ such that $\{x_1,x_2\}<z'<\{y_1,y'_2\}$. As $x_1,z',x_2,y_2$ form a 4-cycle in $\lk(y'_2,\Delta)$ and $\{x_1,x_2\}<\{z',y_2\}$, by Assumption 1 there is $z$ such that $\{x_1,x_2\}<z<\{z',y_2\}$. Note that $z<z'<y_1$, thus $z<y_1$ by Lemma~\ref{lem:poset}, as desired.
\end{proof}

%\begin{cor}
%	\label{cor:bowtie free criterion}
%Suppose $\Lambda'$ is an admissible linear subgraph of the Dynkin diagram $\Lambda$ of $A_S$ with consecutive vertices of $\Lambda'$ being $\{s_i\}_{i=1}^n$. Let $\Delta=\Delta_{\Lambda,\Lambda'}$ with its vertex set $V$ endowed with the order in Lemma~\ref{lem:poset} (2). 
%We assume that for any $\{x_1,x_2,y_1,y_2\}$ in $\Delta$ such that $x_i\le y_j$ for $1\le i,j\le 2$, $x_1,x_2$ have the same type and $y_1,y_2$ have the same type, there is a vertex $z\in V$ such that $x_i\le z\le y_j$ for $1\le i,j\le 2$.
%Then $(V,\le)$ is bowtie free. 
%\end{cor}

\begin{lem}
	\label{lem:connected intersection}
	Suppose $\Lambda'$ is an admissible linear subgraph of the Dynkin diagram $\Lambda$ of $A_S$ with consecutive nodes of $\Lambda'$ being $\{s_i\}_{i=1}^n$. Let $Y=\Delta_{\Lambda,\Lambda'}$ with its vertex set $V$ endowed with the order in Lemma~\ref{lem:poset} (2).
	Assume
	\begin{enumerate}
		\item for each $v\in V$ of type $\hat s_1$ or $\hat s_n$, the vertex set of $\lk(v,Y)$ with the induced order from $(V,\le)$ is a bowtie free poset; 
		\item for any pair of vertices $y_1,y_2$ in $Y$ of type $\hat s_n$, the induced subcomplex of $Y$ spanned by all vertices that are adjacent to both $y_1$ and $y_2$ is either empty or connected.
	\end{enumerate}
	Then $(V,\le)$ is bowtie free. 
\end{lem}

\begin{proof}
It suffices to verify Assumption 2 of Lemma~\ref{lem:bowtie free criterion}. Take a 4-cycle $x_1y_1x_2y_2$ in $Y$ such that $x_1$ and $x_2$ have type $\hat s_1$ and $y_1$ and $y_2$ have type $\hat s_n$. We assume the 4-cycle is embedded.
For $i=1,2$, let $Y_i$ be the induced subcomplex of $X$ spanned by vertices of $X$ that are adjacent to $y_i$. 
Note that $\{x_1,x_2\}\subset Y_1\cap Y_2$. As $Y_1\cap Y_2$ is connected and $x_1\neq x_2$, there exists an edge path $\omega\subset Y_1\cap Y_2$ from $x_1$ to $x_2$. In particular, each vertex of $\omega$ is not an isolated vertex in $Y_1\cap Y_2$. Let $\{w_i\}_{i=1}^k$ be consecutive vertices of $\omega$ with $w_1=x_1$ and $w_k=x_2$.
	
	For each $w_i$, define $\Delta_i$ to be the induced subcomplex of $\lk(w_i,Y)$ spanned by vertices $x$ with $\type(x)>\type(w_i)$. Let $\Delta'_i=\{w_i\}*\Delta_i$. The vertex set $V'_i$ of $\Delta'_i$ is endowed with the order inherit from $Y$. Note that $w_i$ is a lower bound of $V'_i$. By Lemma~\ref{lem:link}, if $\type(w_i)=1$, then $\lk(w_i,Y)$ is bowtie free by Assumption 1, hence the same is true for $\Delta'_i$. If $\type(w_i)>1$, then we take $v_i$ with $\type(v_i)=1$ and $v_i<w_i$. Note that $w_i\in\lk(v_i,Y)$ and  $\Delta_i$ is the induced subcomplex of $\lk(v_i,Y)$ spanned by vertices which are bigger than $w_i$. As $\lk(v_i,Y)$ is bowtie free, $\Delta_i$ is bowtie free.
	
	For $1\le i\le k$, let $z_i$ be the meet of $y_1$ and $y_2$ in $(V'_i,\le)$. Such $z_i$ exists as $y_1$ and $y_2$ has lower bound in $V'_i$ (which is $w_i$) and $(V'_i,\le)$ satisfies the assumption of Lemma~\ref{lem:posets}. It suffices to show $z_i=z_{i+1}$, as this implies $z=z_1=z_2=\cdots=z_k$ is adjacent in $Y$ to each of $\{x_1,x_2,y_1,y_2\}$, which implies Assumption 2 of Lemma~\ref{lem:bowtie free criterion}.
	We assume without loss of generality that $\type(w_i)<\type(w_{i+1})$. As elements in $V'_{i+1}$ is lower bounded by $w_{i+1}$, hence they are lower bounded by $w_i$. Thus $\Delta'_{i+1}\subset \Delta'_{i}$. As $w_{i+1}$ is a lower bound of $y_1$ and $y_2$ in $\Delta'_i$, we know $z_i\ge w_{i+1}$ in $\Delta'_i$. Thus $z_i\in \Delta'_{i+1}$. Then $z_i=z_{i+1}$ follows.
\end{proof}

Sometimes it is more convenient to use the following small variation of Lemma~\ref{lem:connected intersection}.
\begin{cor}
	\label{cor:connected intersection}
Suppose $\Lambda'$ is an admissible linear subgraph of the Dynkin diagram $\Lambda$ of $A_S$ with consecutive nodes of $\Lambda'$ being $\{s_i\}_{i=1}^n$. Let $Y=\Delta_{\Lambda,\Lambda'}$ with its vertex set $V$ endowed with the order in Lemma~\ref{lem:poset} (2). For $1\le k\le n$, let $\Lambda_k$ be the linear subgraph of $\Lambda'$ spanned by $\{s_i\}_{i=1}^k$. Assume
\begin{enumerate}
	\item for each $v\in V$ of type $\hat s_1$ or $\hat s_n$, the vertex set of $\lk(v,Y)$ with the induced order from $(V,\le)$ is a bowtie free poset; 
	\item there exists $k\le n$ such that the $(\Lambda,\Lambda_k)$-relative Artin complex $X$ satisfies the following property: for any pair of vertices $y_1,y_2$ in $Y$ of type $\hat s_n$, the induced subcomplex of $X$ spanned by all vertices that are adjacent to both $y_1$ and $y_2$ in $Y$ is either empty or connected.
\end{enumerate}
Then $(V,\le)$ is bowtie free. 
\end{cor}

\begin{proof}
We will view 	$\lk(y_1,Y)$ as a subcomplex of $Y$.
Note that for any vertex $x\in \lk(y_1,Y)\cap \lk(y_2,Y)$, there exists a vertex $y$ of type $\hat s_1$ such that $y\le x$. As $x\le y_1$ and $x\le y_2$, we know $y\le y_i$ for $i=1,2$. Then Lemma~\ref{lem:poset} implies that 
$y\in \lk(y_1,Y)\cap \lk(y_2,Y)$. It follows that $$\lk(y_1,Y)\cap \lk(y_2,Y)\neq \emptyset\ \textrm{if}\ \textrm{and}\ \textrm{only}\ \textrm{if}\lk(y_1,Y)\cap \lk(y_2,Y)\cap X\neq\emptyset,$$
and $\lk(y_1,Y)\cap \lk(y_2,Y)$ is connected if $\lk(y_1,Y)\cap \lk(y_2,Y)\cap X$ is connected. Now the corollary follows from Lemma~\ref{lem:connected intersection}.
\end{proof}

\subsection{The labeled 4-cycle condition on relative Artin complexes}
\label{subsec:labeled 4-cycle}
Let $\Delta_{\Lambda,\Lambda'}$ be a $(\Lambda,\Lambda')$-relative Artin complex, for a Dynkin diagram $\Lambda$ and $\Lambda'$ being an induced subgraph of $\Lambda$. We will introduce a combinatorial condition on $\Delta_{\Lambda,\Lambda'}$, which is only defined when $\Lambda'$ is a tree that is also admissible in $\Lambda$. This condition serves as a replacement for the bowtie free condition in Definition~\ref{def:bowtie free} when $\Lambda'$ is not a linear graph, and it does not use a partial ordering on the vertices of $\Delta_{\Lambda,\Lambda'}$ (which is needed for the bowtie free condition). Moreover, as we will see later in Lemma~\ref{lem:connect}, when $\Lambda'$ is
a straight line segment, the two definitions agree.

\begin{definition}
	\label{def:labeled 4-cycle}
We refer to Figure~\ref{fig:l4cycle}.	
Suppose $X$ is the $(\Lambda,\Lambda')$-relative Artin complex, with $\Lambda'$ being an admissible tree subgraph of $\Lambda$. Recall that each vertex $x$ of $X$ is labeled by its type, $\hat s$, with $s$ being a node of $\Lambda'$. In this case, $s$ is called the \emph{corresponding node} for the vertex $x$. We say $X$ satisfies the \emph{labeled 4-cycle} condition if for any induced 4-cycle in $X$ with consecutive vertices being $\{x_i\}_{i=1}^4$, the following conditions are satisfied:
\begin{enumerate}
	\item there exists a vertex $x\in X$ adjacent to each of $x_i$;
	\item if $x_i$ is of type $\hat s_i$ for the corresponding node $s_i\in \Lambda'$ for $1\le i\le 4$, then the vertex $x\in X$ in the previous item can be chosen so it satisfies, in addition, that $x$ is of type $\hat s$ for corresponding node $s\in \Lambda'$ such that the $s$ is contained the smallest subtree of $\Lambda'$ spanned by $\{s_i\}_{i=1}^4$. As $\Lambda'$ is a tree, equivalently $s$ is contained in the convex hull of $\{s_i\}_{i=1}^4$ in $\Lambda'$.
\end{enumerate}
\end{definition}
	\begin{figure}[h]
	\centering
	\includegraphics[scale=0.8]{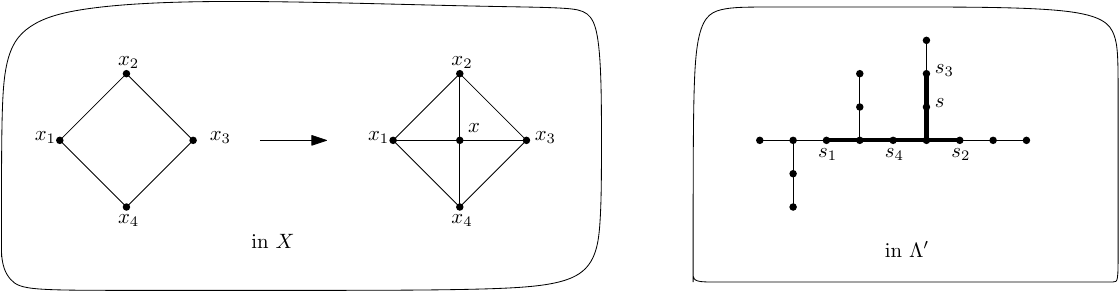}
	\caption{The left picture indicates that, given an induced 4-cycle, we can always find a vertex $x$ in the middle connecting to all vertices of the 4-cycle. The thickened subtree of the right picture is the convex hull of the nodes $\{s_1,s_2,s_3,s_4\}$ in the Dynkin diagram $\Lambda'$.}
	\label{fig:l4cycle}
\end{figure}

\begin{lem}
	\label{lem:connect}
If $\Lambda'$ is an admissible linear subgraph of $\Lambda$, then $\Delta_{\Lambda,\Lambda'}$ satisfies bowtie free condition if and only if it satisfies the labeled 4-cycle condition.
\end{lem}

\begin{proof}
Suppose $\Delta_{\Lambda,\Lambda'}$ satisfies bowtie free condition. Take an induced 4-cycle $x_1y_1x_2y_2$ in $\Delta_{\Lambda,\Lambda'}$. We cannot have $x_1<y_1<x_2$ or $x_1>y_1>x_2$, otherwise $x_1$ is adjacent to $x_2$ by Lemma~\ref{lem:link}, contradicting that we have an induced 4-cycle. Thus up to reordering, we can assume $x_i<y_j$ for $1\le i,j\le 2$. By Definition~\ref{def:bowtie free0}, there is a vertex $x\in \Delta_{\Lambda,\Lambda'}$ such that $x_i\le x\le y_j$ for $1\le i,j\le 2$. As the cycle $x_1y_1x_2y_2$ is induced, we know $x\notin \{x_1,y_1,x_2,y_2\}$. So the vertex $x$ is adjacent to each of $\{x_1,y_1,x_2,y_2\}$ in $\Delta_{\Lambda,\Lambda'}$. Moreover, the type of $x$ satisfies the desired condition in Definition~\ref{def:labeled 4-cycle} (2).

Conversely suppose $\Delta_{\Lambda,\Lambda'}$ satisfies the labeled 4-cycle condition. If we have $\{x_1,y_1,x_2,y_2\}$ satisfying $x_i< y_j$ for $1\le i,j\le 2$, then these 4 vertices form a 4-cycle in $\Delta_{\Lambda,\Lambda'}$. If $x_1$ and $x_2$ are comparable, or if $y_1$ and $y_2$ are comparable, then the bowtie free condition clearly holds for $\{x_1,y_1,x_2,y_2\}$. Now we assume $x_1$ and $x_2$ are not comparable, and $y_1$ and $y_2$ are not comparable. Then the 4-cycle is an induced 4-cycle. Suppose $x_1$ has type $\hat s_{x_1}$ for node $s_{x_1}\in \Lambda'$. Similarly we define nodes $s_{x_2},s_{y_1}$ and $s_{y_2}$ of $\Lambda'$. We assume without loss of generality that the segment $\Lambda''$ from $s_{x_1}$ to $s_{y_2}$ contains all of $\{s_{x_1},s_{x_2},s_{y_1},s_{y_2}\}$. Then the labeled 4-cycle condition implies that there is a vertex $z$ adjacent to each vertex of the 4-cycle such that $z$ has type $\hat s_z$ with the node $s_z\in \Lambda''$. Clearly $x_1<z<y_2$. Now we show $x_2<z$. If this is not true, as $x_2$ and $z$ are adjacent, we must have $x_2>z$, then $x_1<x_2$, contradicting the assumption that $x_1$ and $x_2$ are not comparable. Similarly, $z<y_1$. Thus the bowtie free condition is satisfied. 
 \end{proof}

%
%\begin{definition}
%	\label{def:order tripod}
%Let $\Lambda'$ be an admissible tripod subgraph of $\Lambda$ with the valence one vertices of $\Lambda'$ being $\{s_i\}_{i=1}^3$ and the center vertex of $\Lambda'$ being $s$. Let $\Delta$ be the $(\Lambda,\Lambda')$-relative Artin complex.
%Each $s_i$ for $1\le i\le 3$ will induce an order on the vertex set $V$ of $\Delta$. We now define the order induced by $s_1$.
%For two vertices $x_1,x_2\in \Delta$ of types $\hat t_1$ and $\hat t_2$ with $t_1,t_2\in \Lambda'$, $x_1<x_2$ if $t_2$ lies in the geodesic in $\Lambda'$ from $t_1$ to $s_1$ (it is possible that $t_2=s_1$). As $\Lambda'$ is an admissible subgraph, we know $(V,\le)$ is an order by Lemma~\ref{lem:link}.
%\end{definition}

Now we describe a construction for converting a 4-cycle in the Artin complex $\Delta=\Delta_\Lambda$ of $A_\Lambda$ to a concatenation of four words in $A_\Lambda$. These four words are well-defined up to an appropriate notion of equivalence, explained below.

\begin{cons}
	\label{def:4cycle}
A \emph{chamber} in $\Delta$ is a top-dimensional simplex in $\Delta$. There is a 1-1 correspondence between chambers in $\Delta$ and elements in $A_\Lambda$.	
Let $\{x_i\}_{i=1}^4$ be four consecutive vertices of an induced 4-cycle $\omega$ in $\Delta$ and suppose $x_i$ has type $\hat a_i$ with node $a_i\in \Lambda$. 
For each edge of $\omega$, take a chamber of $\Delta$ containing this edge. We name these chambers by $\{\Theta_i\}_{i=1}^4$ with $\Theta_1$ containing the edge $\overline{x_1x_2}$. Each $\Theta_i$ gives an element $g_i\in A_\Lambda$. Then for $i\in \mathbb Z/4\mathbb Z$, $g_i=g_{i-1}w_{i}$ for $w_i\in A_{\hat a_i}$ (recall that $A_{\hat a_i}$ is defined to be $A_{S\setminus\{a_i\}}$). This defines four words $\{w_1,w_2,w_3,w_4\}$ from the induced $4$-cycle, such that $w_1w_2w_3w_4=1$ in $A_\Lambda$.

The word $w_1w_2w_3w_4$ depends on the choice of $\{\Theta_i\}_{i=1}^4$. A different choice would lead to a word of form $u_1u_2u_3u_4$ such that there exist elements $q_i\in A_{S\setminus\{a_i,a_{i+1}\}}$ such that $u_i=q^{-1}_{i-1}w_i q_i$ for $i\in\mathbb Z/4\mathbb Z$. In this case we will say the words $u_1u_2u_3u_4$ and $w_1w_2w_3w_4$ are equivalent. If, in addition, there exists a parabolic subgroup $A'$ of $A_\Lambda$ such that $q_i\in A_{S\setminus\{a_i,a_{i+1}\}}\cap A'$, then we say $w_1w_2w_3w_4$ is \emph{equivalent} to $u_1u_2u_3u_4$ in $A'$.
\end{cons}

Our next goal is to prove Proposition~\ref{prop:4-cycle}. We start with a preparatory lemma.

\begin{lem}
	\label{lem:4cycle}
Let $\Lambda$ be an arbitrary Dynkin diagram. Let $\omega$ be a 4-cycle in the Artin complex $\Delta_\Lambda$ with consecutive vertices $\{x_i\}_{i=1}^4$, with $x_i$ being of type $\hat a_i$ for node $a_i\in \Lambda$. Suppose $a_1\neq a_3$. Then there exists a vertex $x'_1\in\Delta_\Lambda$ such that $x'_1$ is adjacent to $\{x_1,x_2,x_4\}$, and $x'_1$ is of type $\hat a_3$.
\end{lem}

\begin{proof}
	 Let $w_1w_2w_3w_4$ be a word associated this 4-cycle in $\Delta_\Lambda$ as in Construction~\ref{def:4cycle}.
	Let $\bC$ be the Coxeter complex for $W_\Lambda$. Let $\Si$ and $\od$ be the Davis complex and oriented Davis complex of $A_\Lambda$.
	There is a 1-1 correspondence between vertices of $\bC$ and proper maximal standard subcomplexes of $\Si$. 
	Let $\bar \pi:X\to \bC$ be as in Definition~\ref{def:barpi}.
Let $\bar z_i=\bar\pi(x_i)$. Then $\bar z_i$ has type $\hat a_i$. Let $C_i$ be the proper maximal standard subcomplex in $\Si$ associated with $\bar z_i$ (in the sense of Definition~\ref{def:associated subcomplex}). Let $\widehat{C}_i$ be the subcomplex in $\od$ corresponding to $C_i$.
As $a_1\neq a_3$, we know $C_1\neq C_3$.  Then $w_1w_2w_3w_4$ gives a null-homotopic edge loop in $\s_{\Lambda}$, which lifts to null-homotopic edge loop $P_1P_2P_3P_4$ in $\od$ with $P_i\subset \widehat C_i$.
Let $E_{i,j}=\p_{C_i}(C_j)$ for $1\le i,j\le 4$. 
Let $I=\supp(E_{1,3})$. Suppose $I=I_1\sqcup I_2$ where $I_1$ is the union of irreducible components of $I$ that are contained in $\supp(C_2)=S\setminus\{a_2\}$. By Lemma~\ref{lem:two cells}, there is a vertex $\bar x'\in C_2\cap E_{3,1}$ such that $\bar x=p(\bar x')\in C_2\cap E_{1,3}$, where $p:E_{3,1}\to E_{1,3}$ is the parallel translation. 
For $i=1,2$, let $U_i$ be the face in $E_{1,3}$ containing $\bar x$ such that $\supp(U_i)=I_i$. Then $E_{1,3}=U_1\times U_2$ and $U_1\subset C_{1}\cap C_{2}$. As $\bar x\in C_2\cap C_1$ and $\bar x'\in C_2\cap C_3$, up to passing to a word equivalent to $w_1w_2w_3w_4$, we can assume $P_1$ ends at $\bar x$, $P_2$ travels from $\bar x$ to $\bar x'$ and $P_3$ starts at $\bar x'$.

For $2\le i\le 4$, let $Q_i=\Pi_{\widehat C_1}(P_i)$. Then $Q_i\subset \widehat C_1\cap \widehat C_i$ for $i=2,4$, and $Q_3\subset E_{1,3}$. Then $\Pi_{\widehat C_1}(P_1P_2P_3P_4)=P_1Q_2Q_3Q_4$. As $E_{1,3}=U_1\times U_2$, we know that $\widehat E_{1,3}=\widehat U_1\times \widehat U_2$, hence $Q_3$ is homotopic rel endpoints in $\widehat E_{1,3}\subset \widehat C_1$ to $Q_{31}Q_{32}$ with $$Q_{31}\subset \widehat U_1\subset \widehat C_{1}\cap \widehat C_{2}\ \mathrm{and}\ \supp(Q_{32})\subset \supp(\widehat U_2).$$ As $P_1Q_2Q_3Q_4$ is null-homotopic in $\widehat C_1$, we know $P_1$ is homotopic rel endpoints in $\widehat C_1$ to $$\bar Q_4\bar Q_{32}(\bar Q_{31}\bar Q_2),$$ where $\bar Q_i$ denotes the inverse path of $Q_i$. Then $w_1=w_{11}w_{12}w_{13}$ with $w_{11}=\w(\bar Q_4),w_{12}=\w(\bar Q_{32})$ and $w_{13}=\w(\bar Q_{31}\bar Q_2)$. Let $S$ be the collection of nodes in $\Lambda$. Then $\supp(w_{11})\subset S\setminus\{a_1,a_4\}$, and $\supp(w_{13})\subset S\setminus\{a_1,a_2\}$. Moreover, $\supp(w_{12})\subset \supp(\widehat U_2)\subset S\setminus\{a_1,a_3\}$ by Lemma~\ref{lem:two cells} (2). Thus $w_1w_2w_3w_4$ is equivalent to
$
w_{12}(w_{13}w_2)w_3(w_4w_{11})
$ in the sense of Construction~\ref{def:4cycle}. Let $g_iA_{\hat a_i}$ be the left coset associated with $x_i$ for $1\le i\le 4$.
As $w_{12}\subset A_{S\setminus\{a_1,a_3\}}$, this implies that there exists a coset $gA_{S\setminus\{a_1,a_3\}}$ contained in $g_1A_{\hat a_1}$ such that $gA_{S\setminus\{a_1,a_3\}}$ has nonempty intersection with $g_iA_{\hat a_i}$ for $i=2,4$. Now the lemma follows. 
\end{proof}

\begin{prop}
	\label{prop:4-cycle}
	Suppose $\Lambda'$ is an admissible tree subgraph of $\Lambda$. Then $\Delta=\Delta_{\Lambda,\Lambda'}$ satisfies the labeled 4-cycle condition if and only if 
	for any maximal linear subgraph $\Lambda''\subset \Lambda'$, $\Delta_{\Lambda,\Lambda''}$ satisfies the labeled 4-cycle condition. 
\end{prop}

\begin{proof}
The ``only if'' direction is clear. We prove the ``if'' direction.
Let $\{x_i\}_{i=1}^4$ be four consecutive vertices of an induced 4-cycle $\omega$ in $\Delta$ and suppose $x_i$ has type $\hat a_i$ with node $a_i\in \Lambda'$. We view $\Delta$ as a subcomplex of $\Delta_\Lambda$.  We will be using the following notation. Given nodes $b_1,\ldots,b_k \in \Lambda'$, let $\Lambda'_{b_1,\ldots,b_k}$ be the smallest subtree of $\Lambda'$ containing each of $b_i$ for $1\le i\le k$.
	
%We claim that if the nodes satisfy $a_1\neq a_3$, then there exists a vertex $x'_1\in X$ such that $x'_1$ is adjacent to $x_1,x_2,x_4$ such that $x'_1$ is of type $\hat a_3$. Now we prove this claim. 

If $a_1=a_3$ and $a_2=a_4$, then the proposition follows from our assumption, as these two nodes are contained in a maximal linear subgraph of $\Lambda'$.

Now we consider the case $a_2=a_4$ and $a_1\neq a_3$. As $\omega$ is an induced 4-cycle, by Lemma~\ref{lem:link}, $a_1$ and $a_3$ are in the same component of $\Lambda'\setminus\{a_2\}$. By the assumption of the proposition, it suffices to consider the case when $\Lambda'_{a_1,a_2,a_3}$ is not a linear graph,. Then $a_1,a_2,a_3$ are the valence one nodes of a tripod subgraph of $\Lambda''$ with center $a$ (as in Assertion 2).  Lemma~\ref{lem:4-cycle} implies that there is a vertex $x'_1$ of type $\hat a_3$ such that $x'_1$ is adjacent to $x_1,x_2,x_4$. As $a_2$ and $a_3$ are contained in a maximal linear subgraph of $\Lambda'$, by our assumption, there is a vertex $z\in \Delta$ of type $\hat a'$ with node $a'\in \Lambda'_{a_2,a_3}$ such that $z$ is adjacent to $x'_1,x_2,x_3,x_4$.

If $a'\in \Lambda'_{a,a_2}$, then $\{a_1,a,a',a_2\}$ is contained in a maximal linear subgraph of $\Lambda'$. By Lemma~\ref{lem:connect}, $\Delta_{\Lambda,\Lambda'_{a_1,a_2}}$ is bowtie free. We consider the linear order on $\Lambda'_{a_1,a_2}$ from $a_1$ to $a_2$, which induces an order on the set of vertices of $\Delta_{\Lambda,\Lambda'_{a_1,a_2}}$. Then both $x_1$ and $z$ are lower bounds of the set $\{x_2,x_4\}$. Thus there is a vertex $y\in \Delta$ such that $\{x_1,z\}\le y\le \{x_2,x_4\}$ in $\Delta_{\Lambda,\Lambda'_{a_1,a_2}}$. In particular, $y$ adjacent or equal to each of $x_1,x_2,z,x_4$, and $y$ is of type $\hat a_y$ with node $a_y\in \Lambda'_{a',a_2}$. If $y$ and $z$ have the same type, then $y=z$. If $y$ and $z$ have different types, then $a_y$ and $a_3$ are in different components of $\Lambda\setminus\{a'\}$, implying $y$ and $x_3$ are adjacent by Lemma~\ref{lem:link}. Thus in either case $y$ is adjacent to each of $x_1,x_2,x_3,x_4$.

If $a'\notin \Lambda'_{a,a_2}$, then we consider the 4-cycle $x_1x_2zx_4$ and repeat the previous argument. Namely Lemma~\ref{lem:4-cycle} implies that there is a vertex $x''_1$ adjacent to $x_1,x_2,x_4$ such that $x''_1$ has type $a'$. As before our assumption implies that there is a vertex $z'$ adjacent to $x''_1,x_2,z,x_4$ such that $z'$ is of type $a''$ with $a''\subset \Lambda'_{a',a_2}$. As $a_3$ and $a''$ are in different connected components of $\Lambda\setminus\{a'\}$, we know $x_3$ and $z''$ are adjacent by Lemma~\ref{lem:link}. If $a''\in \Lambda_{a,a_2}$, then we are done by the previous paragraph. Otherwise we repeat this argument on the 4-cycle $x_1x_2z''x_4$ to produce 4-cycle $x_1x_2z'''x_4$ with vertex $z'''$ adjacent to $x_3$ and the type of $z'''$ gives a node in $\Lambda'_{a_2,a_3}$ which is further away from the node $a_3$.
After finitely many iterations, we find a vertex $y$ with its type in $\Lambda_{a,a_2}$ which is adjacent to each of $x_1,x_2,x_3,x_4$. This justifies the labeled 4-cycle condition.

It remains to treat the case that $\{a_1,a_2,a_3,a_4\}$ are mutually distinct. The case when $\Lambda'_{a_1,a_2,a_3,a_4}$ is linear follows from our assumption. Now assume $\Lambda'_{a_1,a_2,a_3,a_4}$ is a tripod with center node $a$. Without loss of generality, we assume $a_1,a_2,a_3$ are valence one nodes of this tripod. Lemma~\ref{lem:link} implies that $a_4\notin \Lambda_{a_1,a_3}$ (otherwise $x_1$ is adjacent to $x_3$). Then $a_4\in \Lambda'_{a,a_2}$. Lemma~\ref{lem:4-cycle} implies that there is a vertex $x'_2$ of type $\hat a_4$ such that $x'_2$ is adjacent to $x_1,x_2,x_3$. If $x'_2$ is adjacent to $x_4$, then we are done, otherwise $x_1x'_2x_3x_4$ is an induced 4-cycle in the 1-skeleton of $\Delta$. Now the previous paragraph implies that there is a vertex $y\in \Delta$ of type $\hat a_y$ with corresponding node $a_y\in \Lambda'_{a_1,a_3,a_4}$ such that $y$ is adjacent to $x_1,x'_2,x_3,x_4$. Note that the nodes $a_2$ and $a_y$ are in different components of $\Lambda\setminus\{a_4\}$, thus $y$ is adjacent to $x_2$ by Lemma~\ref{lem:link}, as desired.

The remaining case is $\Lambda_{a_1,a_2,a_3,a_4}$ is neither linear or a tripod subgraph. We induct on the distance from $a_4$ to the tripod subgraph $\Lambda'_{a_1,a_2,a_3}$. The base case when the distance is 0 is already treated before. Lemma~\ref{lem:4-cycle} implies that there is a vertex $x'_2$ of type $\hat a_4$ such that $x'_2$ is adjacent to $x_1,x_2,x_3$. If $x'_2$ is adjacent to $x_4$, then we are done, otherwise $x_1x'_2x_3x_4$ is an induced 4-cycle in the 1-skeleton of $\Delta$. Now the previous discussion implies that there is a vertex $x'_4\in \Delta$ of type $\hat a'_4$ with the corresponding node $a'_4\in \Lambda'_{a_1,a_3,a_4}$ such that $x'_4$ is adjacent to $x_1,x'_2,x_3,x_4$. Note that $a'_4\notin \Lambda'_{a_1,a_3}$, otherwise $x_1$ and $x_3$ are adjacent by Lemma~\ref{lem:link}, contradicting that $\omega$ is an induced cycle.
	As $a'_4\neq a_4$, $a'_4$ is smaller distance to $\Lambda'_{a_1,a_2,a_3}$ compared to $a_4$. We divide into two subcases.
	\begin{enumerate}
		\item If $a'_4\notin \Lambda'_{a_1,a_2,a_3}$, then by induction there is a vertex $y$ adjacent to each of $x_1,x_2,x_3,x'_4$ such that $y$ has type $\hat a_y$ with node $a_y\in \Lambda'_{a_1,a_2,a_3,a'_4}$. As $a'_4$ separates $a_4$ from $a_y$ in $\Lambda'$, we know from Lemma~\ref{lem:link} that $y$ is adjacent to $x_4$, thus the labeled 4-cycle condition is satisfied.
		\item If $a'_4\in \Lambda'_{a_1,a_2,a_3}$, then $a'_4\in \Lambda'_{a_2,a}$ with $a$ being the center of $\Lambda'_{a_1,a_2,a_3}$ (as $a'_4\notin \Lambda'_{a_1,a_3}$). By the previous discussion, there is a vertex $y$ adjacent to each of $x_1,x_2,x_3,x'_4$ such that $y$ has type $\hat a_y$ with node $a_y\in \Lambda'_{a_1,a_3,a'_4}$. Then it is still true that  $a'_4$ separates $a_4$ from $a_y$ in $\Lambda'$, we know from Lemma~\ref{lem:link} that $y$ is adjacent to $x_4$, as desired.
	\end{enumerate}
This finishes the proof.
\end{proof}

We also record the following consequence of the above proof.
\begin{lem}
	\label{lem:tripod}
	Suppose $\Lambda'$ is an admissible tree subgraph of $\Lambda$. Let $\{x_i\}_{i=1}^4$ be four consecutive vertices of an induced 4-cycle $\omega$ in $\Delta=\Delta_{\Lambda,\Lambda'}$ and suppose $x_i$ has type $\hat a_i$ with $a_i\in \Lambda'$. We assume
	\begin{enumerate}
		\item $\Delta$ satisfies the labeled 4-cycle condition;
		\item  $\{a_1,a_2,a_3\}$ are the valence one nodes of a tripod subgraph $\Lambda''$ of $\Lambda'$ and $a_2=a_4$.
	\end{enumerate}
Then there is a vertex $y\in\Delta$ adjacent to each of $x_i$ such that $y$ has type $\hat a_y$ with the corresponding node $a_y$ contained in the linear subgraph between the valence three node of $\Lambda''$ and the node $a_2$.
\end{lem}

\subsection{An inherited property of the labeled 4-cycle condition}
\label{subsec:labeled 4-cycle property}
This subsection discusses when the labeled 4-cycle condition are inherited by certain subcomplexes of a relative Artin complex.
We start with an elementary lemma, whose proof is left to the reader.
\begin{lem}
	\label{lem:subgraph}
	Let $\Lambda$ be a connected simplicial graph. Let $\Lambda'\subset \Lambda$ be an induced subgraph which is also connected. Then there exists a finite sequence of graphs $\Lambda_1=\Lambda,\Lambda_2,\ldots,\Lambda_n=\Lambda'$ such that for each $1\le i\le n-1$, there is a node $s_i\in \Lambda_i$ such that $\Lambda_{i}\setminus\{s_i\}$ has only one component, which is $\Lambda_{i+1}$.
\end{lem}

\begin{prop}
	\label{prop:inherit}
	Let $\Lambda'$ be an admissible linear subgraph of $\Lambda$. Suppose $\Delta_{\Lambda,\Lambda'}$ is bowtie free. Then for any induced subgraph $\Lambda_0$ of $\Lambda$ satisfying $\Lambda'\subset \Lambda_0$, $\Delta_{\Lambda_0,\Lambda'}$ is bowtie free.
\end{prop}

\begin{proof}
	Note that if $\Lambda'_0$ is the connected component of $\Lambda_0$ that contains $\Lambda'$, then $\Delta_{\Lambda_0,\Lambda'}=\cong \Delta_{\Lambda'_0,\Lambda'}$. Thus it suffices to consider the case when $\Lambda_0$ is connect. By a similar reasoning, we can assume $\Lambda$ is connect. We use a double induction on the the number of nodes in $\Lambda$, then on the number of nodes in $\Lambda'$. 
	The base case when $\Lambda$ has one node is trivial. Now we assume the statement of the proposition is true for any $\Lambda$ with $n-1$ nodes and any admissible linear subgraph of $\Lambda'$ of $\Lambda$. Take $\Lambda$ with $n$ nodes such that $\Delta_{\Lambda,\Lambda'}$ is bowtie free. We wish to show $\Delta_{\Lambda_0,\Lambda'}$ is bowtie free. 
	By Lemma~\ref{lem:subgraph} (applying to $\Lambda_0\subset \Lambda$) and the induction assumption, it suffices to consider the special case when $\Lambda_0$ is the only component of $\Lambda\setminus\{s\}$ for some node $s\in \Lambda\setminus\Lambda'$. Then Lemma~\ref{lem:link} implies that $\Delta_{\Lambda_0,\Lambda'}$ can be identified with $\lk(x,\Delta_{\Lambda,\Lambda'})$ for a vertex $x$ of type $\hat s$. 
	
	Now we employ an inner layer of induction on the number of nodes in $\Lambda'$. The base case when $\Lambda'$ has one node is clear. For more general $\Lambda'$, we will verify that the assumptions of Lemma~\ref{lem:bowtie free criterion} hold for $\Delta_{\Lambda_0,\Lambda'}$. For Assumption 1 of Lemma~\ref{lem:bowtie free criterion}, we label consecutive nodes of $\Lambda'$ by $\{s_i\}_{i=1}^n$, and endow vertices of $\Delta_{\Lambda,\Lambda'}$ with the associated order.
	Let $\mathsf{C}_{s_1}$ be the component of $\Lambda\setminus \{s_1\}$ that contains the remaining nodes of $\Lambda'$ and let $\Lambda'_{s_1}=\mathsf{C}_{s_1}\cap \Lambda'$. Then for any vertex $x'\in \Delta_{\Lambda_0,\Lambda'}$ of type $\hat s_1$, by Lemma~\ref{lem:link} (2), $\lk(x',\Delta_{\Lambda_0,\Lambda'})$ can be identified with the $(\Lambda_0\cap \mathsf{C}_{s_1},\Lambda'_{s_1})$-relative Artin complex. As we are assuming $\Delta_{\Lambda,\Lambda'}$ is bowtie free, it follows from the definition that $\Delta_{\Lambda,\Lambda'_{s_1}}$ is bowtie free. As $\Lambda'_{s_1}$ has less nodes compared to $\Lambda'$, by the inner layer induction assumption, the $(\Lambda_0\cap \mathsf{C}_{s_1},\Lambda'_{s_1})$-relative Artin complex is bowtie free, which justifies Assumption 1 of Lemma~\ref{lem:bowtie free criterion}.
	
	Now we justify Assumption 2 of Lemma~\ref{lem:bowtie free criterion}. 
	Let $x_1,x_2,x_3,y_4$ be four vertices in $\Delta_{\Lambda_0,\Lambda'}$. Suppose $x_1$ and $x_3$ are of type $\hat s_1$; and $x_2$ and $x_4$ are of type $\hat s_n$. We will also assume $\{x_i\}_{i=1}^4\subset \lk(x,\Delta_{\Lambda,\Lambda'})$, where $x$ is the vertex of type $\hat s$ in the previous paragraph and we view $\lk(x,\Delta_{\Lambda,\Lambda'})$ as a subcomplex of $\Delta_{\Lambda,\Lambda'}$. As $\Delta_{\Lambda,\Lambda'}$ is bowtie free, there exists $z\in \Delta_{\Lambda,\Lambda'}$ of type $\hat s_{i_0}$ with $1<i_0<n$ such that $z$ is adjacent to each of $\{x_1,x_2,x_3,x_4\}$. If $z$ is adjacent to $x$, then we are done as $\Delta_{\Lambda_0,\Lambda'}$ is identified with $\lk(x,\Delta_{\Lambda,\Lambda'})$. Now we assume $z$ is not adjacent to $x$.
	
	Note that $\{x_1,x_2,x_3,x_4,x,z\}$ form the vertex set of an embedded boundary of octahedron in $\Delta_{\Lambda,\Lambda'}$, where the 4-cycle $\omega$ through $x_1\to x_2\to x_3\to x_4\to x_1$ is in the belt of the octahedron, $z$ is the north pole and $x$ is the south pole.
	For $1\le i\le 4$, let $C_{\bar x_i}$ be the proper maximal standard subcomplex in the Davis complex $\Si$ associated with $\bar\pi(x_i)=\bar x_i$ in the sense of Definition~\ref{def:associated subcomplex} (where $\bar\pi$ is defined in Definition~\ref{def:barpi}), and let $\whC_{\bar x_i}$ be the associated subcomplex in $\widehat \Si$. Similarly we define $C_{\bar x}$ and $C_{\bar z}$, as well as their hat versions. We write $C_i=C_{\bar x_i}$ and $\whC_i=\whC_{\bar x_i}$.
	
	Let $e_i=\overline{x_ix_{i+1}}$. For $1\le i\le 4$, let $K_i$ (resp. $K'_i$) be the triangle spanned by $e_i$ and $x$ (resp. $z$). For each $i$, let $\Delta_i$ (resp. $\Delta'_i$) be a chamber (i.e. a top-dimensional simplex) of $\Delta_\Lambda$ that contains $K_i$ (resp. $K'_i$). As there is a 1-1 correspondence between chambers in $\Delta_\Lambda$ and elements in $A_\Lambda$, we let $g_i,g'_i\in A_\Lambda$ be the element associated with $\Delta_i,\Delta'_i$. Suppose $g_{i+1}=g_iw_i$, $g'_{i+1}=g'_iw'_i$ and $g'_i=g_ih_i$. Then $w_1,w_3\in A_{S\setminus\{s,s_n\}}$, $w_2,w_4\in A_{S\setminus\{s,s_1\}}$,  $w'_1,w'_3\in A_{S\setminus\{s_{i_0},s_n\}}$, $w'_2,w'_4\in A_{S\setminus\{s_{i_0},s_1\}}$, and $h_i\in A_{S\setminus\{s_1,s_n\}}$.
	
In the rest of the proof, we will show that  $w_1w_2w_3w_4$ is equivalent in $A_{S\setminus\{s\}}$  (cf. Construction~\ref{def:4cycle}) to a word $u_1u_2u_3u_4$ such that $u_j\in A_{S\setminus\{s_{i_0}\}}$ for $1\le j\le 4$. Note that this statement implies that up to modifying $\Delta_i$,  $\{g_i\}_{i=1}^4$ are contained in the same left $A_{\hat s_{i_0}}$-coset (denoted by $g'A_{\hat s_{i_0}}$) which has non-empty intersection with the left $A_{\hat s}$-coset associated with $x$. Then the vertex in $\Delta_{\Lambda,\Lambda'}$ corresponding to $g'A_{\hat s_{i_0}}$ is contained in $\lk(x,\Delta_{\Lambda,\Lambda'})$, and this vertex is adjacent to each of $x_i$ for $1\le i\le 4$. This implies Assumption 2 of Lemma~\ref{lem:bowtie free criterion}.
	
	\begin{figure}
		\centering
		\includegraphics[scale=1]{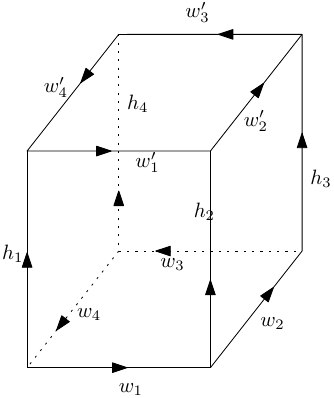}
		\caption{Words on the 1-skeleton of the cube.}
		\label{fig:1}
	\end{figure}
	
	We can put these words on the 1-skeleton $Z^1$ of a 3-cube $Z$ as in Figure~\ref{fig:1}. These words induces a map from $Z^1$ to the Salvetti complex $\s$ of $A_\Lambda$, which extends to $Z^{(2)}\to\s$ as the boundary of each 2-face gives null-homotopic loops. As $Z^{(2)}$ is simply-connected, this map lifts to $Z^{(2)}\to \od$. Thus $w_1w_2w_3w_4$ (resp. $w'_1w'_2w'_3w'_4$) gives four paths $P_1P_2P_3P_4$ (resp. $P'_1P'_2P'_3P'_4$) whose concatenation is a null-homotopic loop in $\od$ such that $P_i\in \whC_i\cap \whC_{\bar x}$ (resp. $P'_i\in \whC_i\cap \whC_{\bar z}$). Moreover, for each $i\in \mathbb Z/4\mathbb Z$, $h_i$ gives a path $Q_i\subset \whC_i\cap \whC_{i-1}$ from the starting point of $P_i$ to the starting point of $P'_i$, such that $Q_iP'_i$ is homotopic rel endpoints to $P_iQ_{i+1}$ in $\od$.

	We use $\bar Q_i$ to denote the inverse path of $Q_i$. By the previous discussion, $$\Pi_{\whC_{\bar{x}}}(Q_iP'_i\bar Q_{i+1}\bar P_i)=\Pi_{\whC_{\bar{x}}}(Q_i)\Pi_{\whC_{\bar{x}}}(P'_i)\Pi_{\whC_{\bar{x}}}(\bar Q_{i+1})\bar P_i$$ is a null-homotopic loop in $\widehat C_{\bar x}$, thus $P_i$ is homotopic rel endpoints in $\widehat C_{\bar x}$ to $$\Pi_{\whC_{\bar{x}}}(Q_i)\Pi_{\whC_{\bar{x}}}(P'_i)\Pi_{\whC_{\bar{x}}}(\bar Q_{i+1}).$$
	Thus 
	$$
	w_i=\w(\Pi_{\whC_{\bar{x}}}(Q_i))\w(\Pi_{\whC_{\bar{x}}}(P'_i))\w(\Pi_{\whC_{\bar{x}}}(\bar Q_{i+1})).$$
	Note that $\Pi_{\whC_{\bar x}}(Q_i)\subset \whC_{i}\cap \whC_{i-1}$ (by Lemma~\ref{lem:more gate} (1) and Lemma~\ref{lem:retraction property}), so the first word and the third word in the above decomposition of $w_i$ is contained in $A_{S\setminus\{s_1,s_n\}}\cap A_{\hat s}$.
	It follows that $$\w(\Pi_{\whC_{\bar{x}}}(P'_1))\w(\Pi_{\whC_{\bar{x}}}(P'_2))\w(\Pi_{\whC_{\bar{x}}}(P'_3))\w(\Pi_{\whC_{\bar{x}}}(P'_4))$$ is equivalent to $w_1w_2w_3w_4$ in $A_{\hat s}$. Thus by replacing $P_i$ by $\Pi_{\whC_{\bar x}}(P'_i)$, we can assume $P_i\subset \Pi_{C_{\bar x}}(C_{\bar z})$.
	
	Let $E=\Pi_{C_{\bar{x}}}(C_{\bar z})$. Note that $C_{i}\cap C_{i+1}$ has nonempty intersection with both $C_{\bar z}$ and $C_{\bar x}$. Let $I=\supp(E)$. Let $I_1$ be the union of irreducible components of $I$ that are contained entirely in $$\supp(C_i\cap C_{i+1})=S\setminus\{s_1,s_n\}.$$ Note that $I_1$ does not depend on $i$. Let $I_2$ be the union of the remaining irreducible components of $I$. By Lemma~\ref{lem:two cells} (2), $s_{i_0}\notin I_2$.
	Then $E=E_1\times E_2$ with $\supp(E_i)=I_i$. Hence $\widehat E=\widehat E_1\times \widehat E_2$. Thus $P_i$ is homotopic rel endpoints in $\widehat E$ to $P_{i1}P_{i2}$ where $P_{ij}$ is contained in a $\widehat E_j$-slice of $\widehat E$ for $j=1,2$. Hence $w_i=w_{i1}w_{i2}$ with $\supp(w_{ij})\in I_j$ for $j=1,2$. Note that $w_{i1}\subset A_{S\setminus\{s_1,s_n\}}\cap A_{\hat s}$. Thus $w_1w_2w_3w_4$ is equivalent to $w_{12}w_{22}w_{32}w_{42}$ in $A_{\hat s}$ via the following sequence of equivalent words:
	\begin{align*}
		& w_{21}w_{11}\cdot w_{21}w_{22}\cdot w_{31}w_{32}\cdot w_{41}w_{42}\\
		& w_{12}\cdot (w_{11}w_{21}) w_{22}\cdot w_{31}w_{32}\cdot w_{41}w_{42}\\
		& w_{12}\cdot w_{22}\cdot (w_{11}w_{21}w_{31}) w_{32}\cdot w_{41}w_{42}\\
		& w_{12}\cdot w_{22}\cdot w_{32}\cdot (w_{11}w_{21}w_{31}w_{41}) w_{42}
	\end{align*}
	note that $w_{11}w_{21}w_{31}w_{41}=1$. As $s_{i_0}\notin I_2$, we know $w_{i2}\in S\setminus\{s_{i_0}\}$. This finishes the proof.
\end{proof}

\begin{cor}
	\label{cor:inherit}
	Let $\Lambda'$ be an admissible tree subgraph of a Dynkin diagram $\Lambda$. 
	Suppose $\Delta_{\Lambda,\Lambda'}$ satisfies the labeled 4-cycle condition. Then for any induced subgraph $\Lambda_0$ of $\Lambda$, $\Delta_{\Lambda_0,\Lambda'\cap \Lambda_0}$ satisfies the labeled 4-cycle condition.
\end{cor}

\begin{proof}
It follows from the definition of labeled 4-cycle condition that $\Delta_{\Lambda,\Lambda''}$ satisfies the labeled 4-cycle condition for any subtree $\Lambda''$ of $\Lambda$. So we can assume without loss of generality that $\Lambda'\subset \Lambda_0$.
By Proposition~\ref{prop:4-cycle} and Lemma~\ref{lem:inherit}, for any linear subgraph $\Lambda''\subset \Lambda'$, $\Delta_{\Lambda,\Lambda''}$ is bowtie free. By Proposition~\ref{prop:inherit}, $\Delta_{\Lambda_0,\Lambda''}$ is bowtie free. Thus by Proposition~\ref{prop:4-cycle} again, $\Delta_{\Lambda_0,\Lambda'}$ satisfies the labeled 4-cycle condition.
\end{proof}

\section{On homotopy types of relative Artin complexes}
\label{subsec:homotopy}
\begin{lem}
	\label{lem:dr}
	Let $\Lambda_1\subset \Lambda_2$ be two induced subgraphs of $\Lambda$ such that $\Lambda_2\setminus\Lambda_1$ contains exactly one node, denoted by $s$. If $\lk(x,\Delta_{\Lambda,\Lambda_2})$ is contractible for some (hence any) vertex $x\in \Delta_{\Lambda,\Lambda_2}$ of type $\hat s$, then $\Delta_{\Lambda,\Lambda_2}$ deformation retracts onto $\Delta_{\Lambda,\Lambda_1}$.
\end{lem}

\begin{proof}
	Let $\Theta$ be the collection of type $\hat s$ vertices in $\Delta_{\Lambda,\Lambda_2}$. Note that if $x_1$ and $x_2$ are distinct elements in $\Theta$, then the intersection of the open stars $$\st^o(x_1,\Delta_{\Lambda,\Lambda_2})\cap \st^o(x_2,\Delta_{\Lambda,\Lambda_2})=\emptyset.$$
	Then $\Delta_{\Lambda,\Lambda_2}$ is a disjoint union  $$\Delta_{\Lambda,\Lambda_1}\sqcup (\sqcup_{x\in \Theta}\st^o(x,\Delta_{\Lambda,\Lambda_2})).$$
	As $\lk(x,\Delta_{\Lambda,\Lambda_2})$ is contractible, we know that $\st(x,\Delta_{\Lambda,\Lambda_2})$ deformation retracts onto $\st(x,\Delta_{\Lambda,\Lambda_2})\cap \Delta_{\Lambda,\Lambda_1}$ for each $x\in \Theta$. Thus the lemma follows.
\end{proof}

\begin{prop}
	\label{prop:contractible}
Let $(\mathcal C_1,\mathcal C_2)$ be two classes of connected Dynkin diagrams such that
\begin{enumerate}
	\item if $\Lambda'$ is an induced subgraph of $\Lambda$ with $\Lambda'\in \mathcal C_1$ and $\Lambda\in \mathcal C_2$,
	then $\Delta_{\Lambda,\Lambda'}$ is contractible;
	\item if $\Lambda'$ is an induced subgraph of $\Lambda$ with $\Lambda'\in \mathcal C_1$ and $\Lambda\in \mathcal C_2$, then for any connected induced subgraph $\Lambda''$ of $\Lambda$ containing $\Lambda'$, we have $\Lambda''\in\mathcal C_2$.
\end{enumerate}
Then $\Delta_\Lambda$ is contractible for each $\Lambda\in\mathcal C_2$ which contains an induced sub-diagram that is in $\mathcal C_1$. If we assume, in addition, that 
\begin{enumerate}[resume]
	\item if $\Lambda\in \mathcal C_2$ and $\Lambda'$ is an induced subgraph of $\Lambda$ such that $\Lambda'$ does not contain any induced subgraphs that are in $\mathcal C_1$, then $A_{\Lambda'}$ satisfies the $K(\pi,1)$-conjecture.
\end{enumerate}
 Then $A_\Lambda$ satisfies the $K(\pi,1)$-conjecture for any $\Lambda\in \mathcal C_2$.
\end{prop}

\begin{proof}
The contractibility of $\Delta_\Lambda$ is a consequence of a more general claim, namely $\Delta_{\Lambda,\Lambda''}$ is contractible whenever $\Lambda\in \mathcal C_2$, $\Lambda''$ is connected and $\Lambda''$ contains an induced subgraph $\Lambda'$ which is in $\mathcal C_1$.
	We prove this claim by induction on the number of nodes in $\Lambda''\setminus \Lambda'$. The base case when $\Lambda''=\Lambda'$ follows from Assumption 1. For the general case, we apply Lemma~\ref{lem:subgraph} to the pair $\Lambda'\subset \Lambda''$ to obtain $\{\Lambda_i\}_{i=1}^n$ and $\{s_i\}_{i=1}^{n-1}$ as in Lemma~\ref{lem:subgraph}, such that $\Lambda_1=\Lambda''$ and $\Lambda_n=\Lambda'$.
	As $\Delta_{\Lambda,\Lambda'}$ is contractible by Assertion 1, it suffices to show that for each $1\le i\le n-1$, $\Delta_{\Lambda,\Lambda_i}$ deformation retracts onto $\Delta_{\Lambda,\Lambda_{i+1}}$. By Lemma~\ref{lem:dr}, we need to show $\lk(v,\Delta_{\Lambda,\Lambda_i})$ is contractible for any vertex  $v\in X=\Delta_{\Lambda,\Lambda_i}$ of type $\hat s_i$. By Lemma~\ref{lem:link}, $\lk(v,\Delta_{\Lambda,\Lambda_i})$ admits a type-preserving isomorphism to $\Delta_{\Lambda_{s_i},\Lambda_{i+1}}$ where $\Lambda_{s_i}$ is the component of $\Lambda\setminus \{s_i\}$ that contains $\Lambda_{i+1}$. As $\Lambda'\subset \Lambda_{s_i}\subset\Lambda$, we know $\Lambda_{s_i}\in \mathcal C_2$ by Assumption 2. Note that $\Lambda'\subset \Lambda_{i+1}\subset\Lambda_{s_i}$. As $\Lambda_{i+1}\subsetneq \Lambda''$, $\Lambda_{i+1}\setminus \Lambda'$ has less nodes compared to $\Lambda''\setminus \Lambda'$. By our induction assumption, $\Delta_{\Lambda_{s_i},\Lambda_{i+1}}$ is contractible. The claim follows. In particular, $\Delta_{\Lambda}$ is contractible for any $\Lambda\in \mathcal C_2$ that contains an induced sub-diagram in $\mathcal C_1$.
	
Now we prove the last sentence of the proposition by induction on the number of nodes in $\Lambda\in \mathcal C_2$. The base case when $\Lambda$ has one node is clear. For the general case, if $\Lambda$ is spherical, then we are done by \cite{deligne}. Note that the $\Lambda=\Lambda'$ case of Assumption 3 implies that if $\Lambda\in\mathcal C_2$ does not contain any sub-diagrams which are in $\mathcal C_1$, then $A_\Lambda$ satisfies the $K(\pi,1)$-conjecture. Now assume $\Lambda$ contains an induced sub-diagram that is in $\mathcal C_1$ and $\Lambda$ is not spherical, then $\Delta_\Lambda$ is contractible by the previous paragraph. Now we use Theorem~\ref{thm:combine}. For $s\in \Lambda$, we need to show $A_{\Lambda\setminus \{s\}}$ satisfies the $K(\pi,1)$-conjecture. As $A_{\Lambda\setminus\{s\}}$ is the direct sum of $A_{\Theta}$ where $\Theta$ runs over the components of $\Lambda\setminus \{s\}$. So, the $K(\pi,1)$-conjecture holds for $A_{\Lambda\setminus\{s\}}$ if and only if it holds for each $A_\Theta$.
It suffices to show $A_\Theta$ satisfies the $K(\pi,1)$-conjecture. If $\Theta$ contains an induced subgraph in $\mathcal C_1$, then by Assumption 2, $\Theta\in \mathcal C_2$. As $\Theta$ has fewer nodes compared to $\Lambda$, we know $A_{\Theta}$ satisfies the $K(\pi,1)$-conjecture by induction assumption. If $\Theta$ does not contain any induced subgraphs which are members in $\mathcal C_1$, then $A_\Theta$ satisfies the $K(\pi,1)$-conjecture by Assumption 3. 
\end{proof}
		
Sometimes it is more convenient to use the following small variation of Proposition~\ref{prop:contractible}.
	\begin{cor}
		\label{cor:contractible}
		Let $(\mathcal C_1,\mathcal C_2)$ be two classes of connected Dynkin diagrams such that
		\begin{enumerate}
			\item if $\Lambda'$ is an induced subgraph of $\Lambda$ with $\Lambda'\in \mathcal C_1$ and $\Lambda\in \mathcal C_2$,
			then $\Delta_{\Lambda,\Lambda'}$ is contractible;
			\item if $\Lambda'$ is an induced subgraph of $\Lambda$ with $\Lambda'\in \mathcal C_1$ and $\Lambda\in \mathcal C_2$, then for any connected induced subgraph $\Lambda''$ of $\Lambda$ containing $\Lambda'$, we have $\Lambda''\in\mathcal C_2$.
			\item any $\Lambda\in \mathcal C_2$ contains an induced subgraph in $\mathcal C_1$;
			\item for each pair of Dynkin diagrams $\Lambda'\subset \Lambda$ such that $\Lambda'\in \mathcal C_1,\Lambda\in\mathcal C_2$ and $\Lambda'$ is an induced subgraph of $\Lambda$, and for any node $s\in \Lambda'$, we know each component of $\Lambda\setminus \{s\}$ satisfies the $K(\pi,1)$-conjecture.
		\end{enumerate}
		Then $A_\Lambda$ satisfies the $K(\pi,1)$-conjecture, and $\Delta_\Lambda$ is contractible for any $\Lambda\in \mathcal C_2$.
	\end{cor}

\begin{proof}
We induct on the number of nodes in $\Lambda$ and assume $\Lambda$ is not spherical as before. Note that for any $\Lambda\in\mathcal C_2$, $\Delta_\Lambda$ is contractible by Proposition~\ref{prop:contractible} and Assumption 3.
By Theorem~\ref{thm:combine}, it remains to show $A_{\hat s}$ satisfies the $K(\pi,1)$-conjecture for any $s\in S$. By Assumption 3, 
there is an induced subgraph $\Lambda'\subset \Lambda$ with $\Lambda'\in \mathcal C_1$. If $s\in \Lambda'$, then we are done by Assumption 4. If $s\notin \Lambda'$, we need to show $A_\Theta$ satisfies the $K(\pi,1)$-conjecture for each component $\Theta$ of $\Lambda\setminus \{s\}$. If $\Theta$ and $\Lambda'$ are in the same component of $\Lambda\setminus \{s\}$, then $\Lambda'\subset \Theta\subset \Lambda$. By Assumption 2, $\Theta\in \mathcal C_2$. As $\Theta$ has fewer nodes compared to $\Lambda$, we know $A_\Theta$ satisfies the $K(\pi,1)$-conjecture by induction assumption. Now assume $\Theta$ and $\Lambda'$ are not in the same component of $\Lambda\setminus \{s\}$. Then $\Theta\cap \Lambda'=\emptyset$. In particular, for any $s'\in \Lambda'$, $\Theta$ is an induced subgraph of a component of $\Lambda\setminus \{s'\}$. By Assumption 4, each direct summand of $A_{\Lambda\setminus\{s'\}}$ coming from components of $\Lambda\setminus\{s'\}$ satisfies the $K(\pi,1)$-conjecture, we know $A_{\Lambda\setminus \{s'\}}$ satisfies the $K(\pi,1)$-conjecture. Thus $A_\Theta$ satisfies the $K(\pi,1)$-conjecture by \cite[Corollary 2.4]{godelle2012k}. This finishes the proof.
\end{proof}

\begin{cor}
	\label{cor:contractible1}
Suppose for any connected almost spherical Dynkin diagram $\Lambda'$ and any connected Dynkin diagram $\Lambda$ which contains $\Lambda'$ as an induced subgraph, $\Delta_{\Lambda,\Lambda'}$ is contractible. Then the $K(\pi,1)$-conjecture holds for any Artin group.
\end{cor}

\begin{proof}
Let $\mathcal C_2$ be the class of all connected Dynkin diagrams. Let $\mathcal C_1$ be the class of all connected almost spherical Dynkin diagrams. By \cite{godelle2012k}, it suffices to show $A_\Lambda$ satisfies the $K(\pi,1)$-conjecture for any $\Lambda\in \mathcal C_2$. We verify the three assumptions of Proposition~\ref{prop:contractible} are satisfied. Assumptions 1 and 2 are clear. Let $\Lambda,\Lambda'$ be as in Assumption 3 of Proposition~\ref{prop:contractible}. We assume without loss of generality that $\Lambda'$ is connected. If $\Lambda'$ does not contain any irreducible almost spherical sub-diagrams, then $\Lambda'$ must be spherical, hence  $A_{\Lambda'}$ satisfies the $K(\pi,1)$-conjecture by \cite{deligne}. 
\end{proof}

\section{Filling 4-cycles in relative Artin complexes}
\label{sec:4-cycle}
The main goal of this section is to prove the following.
\begin{thm}
	\label{thm:bowtie free}
Let $A_\Gamma$ be an irreducible spherical Artin group with its Dynkin diagram $\Lambda$ and its set of nodes $S$. Take a linear subgraph $\Lambda'\subset\Lambda$.
Then $\Delta_{\Lambda,\Lambda'}$ is bowtie free.
 \end{thm}

\begin{proof}
As $\Lambda$ is a tree, $(V,\le)$ is indeed a poset by Lemma~\ref{lem:poset} (2). Now we are done by Corollary~\ref{cor:algebraic} below, \cite[Theorem 9.5]{cumplido2019parabolic} and Proposition~\ref{prop:key} below.
\end{proof} 
We will prove Corollary~\ref{cor:algebraic} in Section~\ref{subsec:lattice} and Proposition~\ref{prop:key} in Section~\ref{subsec:conjugator}.

We record the following consequence of Theorem~\ref{thm:bowtie free}, Lemma~\ref{lem:connect} and Proposition~\ref{prop:4-cycle}.

\begin{cor}
	\label{cor:wheel}
	Let $A_\Gamma$ be an irreducible spherical Artin group with its Dynkin diagram $\Lambda$ and its set of nodes $S$. Then its Artin complex (i.e. its spherical Deligne complex) satisfies the labeled 4-cycle condition.
\end{cor}

\subsection{An algebraic condition for trimming 4-cycles}
\label{subsec:lattice}
We prove a preparatory results for dealing with 4-cycles in general Artin complexes (whose associated Artin groups are not necessarily spherical). 

\begin{prop}
	\label{prop:group bowtie}
Let $\Lambda$ be a Dynkin graph (not necessarily spherical).
Suppose $\Lambda'$ is an admissible linear subgraph of $\Lambda$ with its consecutive nodes denoted by $\{s_i\}_{i=1}^n$. For $1\le i\le n$, let $A_i=A_{S\setminus\{s_i\}}$. Suppose for any commuting pair $u_1\in A_{i_1}$ and $u_2\in A_{i_2}$ ($i_1\le i_2$), there exists $g\in A_{i_1}\cap A_{i_2}$ and $i_1\le i_3\le i_2$ such that $gu_jg^{-1}\in A_{i_3}$ for $j=1,2$. Then $\Delta_{\Lambda,\Lambda'}$ is bowtie free.
\end{prop}

We endow the set of nodes of $\Lambda'$ with the order as in Lemma~\ref{lem:poset}. Let $\bC$ be the Coxeter complex for $W_\Lambda$. Let $\Si$ and $\od$ be the Davis complex and oriented Davis complex of $A_\Lambda$. Let $\bar \pi:\Delta_{\Lambda}\to \bC$ be the quotient map in Definition~\ref{def:barpi}. We view $\Delta_{\Lambda,\Lambda'}$ as a subcomplex of $\Delta_\Lambda$, hence $\bar\pi$ is also defined on $\Delta_{\Lambda,\Lambda'}$. The following is the main lemma for proving Proposition~\ref{prop:group bowtie}.

\begin{lem}
	\label{lem:commute}
Let $\Lambda,\Lambda'$ be as in Proposition~\ref{prop:group bowtie}.
Let $x_1,y_1,x_2,y_2$ be four vertices in $\Delta_{\Lambda,\Lambda'}$ such that $x_i<y_j$ for all $i,j\in\{1,2\}$, $x_1$ and $x_2$ are of type $\hat s_n$; and $y_1$ and $y_2$ are of type $\hat s_1$. Let $\omega$ be the 4-cycle $x_1y_1x_2y_2x_1$ in $\Delta_{\Lambda,\Lambda'}\subset\Delta_\Lambda$ and let $w_nw_1w'_nw'_1$ be an associated word as in Construction~\ref{def:4cycle}.
We assume this 4-cycle is \emph{non-degenerate}, i.e. $x_1\neq x_2$ and $y_1\neq y_2$. 
Then $w_nw_1w'_nw'_1$ is equivalent to a word of form $u_nu_1u^{-1}_nu^{-1}_1$ in the sense of Construction~\ref{def:4cycle}.
\end{lem}

\begin{proof}
 We carry out a case analysis for the proof. Recall that edges in $\od$ are labeled by elements in $S$. For a subset $E\subset\od$, we define $\supp(E)$ to be the collection of labels of edges in $E$. In the following, we write $A_1$ for $A_{\hat s_1}$ and $A_n$ for $A_{\hat s_n}$. Then $w_n,w'_n\in A_n$, and $w_1,w_1'\in A_1$.

\medskip
\noindent
\underline{Case 1: the $\bar\pi(\omega)$ is a single edge $\bar x\bar y$ in $\bC$.} Recall that $\bar x$ corresponds a left coset in $W_\Lambda$. Let $C_{\bar x}$ be the standard subcomplex of $\Si_\Lambda$ spanned by this left coset. Similarly we define $C_{\bar y}$. Let $\widehat C_{\bar x}$ and $\widehat C_{\bar y}$ be the associated standard subcomplexes of $\od$. Then $w_nw_1w'_nw'_1=1$ gives an edge loop $P_nP_1P'_nP'_1$ in $\od$ such that $P_n,P'_n\subset \widehat C_{\bar x}$ and $P_1,P'_1\subset \widehat C_{\bar y}$. Up to replacing $w_nw_1w'_nw'_1$ by an equivalent word, we can assume $P_n,P_1,P'_n,P'_1$ are loops. For any edge path $P$ in $\od$, we use $\bar P$ to denote the inverse path of $P$.

By Lemma~\ref{lem:retraction property}, $\p_{\widehat C_{\bar y}}(\widehat C_{\bar x})=\widehat{\p_{C_{\bar y}}(C_{\bar x})}=\whC_{\bar x}\cap \whC_{\bar y}$. As $P_nP_1P'_nP'_1$ is null-homotopic in $\od$, we know $$\p_{\widehat C_{\bar x}}(P_nP_1P'_nP'_1)=P_nQ_1P'_nQ'_1$$ is null-homotopic in $\whC_{\bar x}$, where $Q_1=\p_{\widehat C_{\bar x}}(P_1)$ and $Q'_1=\p_{\widehat C_{\bar x}}(P'_1)$ are loops in $\whC_{\bar x}\cap \whC_{\bar y}$. Thus $P_nQ_1$ and $P'_nQ'_1$ represent elements in $\pi_1(\whC_{\bar x}\cap \whC_{\bar y})$ which are inverse of each other. By replacing $P_nP_1P'_nP'_1$ with $$(P_nQ_1)(\bar Q_1P_1)(P'_nQ'_1)(\bar Q'_1P'_1),$$ we can assume without loss of generality that $w'_n=w^{-1}_n$ and $P'_n=\bar P_n$. Thus $w_nw_1w^{-1}_{n}w'_1=1$. 
Now we consider the edge loop $$\p_{\widehat C_{\bar y}}(P_nP_1\bar P_nP'_1),$$ which gives a loop $R_nP_1\bar R_nP'_1$ that is null-homotopic in $\whC_{\bar y}$, where $R_n=\p_{\widehat C_{\bar y}}(P_n)$ is a loop in $\whC_{\bar x}\cap \whC_{\bar y}$. Thus $P'_1$ is homotopic rel endpoints in $\whC_{\bar y}$ to $R_n\bar P_1\bar R_n$. Then we can assume $P'_1=R_n\bar P_1\bar R_n$. Now the loop $P_nP_1\bar P_nP'_1$ is homotopic to $P_nP_1\bar P_nR_n\bar P_1\bar R_n$ in $\od$, which we read off the word of form 
$$w_nw_1w^{-1}_n (r_nw^{-1}_1r^{-1}_n).$$
As $R_n\subset\whC_{\bar x}\cap \whC_{\bar y}$, we know $r_n\in A_1\cap A_n$. Thus $w_nw_1w^{-1}_n (r_nw^{-1}_1r^{-1}_n)$ is equivalent to  $(r^{-1}_nw_n)w_1w^{-1}_n (r_nw^{-1}_1)$, which is equivalent to $(r^{-1}_nw_n)w_1(w^{-1}_n r_n)w^{-1}_1$. Then we are done by taking $u_n=r^{-1}_nw_n$ and $u_1=w_1$.

%It remains to prove the claim. Suppose the contrary is true, i.e. each vertex of $\Lambda'$ is contained in either $S_1$ or $S_n$. As $c_{S_i}\subset A_i$, we know  $S_i\subset S\setminus\{s_i\}$ for $i=1,n$. Thus $s_1\in S_n\setminus S_1$. Let $T_n$ be the irreducible component of $S_n\cap \Lambda'$ that contains $s_1$ and suppose $T_n=\{s_1,\ldots,s_i\}$. We now prove by contradiction that $i=n$. If $i<n$, then $s_{i+1}\notin S_n$ (otherwise $s_{i+1}\in T_n$). Thus $s_{i+1}\in S_1$. Let $T_1$ be the irreducible component of $S_1\cap \Lambda'$ that contains $s_{i+1}$. As $s_i$ and $s_{i+1}$ do not commute, and $s_{i+1}\notin S_n$, by Lemma~\ref{lem:commute}, the only possibility left is that $T_n\subset T_1$. But this is still impossible as $s_1\in T_n$ and $s_1\notin S_1$. Thus $i=n$. However, this contradicts that $s_n\notin S_n$.  Thus the claim is proved.
\medskip
\noindent
\underline{Case 2: $\bar\pi(\omega)$ is a union of two edges $\bar x\bar y$ and $\bar y\bar z$ in $\bC$.} Assume without loss of generality that $\bar x$ and $\bar z$ are of type $\hat s_n$ and $\bar y$ is of type $\hat s_1$. We define $\widehat C_{\bar x}$, $\widehat C_{\bar y}$, $\widehat C_{\bar z}$ and $P_nP_1P'_nP'_1$ as before. Up to replacing $w_nw_1w'_nw'_1$ by an equivalent word, we assume $P_n$ and $P'_n$ are loops.
Note that 
\begin{enumerate}
	\item $P_n\subset \widehat C_{\bar x}$, $P_1,P'_1\subset \widehat C_{\bar y}$ and $P'_n\subset \widehat C_{\bar z}$;
	\item  $\supp(\p_{\widehat C_{\bar x}}(\widehat C_{\bar y}))\subset S\setminus\{s_1,s_n\}$.
\end{enumerate}
Note that $\p_{\whC_{\bar x}}(P'_n)\subset\p_{\whC_{\bar x}}(\whC_{\bar z})$. Let $U=\Pi_{C_{\bar x}}(C_{\bar z})$ and let $I=\supp(U)$. Suppose $I=I_1\sqcup I_2$ where $I_1$ is the union of irreducible components of $U$ that are contained in $\supp(C_{\bar y})=S\setminus\{s_1\}$. Let $x'\in P'_n$ be the base point of the loop $P'_n$ and let $x=\Pi_{C_{\bar x}}(x')$. As $x'\in C_{\bar y}\cap C_{\bar z}$, by Lemma~\ref{lem:two cells}, $x\in C_{\bar x}\cap C_{\bar y}$. For $i=1,2$, let $U_i$ be the standard subcomplex in $U$ containing $x$ such that $\supp(U_i)=I_i$. Then $U=U_1\times U_2$ and $U_1\subset C_{\bar x}\cap C_{\bar y}$. Thus $\widehat U=\widehat U_1\times \widehat U_2$ and $\widehat U_1\subset \whC_{\bar x}\cap \whC_{\bar y}$. Thus $\p_{\whC_{\bar x}}(P'_n)$ is homotopic to a concatenation of two edge loops $P_{\widehat U_1}P_{\widehat U_2}$ based $x$ such that $P_{\widehat U_i}\subset \widehat U_i$ for $i=1,2$.

Let $P_{n1}=\p_{\whC_{\bar x}}(P_1)P_{\widehat U_1}\subset \whC_{\bar x}\cap \whC_{\bar y}$, $P_{n2}=P_{\widehat U_2}$ and $P_{n3}=\p_{\whC_{\bar x}}(P'_1)\subset \whC_{\bar x}\cap \whC_{\bar y}$. 
Then $$\p_{\widehat C_{\bar x}}(P_nP_1P'_nP'_1)=P_nP_{n1}P_{n2}P_{n3}$$ is null-homotopic in $\widehat C_{\bar x}$. Thus $P_n$ is homotopic rel endpoints in $\widehat C_{\bar x}$ to $\bar P_{n3}\bar P_{n2}\bar P_{n1}$, where $\bar P_{ni}$ denotes the inverse path of $P_{ni}$. Let $w_{ni}$ be the word we read off from $\bar P_{ni}$. Then $w_n=w_{n3}w_{n2}w_{n1}$.  
Note that $\supp(w_{ni})\subset A_1\cap A_n$ for $i=1,3$. By our choice of $P_{n2}$ and Lemma~\ref{lem:two cells}, we know $\bar P_{n2}$ and $\Pi_{\whC_{\bar z}}(\bar P_{n2})$ correspond to the same word in $A_\Lambda$. 
As $w_nw_1w'_nw'_1$ is equivalent to $w_{n2}(w_{n1}w_1)w'_n(w'_1w_{n3})$, we can assume without loss of generality that $P_n=P_{n2}\subset \widehat U_2$ is a loop based at $x$ and $P'_n$ is still a loop based at $x'$.

Now let $U'=\Pi_{C_{\bar z}}(C_{\bar x})$ and let $I'=\supp(U')$. Suppose $I'=I'_1\sqcup I'_2$ where $I'_1$ is the union of irreducible components of $U'$ that are contained in $\supp(C_{\bar y})=S\setminus\{s_1\}$. For $i=1,2$, let $U'_i$ be the standard subcomplex in $U'$ containing $x'$ such that $\supp(U'_i)=I'_i$. Then $U'=U'_1\times U'_2$ and $U'_1\subset C_{\bar z}\cap C_{\bar y}$. Thus $\widehat U'=\widehat U'_1\times \widehat U'_2$ and $\widehat U'_1\subset \whC_{\bar z}\cap \whC_{\bar y}$. 
By Lemma~\ref{lem:two cells} and our choice of $P_n$, 
$\p_{\whC_{\bar z}}(P_n)\subset \widehat U'_2$, moreover, $P_n$ and $\p_{\whC_{\bar z}}(P_n)$ trace out the same word in the Artin group.

Consider $\p_{\widehat C_{\bar z}}(P_nP_1P'_nP'_1)$ and argue as before, we know $P'_n$ is homotopic rel endpoints in $\widehat C_{\bar z}$ to $P'_{n3}P'_{n2}P'_{n1}$ satisfying:
\begin{enumerate}
	\item  $P'_{ni}\subset \p_{\whC_{\bar z}}\cap\p_{\whC_{\bar y}}$ for $i=1,3$;
	\item $P'_{n2}\subset \widehat U'_2$ and $P'_{n2}$ is the inverse path of $\p_{\whC_{\bar z}}(P_n)$.
\end{enumerate}
Let $w'_{ni}$ be the word we read off from $\bar P'_{ni}$. Then $w'_n=w'_{n3}w'_{n2}w'_{n1}$ and $w'_{ni}\in A_1\cap A_n$ for $i=1,3$. Moreover, $w'_{n2}=w^{-1}_n$ by our choice of $P_n$ (Lemma~\ref{lem:two cells} (3)).
Thus $$w_nw_1w'_nw'_1=w_nw_1(w'_{n3}w^{-1}_nw'_{n1})w'_1$$ 
is equivalent to $w_n(w_1w'_{n3})w^{-1}_n(w'_{n1}w'_1)$. From now on, we may assume $P_n$ is a loop in $\widehat U_2$ based at $x$, $P'_n$ is a loop in $\widehat U'_2$ based at $x'$, $\bar P_n=\Pi_{\whC_{\bar x}}(P'_n)$ and $\bar P'_n=\Pi_{\whC_{\bar z}}(P_n)$.

Now we consider the edge loop $$\p_{\widehat C_{\bar y}}(P_nP_1P'_nP'_1),$$ which gives a loop $R_nP_1R'_nP'_1$ that is null-homotopic in $\whC_{\bar y}$, where $R_n=\p_{\whC_{\bar y}}(P_n)$ is a loop in $\whC_{\bar x}\cap \whC_{\bar y}$ and $R'_n=\p_{\whC_{\bar y}}(P'_n)$ is a loop in $\whC_{\bar z}\cap \whC_{\bar y}$. Thus $P'_1$ is homotopic rel endpoints in $\whC_{\bar y}$ to $\bar R'_n\bar P_1\bar R_n$. Then we can assume $P'_1=\bar R'_n\bar P_1\bar R_n$. 

Let $r_n=\w(R_n)$ and $r'_n=\w(R'_n)$. We claim $r'_n=r^{-1}_n$. To see this, let $u$ be a geodesic in $\Si$ from $x$ to $x'$. Then $\supp(u)\perp \supp(U'_2)=\supp(U_2)$ by Lemma~\ref{lem:two cells}. Let $V$ be the smallest standard subcomplex of $\Si$ containing $U_2$ and $U'_2$. Then $\supp(V)=U_2\times V_{u}$ where $V_u$ is the standard subcomplex containing $u$ with support being $\supp(u)$. As $x,x'\in C_{\bar y}$, we know $u\subset C_{\bar y}$ and $V_u\subset C_{\bar y}$. Thus $V\cap C_{\bar y}=(U_2\cap C_{\bar y})\times V_u$. Thus by Lemma~\ref{lem:more gate} (1), for any vertex $z\in V$, $\prj_{C_{\bar y}}(z)$ is the vertex in $(U_2\cap C_{\bar y})\times V_u$ that are closest to $z$. Then the map $\prj_{C_{\bar y}}|_{\vertex V}$ respects the product decomposition of $V$, more precisely,
\begin{equation}
	\label{e:projection}
	\prj_{C_{\bar y}}|_{\vertex V}= (\prj_{C_{\bar y}}|_{U_2})\times (\mathrm{Id}|_{V_u})= (\prj_{C_{\bar y}}|_{U'_2})\times(\mathrm{Id}|_{V_u}).
\end{equation}
Thus for any edge path $P\subset U_2$, $\w(\Pi_{C_{\bar y}}(P))=\w(\Pi_{C_{\bar y}}(p(P)))$ where $p:U_2\to U'_2$ is the parallel translation. Now the claim $r'_n=r^{-1}_n$ follows.

The claim implies $\w(P'_1)=r_nw^{-1}_1r^{-1}_n$ for some $r_n\in A_1\cap A_n$. The word $w_nw_1w'_nw'_1$ becomes
$$w_nw_1w^{-1}_n (r_nw^{-1}_1r^{-1}_n),$$
which is equivalent to $(r^{-1}_nw_n)w_1(w^{-1}_n r_n)w^{-1}_1$. Then we are done.

%Up to passing to equivalent words of $w_nw_1w'_nw'_1$, (i.e. replace $w_1$ (resp. $w'_1$) by $b^{-k_1}w_1b^{-k_6}$  (resp. $b^{-k_4}w'_1b^{-k_3}$)), we have $a^{-k_2}w_1a^{-k_4}w'_1=1$.

%Note that $\p_{\widehat C}(\widehat C_{\bar y})$ is a single point. Thus $\p_{\widehat C}(P_1)$ and $\p_{\widehat C}(P'_1)$ give trivial path. Thus by considering $\p_{\widehat C}(P_nP_1P'_nP'_1)$, we conclude that $a^{-k_2}a^{-k_4}=1$, which implies $k_4=-k_2$.  Thus $a^{-k_2}w_1a^{k_2}w'_1=1$. As $P_n\subset\widehat C$ and $\p_{\widehat C_{\bar y}}(\widehat C)$ is a single point, we know $\p_{\widehat C_{\bar y}}(P_n)$ is a constant path. Similarly $\p_{\widehat C_{\bar y}}(P'_n)$ is a constant path. By considering $\p_{\widehat C_{\bar y}}(P_nP_1P'_nP'_1)$, we deduce that $w'_1=(w_1)^{-1}$. Thus $w_n=a^{k_2}$ and $w_1$ commute, and we conclude as in Case 1.

\medskip
\noindent
\underline{Case 3: the $\bar\pi(\omega)$ is a 4-cycle in $\bC$.} We denote the consecutive vertices of $\bar \pi(\omega)$ by $\{\bar z_i\}_{i=1}^4$. Then $\bar z_1$ and $\bar z_3$ have type $\hat s_n$, and $\bar z_2$ and $\bar z_4$ have type $\hat s_1$. Let $C_i$ be the standard subcomplex of $\Si$ associated with $\bar z_i$.  Let $E_{i,j}=\p_{C_i}(C_j)$ for $1\le i,j\le 4$. 

Let $I=\supp(E_{1,3})$. Suppose $I=I_1\sqcup I_2$ where $I_1$ is the union of irreducible components of $I$ that are contained in $\supp(C_2)=\supp(C_4)=S\setminus\{s_1\}$. By Lemma~\ref{lem:two cells}, there is a vertex $x'\in C_2\cap E_{3,1}$ such that $x=p(x')\in C_2\cap E_{1,3}$, where $p:E_{3,1}\to E_{1,3}$ is the parallel translation. 
For $i=1,2$, let $U_i$ be the standard subcomplex in $E_{1,3}$ containing $x$ such that $\supp(U_i)=I_i$. Then $E_{1,3}=U_1\times U_2$ and $U_1\subset C_{1}\cap C_{2}$. Similarly, by Lemma~\ref{lem:two cells}, there is a vertex $y'\in C_4\cap E_{3,1}$ such that $y=p(y')\in C_4\cap E_{1,3}$. We can assume $y\in U_2$. Indeed, if $y\notin U_2$, as both $x$ and $y$ are in $E_{1,3}=U_1\times U_2$, there is geodesic $v$ in $E_{1,3}$ from $x$ to $y$ such that $v=v_2v_1$ with $\supp(v_i)\subset \supp(U_i)$ for $i=1,2$. As $y\in C_4\cap E_{1,3}$ and $$\supp(v_1)\subset \supp(C_2\cap E_{1,3})=\supp(C_4\cap E_{1,3})=\supp(C_4)\cap\supp(E_{1,3}),$$
we know $v_1\subset C_4\cap E_{1,3}$. By replacing $y$ by the other endpoint of $v_1$ and modifying $y'$ accordingly, we still have $y'\in C_4\cap E_{3,1}$ (by Lemma~\ref{lem:two cells} (1)) and $y\in C_4\cap E_{1,3}$. Now $\supp(v)\subset \supp(U_2)$. Let $v'=p^{-1}(v)$, which is a geodesic in $E_{3,1}$ connecting $x'$ and $y'$. By Lemma~\ref{lem:two cells} $\supp(U_2)\perp\supp(u)$, where $u$ is a geodesic path in $C_2$ from $x$ to $x'$. Hence $$\supp(v')=\supp(v)\subset \supp(U_2).$$ 
Let $U'_2$ be the standard subcomplex in $E_{3,1}$ containing $v'$ with $\supp(U'_2)=I_2$. Let $V$ be the smallest standard subcomplex of $\Si$ containing $U_2$ and $U'_2$. Then $V=U_2\times V_{u}$ where $V_u$ is the standard subcomplex containing $u$ with support being $\supp(u)$.

Up to passing to a word equivalent to $w_nw_1w'_nw'_1$, we may assume $P_n$ is a path from $y$ to $x$, $P_1$ is a path from $x$ to $x'$, $P'_n$ is a path from $x'$ to $y'$ and $P'_1$ is a path from $y'$ to $y$. Then $\Pi_{\whC_1}(P'_n)\subset \widehat E_{1,3}=\widehat U_1\times \widehat U_2$. As $\Pi_{\whC_1}(P'_n)$ is a path connecting $x\in \widehat U_2$ and $y\in \widehat U_2$, $\Pi_{\whC_1}(P'_n)$ is homotopic rel endpoints in $\whC_1$ to a concatenation to two paths $P_{\widehat U_1}P_{\widehat U_2}$ such that $P_{\widehat U_1}$ is a loop in $\widehat U_1$ based at $x$ and $P_{\widehat U_2}$ is a path in $\widehat U_2$ connecting $x$ and $y$. Let $P_{n1}=\Pi_{\whC_1}(P_1)P_{\widehat U_1}\subset \whC_1\cap \whC_2$, $P_{n2}=P_{\widehat U_2}$ and $P_{n3}=\Pi_{\whC_1}(P'_1)\subset \whC_1\cap \whC_4$. Now we argue as in Case 2 to show that up to passing to an equivalent word, we may assume $P_n=P_{n2}\subset\widehat U_2$. 

Now considering $\Pi_{\whC_3}(P_nP_1P'_nP'_1)$. Arguing as in Case 2, up to passing to an equivalent word, we may assume $P'_n$ is a path in $\widehat U'_2$ connecting $x'$ and $y'$, $\bar P'_n=\Pi_{\whC_3}(P_n)$ and $\bar P_n=\Pi_{\whC_1}(P'_n)$. Hence $w'_n=w^{-1}_n$ by Lemma~\ref{lem:two cells} (3). Now we consider $\Pi_{\whC_2}(P_nP_1P'_nP'_1)$. By the argument in Case 2 again (using \eqref{e:projection}), we know $\Pi_{\whC_2}(P_n)$ and $\Pi_{\whC_2}(P'_n)$ give two words that are inverse of each other. Then we can finish the proof as in Case 2.
\end{proof}

\begin{proof}[Proof of Proposition~\ref{prop:group bowtie}]	
First we show for $x_1,y_1,x_2,y_2$ as in Lemma~\ref{lem:commute}, there exists vertex $z\in \Delta_{\Lambda,\Lambda'}$ such that $x_i\le z\le y_j$ for all $i,j\in \{1,2\}$. We choose the four chambers $\{\Delta_i\}_{i=1}^4$ of $\Delta_\Lambda$ as in Construction~\ref{def:4cycle} with $\overline{x_1y_1}\subset\Delta_1$ and $\overline{y_1x_2}\subset \Delta_2$. Let $g_i\in A_\Lambda$ be the element associated with $\Delta_i$. By Lemma~\ref{lem:commute}, there exists $u_i\in A_i$ for $i=1,n$ such that $g_2=g_1u_n$, $g_3=g_2u_1$, $g_4=g_3u^{-1}_n$ and $g_1=g_4u^{-1}_1$.
Let $g\in A_1\cap A_n$ and $s\in \Lambda'$ be as in the assumption of Proposition~\ref{prop:group bowtie} such that $gu_ig^{-1}\in A_{S\setminus\{s\}}$ for $i=1,n$. In particular, for $i=1,n$, we have $u_i=g^{-1}v_ig$ for some $v_i\in A_{S\setminus\{s\}}$. As $g\in A_1\cap A_n$, the word $u_nu_1u^{-1}_nu^{-1}_1$ is equivalent to $v_nv_1^{-1}v_nv^{-1}_1$. In other words, we can choose the four chambers $\{\Delta_i\}_{i=1}^4$ in Construction~\ref{def:4cycle} such that $g_2=g_1v_n$, $g_3=g_2v_1$, $g_4=g_3v^{-1}_n$ and $g_1=g_4v^{-1}_1$. Up to a left translation we can assume $g_1$ is the trivial element.
Let $z\in \Delta_{\Lambda,\Lambda'}$ be the vertex corresponding to $A_{\hat s}$. Then $z\in \Delta_i$ for $1\le i\le 4$. In particular $z$ is adjacent or equal to each of $x_1,x_2,y_1,y_2$. Thus $x_i\le z\le y_j$ for $i,j\in\{1,2\}$, as desired.

To show Proposition~\ref{prop:group bowtie} in full generality, we induct on the number of nodes in $\Lambda'$, and use Lemma~\ref{lem:bowtie free criterion}. Assumption 2 of Lemma~\ref{lem:bowtie free criterion} follows from the previous paragraph. For $i=1,n$, let $\Lambda_i$ be the component of $\Lambda\setminus\{s_i\}$ containing all vertices of $\Lambda'$ expect $s_i$. Let $\Lambda'_i=\Lambda_i\cap\Lambda'$. By induction, $\Delta_{\Lambda,\Lambda'_i}$ is bowtie free for $i=1,n$. Hence $\Delta_{\Lambda'_i,\Lambda_i}$ is bowtie free for $i=1,n$ by Proposition~\ref{prop:inherit}. Thus Assumption 1 of Lemma~\ref{lem:bowtie free criterion} holds by Lemma~\ref{lem:link}, as desired.
\end{proof}

%Now we prove by induction on $n$ that each $n\ge 2$, if $\Lambda'$ has $n$ vertices, the conclusion of Theorem~\ref{thm:bowtie free} is true (for all possible ambient Dynkin diagram $\Lambda$). The base case $n=2$ is follows immediately from the previous paragraph. For arbitrary $n$, note that Assumption 1 of Lemma~\ref{lem:bowtie free criterion} holds true by induction and Lemma~\ref{lem:link}. Assumption 2 of Lemma~\ref{lem:bowtie free criterion} is true by the previous paragraph. This finishes the proof.

%Recall that the notion of almost spherical is in Definition~\ref{def:almost spherical}.

\begin{cor}
	\label{cor:algebraic}
Suppose $\Lambda$ is a Dynkin diagram, and $\Lambda\setminus \{s\}$ is spherical for any node $s\in \Lambda$. Suppose$\Lambda'\subset \Lambda$ is an admissible linear subgraph with its consecutive nodes being $\{s_i\}_{i=1}^n$. Assume that
\begin{enumerate}
	\item the intersection of any two parabolic subgroups of $A_\Lambda$ is parabolic;
	\item for any commuting pair of parabolic subgroups $\{\mathsf{P_i}\}_{i=1,2}$ such that $c_{\mathsf P_1}c_{\mathsf P_2}=c_{\mathsf P_2}c_{\mathsf P_1}$, $c_{\mathsf P_1}\in A_{i_1}$ and $c_{\mathsf P_2}\in A_{i_2}$, there exists $g\in A_{i_1}\cap A_{i_2}$ and $i_1\le i_3\le i_2$ such that $gc_{\mathsf P_j}g^{-1}\in A_{i_3}$ for $j=1,2$ (see Section~\ref{subsec:Artin} for the definition of $c_{\mathsf P_1}$).
\end{enumerate}
Then $\Delta_{\Lambda,\Lambda'}$ is bowtie free.
\end{cor}

\begin{proof}
The \emph{parabolic closure} of an element $g\in A_S$ is the smallest parabolic subgroup containing $g$. By Theorem~\ref{thm:parabolic}, there is a upper bound on the length of strictly increasing chain of parabolic group, thus each element in $A_S$ has a well-defined parabolic closure by Assumption 1. 

We claim that for any subset $S_1\subsetneq S$ and $g\in A_S$, $gA_{S_1}g^{-1}=A_{S_1}$ implies $gc_{S_1}g^{-1}=c_{S_1}$.
The conjugation $gA_{S_1}g^{-1}=A_{S_1}$ induces an isomorphism $\varphi:A_{S_1}\to A_{S_1}$. Let $\{T_i\}_{i=1}^k$ be the irreducible components of $S_1$. By Theorem~\ref{thm:parabolic}, $gA_{T_i}g^{-1}=h_iA_{T'_i}h^{-1}_i$ for some $T'_i\subset S_1$ and $h_i\in A_{S_1}$. As $T'_i$ is irreducible, there exists $j$ such that $T'_i\subset T_j$. As $h_i\in A_{S_1}$, we know $\varphi(A_{T_i})\subset A_{T_j}$. Let $\{T_{i_m}\}_{m=1}^n\subset \{T_i\}_{i=1}^k$ be the collection of irreducible factors such that $\varphi(A_{T_{i_m}})\subset A_{T_j}$. Then $\varphi$ induces an isomorphism between the direct sum of $A_{T_{i_m}}$ with $1\le i\le n$ and $A_{T_j}$. This implies $n=1$ by \cite{MR2098788}. Thus $\varphi(A_{T_i})=A_{T_j}$ and $\varphi$ permutes the irreducible factors of $S_1$. As $c_{T_i}$ generates the center of $A_{T_i}$, we know $\varphi(c_{T_i})=c_{T_j}$ or $c^{-1}_{T_j}$. By considering the abelianization of $A_\Lambda$, we know $\varphi(c_{T_i})=c_{T_j}$. Thus $gc_{S_1}g^{-1}=c_{S_1}$.  It follows from this claim that for any proper parabolic subgroup $\mathsf{P}$ of $A_S$, $c_{\mathsf{P}}$ is well-defined.

%Conversely, if $gc_{S_1}g^{-1}=c_{S_1}$, then by Assumption 1, $c_{S_1}$ is contained in the parabolic subgroup $\mathsf P=gA_{S_1}g^{-1}\cap A_{S_1}$ of $A_S$. By Theorem~\ref{thm:parabolic}, $\mathsf P$ can also be viewed as a parabolic subgroup of the Artin group $A_{S_1}$. As $c_{S_1}$ is not contained in any proper standard parabolic subgroup of $A_{S_1}$ and $c_{S_1}$ is central, we know $c_{S_1}$ is not contained in any proper parabolic subgroup of $A_{S_1}$. Thus $A_{S_1}\subset gA_{S_1}g^{-1}$. The other inclusion is similar, and the claim is proved.

It suffices to justify the assumption of Proposition~\ref{prop:group bowtie}. Let $u_1$ and $u_2$ be as in Proposition~\ref{prop:group bowtie}. 
Let $\mathsf{P}_i$ be the parabolic closure of $u_i$ for $i=1,2$. Note that $\mathsf{P}_i$ is spherical for $i=1,2$.
As $u_2=u_1u_2u^{-1}_1$, we know $u_1\mathsf{P}_2u^{-1}_1$ and $\mathsf{P}_2$ are parabolic subgroups containing $u_2$. 
The minimality of $\mathsf{P}_2$ implies that $\mathsf{P}_2\subset u_1\mathsf{P}_2u^{-1}_1$ and $\mathsf{P}_2\subset u^{-1}_1\mathsf{P}_2u_1$. Thus $\mathsf{P}_2=u_1\mathsf{P}_2u^{-1}_1$. 
We define $Z_1=c_{\mathsf{P}_1}$ and $Z_2=c_{\mathsf{P}_2}$. By the above claim, $Z_2=u_1Z_2u^{-1}_1$. Thus $Z_2u_1Z^{-1}_2=u_1$, which implies that $Z_2\mathsf{P}_1Z^{-1}_2=\mathsf{P}_1$. Thus $Z_1Z_2=Z_2Z_1$. By Assumption 2, there is $g\in A_{i_1}\cap A_{i_2}$ and $i_1\le i_3\le i_2$ such that $gZ_jg^{-1}\in A_{i_3}$ for $j=1,2$. Next we will show  $g\mathsf{P}_jg^{-1}\subset A_{i_3}$, which would imply $gu_jg^{-1}\in A_{i_3}$. 
We write $g\mathsf{P}_jg^{-1}=hA_{S_j}h^{-1}$ for some $S_j\subset S$, then $gZ_jg^{-1}=hc_{S_j}h^{-1}$ by the discussion on the previous paragraph on $c_{\mathsf{P}}$ being well-defined. By Assumption 1, $\mathsf{P}=hA_{S_j}h^{-1}\cap A_{i_3}$ is a parabolic subgroup of $A_\Lambda$ containing $hc_{S_j}h^{-1}$. Then $h^{-1}\mathsf{P}h$ is a parabolic subgroup of $A_{S_j}$ (by Theorem~\ref{thm:parabolic}) containing $c_{S_j}$. As $c_{S_j}$ is contained in the center of $A_{S_j}$, we know $c_{S_j}$ is not contained in any proper parabolic subgroup of $A_{S_j}$. Thus $h^{-1}\mathsf{P}h=A_{S_j}$, which implies that $\mathsf{P}=hA_{S_j}h^{-1}$. Hence $g\mathsf{P}_jg^{-1}\subset A_{i_3}$. 
\end{proof}

\subsection{Control of conjugator}
\label{subsec:conjugator}
Now we show that Assumption of Corollary~\ref{cor:algebraic} holds when $A_\Lambda$ is spherical.
Let $A_S,\Lambda,\Lambda',s_1,s_n$ be as in the beginning of Section~\ref{sec:4-cycle}.  Let consecutive nodes in $\Lambda'$ be $\{s_i\}_{i=1}^n$. 
For each $s\in S$, let $A_{\hat s}$ be the parabolic subgroup generated by $S\setminus\{s\}$. We define $A_{1}=A_{\hat s_1}$ and $A_n=A_{\hat s_n}$.

\begin{prop}
	\label{prop:key}
Suppose $A_S$ is spherical.	For $i=1,n$, suppose $\mathsf{P}_i$ is a parabolic subgroup $A_S$ which is contained in $A_i$. Let $Z_i=c_{\mathsf{P}_i}$ (see Section~\ref{subsec:Artin} for the definition of $c_{\mathsf{P}_i}$). We assume that $Z_1Z_n=Z_nZ_1$. Then there exists a node $s\in \Lambda'$ and $g\in A_1\cap A_n$ such that $gZ_ig^{-1}\in A_{\hat s}$ for $i=1,n$.
\end{prop}
Note that without the requirement  $g\in A_1\cap A_n$, the proposition follows immediately from \cite[Theorem 2.2]{cumplido2019parabolic} when $\mathsf{P_i}$ is irreducible for $i=1,2$. The reducible case is not that much harder. However, most of the work here is for arranging the conjugator $g$ such that it belongs to $A_1\cap A_n$. Now we will prove  Proposition~\ref{prop:key}, assuming a technical lemma, which is Lemma~\ref{lem:triming}. The proof of Lemma~\ref{lem:triming} is postponed to the next subsection.

\begin{proof}
We will be assuming $Z_i\notin A_1\cap A_n$ for $i=1,n$, otherwise the proposition is trivial.
	We first claim there exists positive words $b,d,f,g$ on $S$ such that
	\begin{enumerate}
		\item $Z_1=b^{-1}c_{S_1}b$ and $Z_n=dc_{S_n}d^{-1}$ for some $S_1\subset S\setminus \{s_1\}$ and $S_n\subset S\setminus\{s_n\}$;
		\item $b\in A_1$ and $d\in A_n$;
		\item $bd=fg$;
		\item $f^{-1}c_{S_1}f=c_{S'_1}$ for some $S'_1\subset S$, and $gc_{S_n}g^{-1}=c_{S'_n}$ for some $S'_n\subset S$;
		\item $c_{S'_1}$ commutes with $c_{S'_n}$.
	\end{enumerate}

Now we prove the claim. By Theorem~\ref{thm:parabolic}, for $i=1,n$, we can assume $\mathsf{P}_i$ is also a parabolic subgroup of the Artin group $A_i$, i.e. $\mathsf{P}_i=g_iA_{T_i}g^{-1}_i$ for $g_i\in A_i$ and $T_i\subset S\setminus\{s_i\}$. 
	
	Let $Z_{n}=cd^{-1}$ be the $pn$-normal form of $Z_n$ in $A_n$. Then $c,d\in A_n$. By \cite[Theorem 4]{cumplido2019minimal}, $d^{-1}\mathsf{P}_nd=A_{S_n}$ for some $S_n\subset S\setminus\{s_n\}$. By \cite[Lemma 34]{cumplido2019minimal}, $Z_n=dc_{S_n}d^{-1}$.
	Let $Z^{-1}_{1}=a^{-1}b$ be the $np$-normal form of $Z^{-1}_1$ in $A_1$. Now consider $A_1$ with a new generating set $\bar T$ whose generators are inverses of the original generators.  Then $a^{-1}b$ can be viewed as the $pn$-normal form of $Z^{-1}_1$ in $A_1$ with respect to $\bar T$. Applying \cite[Theorem 4]{cumplido2019minimal} with $\bar T$ in mind, we know $b\mathsf{P}_1b^{-1}=A_{S_1}$ for some $S_1\subset S\setminus\{s_1\}$. By \cite[Lemma 34]{cumplido2019minimal} (with respect to $\bar T$), $bZ^{-1}_1b^{-1}=c^{-1}_{S_1}$. Thus $c_{S_1}=bZ_1b^{-1}$ and $b^{-1}c_{S_1}b=Z_1$.
	
	As $b^{-1}c_{S_1}bdc_{S_n}d^{-1}=dc_{S_n}d^{-1}b^{-1}c_{S_1}b$, $c_{S_1}$ commutes with $bdc_{S_n}d^{-1}b^{-1}$. 
	Let $ef^{-1}$ by the $pn$-normal form of $bdc_{S_n}d^{-1}b^{-1}$ in $A_\Gamma$. Then by using \cite{cumplido2019minimal} as before we know $bdc_{S_n}d^{-1}b^{-1}=fc_{S'_n}f^{-1}$ for some $S'_n\subset S$. By \cite[Theorem 2.6]{charney1995geodesic}, we know $bd=fg$ for some positive word $g$. Then $gc_{S_n}g^{-1}=c_{S'_n}$.

	As $c_{S_1}$ commutes with $fc_{S'_n}f^{-1}$, we know $c_{S_1}fc_{S'_n}f^{-1}c^{-1}_{S_1}=fc_{S'_n}f^{-1}$. As $fc_{S'_n}f^{-1}$ is a $pn$-normal form, by \cite[Theorem 2.6]{charney1995geodesic}, $fk= c_{S_1}f$ for some positive word $k$. In particular, $f^{-1}c_{S_1}f$ is positive. By \cite[Corollary 6.5]{cumplido2019parabolic}, $f^{-1}c_{S_1}f=c_{S'_1}$ for some $S'_1\subset S$. From $c_{S_1}fc_{S'_n}f^{-1}c^{-1}_{S_1}=fc_{S'_n}f^{-1}$, we know $f^{-1}c_{S_1}fc_{S'_n}f^{-1}c^{-1}_{S_1}f=c_{S'_n}$. Thus $c_{S'_1}$ commutes with $c_{S'_n}$.
	Thus the claim is proved.
	
	Let $\Si$ and $\od$ be the Davis complex and the oriented Davis complex associated with $W_S$ and $A_S$.  Recall that vertices of $\od^1$ are identified with $W_\Ga$. Let $x_0$ be the identity vertex. Let $E_i$ be the face of $\Si$ associated with $A_i$ for $i=1,n$.
	\begin{figure}
		\centering
		\includegraphics[scale=1]{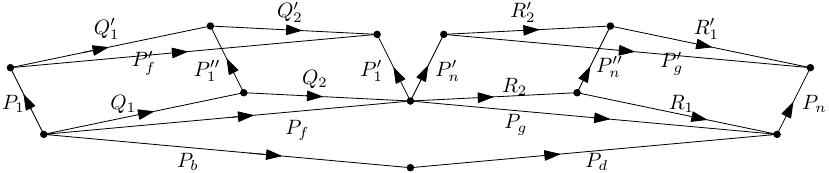}
		\caption{Some paths.}
			\label{fig:7}
	\end{figure} 
	We define positive paths $P_b,P_d,P_f,P'_f,P_g,P'_g,P_1,P'_1,P_n,P'_n$ on $\od^1$ through the following properties (see Figure~\ref{fig:7}):
	\begin{itemize}
		\item $\sft(P_b)=x_0$ and $\w(P_b)=b$; $\sfs(P_d)=x_0$ and $\w(P_d)=d$;
		\item $\sfs(P_f)=\sfs(P_b)=\sfs(P_1)$, $\w(P_f)=f$ and $\w(P_1)=c_{S_1}$;
		\item $\sfs(P'_f)=\sft(P_1)$ and $\w(P'_f)=f$;
		\item $\sfs(P'_1)=\sfs(P'_n)=\sfs(P_g)=\sft(P_f)$, $\w(P'_1)=c_{S'_1}$, $\w(P'_n)=c_{S'_n}$ and $\w(P_g)=g$;
		\item  $\sfs(P_n)=\sft(P_g)$ and $\w(P_n)=c_{S_n}$;
		\item  $\sft(P'_g)=\sft(P_n)$ and $\w(P'_g)=g$.
	\end{itemize}
	
	Note that $P_fP_g\eq P_bP_d$, $P_1P'_f\eq P_fP'_1$ and $P'_nP'_g\eq P_gP_n$. 
	Let $K$ be the convex hull of $\vertex E_1\cup\vertex E_n$ (see the beginning of Section~\ref{subsec:convex hull} for relevant definitions). We claim \begin{equation}
		\label{e:containment}
		\vertex([P_fP_g]\vee_p[\Pi_{\widehat E_1}(P_fP_g)])\subset K.
	\end{equation}
Now we prove this claim. Note that $[P_fP_g]\vee_p[\Pi_{\widehat E_1}(P_fP_g)]=[P_bP_d]\vee_p[\Pi_{\widehat E_1}(P_bP_d)]=[P_b]([P_d]\vee_p[\Pi_{\widehat E_1}(P_d)])$. As $\Pi_{\widehat E_1}(\widehat E_n)=\widehat{E_1\cap E_n}$ by Lemma~\ref{lem:more gate} and Lemma~\ref{lem:retraction property}, we know $\Pi_{\widehat E_1}(P_d)\subset \widehat E^1_n$. Then both $[\Pi_{\widehat E_1}(P_d)]$ and $[P_d]$ are $\prec (\delta_{x_0}(E_n))^k$ for some $k$. Thus $\vertex([P_d]\vee_p[\Pi_{\widehat E_1}(P_d)])\subset \vertex E_n$ (recall that this means any path representing $[P_d]\vee_p[\Pi_{\widehat E_1}(P_d)]$ has vertex set in $E_n$). Thus one path representing $[P_fP_g]\vee_p[\Pi_{\widehat E_1}(P_fP_g)]$ has vertex set in $\vertex E_1\cup\vertex E_n$. Thus \eqref{e:containment} follows.

	Next we claim there exist positive paths $Q_1$, $Q_2$, $Q'_1$, $Q'_2$ and $P''_1$ such that
	\begin{enumerate}
		\item $[P_f]=[Q_1Q_2]$, $[P'_f]=[Q'_1Q'_2]$, $[P_1Q'_1]=[Q_1P''_1]$ and $[P''_1Q'_2]=[Q_2P'_1]$;
		\item $Q_1\subset\widehat E^1_1$;
		\item $\w(Q_i)=\w(Q'_i)$ for $i=1,2$ and $\w(P''_1)=c_{S''_1}$ for some $S''_1\subset S\setminus \{s_1\}$;
		\item if $Q_2$ is non-trivial, then the first edge of $Q_2$ is labeled by $s_1$ and $Q_2$ is left irreducible.
	\end{enumerate}

	To see this, as $[P_fP'_1]=[P_1P'_f]$, $\w(P'_f)=\w(P_f)$ and $\w(P'_1)=c_{S'_1}$, we know from the type II decomposition of Corollary~\ref{cor:decomposition} that $$P_f\eq v_{n+1}v_n\cdots v_1\ \mathrm{and}\ P'_f\eq v'_{n+1}v'_n\cdots v'_1$$ with $v_i$ and $v'_i$ are elementary $B$-segments satisfying the requirements in Corollary~\ref{cor:decomposition}, for some suitable $B\in \cq_\Gamma$. In particular $v_{n+1}\subset\widehat E^1_1$. Let $i_0$ be the biggest possible such that $s_1\in\supp(v_{i_0})$ (note that $i_0\le n$). Let $u=v_{i_0}\cdots v_1$. 
	
	We first show $[u]=[\bar v\bar v_1\cdots \bar v_m]$ where each $\bar v_i$ is an elementary $B$-segment, $s_1\in\supp(\bar v_1)$, $\bar v$ is a concatenation of elementary $B$-segments such that $\bar v\subset \widehat E^1_1$, and $\supp(\bar v_1\cdots \bar v_i)$ is irreducible for $1\le i\le m$. This can be proved by induction on $i_0$. Namely by induction we can assume $[v_{i_0}\cdots v_2]=[\bar v\bar v_1\cdots \bar v_m]$ satisfying all the requirements. As $\supp(v_1)$ is irreducible (Lemma~\ref{lem:irreducible segment} (1)), if $\supp(v_1)\perp \supp(\bar v_1\cdots \bar v_m)$ is not true, then $\supp(\bar v_1\cdots \bar v_mv_1)$ is irreducible and we are done. If $\supp(v_1)\perp \supp(\bar v_1\cdots \bar v_m)$, by repeatedly applying  Lemma~\ref{lem:irreducible segment} (3), we know $[\bar v_1\cdots \bar v_mv_1]=[u_1\bar u_1\cdots \bar u_m]$ where $u_1$ and each $\bar u_i$ are elementary $B$-segments, $\supp(u_1)=\supp(v_1)$ and $\supp(\bar u_i)=\supp(\bar v_i)$ for each $i$. As $s_1\in\supp(\bar v_1)$, we know $s_1\notin \supp(v_1)=\supp(u_1)$. Then we can merge $u_1$ and $\bar v$ to form our new $\bar v\subset\widehat E^1_1$.
	
	We define $Q_1=v_{n+1}\cdots v_{i_0+1}\bar v$ and $Q_2=\bar v_1\cdots \bar v_m$. Lemma~\ref{lem:left irreducible} (1) implies that $Q_2$ is left irreducible. Corollary~\ref{cor:decomposition} implies that there exists a face $F$ (dual to $B$) containing $x=\sft(v_{n+1}\cdots v_{i_0+1})$ and $p>0$ such that $$[v_{n+1}\cdots v_{i_0+1}](\delta_x(F))^p=[P_1][v'_{n+1}\cdots v'_{i_0+1}].$$
	By repeatedly applying Lemma~\ref{lem:irreducible segment} (2), there exist a face $F'$ dual to $B$ containing $x'=\sft(\bar v)$ and positive paths $\bar v'$ and $Q'_2$ made of elementary $B$-segments with
	$$
	(\delta_x(F))^p[\bar v]=[\bar v'](\delta_{x'}(F'))^p\ \mathrm{and}\ (\delta_{x'}(F'))^p[Q'_2]=[Q_2P'_1],
	$$
	$\w(\bar v)=\w(\bar v')$ and $\w(Q'_2)=\w(Q_2)$. We define $P''_1=(\delta_{x'}(F))^p$ and $Q'_1=v'_{n+1}\cdots v'_{i_0+1}\bar v'$. 
	As $\bar v_1$ starts with an edge labeled by $s_1$, the claim is proved.
	
	Similarly, there exist positive paths $R_1$, $R_2$, $R'_1$, $R'_2$ and $P''_n$ such that
	\begin{enumerate}
		\item $[P_g]=[R_2R_1]$, $[P'_g]=[R'_2R'_1]$, $[P'_nR'_2]=[R_2P''_n]$ and $[P''_nR'_1]=[R_1P_n]$;
		\item $R_1\subset\widehat E^1_n$;
		\item $\w(R_i)=\w(R'_i)$ for $i=1,2$ and $\w(P''_n)=c_{S''_n}$ for some $S''_n\subset S\setminus \{s_n\}$;
		\item if $R_2$ is non-trivial, then the last edge of $R_2$ is labeled by $s_n$ and $R_2$ is right irreducible.
	\end{enumerate}
	
	Let $u_0=Q_2R_2$. The next claim we want to prove is that $\supp(u_0)$ does not contain each node of $\Lambda'$. We argue by contradiction and assume each node of $\Lambda'$ is contained in $\supp(u_0)$. 
	Let $k$ be the largest number such that $s_k\in \supp(Q_2)$. As $s_1\in \supp(Q_2)$, $\supp(Q_2)$ is irreducible and the Dynkin diagram $\Lambda$ is tree, we know $s_i\in\supp(Q_2)$ for $1\le i\le k$. Write $R_2=f_1f_2\cdots f_{n'}$ where each $f_i$ is an edge of $R_2$.
	As $R_2$ is right irreducible and ends with an edge labeled by $s_n$, for any $k<i<n$, we can always find an edge labeled by $s_{i+1}$ after the first occurrence of an edge in $R_2$ labeled $s_i$ (this uses that the Dynkin diagram $\Lambda$ is a tree as well).
	
	We turn $Q_2R_2$ into a left irreducible path in a similar way as before: either $Q_2f_1$ is left irreducible, or we must have $\supp(f_1)\perp\supp(Q_2)$. In the latter case $Q_2f_1=\bar f_1\bar Q_2$ with $\w(\bar f_1)=\w(f_1)$ and $\w(Q_2)=\w(\bar Q_2)$. As $s_1\in\supp(Q_2)$, we know $\w(f_1)\neq s_1$, thus $\bar f_1\subset\widehat E^1_1$. Now we run the same process by adding $f_2$ to the end of word produced in the first step. Repeating this process we will obtain $[Q_2R_2]=[UV]$ where $U$ is a positive path in $\widehat E^1_1$ and $V$ is a left irreducible positive path starting with an edge labeled by $s_1$. Moreover, the discussion on right irreducibility of $R_2$ in the previous paragraph and our construction of $V$ imply that $\supp(V)$ contains all the nodes in $\Lambda'$. To finish the prove the claim, it remains to show 
	\begin{equation}
		\label{e:join}
		\vertex([V]\vee_p[\Pi_{\widehat E_1}(V)])\subset K,
	\end{equation}
	as \eqref{e:join} and Lemma~\ref{lem:triming} in the next subsection imply that $s_n\notin \supp(V)$, which leads to a contradiction.
	Now we prove \eqref{e:join}.
	Recall from \eqref{e:containment} that $$\vertex([Q_1Q_2R_2R_1]\vee_p[\Pi_{\widehat E_1}(Q_1Q_2R_2R_1)])\subset K.$$ As $Q_1\subset\widehat E^1_1$, we know $\Pi_{\widehat E_1}(Q_1Q_2R_2R_1)=Q_1\Pi_{\widehat E_1}(Q_2R_2R_1)$. Thus $$[Q_1Q_2R_2R_1]\vee_p[\Pi_{\widehat E_1}(Q_1Q_2R_2R_1)]=[Q_1]([Q_2R_2R_1]\vee_p[\Pi_{\widehat E_1}(Q_2R_2R_1)]).$$ Then $\vertex([Q_2R_2R_1]\vee_p[\Pi_{\widehat E_1}(Q_2R_2R_1)])\subset K$. As $$[Q_2R_2]\vee_p[\Pi_{\widehat E_1}(Q_2R_2)]\prec [Q_2R_2R_1]\vee_p[\Pi_{\widehat E_1}(Q_2R_2R_1)],$$ we know $\vertex([Q_2R_2]\vee_p[\Pi_{\widehat E_1}(Q_2R_2)])\subset K$. As $[Q_2R_2]=[UV]$ and $U\subset \widehat E^1_1$, we deduce \eqref{e:join}.
	
Let $s'$ be a node of $\Lambda'$ not contained in $\supp(Q_2R_2)$. As $\supp(Q_2)$ is irreducible, and $\supp(R_2)$ is irreducible, we know each of $\supp(Q_2)$ and $\supp(R_2)$ is contained in a connected component of $\Lambda\setminus\{s'\}$. However, they are contained in different components of $\Lambda\setminus\{s'\}$ if they are both nonempty, as
$s_1\in\supp(Q_2)$ if $Q_2$ is nontrivial, $s_n\in\supp(R_2)$ if $R_2$ is nontrivial. It follows that
	$\supp(Q_2)\perp\supp(R_2)$, $s_n\notin \supp(Q_2)$ and $s_1\notin \supp(R_2)$.

%,  thus $\supp(Q_2)$ and $\supp(R_2)$ are in different connected components of $\Lambda\setminus\{s'\}$. 
	
	It follows that $[Q_2R_2]=[\widetilde R_2\widetilde Q_2]$ for positive paths $\widetilde R_2$ and $\widetilde Q_2$ with $\w(\widetilde Q_2)=\w(Q_2)$ and $\w(\widetilde R_2)=\w(R_2)$. As $\widetilde R_2$ starts with a vertex in $E_1$ and $\supp(\widetilde R_2)\subset\supp(E_1)$, we know $\widetilde R_2\subset\widehat E^1_1$. Similarly, $\widetilde Q_2\subset\widehat E^1_n$. Then $$[Q_1\widetilde R_2\widetilde Q_2R_1]=[P_bP_d].$$ Thus $$\w(Q_1\widetilde R_2\widetilde Q_2R_1)=\w(P_bP_d)=bd.$$ Let $q_1=\w(Q_1)$, $r_2=\w(\widetilde R_2)$, $q_2=\w(\widetilde Q_2)$ and $r_1=\w(R_1)$. Then $q_1r_2q_2r_1=bd$ in $A_\Gamma$. Then $b^{-1}q_1r_2=dr^{-1}_1q^{-1}_2$. Note that $b,q_1,r_2\in A_1$ and $d,r_1,q_2\in A_n$. Define $\alpha=b^{-1}q_1r_2$, then $\alpha\in A_1\cap A_n$. 
	From $[P_1Q'_1]=[Q_1P''_1]$, we know $c_{S_1}q_1=q_1c_{S''_1}$. Thus $$Z_1=b^{-1}c_{S_1}b=(b^{-1}q_1r_2)r^{-1}_2(q^{-1}_1c_{S_1}q_1)r_2(b^{-1}q_1r_2)^{-1}=\alpha(r^{-1}_2 c_{S''_1}r_2)\alpha^{-1}.$$ Similarly, $$Z_n=\alpha(q_2c_{S''_n}q^{-1}_2)\alpha^{-1}.$$ Thus $$\alpha^{-1}Z_1\alpha=r^{-1}_2 c_{S''_1}r_2\ \mathrm{and}\ \alpha^{-1}Z_n\alpha=q_2c_{S''_n}q^{-1}_2.$$ To finish the proof of the proposition, it suffices to show there exists $s\in\Lambda'$ such that $$r_2,c_{S''_1},q_2,c_{S''_n}\in A_{\hat s}.$$ 
	
	Let $I_1$ be the union of irreducible components of $\supp(c_{S''_1})$ that are not contained in $\supp(q_2)$. Let $I_n$ be the union of irreducible components of $\supp(c_{S''_n})$ that are not contained in $\supp(r_2)$. Lemma~\ref{lem:left irreducible} (1) implies that 
	\begin{align*}
		&I_1\perp \supp(q_2),\ \ I_n\perp \supp(r_2)\\
		&\supp(q_2)\cup \supp(c_{S''_1})= I_1\cup\supp(q_2)\\
		&\textrm{and}\ \supp(r_2)\cup \supp(c_{S''_n})= I_n\cup\supp(r_2).
	\end{align*}
	Lemma~\ref{lem:left irreducible} also implies that $I_1$ (resp. $I_n$) is the union of a collection of irreducible components of $\supp(c_{S'_1})$ (resp. $\supp(c_{S'_n})$). 
As $c_{S'_1}$ commutes with $c_{S'_n}$, by Lemma~\ref{cor:commute}, for an irreducible component from $I_1$ and an irreducible component from $I_2$, either they commute, or one is contained in another. 
	
	If both $q_2$ and $r_2$ are trivial, then $c_{S'_1}=c_{S''_1}\in A_1$ and $c_{S'_n}=c_{S''_n}\in A_n$. So $s_1\notin I_1$ and $s_n\notin I_n$. If $s_n\notin I_1$ or $s_1\notin I_n$, then we are done. Now assume $s_n\in I_1$ and $s_1\in I_n$. Then $s_n\in I_1\setminus I_n$. We claim there is a node of $\Lambda'$ which is not contained in  $I_1\cup I_n$. Suppose the contrary is true.
	Let $T_1$ be the irreducible component of $I_1\cap \Lambda'$ that contains $s_n$ and suppose $T_1=\{s_n,\ldots,s_i\}$. We now prove by contradiction that $i=1$. If $i>1$, then $s_{i-1}\notin I_1$ (otherwise $s_{i-1}\in T_1$). Thus $s_{i-1}\in I_n$. Let $T_n$ be the irreducible component of $I_n\cap \Lambda'$ that contains $s_{i-1}$. As $s_i$ and $s_{i-1}$ do not commute, and $s_{i-1}\notin I_1$, by Lemma~\ref{cor:commute}, the only possibility left is that $T_1\subset T_n$. But this is still impossible as $s_n\in T_1$ and $s_n\notin I_n$. Thus $i=1$. However, this contradicts that $s_1\notin I_1$. Thus $I_1\cup I_n$ cannot contain all nodes of $\Lambda'$.
	
	%If exactly one of $q_2$ and $r_2$, say $q_2$ is trivial, then $c_{S'_1}=c_{S''_1}\in A_1$ and $s_1\notin I_1$. As $s_n\in \supp(r_2)$, we put $i_1$ to be the largest number such that $s_{i_1}\notin \supp(r_2)$. As $\supp(Q_2R_2)$ does not contain each node of $\Lambda'$, we know $i_1$ exists. As $I_n\perp \supp(r_2)$, $s_{i_1}\notin I_n$. If $s_{i_1}\notin I_1$ or $s_1\notin I_n$, then we are done. Now assume $s_{i_1}\in I_1$ and $s_1\in I_n$. Then by the same argument as in the previous paragraph, there is a node $s$ in $\{s_i\mid 1\le i\le i_1\}$ which is not contained in $I_1\cup I_n$. Such node is not in $\supp(r_2)$ as well since $\supp(r_2)$ is irreducible. Thus  $r_2,c_{S''_1},q_2,c_{S''_n}\in A_{\hat s}$. 
	
	Now we assume at least one of $r_2$ and $q_2$ are nontrivial. 
Let $i_n$ be the biggest number such that $s_{i_n}\notin \supp(r_2)$, and let $i_1$ be the smallest number such that $s_{i_1}\notin \supp(q_2)$ (it is possible $i_n=n$ or $i_1=1$). As $\supp(q_2)\perp\supp(r_2)$, we know $i_1\le i_n$, $s_{i_n}\notin \supp(q_2)$ and $s_{i_1}\notin \supp(r_2)$. Now we show 
	\begin{equation}
		\label{e:ncontain}
		s_{i_1}\notin I_1 \ \mathrm{and}\  s_{i_n}\notin I_n.
	\end{equation}
	If $q_2$ and $r_2$ are both nontrivial, then $s_1\in \supp(q_2)$ and $s_n\in \supp(r_2)$. As $I_1\perp\supp(q_2)$ and $I_n\perp \supp(r_2)$, \eqref{e:ncontain} follows. If one of $q_2$ or $r_2$ is trivial, in which case $i_n=n$ or $i_1=1$, hence \eqref{e:ncontain} holds.
	If $s_{i_n}\notin I_1$ or $s_{i_1}\notin I_n$, then $s_{i_n}$ or $s_{i_1}$ is not contained in $$\supp(q_1)\cup\supp(r_2)\cup\supp_{c_{S''_1}}\cup\supp_{c_{S''_n}}.$$
	Now assume $s_{i_n}\in I_1$ and $s_{i_1}\in I_n$. As $s_{i_n}=I_1\setminus I_n$ and $s_{i_1}=I_n\setminus I_1$, by the argument in the previous paragraph, there is a node $s$ in $\{s_i\mid i_1\le i\le i_n\}$ such that $s\notin I_1\cup I_n$. Then $s\notin \supp(q_1)\cup\supp(r_2)\cup\supp(c_{S''_1})\cup\supp(c_{S''_n})$, as desired.
\end{proof}

\subsection{Proof of the main lemma}
Let $\od_\Ga$, $\g^+=\g^+(\Ga)$, $\g=\g(\Ga)$, and $\pi:\od_\Ga\to \Si_\Ga$ be as in Section~\ref{subsec:Deligne}.
Recall that edges of $\Si_\Ga$ and $\od_\Ga$ are labeled by elements in the generating set $S$ (cf. Definition~\ref{def:label}). We will write $\Si=\Si_\Ga$ and $\od=\od_\Ga$. In this subsection we allow the associated Artin group $A_S$ to be arbitrary (not necessarily spherical). We will prove the main lemma, which is Lemma~\ref{lem:triming}, needed in the proof of Proposition~\ref{prop:key}.
\begin{lem}
	\label{lem:tracking label}
	Let $F$ be a standard subcomplex of $\Sigma$. Let $e\subset \widehat F$ be a positive edge from $x$ to $y$. Let $u\subset\widehat F$ be a minimal positive path from $x$ to $z$ which is $F$-escaping. Suppose $[u]$ and $[e]$ have an upper bound in $\g^+_{x\to}$ (cf. Section~\ref{subsec:Deligne}). Then
	\begin{enumerate}
		\item there exists an $F$-escaping minimal positive path $\bar u$ starting from $y$ such that $\supp(\bar u)=\supp(u)$ and $[e\bar u]\prec ([u]\vee_p [e])$;
		\item if $\supp(e)\notin \supp(u)$ and $\supp(e)$ does not commute with $\supp(u)$, then there exists an $F$-escaping minimal positive path $\bar u$ starting from $y$ such that $\supp(\bar u)=\supp(u)\cup \supp(e)$ and $[e\bar u]\prec ([u]\vee_p [e])$.
	\end{enumerate}
\end{lem}

\begin{proof}
	For (1), let $H$ be the wall dual to $e$. As $u$ and $e$ are minimal positive paths, we know $[u]\vee_p [e]$ is represented by a minimal positive path $u'$. Suppose $[u']=[uv]$ for some positive path $v$. As $u'$ is minimal, $u'$ crosses $H$ exactly once, hence the same holds for $v$, as $u$ does not cross $H$ because $u$ is $F$-escaping. Let $e'$ be the edge of $v$ dual to $H$ and let $v_1$ be the subpath of $v$ made of all the edges before $e'$. Let $y$ (resp. $y'$) be the endpoint of $e$ (resp. $e'$). By Lemma~\ref{lem:geodesic in carrier}, there is a minimal positive path $u_1$ from $y$ to $y'$ made of elementary $H$-segments. Then $eu_1$ is a minimal positive path, implying that $[e]\prec[eu_1]$. On the other hand, $[uv_1e']=[eu_1]$. Thus $[u]\prec[eu_1]$, which implies that $[u']=[eu_1]$.
	
	Let $r\in W_\Ga$ be the reflection fixing $H$. Viewing $\od^1$ as the oriented Cayley graph of $W_\Ga$, we  have a left action of $W_\Gamma$ on $\od^1$ which preserves labeling and orientation of edges and sends minimal positive paths to minimal positive paths. Let $\bar u=r(u)$ and $\bar v_1=r(v_1)$. Then $\bar u\bar v_1$ is a minimal positive path from $y$ to $y'$. Thus $$[eu_1]=[e\bar u\bar v_1]=[u]\vee_p [e].$$ Clearly $\supp(\bar u)=\supp(u)$. As $u$ is $F$-escaping and $r(\widehat F^1)=\widehat F^1$, we know $\bar u$ is $F$-escaping. 
	
	%For (2), let $E$ be the standard subcomplex containing $x$ with $\supp(E)=\supp(u)\cup\supp(e)$. As $[u]$ and $[e]$ have an upper bound, by applying the retraction $\Pi_{\widehat F}:\od\to\widehat E$ as in Definition~\ref{def:retraction} to the positive path representing this upper bound, we can assume $[u]$ and $[e]$ have an upper bound represented by an $E$-path. Thus $\supp(u')\subset \supp(u)\cup \supp(e)$. 
	
	For (2), let $e''$ be the edge labeled by $\supp(e)$ starting from the endpoint of $u$. We claim $[u e'']\prec [uv_1]$. Assuming the claim, we can finish the proof by taking $\bar u=r(ue'')$ and repeating the argument of the previous paragraph. As $\supp(e)$ and $\supp(u)$ do not commute, Lemma~\ref{lem:extra edge} (2) implies that $ue''$ is $F$-escaping, so is $\bar u$ by the same argument as before.
	
	It remains to prove the claim. We write $u$ as concatenation of edges $e_1\cdots e_m$. For $2\le i\le m+1$, let $e'_i$ be the edge whose starting point is equal to the ending point of $e_{i-1}$, and $\supp(e'_i)=\supp(e)$. We will prove by induction that $[e_1\cdots e_{i-1}e'_i]\prec[u']$ for $2\le i\le m+1$. The case $i=2$ is clear as $[e_1e'_2]\prec[e_1]\vee_p[e]\prec[u]\vee_p[e]$. Now suppose $[e_1\cdots e_{i-1}e'_i]\prec[u']$ ($i\le m$). Then $$[e_1\cdots e_{i-1}e_i]\vee_p[e_1\cdots e_{i-1}e'_i]\prec[u'].$$ And it follows that $$[e_1\cdots e_{i-1}]([e_i]\vee_p[e'_i])\prec[u'].$$ As $[e_ie'_{i+1}]\prec [e_i]\vee_p[e'_i]$, $[e_1\cdots e_{i}e'_{i+1}]\prec[u']$. Taking $i=m+1$, we know $[ue''] \prec[u']$. By Lemma~\ref{lem:extra edge} (2), $ue''$ is $F$-escaping, it is contained on one side of $H$. Thus there is a positive minimal path $u''$ from the endpoint of $ue''$ to $y'$ such that $e'$ is the last edge of $u''$. As $[ue'']\prec[u']$ and $u'$ is a minimal positive path, we know $[u']=[ue''u'']$. As $u''=u'''e'$ for some minimal positive path $u'''$, we have $[ue''u'''e']=[uv_1e']$. Then $[ue''u''']=[uv_1]$ by Theorem~\ref{thm:deligne} (1). Hence $[ue'']\prec [uv_1]$.
\end{proof}

Recall that a proper standard subcomplex $F$ of $\Si$ is \emph{maximal} if $\supp(F)=S\setminus\{s\}$ for some $s\in S$, where $S$ is the vertex set of $\Gamma$. In this case we define the \emph{type} of $F$ to be $\hat s$.

\begin{lem}
	\label{lem:outside}
	Suppose $E_1$ and $E_n$ are proper maximal standard subcomplexes of $\Si$ of type $\hat s_1$ and $\hat s_n$ respectively such that $E_1\cap E_n\neq\emptyset$.
	Let $u$ be an $E_1$-escaping geodesic in $\Sii$ (or in $\od^1$) ending at $y\notin E_1$. Suppose $\supp(u)$ is irreducible and $s_1,s_n\in \supp(u)$. Then $y$ is not in the convex hull of the union of $\vertex E_1$ and $\vertex E_n$.
\end{lem}

\begin{proof}
	Let $x$ be the other endpoint of $u$. Then $x=\prj_{E_1}(y)$. We first show $y\notin E_n$. Indeed, if $y\in E_n$, then $x\in E_1\cap E_n$ by Lemma~\ref{lem:more gate}, which implies $u\subset E_n$ by the convexity of $\vertex E_n$. However, no edges of $E_n$ are labeled by $s_n$, which contradicts that $s_n\in \supp(u)$. Now let $u'$ be a geodesic in $\Sii$ from $y$ to $\prj_{E_n}(y)$. If $y$ is in the convex hull of $\vertex E_1\cup\vertex E_2$, then by Lemma~\ref{lem:hull}, we know $\supp(u')\perp\supp(u)$. In particular, $s_n\in \supp(u)$ implies $s_n\notin \supp(u')$. However, $y\notin E_n$ and the maximality of $E_n$ implies that  $s_n\in \supp(u')$, which is a contradiction.
\end{proof}

Recall that $\od^1$ denotes the 1-skeleton of $\od$. Let $\Pi_{\widehat F}:\od\to \widehat F$ be the map in Definition~\ref{def:retraction}. It also induces a map $\od^1\to \widehat F^1$. Recall that $\Pi_{\widehat F}\mid_{\vertex \Si}=\prj_{F}:\vertex\Si\to \vertex F$. Note that $\Pi_{\widehat F}$ sends positive paths to positive paths, though it is possible to send a non-trivial positive path to a trivial positive path.

\begin{lem}
	\label{lem:triming}
	Let $A_\Gamma$ be an Artin group (not necessarily spherical). Let $E_1,E_n$ be proper maximal standard subcomplexes of $\Si$ with $\supp(E_1)=S\setminus\{s_1\}$, $\supp(E_n)=S\setminus\{s_n\}$ such that $E_1\cap E_n\neq\emptyset$ and $s_1\neq s_n$.
	Let $K$ be the convex hull of $\vertex E_1\cup\vertex E_n$.	Suppose $u\subset \od^1$ is a positive path (not necessarily minimal) starting at a vertex $x_u\in E_1$ 
	such that
	\begin{enumerate}
		\item $[u]$ and $[\Pi_{\widehat E_1}(u)]$ have an upper bound in $\g^+_{x_u\to}$ (cf. Section~\ref{subsec:Deligne}), and the vertex set of one representative (hence any representative) of $[u]\vee_p[\Pi_{\widehat E_1}(u)]$ is contained in $K$;
		\item $u$ is left irreducible and starts with an edge labeled by $s_1$.
	\end{enumerate}
	Then $s_n\notin \supp(u)$.
\end{lem}

\begin{proof}
	%First we write $\w(u)=a_1a_2$ where $\supp(a_1)=I$ and $\supp(a_2)=\supp(u)\setminus I$. As $a_1$ and $a_2$ commute, $[u]=[u_{a_2}u_{a_1}]$ with $\supp(u_{a_i})=a_i$ for $i=1,2$. As $s_n\in I$, $s_n\notin \supp(u_{a_2})$. Thus $\supp(u_{a_2})\subset \supp(E_1)$ and $u_{a_2}\subset\Ga(\ca,E_1)$. Note that $u_{a_1}$ starts with an edge labeled by $s_n$ and the vertex set of any representative of $[u_{a_1}]\vee_p[\Pi_{\widehat E_1}(u_{a_1})]$ is contained in $K$. Thus it suffices to prove the lemma in the special case when $\supp(u)$ is irreducible, which we will assume from now on.
	
	We refer to Figure~\ref{fig:8} for a schematic picture of the proof.	
	\begin{figure}
		\centering
		\includegraphics[scale=1]{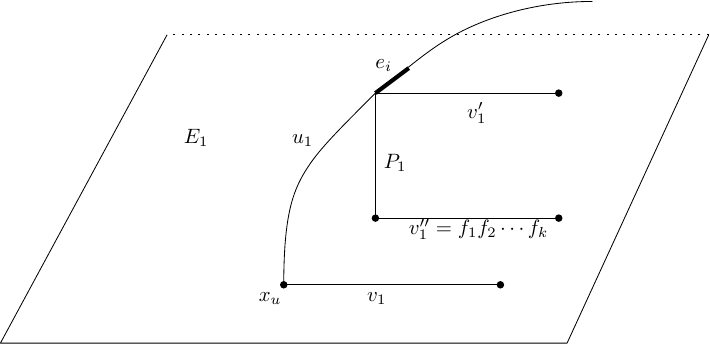}
		\caption{Proof of Lemma~\ref{lem:triming}.}
		\label{fig:8}
	\end{figure}
	We argue by contradiction and assume $s_n\in\supp(u)$. Subsequently we will assume that $\emptyset\perp T$ for any subset $T\subset S$.
	
	%We claim there exists a path $u_0\subset \Gamma(\ca,E_1)$ and a sequence of edges $\{e_i\}_{i=1}^m$ such that 
	%\begin{itemize}
	%	\item $[u]=[u_0e_1e_2\cdots e_m]$;
	%	\item $\supp(e_i)$ does not commute with $\supp(e_1\cdots e_{i-1})$ for any $1\le i\le m$;
	%	\item $\supp(e_1)=s_n$.
	%\end{itemize}
	%We prove the claim by induction on the length of $u$, we write $u$ as a sequence of edges $e'_1e'_2\cdots e'_{m'}$. Note that $\supp(e'_1)=s_n$ by assumption 2. By induction, we can assume $[e'_1\cdots e'_{m'-1}]=[u_0e_1\cdots e_{m}]$ with $u_0$ and $e_i$ satisfying the requirements in the claim. If $\supp(e'_{m'})$ does not commute with $\supp(e_1\cdots e_{m})$, then we write $[u]=[u_0(e_1e_2\cdots e_me'_{m'})]$. Otherwise $\w(e_1\cdots e_{m-1})$ and $\w(e'_{m'})$ commutes in $A_\Ga$, thus
	%$e_1\cdots e_{m-1}e'_{m'}\eq e''_{m'}e''_1\cdots e''_{m-1}$ with $\w(e''_1\cdots e''_{m-1})=\w(e_1\cdots e_{m-1})$ and $\w(e''_{m'})=\w(e'_m)$. As $s_n\in \supp(e_1\cdots e_{m-1})$, $\supp(e'_{m'})\in E_1$. Thus $e''_{m'}\subset E_1$ and  $[u]=[(u_0e''_{m'})(e''_1\cdots e''_{m-1})]$. Then the claim is proved.	

	Let $e$ be the first edge of $u$ labeled by $s_n$. Let $P$ be an $E_1$-escaping minimal positive path ending at the starting point $x$ of $e$ (it is possible that $P$ is trivial). By Lemma~\ref{lem:irreducible}, $\supp(P)$ is irreducible whenever it is nonempty. By Lemma~\ref{lem:support contain}, $\supp(P) \subset \supp(u')$ where $u'$ is the subpath of $u$ made of edges before $e$. Thus $\supp(e)\notin \supp(P)$ as $s_n\notin \supp(u')$.
	
	If $\supp(P)\neq\emptyset$ and $\supp(P)\cup \supp(e)$ is irreducible, then Lemma~\ref{lem:extra edge} implies that $Pe$ is an $E_1$-escaping minimal positive path. As $P$ leaves $E_1$, $s_1\in \supp(P)$. Thus $\{s_1,s_n\}\subset \supp(Pe)$. Then Lemma~\ref{lem:outside} implies that the endpoint of $e$ is not contained in $K$, which is a contradiction.
	
	Now we assume $\supp(e)\perp\supp(P)$ (including the case $\supp(P)=\emptyset$). We write $u=e_1\cdots e_m$ where $e_i$'s are edges of $u$. Let $v=\Pi_{\widehat E_1}(u)$. Then $v$ is non-trivial and it has an edge labeled by $s_n$, namely $\Pi_{\widehat E_1}(e)$. 
	Let $I$ be the irreducible component of $\supp(v)$ that contains $s_n$. Note that $\w(v)=w_1w_2$ in $A_\Gamma$, where $w_1$ and $w_2$ are positive words with $\supp(w_1)=I$ and $\supp(w_2)=\supp(u)\setminus I$. Thus $v\eq v_1v_2$ where $v_1$ is a positive path with $\supp(v_1)=I$.
	
	Recall that $\supp(e_1\cdots e_m)$ is irreducible, hence spans a connected subgraph of the Dynkin diagram $\Lambda$ of $W_\Gamma$. Moreover $\{s_1,s_n\}\in\supp(e_1\cdots e_m)$. On the other hand, $s_1\notin I$ (as $I\subset \supp(E_1)$) and $s_n\in I$. Thus there must exist one edge $f$ in $\{e_j\}_{j=1}^m$ such that $\supp(f)\notin I$ such that $\supp(f)$ does not commute with $I$. Let $e_i$ be the first such $f$ in $e_1\cdots e_m$. 
	
	We now show $\supp(e_1\cdots e_{i-1})\perp I$. Let $I_j=\supp(e_1\cdots e_{j})$. Then $I_1\subset I_2\subset\cdots\subset I_m$, and $I_j\setminus I_{j-1}$ contains at most one vertex.  By the left irreducibility of $u$, each $I_j$ spans a connected subgraph of $\Lambda$. Note that $I_1\cap I=\emptyset$ and $i$ is the smallest number such that $I_i\cap I$ contains a vertex not in $I$ that is adjacent to a vertex in $I$ in $\Lambda$. Then $I_{i}\cap I=\emptyset$, indeed, if by contradiction we assume $s\in I_i\cap I$, then $I_i$ contains an edge path in $\Lambda$ from $s_1\notin I$ to $s\in I$. This path then contains a vertex $s'\notin I$ such that $s'$ is adjacent to a vertex in $I$. By the minimality of $i$, $s'=I_i\setminus I_{i-1}$. As $s\neq s'$, $s\in I_{i-1}$. Then $I_{i-1}$ spans a connected subgraph of $\Lambda$ that contains both vertices in $I$ and vertices outside $I$. Thus $I_{i-1}$ must contain a vertex outside $I$ that is adjacent to a vertex in $I$. This contradicts the minimality of $i$. Hence $I_i\cap I=\emptyset$ is proved. The minimality of $i$ implies there is exactly one vertex in $I_i$  which does not commute with $I$, and this vertex is in $I_i\setminus I_{i-1}$. Thus $I_{i-1}\perp I$.
	
	Let $P_1$ be an $E_1$-escaping minimal positive path ending at the starting point of $e_i$ (it is possible that $P_1$ is the trivial path). Let $P'_1=P_1e_i$. We now show $\supp(P'_1)$ is irreducible. By Lemma~\ref{lem:irreducible}, $\supp(P_1)$ is irreducible. Thus if $\supp(P'_1)$ is reducible, then $$\supp(P_1)\neq\emptyset\ \mathrm{and}\ \supp(P_1)\perp \supp(e_i).$$ As $P_1$ is a non-trivial path leaving $E_1$, $s_1\in \supp(P_1)$. Hence $\supp(e_i)\in \supp(E_1)$.
	Then $d_i=\Pi_{\widehat E_1}(e_i)$ is an edge with the same support as $e_i$. As $\supp(d_i)\notin I$ and $\supp(d_i)$ does not commute with $I$, $I$ is not an irreducible component of $\supp(v)$, contradicting the choice of $I$.
	
	We also have $s_1\in \supp(P'_1)$. Indeed, if this is not true, then $P_1$ is the trivial path and $e_i\subset E_1$. Then $e_i=\Pi_{\widehat E_1}(e_i)$. As $\supp(e_i)\notin I$ and $\supp(e_i)$ does not commute with $I$, we deduce a contradiction to the choice of $I$ as before.
	
	Let $u_1=e_1\cdots e_{i-1}$. Let $[u_2]=[u_1e_i]\vee_p [v_1]$. Let $v'_1$ be the positive path starting at the endpoint of $u_1$ such that $\w(v'_1)=\w(v_1)$, (see Definition~\ref{def:label} for $\w(v_1)$). Since $\supp(u_1)\perp\supp(v_1)$ as we proved before, we know $[u_1]\vee_p [v_1]=[u_1v'_1]$. Thus $$[u_2]=[u_1]([e_i]\vee_p [v'_1]).$$ 
	
	Let $v''_1$ be the positive path with the same starting point as $P_1$ such that $\w(v''_1)=\w(v'_1)$. As $\supp(P_1)\subset\supp(u_1)$ by Lemma~\ref{lem:support contain}, we know $$\supp(P_1)\perp\supp(v_1).$$ Thus $[P_1]\vee_p [v''_1]=[P_1v'_1]$. Hence $$[P_1e_i]\vee_p [v''_1]=[P_1]([e_i]\vee_p [v'_1]).$$

	As $\supp(P'_1)$ and $\supp(v''_1)$ are irreducible, and $\supp(P'_1)$ contains $\supp(e_i)$ which is adjacent to a vertex in $\supp(v''_1)=\supp(v_1)$ in the Dynkin diagram $\Lambda$, we know that $\supp(P'_1)\cup \supp(v''_1)$ is irreducible. We write $v''_1$ as a sequence of edges $f_1f_2\cdots f_k$. By Lemma~\ref{lem:extra edge} (2), $P'_1$ is $E_1$-escaping. We claim for $1\le i\le k+1$, there is an $E_1$-escaping minimal positive path $P'_i$ with the same endpoint as $f_{i-1}$ such that $$[f_1\cdots f_{i-1}P'_i]\prec ([P'_1]\vee_p [f_1\cdots f_{i-1}])$$ and $$\supp(P'_i)=\supp(P'_1)\cup \supp (f_1\cdots f_{i-1}).$$ Note that the claim is true for $i=1$. Suppose the claim is true for some $i$. Let $Q_i$ be a positive path such that $$[P'_1]\vee_p [f_1\cdots f_{i-1}]= [f_1\cdots f_{i-1}Q_i].$$ Then $[P'_i]\prec [Q_i]$. Note that $$[P'_1]\vee_p [f_1\cdots f_{i}]=[f_1\cdots f_{i-1}]([Q_i]\vee_p [f_i]).$$ 
	By Lemma~\ref{lem:tracking label}, there exists an $E_1$-escaping minimal positive path $P'_{i+1}$ with $\supp(P'_{i+1})=\supp(P'_{i})\cup\supp(f_i)$ and $[f_iP'_{i+1}]\prec [P'_{i}]\vee_p [f_i]$. Thus 
	\begin{align*}
		[f_1\cdots f_{i}P'_{i+1}]&\prec [f_1\cdots f_{i-1}]([P'_i]\vee_p [f_i])\prec [f_1\cdots f_{i-1}]([Q_i]\vee_p [f_i])\\
		&=[f_1\cdots f_{i-1}Q_i]\vee_p [f_1\cdots f_{i-1}f_i]=
		[P'_1]\vee_p [f_1\cdots f_{i}]
	\end{align*}
	and $\supp(P'_{i+1})=\supp(P'_1)\cup \supp (f_1\cdots f_{i})$. Thus the claim is proved.
	Take $i=k+1$, then 
	$$[v''_1P'_{k+1}]\prec [P'_1]\vee_p[v''_1]\ \textrm{and}\ \supp(P'_{k+1})=\supp(P'_1)\cup \supp (v''_1).$$ 
	Recall that $\supp(P'_1)\cup \supp(v''_1)$ is irreducible, $s_1\in \supp(P'_1)$ and $s_n\in\supp(v''_1)=\supp(v_1)$. By Lemma~\ref{lem:outside}, $\vertex P'_{k+1}\nsubseteq K$. On the other hand, we will now show $\vertex P'_{k+1}\subset K$, which leads to a contradiction. For a class of positive path $[w]$, we write $\vertex[w]\subset K$ if any member in $[w]$ has vertex set contained in $K$. The convexity of $K$ implies that this is equivalent to one member in $[w]$ has vertex set contained in $K$.
	By assumption 1, $$\vertex ([u]\vee_p[\Pi_{\widehat E_1}(u)])\subset K.$$
	As $[ u_1e_i]\prec[u]$ and $[ v_1]\prec[\Pi_{\widehat E_1}(u)]$,  we know $$\vertex([ u_1e_i]\vee_p[ v_1])\subset K.$$ 
	Recall that $[ u_1e_i]\vee_p[v_1]=[u_1]([e_i]\vee_p [v'_1])$, then it follows that $$\vertex([e_i]\vee_p [v'_1])\subset K.$$ 
	Note that $P_1$ is a geodesic between the starting point of $e_i$, which is in $K$, and a point in $E_1$, which is in $K$. Then the convexity of $K$ implies $\vertex P_1=K$. Then $$\vertex([P_1]([e_1]\vee_p [v'_1]))\subset K.$$ 
	Recall that $[P_1]([e_1]\vee_p [v'_1])=[P'_1]\vee_p [v''_1]$, we know that $$\vertex([P'_1]\vee_p [v''_1])\subset K.$$ As $[v''_1P'_{k+1}]\prec [P'_1]\vee_p[v''_1]$, it follows that $\vertex [v''_1P'_{k+1}]\subset K$. Thus $\vertex P'_{k+1}\subset K$, as desired.
\end{proof}
\section{Artin groups with tree Dynkin diagrams}
\label{sec:tree}

Throughout this section, we will use the following convention. Let $\Lambda$ be a simplicial graph and let $E$ be an arbitrary subset of $\Lambda$. We define a \emph{component} of $\Lambda\setminus E$, to be the induced subgraph spanned by nodes in a connected component of $\Lambda\setminus E$. Note that each component is connected.

The key result in this Section is Proposition~\ref{prop:tree contractible}. Section~\ref{subsec:preserve} and Section~\ref{subsec:contractible} is devoted to the proof of Proposition~\ref{prop:tree contractible}, and Section~\ref{subsec:application} discusses several applications of Proposition~\ref{prop:tree contractible}. The key step in the proof of Proposition~\ref{prop:tree contractible} is in Section~\ref{subsec:preserve}, where we discuss the relation between the $\CAT(0)$ property, the bowtie free property, and the labeled 4-cycle condition. 

\subsection{Two motivating examples}
\label{subsec:examples}
In this subsection we give an informal discussion about two examples, which motivates the rest of this section.
Let $\Lambda$ be a linear Dynkin diagram with consecutive nodes being $\{s_i\}_{i=1}^n$ with $n\ge 3$ such that the edge between $s_1$ and $s_2$ is labeled $6$, and all other edges are labeled $3$. 

We want to show $A_\Lambda$ satisfies the $K(\pi,1)$-conjecture. A key step is to show $\Delta_\Lambda$ is contractible. The proof has two steps. First is to show (under an appropriate induction assumption) that $\Delta_{\Lambda}$ deformation retracts to its core, $\Delta_{\Lambda,\Lambda'}$, where $\Lambda'$ is the full subgraph spanned by $\{s_1,s_2,s_3\}$. The second step is to show $\Delta_{\Lambda,\Lambda'}$ is contractible. As the second step takes more effort, we will focus on that. 

The plan is to metrize the 2-dimensional complex $\Delta_{\Lambda,\Lambda'}$ so that its 2-simplices are flat triangles, with angle $\pi/2$ at vertices of type $\hat s_2$, angle $\pi/3$ at vertices of type $\hat s_1$, and angle $\pi/6$ at vertices of type $\hat s_3$. As we already know $\Delta_{\Lambda,\Lambda'}$ is simply-connected (Lemma~\ref{lem:relative sc}), if it is locally CAT$(0)$, then it is CAT$(0)$, and contractibility of $\Delta_{\Lambda,\Lambda'}$ follows. Now we are reduced to checking link conditions at each vertex. The most interesting case is the link condition for vertices of type $\hat s_1$, where we need to check the links have girth $\ge 6$. Let $\Lambda_1$ be the component of $\Lambda\setminus\{s_1\}$, and $\Lambda'_1=\Lambda_1\cap\Lambda'$. Then for any vertex $x\in \Delta_{\Lambda,\Lambda'}$ of type $\hat s_1$, there is a label-preserving isomorphism between $\lk(x,\Delta_{\Lambda,\Lambda'})$ and $\Delta_{\Lambda_1,\Lambda'_1}$ by Lemma~\ref{lem:link}. As $\Lambda_1$ is of type $A_n$, it follows from the labeled 4-cycle condition on $\Delta_{\Lambda_1}$ (Corollary~\ref{cor:wheel}) that $\Delta_{\Lambda_1,\Lambda'_1}$ has girth $\ge 6$.

Now we look at our second motivating example. Let $\Lambda$ and $\{s_i\}_{i=1}^n$ be as before. However, we require the first edge $\overline{s_1s_2}$, and the last edge $\overline{s_{n-1}s_n}$ have label $=6$, and all other edges have label $=3$. We also assume $n\ge 4$. Let $\Lambda'$ be as before. We still want to show $\Delta_{\Lambda}$ deformation retracts onto a core $\Delta_{\Lambda,\Lambda'}$, and this core can be equipped as a CAT$(0)$ metric as before, hence contractible. Again, the interesting part is to check the link condition on vertices of type $\hat s_1$. 

Let $\Lambda_1$ and $\Lambda'_1$ be as before.
Then $\Lambda_1$ is no longer spherical. Actually $A_{\Lambda_1}$ is an Artin group about which we know very little from previous literature (when $n\ge 6$). It seems the only way we know how to prove $\Delta_{\Lambda,\Lambda_1}$ has girth $\ge 6$, is to prove a much stronger statement, namely $\Delta_{\Lambda_1}$ satisfies the labeled 4-cycle condition. While we do not know this property for $\Delta_{\Lambda_1}$, we know $\Delta_{\Lambda_1\setminus \{s_n\}}$ has this property, as $\Lambda_1\setminus\{s_n\}$ is a diagram of type $A_{n-1}$. So we need a mechanism for promoting the labeled 4-cycle condition on $\Delta_{\Lambda_1\setminus \{s_n\}}$ to the same condition on $\Delta_{\Lambda_1}$.

We commented earlier in Section~\ref{subsec:proof2} (after Definition~\ref{def:4wheel}) that the labeled 4-cycle condition is connected to the CAT$(1)$ property. However, if we know a much stronger property, say that $\Delta_{\Lambda_1}$ is CAT$(0)$, then this will be helpful for proving labeled 4-cycle condition, as this property is about filling a 4-cycle by certain type of combinatorial disks in the 2-skeleton. While $\Delta_{\Lambda_1}$ being CAT$(0)$ does not directly imply this, it becomes relevant as we can use CAT$(0)$ geometry to produce nice filling disk for 4-cycles. Again, the only way we know how to prove $\Delta_{\Lambda_1}$ satisfies the labeled 4-cycle property, is to prove a seemingly much stronger property that $\Delta_{\Lambda_1}$ has some form of non-positive curvature. We do not know how to prove $\Delta_{\Lambda_1}$ is CAT$(0)$. However, let $\Theta$ be the full subgraph of $\Lambda_1$ spanned by $\{s_{n-2},s_{n-1},s_n\}$. Then the discussion in the previous example implies that $\Delta_{\Lambda_1,\Theta}$ is CAT$(0)$. It turns out this CAT$(0)$ property does help in promoting the labeled 4-cycle condition on $\Delta_{\Lambda_1\setminus \{s_n\}}$ to the same condition on $\Delta_{\Lambda_1}$, see Lemma~\ref{lem:enlarge1}. The general principle for this kind of promotion, as we mentioned in the induction scheme in Section~\ref{subsec:core}, is that the relative Artin complex $\Delta_{\Lambda_1,\Theta}$ and $\Delta_{\Lambda_1}$ are acted upon by the same group $A_{\Lambda_1}$, which helps us to relate properties of $\Delta_{\Lambda_1}$ to non-positive curvature of $\Delta_{\Lambda_1,\Theta}$.

This example also shows why it is relevant in our strategy that we need to promote labeled 4-cycle condition in a smaller relative Artin complex, to the same condition in a bigger relative Artin complex. This is actually the key technical step in the proof of Proposition~\ref{prop:tree contractible}. In Section~\ref{subsec:preserve}, we discuss three scenarios that such promotion is possible, namely Lemma~\ref{lem:enlarge1}, Lemma~\ref{lem:enlarge2} and Lemma~\ref{lem:enlarge3}.

\subsection{Preservation of bowtie free properties}
\label{subsec:preserve}
\begin{definition}
Let $\Lambda$ be a Dynkin diagram. We take an edge $e\subset \Lambda$. We say $e$ is \emph{$m$-solid} in $\Lambda$ if the $(\Lambda,e)$-relative Artin complex is a graph with girth $\ge 2m$. 
\end{definition}

The following is a direct consequence of Lemma~\ref{lem:embedding}.

\begin{lem}
	\label{lem:solid}
Let $\Lambda'$ be an induced subgraph of $\Lambda$ that contains $e$. If $e$ is $m$-solid in $\Lambda$, then $e$ is $m$-solid in $\Lambda'$.
\end{lem}

\begin{lem}
	\label{lem:enlarge1}
Suppose the Dynkin diagram $\Lambda$ is a tree. Let $\Lambda'$ be a linear subgraph of $\Lambda'$ with its consecutive nodes being $\{s_i\}_{i=1}^n$. 
We assume
\begin{enumerate}
	\item the edge $\overline{s_1s_2}$ is $6$-solid in the component of $\Lambda\setminus\{s_3\}$ that contains $\overline{s_1s_2}$;
	\item if $\Lambda'_1$ is the subgraph of $\Lambda'$ spanned by all the nodes in $\Lambda'$ except $s_1$ and $C_1$ is the component of $\Lambda\setminus\{s_1\}$ containing $\Lambda'_1$, then the  $(C_1,\Lambda'_1)$-relative Artin complex is bowtie free (cf. Definition~\ref{def:bowtie free}).
\end{enumerate}
Then the $(\Lambda,\Lambda')$-relative Artin complex is bowtie free.
\end{lem}

\begin{proof}
In the following proof, we will view links of vertices as subcomplexes of the ambient complex.	
Let $\Lambda_0$ be the subgraph of $\Lambda$ spanned by $\{s_1,s_2,s_3\}$. Let $X$ be the $(\Lambda, \Lambda_0)$-relative Artin complex. Note that $X$ is made of triangles. We metrize the triangles  such that they are geodesic triangles in the Euclidean plane with angle $\pi/2$ at vertices of type $\hat s_2$, angle $\pi/6$ at vertices of type $\hat s_3$ and angle $\pi/3$ at vertices of type $\hat s_1$.  Lemma~\ref{lem:link} implies that if $v$ is a vertex of type $\hat s_3$ in $X$, then $\lk(v,X)$ is isomorphic to the $(C,\overline{s_1s_2})$-relative Artin complex, where $C$ is the component of $\Lambda\setminus\{s_3\}$ that contains $\overline{s_1s_2}$. Then Assumption 1 implies $\lk(v,X)$ with its angular metric is $\CAT(1)$. Now we consider the case when $v$ is of type $\hat s_1$. Then $\lk(v,X)$ is isomorphic to the $(C_1,\overline{s_2s_3})$-relative Artin complex. By Lemma~\ref{lem:inherit} and Assumption 2, $\Delta_{C_1,\overline{s_2s_2}}$ is bowtie free, thus  $\overline{s_2s_3}$ is $3$-solid in $C_1$. Thus $\lk(v,X)$ is $\CAT(1)$ when $v$ is of type $\hat s_1$. By Lemma~\ref{lem:link}, if $v$ is of type $\hat s_2$, then $\lk(v,X)$ is a complete bipartite graph with each edge having length $\pi/2$. Thus $\lk(v,X)$ is $\CAT(1)$ as well. This implies $X$ is locally $\CAT(0)$. As $X$ is simply-connected by Lemma~\ref{lem:relative sc}, $X$ is $\CAT(0)$. 

Let $Y=\Delta_{\Lambda,\Lambda'}$. We endow the vertex set of $Y$ with the order described in Lemma~\ref{lem:poset} (2).  As $\Lambda_0\subset\Lambda'$, we know there is a natural embedding $X\to Y$, whose image is the induced subcomplex of $Y$ spanned by vertices of type $\hat s_1,\hat s_2$ or $\hat s_3$.  We prove the lemma by induction on the number $n$ of nodes of $\Lambda'$. 

The base case is $n=3$. We will use Lemma~\ref{lem:connected intersection}. Assumption 1 of Lemma~\ref{lem:connected intersection} follows Assumption 1 and Assumption 2 of this lemma. Now we verify Assumption 2 of Lemma~\ref{lem:connected intersection}. Let $y_1,y_2$ be two vertices in $Y$ with type $\hat s_n$. Note that the closed stars of $y_1$ and $y_2$ are convex subset of $X$, thus their intersection is empty or convex (hence connected).

Now we look at case $n>3$ and we will use Corollary~\ref{cor:connected intersection}. If $y\in Y$ is a vertex of type $\hat s_1$, then $\lk(y,Y)$ is bowtie free by Assumption 2 and Lemma~\ref{lem:link}. If $y$ has type $\hat s_n$, then $\lk(y,Y)$ admits a type-preserving isomorphism to the $(\Theta,\Theta')$-relative Artin complex, where $\Theta$ is the component of $\Lambda\setminus\{s_n\}$ that contains $\{s_i\}_{i=1}^{n-1}$ and $\Theta'$ is the linear subgraph spanned by $\{s_i\}_{i=1}^{n-1}$. We now show $\lk(y,Y)$ is bowtie free, which implies Assumption 1 of Corollary~\ref{cor:connected intersection}. By induction, it suffices to show the $(\Theta,\Theta')$-relative Artin complex satisfies assumptions of Lemma~\ref{lem:enlarge1} with $(\Lambda,\Lambda')$ replaced by $(\Theta,\Theta')$. Note that Assumption 1 follows from Lemma~\ref{lem:solid}. Take vertex $y'\in\lk(y,Y)$ with type $\hat s_1$. For Assumption 2, by Lemma~\ref{lem:link} (2) we need to show $\lk(y',\lk(y,Y))$ is bowtie free. Note
 $$\lk(y',\lk(y,Y))=\lk(\overline{y'y},Y)=\lk(y,\lk(y',Y)).$$
As $\lk(y',Y)$ is bowtie free by assumption and $\lk(y,\lk(y',Y))$ is spanned by all vertices of $\lk(y',Y)$ that is $<y$, we know $\lk(y,\lk(y',Y))$ is bowtie free, so is $\lk(y',\lk(y,Y))$. It remains to verify Assumption 2 of Corollary~\ref{cor:connected intersection}, but it is an immediate consequence of Lemma~\ref{lem:convex} below.
\end{proof}

\begin{lem}
	\label{lem:convex}
Under the same assumption of Lemma~\ref{lem:enlarge1}, let  $\Lambda_0$ be the subgraph of $\Lambda$ spanned by $\{s_1,s_2,s_3\}$. Let $X=\Delta_{\Lambda,\Lambda_0}$, equipped with the $\CAT(0)$ metric as in the proof of Lemma~\ref{lem:enlarge1}. Suppose $n>3$ and let $y_1$ be a vertex in $Y=\Delta_{\Lambda,\Lambda'}$ of type $\hat s_n$ and let $Y_1$ be the induced subcomplex of $X$ spanned by vertices in $X$ that are adjacent to $y_1$ in $\Delta_{\Lambda,\Lambda'}$. Then $Y_1$ is a convex subcomplex of $X$.
\end{lem}

\begin{proof}
As $n>3$, we know $Y_1=\lk(y_1,Y)\cap X \cong \Delta_{\Lambda\setminus\{s_n\},\Lambda_0}$, hence $Y_1$ is connected. It remains to show $Y_1$ is locally convex in $X$.
Take $y\in Y_1$. It is clear that $Y_1$ is locally convex around $y$ in $X$ if $y$ is not a vertex. Now assume $y$ is a vertex of type $\hat s_3$. Then $\lk(y,X)=\lk(y,Y)\cap X$. By Corollary~\ref{cor:adj}, if a vertex of type $\hat s_i$ for $i=1,2$ is adjacent to $y$ in $\Delta_\Lambda$, then this vertex is adjacent to $y_1$ in $\Delta_\Lambda$. 
%By Lemma~\ref{lem:link} (3), $\lk(y,X)$ is the join factor of $\lk(y,Y)$ that is contained in $X$. As $y_1$ is in a different join factor of $\lk(y,Y)$ which is disjoint from $X$, $\lk(y_1,\lk(y,Y))\cap X=\lk(y,X)$. As $\lk(y,Y_1)=\lk(y,Y)\cap Y_1=\lk(y,Y)\cap \lk(y_1,Y)\cap X=\lk(y_1,\lk(y,Y))\cap X$. 
So $\lk(y,X)=\lk(y,Y_1)$ and $Y_1$ is locally convex around $y$.
If $y$ is a vertex of type $\hat s_2$, then $\lk(y,Y_1)$ is a complete bipartite graph, hence $Y_1$ is locally convex in $X$ around $y$. Suppose $y$ has type $\hat s_1$. To show local convexity, it suffices to rule out the possibility that there are two distinct vertices $z_1,z_2$ in $\lk(y,Y_1)$ which can be connected by an edge path $\omega$ in $\lk(y,X)$ such that $\omega$ is made of $2$ edges that are not contained in $\lk(y,Y_1)$. Let $z$ be the vertex in $\omega$ that is outside $\lk(y,Y_1)$. Note that $\type(z_1)=\type(z_2)$. If $\type(z_1)=\hat s_3$, then $\type(z)=\hat s_2$ and $z_1>z$. As $y_1>z_1$, we have $y_1>z$, implying $z\in\lk(y,Y_1)$, contradiction. Thus $\type(z_1)=\type(z_2)=\hat s_2$. Then $\type(z)=\hat s_3$. As $\{z_1,z_2,z,y_1\}\subset\lk(y,Y)$ and $\lk(y,Y)$ is bowtie free by Assumption 2 of Lemma~\ref{lem:enlarge1}, there is $w\in \lk(y,Y)$ such that $\{z_1,z_2\}\le w\le \{z,y_1\}$. This forces $z\le y_1$, implying $z\in\lk(y_1,Y)$, contradiction again. 
\end{proof}

\begin{lem}
	\label{lem:enlarge2}
	Suppose the Dynkin diagram $\Lambda$ is a tree. Let $\Lambda'$ be a linear subgraph with its consecutive nodes being $\{s_i\}_{i=1}^n$. For $1\le i\le 3$, let $\{b_i\}_{i=1}^3$ be the set of nodes of another linear subgraph $\Lambda''$ such that $\Lambda''\cap \Lambda'$ has at most one node. Let $b_3$ be the node of $\Lambda''$ which is furthest away from $\Lambda'$. We assume
	\begin{enumerate}
		\item $\overline{b_2b_3}$ is $6$-solid in the component of $\Lambda\setminus\{b_1\}$ that contains $\overline{b_2b_3}$; 
		\item for the component $C$ of $\Lambda\setminus\{b_3\}$ containing $\Lambda'$, the $(C,T')$-relative Artin complex satisfies the labeled 4-cycle condition where $T'$ is the minimal tripod subgraph of $\Lambda'$ spanned by $b_2$ and $\Lambda'$.
	\end{enumerate}
	Then the $(\Lambda,\Lambda')$-relative Artin complex is bowtie free.
\end{lem}

\begin{proof}
For $i=1,n$, let
 $\Lambda_i$ be the subgraph of $\Lambda$ corresponding the shortest path from $b_2$ to $s_i$. It follows from Assumption 2, Lemma~\ref{lem:connect} and Proposition~\ref{prop:4-cycle} that  $(C,\Lambda_i)$-relative Artin complex is bowtie free.
	
Let $\Delta$ be the Artin complex of $A_\Lambda$. For any pair of induced subgraphs $\Lambda_1\subset \Lambda_2$ in $\Lambda$, let $\Delta_{\Lambda_2,\Lambda_1}$ be the $(\Lambda_2,\Lambda_1)$-relative Artin complex.
Let $X=\Delta_{\Lambda,\Lambda''}$, viewed as a subcomplex of $\Delta$.
Note that $X$ is 2-dimensional. We metric triangles in $X$ such that they are flat triangles in the Euclidean plane with angle $\pi/2$ at vertex of type $\hat b_2$, angle $\pi/6$ at vertex of type $\hat b_1$ and angle $\pi/3$ at vertex of type $\hat b_3$. By consider the embedding $\Delta_{C,\overline{b_1b_2}}\to \Delta_{C,\Lambda_i}$, we deduce from Assumption 2 that $\overline{b_1b_2}$ is $3$-solid in $C$. It follows from Lemma~\ref{lem:relative sc} and Assumption 1 that $X$ is $\CAT(0)$. 

%In particular, 
%\begin{center}
%	$(\star)$\ \ $\overline{b_1b_2}$ is $3$-solid in the component of $\Lambda\setminus\{b_3\}$ that contains $\overline{b_1b_2}$.
%\end{center}
%We will refer statement $(\star)$ as Assumption 1'.

We now prove the lemma by induction on the number of nodes in $\Lambda'$, and assume the lemma holds whenever $\Lambda'$ has $\le n-1$ nodes. We will be assuming that $\Lambda'$ and $\Lambda''$ are not contained in a common linear subgraph of $\Lambda$, otherwise we are reduced to Lemma~\ref{lem:enlarge1}. Let $T$ be the smallest subtree of $\Lambda$ containing $\Lambda'$ and $\Lambda''$. Note that $T$ is a tripod and $b_3,s_1,s_n$ are the three valence one nodes of $T$. Let $T'$ be the maximal subtree in $T\setminus\{b_3\}$.

Now we verify the Assumption 1 of Lemma~\ref{lem:bowtie free criterion}. For $i=1,n$, let $\Lambda_{s_i}$ be the component of $\Lambda\setminus\{s_i\}$ containing the rest of the nodes of $\Lambda'$. Let $\Lambda'_{s_i}=\Lambda'\cap \Lambda_{s_i}$. To justify Lemma~\ref{lem:bowtie free criterion} (1), we need to show $\Delta_{\Lambda_{s_i},\Lambda'_{s_i}}$ is bowtie free. By the induction assumption, it suffices to show the two assumptions of Lemma~\ref{lem:enlarge2} hold with $(\Lambda,\Lambda')$ replaced by $(\Lambda_{s_i},\Lambda'_{s_i})$ for $i=1,n$. Indeed, Assumption 1 follows from Lemma~\ref{lem:solid} and Assumption 2 follows from Corollary~\ref{cor:inherit}.

%For Assumption 2, let $C_{s_1}$ be the component of $\Lambda_{s_1}\setminus\{b_3\}$ containing $\Lambda'_{s_1}$. Take a vertex $v'\in \Delta_{C,\Lambda'}$ of type $\hat s_1$. Then $\Delta_{C_{s_1},\Lambda'_{s_1}}$ can be identified with $\lk(v',\Delta_{C,\Lambda'})$ by Lemma~\ref{lem:link}. Thus  $\Delta_{C_{s_1},\Lambda'_{s_1}}$ is bowtie free by Lemma~\ref{lem:inherit}. For Assumption 3, for $i=1,n$, let $\Lambda_{s_1,i}=\Lambda_{s_1}\cap \Lambda_i$. Take a vertex $v'\in \Delta_{C,\Lambda_i}$ of type $\hat s_1$. Then $\Delta_{C_{s_1},\Lambda_{s_1,i}}$ can be identified with $\lk(v',\Delta_{C,\Lambda_i})$ by Lemma~\ref{lem:link}. Thus $\Delta_{C_{s_1},\Lambda_{s_1,i}}$ is bowtie free by Lemma~\ref{lem:inherit}.

It remains to verify Assumption 2 of Lemma~\ref{lem:bowtie free criterion}. 
Take four mutually distinct vertices $\{x_1,x_2,y_1,y_2\}$ in $\Delta_{\Lambda,\Lambda'}$ such that $x_i<y_j$ for $1\le i,j\le 2$, $x_1$ and $x_2$ have type $\hat s_1$, and $y_1$ and $y_2$ have type $\hat s_n$.
Let $X_i$ be the subcomplex of $X$ spanned by vertices in $X$ that are adjacent to $x_i$ in $\Delta$. We define $Y_i$ similarly. By Assumption 2 of Lemma~\ref{lem:enlarge2}, the $(C,T'_i)$-relative Artin complex satisfies the labeled 4-cycle condition for $i=1,n$, where $T'_i$ is the segment in $T'$ from $b_2$ to $s_i$. Hence $\Delta_{C,T'_i}$ satisfies the bowtie free condition by Lemma~\ref{lem:connect}.
Now  by applying Lemma~\ref{lem:convex} to $T'_1\cup \overline{b_2b_3}$ and $T'_n\cup\overline{b_2b_3}$, we know $X_i$ and $Y_i$ are convex subcomplexes of $X$. As $x_i$ is adjacent to $y_j$ in $\Delta$ for $1\le i,j\le 2$, we know $X_i\cap Y_j\neq\emptyset$.

Next we prove either $Y_1\cap Y_2\cap X_i\neq\emptyset$ for $i=1,2$, or $X_1\cap X_2\cap Y_j\neq\emptyset$ for $j=1,2$. It suffices to verify the assumptions of Lemma~\ref{lem:4convexsubsets} below. Take a vertex $p_1\in X_i\cap Y_j$. We assume without loss of generality that $i=1,j=1$. Suppose $p_1$ is of type $\hat b_1$. By Corollary~\ref{cor:adj}, if a vertex of $\Delta$ of type $\hat b_3$ or $\hat b_2$ is adjacent to $p_1$ in $\Delta$, then this vertex is also adjacent to $x_1$ and $y_1$ in $\Delta$.
%Note that $$
%\lk(p_1,X_1)=\lk(p_1,\Delta)\cap X_1=\lk(p_1,\Delta)\cap \lk(x_1,\Delta)\cap X=\lk(\overline{p_1x_1},\Delta)\cap X.
%$$ By applying Lemma~\ref{lem:link} twice, we know $\lk(\overline{p_1x_1},\Delta)$ and $\lk(p_1,\Delta)$ share a common join factor, which is exactly the join factor that has non-trivial intersection with $X$ (i.e. containing vertices of type $\hat b_2$ and $\hat b_3$). Thus $\lk(\overline{p_1x_1},\Delta)\cap X=\lk(p_1,\Delta)\cap X$, implying 
Thus $\lk(p_1,X)=\lk(p_1,X_1)=\lk(p_1,Y_1)$. 

If $p_1$ is of type $\hat b_2$, then Lemma~\ref{lem:link} implies that $\lk(p_1,X)$, $\lk(p_1,X_1)$ and $\lk(p_1,Y_1)$ are complete bipartite graphs. We write $\lk(p_1,X)=B_1* B_3$, where $B_i$ is made of all vertices in $\lk(p_1,X)$ of type $\hat b_i$ for $i=1,3$. Both $B_1$ and $B_3$ are discrete. By Corollary~\ref{cor:adj}, if a vertex of $\Delta$ of type $\hat b_3$ is adjacent to $p_1$ in $\Delta$, then this vertex is also adjacent to $x_1$ and $y_1$ in $\Delta$.
%By Lemma~\ref{lem:link} and the discussion in the previous paragraph, we know $\lk(\overline{p_1x_2},\Delta)$ and $\lk(p_1,\Delta)$ share a common join factor, namely the join factor containing vertices of type $\hat b_3$. 
Thus $B_3\subset \lk(p_1,X_1)$ and $B_3\subset \lk(p_1,Y_1)$. 

Suppose $p_1$ is of type $\hat b_3$. Take a vertex $\xi\in \lk(p_1,X_1)$ and a vertex $\eta\in\lk(p_1,Y_1)$ such that they are connected by an edge $e$ in $\lk(p_1,X)$. To establish Assumption 3 of Lemma~\ref{lem:4convexsubsets}, it suffices to show $e\subset \lk(p_1,X_1\cup Y_1)$. Suppose (without loss of generality) that $\xi$ is of type $\hat b_2$ and $\eta$ is of type $\hat b_1$. We will view $\lk(p_1,X)$ as a subgraph of $\lk(p_1,\Delta)$, and view $\lk(p_1,\Delta)$ as a subcomplex of $\Delta$. Then $y_1$ and $\eta$ are adjacent in $\Delta$, so is $\eta$ and $\xi$. 
%Moreover, Lemma~\ref{lem:link} implies that $\lk(\eta,\lk(p_1,\Delta))=\lk(\overline{\eta p_1},\Delta)$ admits a join decomposition, with $y_1$ and $\xi$ in different join factors. 
By Corollary~\ref{cor:adj}, $y_1$ and $\xi$ are adjacent in $\lk(p_1,\Delta)$. So $\xi$ is adjacent to both $x_1$ and $y_1$ in $\lk(p_1,\Delta)$, implying that $\xi\in \lk(p_1,X_1\cap Y_1)$. Thus both $\xi$ and $\eta$ are in $\lk(p_1,Y_1)$, hence $e\subset \lk(p_1,Y_1)$, as desired.

%If two adjacent sides of $P$ are trivial, then $X_1\cap X_2\cap Y_1\cap Y_2\neq\emptyset$. Take a vertex $q$ in this intersection.  Then we can always assume $X_1\cap X_2\cap Y_1\cap Y_2$ contains a vertex $q$ of type $\hat b_3$. Let $\Phi_{\Lambda'}$ be the induced subcomplex in $\lk(q,\Delta)$ spanned by vertices of type $\hat s$ with $s\in \Lambda'$. Lemma~\ref{lem:link} implies that $\Phi_{\Lambda'}$ is type isomorphic to $\Delta_{C,\Lambda'}$. As $\{x_1,x_2,y_1,y_2\}\subset \Phi_{\Lambda'}$, Assumption 2 of Lemma~\ref{lem:bowtie free criterion} follows from Assumption 2 of Lemma~\ref{lem:enlarge2}.

Thus the conclusion of Lemma~\ref{lem:4convexsubsets} holds.
We assume without loss of generality that  $Y_1\cap Y_2\cap X_i\neq\emptyset$ for $i=1,2$. Take a vertex $w_1\in Y_1\cap Y_2\cap X_1$, and a vertex $w_2\in Y_1\cap Y_2\cap X_2$. If $w_1$ is type $\hat b_i$ with $i=1,2$, then $b_i$ separates $\{s_1,s_n\}$ from $b_3$. Thus if we have $w'_1$ of type $\hat b_3$ which is adjacent in $w_1$ in $X$, then $w'_1$ is adjacent in each of $\{x_1,y_1,y_2\}$ in $\Delta$ by Corollary~\ref{cor:adj}. Thus $w'_1\in Y_1\cap Y_2\cap X_1$.
Thus up to replacing $w_1$ by $w'_1$, we can assume $w_1$ is of type $\hat b_3$, and similarly that $w_2$ has type $\hat b_3$. If $w_1=w_2$, then $x_1y_1x_2y_2$ is a 4-cycle in $\lk(w_1,\Delta)\cong \Delta_{C,T'}$. By Assumption 2, Proposition~\ref{prop:4-cycle} and Lemma~\ref{lem:connect}, $\Delta_{C,\Lambda'}$ is bowtie free. As $x_1y_1x_2y_2$ can also be viewed a 4-cycle in $\Delta_{C,\Lambda'}$ via the type-preserving isomorphism $\lk(w_1,\Delta)\cong \Delta_{C,T'}$, this finishes the justification of Assumption of Lemma~\ref{lem:bowtie free criterion} of this 4-cycle. Now we assume $w_1\neq w_2$.

%As $Y_1\cap Y_2$ is connected, there is an edge path $\omega\subset Y_1\cap Y_2$ from $w_1$ to $w_2$. Let $\{q_i\}_{i=1}^k$ be consecutive vertices of $\omega$ with $w_1=q_1$ and $w_2=q_k$.

Let $T'$ be as in the statement of the lemma. 
For $i=1,2$,
let $\Phi_{\Lambda',i}$ (resp. $\Phi_{T',i}$) be the induced subcomplex in $\lk(w_i,\Delta)$ spanned by vertices of type $\hat s$ with node $s\in \Lambda'$ (resp. $s\in T'$). Then by Lemma~\ref{lem:link}, $\Phi_{\Lambda',i}$ admits a type-preserving isomorphism to $\Delta_{C,\Lambda'}$ and $\Phi_{T',i}$ admits a type-preserving isomorphism to $\Delta_{C,T'}$. 
Note that $w_i$ is adjacent to $y_1,y_2,x_i$ in $\Phi_{\Lambda',i}$. Vertices of $\Phi_{\Lambda',i}$ are ordered according to the linear order on $\Lambda'$ such that vertices of type $\hat s_n$ are maximal elements. As $\Phi_{\Lambda',i}$ is bowtie free by Assumption 2, Proposition~\ref{prop:4-cycle} and Lemma~\ref{lem:connect}, we know $y_1$ and $y_2$ has a meet in $\Phi_{\Lambda',i}$, denoted by $z_i$. In particular $x_i\le z_i$ in $\Phi_{\Lambda',i}$. Now we show $z_1=z_2$, which will finish the proof.

Let $T''$ be the linear subgraph in $T$ from $b_3$ to $s_n$, with the linear order such that $s_n$ is maximal. Lemma~\ref{lem:enlarge1} implies that $\Delta_{\Lambda,T''}$ is bowtie free. Let $z$ be the meet of $y_1$ and $y_n$ in $\Delta_{\Lambda,T''}$. Then $w_i\le z$ for $i=1,2$ in $\Delta_{\Lambda,T''}$. 
In particular, $w_i$ and $z$ are either adjacent or identical in $\Delta$. As $w_1\neq w_2$, we know $z$ is not of type $\hat b_3$, thus $z\neq w_i$ for $i=1,2$. Suppose $z$ is of type $\hat s_z$ for node $s_z\in \Lambda$. Suppose $a$ is the center of $T'$. We now show that $s_z$ is contained in the smallest subtree $T'_{a,s_n}$ of $T$ containing $a$ and $s_n$. 

Indeed, suppose this is not true. Let $T_z$ be the minimal subtree of $T'$ spanned by $s_1,s_n$ and $s_z$. Then we can viewed the 4-cycle $x_1y_1zy_2$ as a 4-cycle in $\Delta_{C,T_z}$ via the type preserving isomorphism $\lk(w_1,\Delta)\cong \Delta_{C,T'}$. By applying Proposition~\ref{prop:4-cycle} and Lemma~\ref{lem:tripod} to the 4-cycle $x_1y_1zy_2$ in $\Delta_{C,T_z}$, we deduce that $y_1$ and $y_2$ are adjacent in $\Delta$ to a common vertex whose type corresponds to a node in $T'_{a,s_n}$, which contradicts that $z$ is the meet of $y_1$ and $y_n$ in $\Delta_{\Lambda,T''}$. Thus $s_z\in T'_{a,s_n}$.
It follows that $z\in \Phi_{\Lambda',i}$ for $i=1,2$. We now deduce from the choice of $z_1,z_2$ and $z$ that $z_1=z_2=z$.
\end{proof}

\begin{lem}
	\label{lem:4convexsubsets}
Let $X$ be a simplicial complex of type $\{\hat b_1,\hat b_2,\hat b_3\}$ such that its triangles are flat triangles with angle $\pi/2$ at vertex of type $\hat b_2$, angle $\pi/6$ at vertex of type $\hat b_1$ and angle $\pi/3$ at vertex of type $\hat b_3$. Suppose $X$ is $\CAT(0)$ with such metric. Let $X_1,X_2,Y_1,Y_2$ be convex subcomplexes of $X$ such that $X_i\cap Y_j\neq\emptyset$ for $1\le i,j\le 2$. We assume, in addition, that for any $1\le i,j\le 2$, and any vertex $p\in X_i\cap Y_j$, then following holds:
\begin{enumerate}
	\item if $p$ is of type $\hat b_1$, then either $\lk(p,X_i)\subset \lk(p,Y_j)$ or $\lk(p,Y_j)\subset\lk(p,X_i)$;
	\item if $p$ is of type $\hat b_2$, then $\lk(p,X)$, $\lk(p,X_i)$ and $\lk(p,Y_j)$ are complete bipartite graphs such that they share the same join factor made of vertices of type $\hat b_3$;
	\item if $p$ is of type $\hat b_3$, then there does not exist an edge $e$ in $\lk(p,X)$ connecting a vertex in $\lk(p,X_i)$ and a vertex in $\lk(p,Y_j)$ such that $e$ is not contained in $\lk(p,X_i\cup Y_j)$.
\end{enumerate}
Then either $Y_1\cap Y_2\cap X_i\neq\emptyset$ for $i=1,2$, or $X_1\cap X_2\cap Y_j\neq\emptyset$ for $j=1,2$.
\end{lem}

\begin{proof}
A $4$-gon in $X$ is \emph{admissible} if it is made of fours geodesic segments $\{S_i\}_{i=1}^4$ such that $S_1$ goes from $p_1\in X_1\cap Y_1$ to $p_2\in Y_1\cap X_2$, $S_2$ goes from $p_2\in Y_1\cap X_2$ to $p_3\in Y_2\cap X_2$, $S_3$ goes from $p_3\in Y_2\cap X_2$ to $p_4\in Y_2\cap X_1$ and $S_4$ goes from $Y_2\cap X_1$ to $Y_1\cap X_1$. Let $\ell(S_i)$ denotes the length of $S_i$. The \emph{perimeter} of this 4-gon is defined to be $\sum_{i=1}^4\ell(S_i)$.
Let $P$ be a 4-gon with perimeter minimized among all admissible 4-gons. Indeed, as $X$ is a piecewise Euclidean complex with finitely many shapes, it follows from the ``quasi-compactness'' argument in \cite[Chapter I.7]{BridsonHaefliger1999} that the minimum of perimeter achieved by one admissible $4$-gon (the upshot is that each such 4-gon is contained in a finite subcomplex and there are only finitely many combinatorial types of such finite subcomplexes that contain these 4-gons with a given upper bound on their perimeter).

We claim that in $P$, if $\angle_{p_1}(p_2,p_4)<\pi$, then $Y_1\cap S_4=\{p_1\}$ and $X_1\cap S_1=\{p_1\}$. Indeed, if $Y_1\cap S_4=\{p_1\}$ is not true, then by the convexity of $Y_1$ and $X_1$, we know a small subsegment $\overline{p_1p'_1}$ of $S_4$ is contained in $X_1\cap Y_1$. Then we consider the geodesic 4-gon $P'$ with vertices $p_4,p'_1,p_2,p_3$, which is still admissible by the convexity of $X_1$. We write $\overline{p'_1p_2}$ for the geodesic between $p'_1$ and $p_2$. As $\angle_{p_1}(p_2,p_4)<\pi$, by $\CAT(0)$ geometry we have $\ell(\overline{p'_1p_2})<\ell(\overline{p'_1p_1})+\ell(\overline{p_1p_2})$. This implies that $P'$ has perimeter smaller than $P$, which is a contradiction. Thus $Y_1\cap S_4=\{p_1\}$. Similarly we can prove $X_1\cap S_1=\{p_1\}$.

We denote the point in $\lk(p_1,X)$ arising from $\overline{p_1p_2}$ to be $\log_{p_1}(p_2)$.
We claim if $\angle_{p_1}(p_2,p_4)<\pi$, $p_4\neq p_1$ and $p_2\neq p_1$, then the shortest arc in $\lk(p_1,X)$ from $\log_{p_1}(p_4)$ to $\log_{p_1}(p_2)$ does not contain any point in $\lk(p_1,X_1\cap Y_1)$. By the previous paragraph, the two endpoints of this arc is not contained in $\lk(p_1,X_1\cap Y_1)$. Suppose an interior point of this arc is in $\lk(p_1,X_1\cap Y_1)$. Then it corresponds to a small geodesic segment $\overline{p_1p}\subset X_1\cap Y_1$ emanating from $p_1$. Take $p'_4\in S_4$ close to $p_1$ and $p'_2\in S_2$ close to $p_1$. As a small neighborhood of $p_1$ in $X$ is isometric to a Euclidean cone (cf. \cite[Theorem I.7.16]{BridsonHaefliger1999}), $p'_4,p_1,p'_2$ span a small flat triangle in $X$ which is cut through by subsegment of  $\overline{p_1p}$. We assume $p\in\overline{p'_4p'_2}$ by possibly shortening $\overline{p_1p}$. As $\angle_{p_1}(p'_2,p'_4)<\pi$,
\begin{align*}
	&	\ell(\overline{p_4p})+\ell(\overline{pp_2})\le \ell(\overline{p_4p'_4})+\ell(\overline{p'_4p'_2})+\ell(\overline{p'_2p_2})\\
	&<\ell(\overline{p_4p'_4})+\ell(\overline{p'_4p_1})+\ell(\overline{p_1p'_1})+\ell(\overline{p'_1p_2})=\ell(S_4)+\ell(S_1).
\end{align*}
By the convexity of $X_1$ and $Y_1$, we know $\overline{p_4p}\subset X_1$ and $\overline{pp_2}\subset Y_1$. Thus the 4-gon with vertices $\{p_2,p_3,p_4,p\}$ is admissible and its perimeter is smaller than the perimeter of $P$, contradiction.

We claim if $p_4\neq p_1$ and $p_2\neq p_1$, then $\angle_{p_1}(p_4,p_2)\ge 2\pi/3$. It suffices to consider the case that $\angle_{p_1}(p_4,p_2)<\pi$. Then $p_1$ cannot be in the interior of a 2-cell, otherwise this 2-cell is contained in $Y_1\cap X_1$, which means small subsegments of $S_1$ and $S_4$ starting at $p_1$ is contained in $Y_1\cap X_1$, contradicting one of the previous claims. Also $p_1$ cannot be in the interior of an edge $e$, otherwise $p_1\in e\subset Y_1\cap X_1$. As a small neighborhood of $p_1$ in $X$ splits as a Cartesian product of a small segment of $e$ and a tree, we know that $e$ gives rise to a point in the shortest arc in $\lk(p_1,X)$ from $\log_{p_1}(p_4)$ to $\log_{p_1}(p_2)$, which contradicts the previous paragraph. Thus $p_1$ must be a vertex. 

Let $\alpha$ be the shortest arc in $\lk(p_1,X)$ from $\log_{p_1}(p_4)$ to $\log_{p_1}(p_2)$. If $p_1$ is of type $\hat b_1$, then by Assumption 1, one endpoint of $\alpha$ is contained in $\lk(p_1,X_1\cap Y_1)$, contradiction. If $p_2$ is of type $\hat b_2$, then Assumption 2 implies that $X_1\cup Y_1$ is locally convex in $X$ in a small neighborhood around $p_1$. Hence the shortest arc $\alpha$ in $\lk(p_1,X)$ from $\log_{p_1}(p_4)$ to $\log_{p_1}(p_2)$ is contained in $\lk(p_1,X_1\cup Y_1)$. As $\alpha$ joins a point in $\lk(p_1,X_1)$ and a point in $\lk(p_1,Y_1)$, $\alpha$ contains at least one point in $\lk(p_1,X_1\cap Y_1)$, contradiction. The only possibility left is that $p_1$ is of type $\hat b_3$. By the previous claim, $$\alpha\cap \lk(p_1,X_1\cap Y_1)=\emptyset.$$ Let $\alpha_{X_1}=\alpha\cap \lk(p_1,X_1)$ and let $\alpha_{Y_1}=\alpha\cap\lk(p_1,Y_1)$. Then $\alpha_{X_1}\cap \alpha_{Y_1}=\emptyset$. Let $\alpha'$ be the closure of $\alpha\setminus(\alpha_{X_1}\cup \alpha_{Y_1})$. Then $\alpha'$ connects a vertex $\xi\in \lk(p_1,X_1)$ to a vertex $\eta\in\lk(p_1,Y_1)$. Assumption 3 implies that $\alpha'$ is made of at least two edges. Thus $\angle_{p_1}(p_4,p_2)\ge 2\pi/3$.

The same argument implies that for the $4$-gon $P$, for each $p_i$, if the two sides of $P$ containing $p_i$ are non-degenerate, then the angle of these two sides at $p_i$ is $\ge 2\pi/3$. 
Now by \cite[Chapter II.2.12]{BridsonHaefliger1999}, at least two sides of the 4-gon must be a point.  Then the lemma follows.
\end{proof}
\begin{lem}
	\label{lem:enlarge3}
	Suppose the Dynkin diagram $\Lambda$ is a tree. Let $\Lambda'$ be a linear subgraph with its consecutive nodes being $\{s_i\}_{i=1}^n$. Let $b_3$ be a node outside $\Lambda'$ such that $b_3$ is adjacent to a node $b_2\in \Lambda'$. Let $C$ be the component of $\Lambda\setminus\{b_3\}$ containing $\Lambda'$.
    We assume that
	\begin{enumerate}
		\item for any node $b\in \Lambda'$ adjacent to $b_2$, the edge $\overline{b_2b_3}$ is $6$-solid in the component of $\Lambda\setminus\{b\}$ that contains $\overline{b_2b_3}$; 
		\item the $(C,\Lambda')$-relative Artin complex  is bowtie free.
	\end{enumerate}
	Then the $(\Lambda,\Lambda')$-relative Artin complex is bowtie free.
\end{lem}

\begin{proof}
Suppose $b_2$ is an interior vertex of $\Lambda'$, otherwise we are reduced to Lemma~\ref{lem:enlarge1}.	
Let $b_1$ (resp. $b'_1$) be the node in $\Lambda'$ adjacent to $b_2$ such that $b_1$ and $s_1$ (resp. $s_n$) are contained in the same component of $\Lambda'\setminus\{b_2\}$.
	Let $\Lambda_{b_1}$ (resp. $\Lambda_{b'_1}$) be the linear subgraph spanned by $\{b_1,b_2,b_3\}$ (resp. $\{b'_1,b_2,b_3\}$). Let $X=\Delta_{\Lambda,\Lambda_{b_1}}$ and $X'=\Delta_{\Lambda,\Lambda_{b'_1}}$, which are 2-dimensional complexes. We metrize triangles in $X$ (resp. $X'$) such that they are flat triangles in the Euclidean plane with angle $\pi/2$ at vertex of type $\hat b_2$, angle $\pi/6$ at vertex of type $\hat b_1$ (resp. $\hat b'_1$) and angle $\pi/3$ at vertex of type $\hat b_3$. 
By Assumption 2, $\overline{b_1b_2}$ and $\overline{b'_1b_2}$ are $3$-solid in $C$. This together with Assumption 1 and Lemma~\ref{lem:relative sc} imply that $X$ and $X'$ are $\CAT(0)$. 

We now prove the lemma by induction on the number of nodes in $\Lambda'$, and assume the lemma holds whenever $\Lambda'$ has $\le n-1$ nodes. Now we verify the assumptions of Lemma~\ref{lem:bowtie free criterion} for $\Delta_{\Lambda,\Lambda'}$. Let $\Lambda_{s_1}$ be the component of $\Lambda\setminus\{s_1\}$ that contains the rest of the nodes of $\Lambda'$. Let $\Lambda'_{s_1}=\Lambda'\cap \Lambda_{s_1}$ and $C_{s_1}=\Lambda_{s_1}\cap C$. To justify Lemma~\ref{lem:bowtie free criterion} (1), we need to show $\Delta_{\Lambda_{s_1},\Lambda'_{s_1}}$ is bowtie free. By the induction assumption, it suffices to show the assumptions of Lemma~\ref{lem:enlarge3} hold with $(\Lambda,\Lambda')$ and $C$ replaced by $(\Lambda_{s_1},\Lambda'_{s_1})$ and $C_{s_1}$, which follows by repeating the corresponding argument in the proof of Lemma~\ref{lem:enlarge2}.

Now we justify Assumption 2 of Lemma~\ref{lem:bowtie free criterion}. 
 Take mutually distinct vertices $\{x_1,x_2,y_1,y_2\}$ in $\Delta_{\Lambda,\Lambda'}$ such that $x_i<y_j$ for $1\le i,j\le 2$, $x_1$ and $x_2$ have type $\hat s_{1}$, and $y_1$ and $y_2$ have type $\hat s_{n}$. Let $\Lambda_i$ (resp. $\Lambda'_i$) be the linear subgraph of $\Lambda$ from $b_3$ (resp. $b_2$) to $s_i$ ($i=1,n$). Let $\Delta$ be the Artin complex of $A_\Lambda$, and let $X_i$ and $Y_i$ (resp. $X'_i$ and $Y'_i$) be the induced subcomplexes of $X$ (resp. $X'$) spanned by vertices that are adjacent or equal to $x_i$ (resp. $y_i$) in $\Delta$. 

Next we prove $X_i,Y_i$ (resp. $X'_i,Y'_i$) are convex subcomplexes of $X$ (resp. $X'$). By symmetry, we only treat the cases of $X_1$ and $X'_1$. The convexity of $X_1$ is clear when $b_1=s_1$. Now suppose $s_1\neq b_1$.
Note that Assumption 2 and Lemma~\ref{lem:inherit} imply that $\Delta_{C,\Lambda'_i}$ is bowtie free.
The convexity of $X_1$ follows from Lemma~\ref{lem:convex} as $\{b_3,b_2,b_1,s_1\}$ are contained in $\Lambda_1$.  

Now we show $X'_1$ is convex in $X'$. As $X'_1=\operatorname{lk}(x_1,\Delta)\cap X'\cong \Delta_{\Lambda\setminus\{s_1\},\Lambda_{b'_1}}$, we know $X'_1$ is connected. It remains to show $X'_1$ is locally convex in $X'$ around each vertex $q\in X'_1$, which amounts to show $\lk(q,X'_1)$ is $\pi$-convex in $\lk(q,X')$ with respect to the angular metric on both graphs. The cases of $q$ of type $\hat b_3$ and of type $\hat b_2$ is similar to the proof of Lemma~\ref{lem:convex}. Now we assume $q$ is of type $\hat b'_1$. We claim if two vertices of $\lk(q,X'_1)$ has combinatorial distance $\le 5$ in $\lk(q,X')$ (i.e. they can be connected by an edge path in $\lk(q,X')$ with $\le 5$ edges), then the unique geodesic between these two vertices are contained in $\lk(q,X'_1)$. This claim will imply the desired $\pi$-convex property, hence imply the convexity of $X'_1$ in $X$. 

Now we prove the claim. If we view $\lk(q,X')$ (resp. $\lk(q,X'_1)$) as a subgraph of $X'$ (resp. $X'_1$), and view $X'$ and $X'_1$ as subcomplexes of $\Delta$, then $\lk(q,X'_1)$ is the collection of vertices in $\lk(q,X')$ that are adjacent to $x_1$.
Let $C_{b'_1}$ be the component of $\Lambda\setminus\{b'_1\}$ that contains $b_2$. 
%Then $\overline{b_1b_2}$ is 3-solid in $C_{b'_1}$ and $\overline{b_2b_3}$ is 6-solid in $C_{b'_1}$ by Lemma~\ref{lem:solid}. Then we metric $X_{b'_1}$ with flat triangle with angles $(\pi/2,\pi/3,\pi/6)$ as before.  
Let $X_{b'_1}=\Delta_{C_{b'_1},\Lambda_{b_1}}$. By Lemma~\ref{lem:link}, we can identify $X_{b'_1}$ with the induced subcomplex of $\lk(q,\Delta)$ spanned by vertices of type $\hat b_1,\hat b_2,\hat b_3$. Then the induced subgraph of $X_{b'_1}$ spanned by vertices of type $\hat b_2$ and $\hat b_3$ is identified with $\lk(q,X')$. Note that $x_1\in \lk(q,\Delta)$. Let $Z_1$ be the induced subcomplex of $X_{b'_1}$ spanned by vertices that are adjacent to $x_1$ inside $\Delta$. Then $\lk(q,X'_1)$ is identified with the induced subgraph of $Z_1$ spanned by vertices of type $\hat b_2$ and $\hat b_3$. We metric $X_{b'_1}$ with flat triangle with angles $(\pi/2,\pi/3,\pi/6)$ as before.
By Lemma~\ref{lem:girth} below, if we know 
\begin{enumerate}
	\item $X_{b'_1}$ with such metric is $\CAT(0)$ and $Z_1$ is a convex subcomplex of $X_{b'_1}$;
\item for any vertex
	$v$ in $Z_1$ of type $\hat b_1$, we have $\lk(v,Z_1)=\lk(v,X_{b'_1})$;
\end{enumerate}
then the claim follows. Item 2 follows from Corollary~\ref{cor:adj}. For item 1, by Lemma~\ref{lem:convex}, it suffices to show the $(C_{b'_1}\cap C,C_{b'_1}\cap \Lambda')$-relative Artin complex is bowtie free and $\overline{b_2b_3}$ is $6$-solid in the component of $C_{b'_1}\setminus \{b_1\}$ that contains $b_2$. The former is true because $\Delta_{C,\Lambda'}$ is bowtie free, hence by Lemma~\ref{lem:inherit} $\Delta_{C,\Theta}$ is bowtie free with $\Theta$ being the segment in $\Lambda'$ from $s_1$ to $b'_1$, hence the link of a vertex of type $\hat b'_1$ in $\Delta_{C,\Theta}$ is bowtie free, which admits a type-preserving isomorphism to the $(C_{b'_1}\cap C,C_{b'_1}\cap \Lambda')$-relative Artin complex by Lemma~\ref{lem:link}. The latter is true because of Assumption 1 and Lemma~\ref{lem:solid}. This finishes the proof of the convexity of $X_i,Y_i,X'_i,Y'_i$.

\medskip

Next we show either $Y_1\cap Y_2\cap X_i\neq\emptyset$ for $i=1,2$, or $X_1\cap X_2\cap Y_i\neq\emptyset$ for $i=1,2$. It suffices to justify the assumptions of Lemma~\ref{lem:4convexsubsets} for $q\in X_i\cap Y_j$. We only look at the case $i=1,j=1$, as other cases are similar. If $q$ is of type $\hat b_1$, then $\lk(q,X_1)=\lk(q,X)$. Hence $\lk(q,Y_1)\subset\lk(q,X_1)$. If $q$ is of type $\hat b_2$, then $\lk(q,Y_1),\lk(q,X_1)$ and $\lk(q,X)$ are all complete bipartite graphs, and they share a joint factor, namely the factor made of vertices of type $\hat b_3$. If $q$ is of type $\hat b_3$, we need to show there does not exist an edge in $\lk(q,X)$ connecting a vertex in $\lk(q,X_1)$ and a vertex in $\lk(q,Y_1)$ such that this edge is not contained in $\lk(q,X_1\cup Y_1)$. Suppose such edge exists, with it endpoint $\xi\in \lk(q,X_1)$ and $\eta\in \lk(q,Y_1)$. Now we identify $\Delta_{C,\Lambda'}$ as the induced subcomplex of $\lk(q,\Delta)$ spanned by vertices of type $\hat s$ with node $s\in \Lambda'$. Then we view $x_1,\xi,\eta,y_1$ as four vertices in $\Delta_{C,\Lambda'}$, which form a 4-cycle in such order. Note that one of $\{\xi,\eta\}$ has type $\hat b_1$, and the other has type $\hat b_2$. If $\xi$ has type $\hat b_1$, then $x_1$ is adjacent to $\eta$ in $\Delta_{C,\Lambda'}$ by Corollary~\ref{cor:adj}, which implies that $\eta\in \lk(q,X_1)$ and $\overline{\eta\xi}\subset \lk(q,X_1)$, contradiction. If $\xi$ has type $\hat b_2$, by the bowtie free property of $\Delta_{C,\Lambda'}$, there exists $z\in \Delta_{C,\Lambda'}$ such that $\{x_1,\eta\}\le z\le \{\xi,y_1\}$. As $\eta\le z\le \xi$ in $\Delta_{C,\Lambda'}$, we know either $z=\eta$ or $z=\xi$, thus either $\eta\in \lk(q,X_1)$ or $\xi\in \lk(q,Y_1)$, which implies $\overline{\eta\xi}\subset \lk(q,X_1\cup Y_1)$, contradiction again. 

\medskip
\noindent
\underline{Case 1: $X_1\cap X_2\cap Y_i\neq\emptyset$ for $i=1,2$.}
 We first claim $X_1\cap X_2\cap Y_i$ has a vertex of type $\hat b_3$. This follows from Lemma~\ref{lem:link} if $X_1\cap X_2\cap Y_i$ has a vertex of type $\hat b_2$, as the node $b_2$ sits in the shortest path from $b_3$ to $s_j$ for $j=1,n$.
If $X_1\cap X_2\cap Y_i$ has a vertex $w_i$ of type $\hat b_1$, then $$\lk(w_i,Y_i)\subset \lk(w_i,X_1)=\lk(w_i,X_2)=\lk(w_i,X)$$ by the previous paragraph, thus any vertex of type $\hat b_3$ in $\Delta$ which is adjacent to both $w_i$ and $y_i$ (such vertex always exists) is contained in $X_1\cap X_2\cap Y_i$. Thus the claim follows. 
For $i=1,2$, take a vertex $w_i\in X_1\cap X_2\cap Y_i$ of type $\hat b_3$. If $w_1=w_2$, then the 4-cycle $x_1y_1x_2y_2$ is contained in $\lk(w_1,\Delta)$. Note that the induced subcomplex of  $\lk(w_1,\Delta)$ spanned by vertices of type $\hat s$ with node $s\in \Lambda'$ admits a type-preserving isomorphic to $\Delta_{C,\Lambda'}$. As $\Delta_{C,\Lambda'}$ is bowtie free, Assumption 2 of Lemma~\ref{lem:bowtie free criterion} holds for the 4-cycle $x_1y_1x_2y_2$ in $\Delta_{\Lambda,\Lambda'}$. So we assume $w_1\neq w_2$.

%As $X_1\cap X_2$ is connected, there is an edge path $\omega\subset X_1\cap X_2$ from $w_1$ to $w_2$. Let $\{q_i\}_{i=1}^k$ be consecutive vertices of $\omega$ with $w_1=q_1$ and $w_2=q_k$.

Let $\Lambda_1$ (resp. $\Lambda'_1$) be the linear subgraph from $b_3$ (resp. $b_2$) to $s_1$ with the orientation such that $s_1$ is the minimal node. Recall that $\Delta_{C,\Lambda'_1}$ is bowtie free. This together with Assumption 1 and Lemma~\ref{lem:enlarge1} imply that $\Delta_{\Lambda,\Lambda_1}$ is bowtie free. As both $x_1,x_2$ are adjacent to $w_1$ and $w_2$, we know $x_i\le w_j$ in $\Delta_{\Lambda,\Lambda_1}$ for $1\le i,j\le 2$. Let $z$ be the join of $x_1$ and $x_2$ in $\Delta_{\Lambda,\Lambda_1}$. Then $z$ is adjacent to each of $x_1,x_2,w_1,w_2$ in $\Delta_{\Lambda,\Lambda_1}$. As $w_1\neq w_2$, we know $z$ is of type $\hat s_z$ with $s_z\in \Lambda'_1$. 

For $i=1,2$, let $\Phi_{\Lambda',i}$ be the induced subcomplex of $\lk(w_i,\Delta)$ spanned by vertices of type $\hat s$ with node $s\in \Lambda'$. Then $z\in \Phi_{\Lambda',i}$.
By Lemma~\ref{lem:link}, $\Phi_{\Lambda',i}$ admits a type-preserving isomorphism to $\Delta_{C,\Lambda'}$. Note that
$w_i$ is adjacent to $x_1,x_2,y_i,z$ in $\Phi_{\Lambda',i}$. As $\Phi_{\Lambda',i}$ is bowtie free by Assumption 2, we know $x_1$ and $x_2$ have a join in $\Phi_{\Lambda',i}$, denoted by $z_i$. As $y_i$ is a common upper bound of $x_1$ and $x_2$, we know $z_i$ is adjacent to $y_i$ in $\Delta$.
As $z$ is adjacent to both $x_1$ and $x_2$ in $\Delta$, $z$ is a common upper bound for $x_1,x_2$ in $\Phi_{\Lambda',i}$, thus $z_i$ has type $\hat s_{z_i}$ with node $s_{z_i}$ contained in the segment of $\Lambda'$ from $s_1$ to $s_z$. Thus, $s_{z_i}\in \Lambda'_1$. Moreover, by the choice of $z_i$, we know $z$ is either equal or adjacent to $z_i$ in $\Delta$.
If $z_i\neq z$, then $z_i<z$ in $\Delta_{\Lambda,\Lambda_1}$, which is a contradiction as $z$ is the join of $x_1$ and $x_2$ in $\Delta_{\Lambda,\Lambda_1}$. Thus $z=z_1=z_2$ and $z$ is adjacent to each of $\{x_1,x_2,y_1,y_2\}$.

%, as desired.

%In particular $y_i\ge z_i$ in $\Phi_{\Lambda',i}$. 

%Let $T''$ be the linear subgraph in $T$ from $b_3$ to $s_n$. Lemma~\ref{lem:enlarge1} implies that $\Delta_{\Lambda,T''}$ is bowtie free. Let $z$ be the meet of $y_1$ and $y_n$ in $\Delta_{\Lambda,T''}$. Then $q_i\le z$ for $i=1,2$ in $\Delta_{\Lambda,T''}$. In particular, $z\in \Phi_{\Lambda',i}$ for $i=1,2$. As $z_1,z_2\in \Delta_{\Lambda,T''}$ by the previous paragraph, we know $z_1=z=z_2$.
\medskip
\noindent
\underline{Case 2: $Y_1\cap Y_2\cap X_i\neq\emptyset$ for $i=1,2$.} If $Y_1\cap Y_2$ is a single vertex $w$, then $w\in X_i$ for $i=1,2$ and $w$ is adjacent to each of $\{x_1,y_1,x_2,y_2\}$ in $\Delta$. If $w$ is of type $\hat b_1$ or $\hat b_2$, then we are done. The case $w$ is of type $\hat b_3$ can be handled using the argument in Case 1. 

Now we assume $Y_1\cap Y_2$ is not a single vertex.
We claim that for $i=1,2$, $Y_1\cap Y_2\cap X_i$ contains a vertex $w_i$ of type $\hat b_2$ or $\hat b_3$. Indeed, if $w_i$ has type $\hat b_1$, then we can find $w'_i\in Y_1\cap Y_2$ which is adjacent to $w_i$, as $Y_1\cap Y_2$ is connected (from the convexity) and not an isolated point. Then $w'_i$ is of type $\hat b_2$ or $\hat b_3$. Note that the geodesic from $s_1$ to $b_2$ or $b_3$ passes through $b_1$, and recall that $x_i$ is adjacent or equal to $w_i$. So $x_i$ is adjacent to $w'_i$ by Corollary~\ref{cor:adj}, which implies that $w'_i\in Y_1\cap Y_2\cap X_i$. We can replace $w_i$ by $w'_i$ and the claim follows. 

The case $w_1=w_2$ can be handled as before. So we assume $w_1\neq w_2$.

Let $\Lambda_n$ (resp. $\Lambda'_n$) be the linear subgraph from $b_3$ (resp. $b_2$) to $s_n$ with the orientation such that $s_n$ is the maximal node. As in Case 1, we know $\Delta_{\Lambda,\Lambda_n}$ is bowtie free. Note that $y_1$ and $y_2$ have common lower bounds $\{w_1,w_2\}$ in $\Delta_{\Lambda,\Lambda_n}$. Let $z$ be their meet in $\Delta_{\Lambda,\Lambda_n}$. As $w_1\neq w_2$, $z$ has type $\hat s_z$ with node $s_z\in \Lambda'_n$. If $w_i$ is of $\hat b_3$, then by the argument in Case 1, we know $x_i$ is adjacent to $z$. Now suppose $w_i$ is of type $\hat b_2$. Then $w_i\in \Delta_{\Lambda,\Lambda_n}$ and $w_i$ is a lower bound for $\{y_1,y_2\}$ in $\Delta_{\Lambda,\Lambda_n}$. Thus $w_i\le z$ in $\Delta_{\Lambda,\Lambda_n}$ by our choice of $z$. In particular, $w_i$ is adjacent or equal to $z$ in $\Delta$. 
As $x_i$ is adjacent to $w_i$ in $\Delta$, and the segment in $\Lambda'$ from $s_1$ to $s_z$ contains $b_2$ (if $s_z\neq b_2$), by Corollary~\ref{cor:adj}, $x_i$ is adjacent to $z$. Thus we always know that $z$ is adjacent to each of $\{x_1,y_1,x_2,y_2\}$ in $\Delta$, and we finish as before.
\end{proof}

\begin{lem}
	\label{lem:girth}
Suppose $X$ is a simplicial complex of type $\{\hat s_1,\hat s_2,\hat s_3\}$ such that all the triangles in $X$ are convex triangles in the Euclidean plane or the hyperbolic plane with angle $\pi/m_i$ at the vertex of type $\hat s_i$ for $1\le i\le 3$ with $\frac{1}{m_1}+\frac{1}{m_2}+\frac{1}{m_3}\le 1$. We assume $X$ is $\CAT(0)$. Let $Y\subset X$ be a convex subcomplex of $X$.
Let $\Theta_X$ (resp. $\Theta_Y$) be the induced subcomplex of $X$ (resp. $Y$) spanned by vertices of type $\hat s_2$ and $\hat s_3$. Then
\begin{enumerate}
	\item $\Theta_X$ has girth $\ge 2m_1$;
	\item if for any vertex $v$ in $Y$ of type $\hat s_1$, we have $\lk(v,Y)=\lk(v,X)$, then for any two vertices of $\Theta_Y$ of combinatorial distance $<m_1$, any shortest edge path connecting these two vertices in $\Theta_X$ is contained in $\Theta_Y$.
\end{enumerate}
\end{lem}

\begin{proof}
	For Assertion 1, let $\omega$ be an embedded cycle in $\Theta_X$ of shortest length and let $\{v_i\}_{i\in \mathbb Z/k\mathbb Z}$ be consecutive vertices of $\omega$. First we consider the case The case $\frac{1}{m_1}+\frac{1}{m_2}+\frac{1}{m_3}<1$, then each triangle in $X$ has a hyperbolic metric.
	Let $\mathbb D$ be a minimal area singular disk diagram for $\omega$ (see Section~\ref{ss:disk} for relevant definitions). Then 
	\begin{equation}
		\label{eq:GB2}
		\sum_{v\in \bD^{(0)}}\kappa(v)-\sum_{C\in \bD^{(2)}}\mathrm{Area}(C)=2\pi.
	\end{equation}
	We can compare this disk diagram with another disk diagram $\bD_0$ obtained by assembling $2m_1$ copies of some triangle in $X$ in a cyclic fashion around the vertex of type $\hat s_1$. The Gauss-Bonnet formula for $\bD_0$ is 
	\begin{equation}
		\label{eq:GB3}
		\sum_{v\in \bD^{(0)}_0}\kappa(v)-\sum_{C\in \bD^{(2)}_0}\mathrm{Area}(C)=2\pi.
	\end{equation}
	Note that 
	\begin{enumerate}
		\item $\bD$ must has an interior vertex, otherwise we can find two non consecutive vertices of $\omega$ that are adjacent in $X$, contradicting $\omega$ is an embedded cycle of shortest length; moreover, we can find an interior vertex of $\bD$ which maps to a vertex of type $\hat s_3$; thus the area of $\bD$ is $\ge$ the area of $\bD_0$ as $\bD$ is locally $\CAT(0)$ around this vertex;
		\item for any pair of vertices $v\in \partial \bD$ and $v_0\in \partial \bD_0$ of the same type, we must have $\kappa(v)\le \kappa(v_0)$;
		\item if $v$ is a vertex in the interior of $\bD$, then $\kappa(v)\le 0$; if $v$ is a vertex in the interior of $\bD_0$, then $\kappa(v)=0$.
	\end{enumerate}
	Comparing \eqref{eq:GB2} and \eqref{eq:GB3}, we know the number of vertices on $\partial \bD$ is $\ge$ the number of vertices in $\partial \bD_0$. Thus $k\ge 2m_1$. The case $\frac{1}{m_1}+\frac{1}{m_2}+\frac{1}{m_3}=1$ is similar.

	For Assertion 2, take two vertices $q_1$ and $q_2$ of $\Theta_Y$ connecting by a shortest edge path $\omega$ in $\Theta_X$ made of $\le m_1-1$ edges. We can assume $\omega\cap \Theta_Y$ only in its two endpoints. Let $\omega'$ be the image of $\omega$ under the nearest point projection $X\to Y$. The definition of nearest point projection implies that $\omega'$ cannot contain any point in the interior of a 2-face of $Y$. As for any vertex $v$ in $Y$ of type $\hat s_1$, we have $\lk(v,Y)=\lk(v,X)$, thus $\omega'$ does not contain any vertices of type $\hat s_1$, as well as any interior point of an edge with one type $\hat s_1$ endpoint. Thus $\omega'\subset \Theta_Y$. As the length of $\omega'$ is upper bounded by the length of $\omega$, we know $q_1$ and $q_2$ can be connected by an edge path $\omega''$ in $\Theta_Y$ with $\le m_1-1$ edges. Note that $\omega$ and $\omega''$ form a loop in $\Theta_X$, which has girth $2m_1$. As $\omega$ and $\omega''$ are both shortest edge paths connecting two vertices, we have $\omega=\omega''$.
\end{proof}

%
%If $Y_1\cap Y_2\cap X_i$ contains a vertex $w_i$ of type $\hat b_3$ for $i=1,2$, then $w_i$ is also a vertex in $Y'_1\cap Y'_1\cap X'_i$, and by symmetry we are done by repeating the argument in the previous case. 
%
%
%
%
%
%Let $w_1$ be as before. If $X_1\cap Y_1\cap Y_2$ has a vertex of type $\hat b_2$, then it also has a vertex of type $\hat b_3$, and we can handle this case as before.
%
%If $X_1\cap Y_1\cap Y_2$ does not contain a vertex of type $\hat b_2$. If $X_1\cap Y_1\cap Y_2$ is one vertex, then it is not hard. Otherwise it has a vertex of type $\hat b_3$. 
%
%Thus either $X_1\cap Y_1\cap Y_2$ is exactly one vertex, or it has a vertex of type $\hat b_3$. In the latter case, we consider the complex $X'$ for $b_3,b_2,b'_1$. Then we still know $X'_1\cap Y'_1\cap Y'_2$ is nonempty as it has a vertex of type $\hat b_3$. 
%
%
%Now we show $z_1=z_2$, which will finish the proof.
%
%
%
%
%It follows from the above analysis of $\lk(w_1,Y_1)$ and $\lk(w_1,X_1)$ that if $w_1$ is of type $1$ or $2$, then there is a vertex in $Y_1\cap Y_2\cap X_1$ of type $\hat b_3$. Thus we can assume $w_1$ is of type $\hat b_3$, and similarly that $w_2$ has type $\hat b_3$. As $Y_1\cap Y_2$ is connected, there is an edge path $\omega\subset Y_1\cap Y_2$ from $w_1$ to $w_2$. Let $\{q_i\}_{i=1}^k$ be consecutive vertices of $\omega$ with $w_1=q_1$ and $w_2=q_k$.

\subsection{Contractibility of the Artin complexes}
\label{subsec:contractible}

\begin{prop}
	\label{prop:tree}
	Suppose $\Lambda$ is tree Dynkin diagram. Suppose there exists a collection $E$ of open edges with label $\ge 6$ such that for each component $\Lambda'$ of $\Lambda\setminus E$ the Artin complex $\Delta_{\Lambda'}$ satisfies the labeled 4-cycle condition.  Then 
	\begin{enumerate}
		\item $\Delta_{\Lambda,\Lambda'}$ satisfies the labeled 4-cycle condition for each component $\Lambda'$ of $\Lambda\setminus E$; and each edge $e\in E$ is $6$-solid in $\Lambda$;
		\item if there exists an edge $e_1\in E$ and a different edge $e_2$ such that they intersect in a node, then $\Delta_{\Lambda,e_1\cup e_2}$ is contractible.
	\end{enumerate}
\end{prop}

\begin{proof}
We first prove (1) and 	we induct on the number of edges in $\Lambda$. The base case when $\Lambda$ has only one edge follows from \cite[Lemma 6]{appel1983artin}. Now we assume the statement is true for $\Lambda$ with $\le (n-1)$ edges. Take $\Lambda$ with $n$ edges. If $E=\emptyset$, then the statement is clearly true. Now assume $E\neq\emptyset$. Let $\Lambda'$ be a component of $\Lambda\setminus E$, and let $\Upsilon$ be an arbitrary maximal linear subgraph of $\Lambda'$. By Proposition~\ref{prop:4-cycle} and Lemma~\ref{lem:connect}, we need to show $\Delta_{\Lambda,\Upsilon}$ is bowtie free.
	
	As $E\neq\emptyset$, there exists an edge $e=\overline{ab}$ with $e\in E$ such that $e\cap \Lambda'\neq\emptyset$. Suppose $a=e\cap \Lambda'$. Let $C$ be the component of $\Lambda\setminus\{b\}$ that contains $a$. Then $\Lambda'\subset C$. Let $E_C=E\cap C$. Note that components of $C\setminus E_C$ are also components of $\Lambda\setminus E$. Thus $C$ also satisfies the assumption of the proposition, but $C$ has fewer edges compared to $\Lambda$. By induction, $\Delta_{C,\Lambda'}$ satisfies the labeled 4-cycle condition. By Proposition~\ref{prop:4-cycle}, Lemma~\ref{lem:connect} and Lemma~\ref{lem:inherit}, for any linear subgraph $\Lambda''\subset \Lambda'$, $\Delta_{C,\Lambda''}$ is bowtie free.
	
	Now we go back to $\Upsilon$. Note that either $e\cap \Upsilon$ is an endpoint of $\Upsilon$, or $e\cap \Upsilon$ is an interior point of $\Upsilon$, or $e\cap \Upsilon=\emptyset$. We first consider the case $e\cap \Upsilon$ is an endpoint of $\Upsilon$. Let $c\in \Up$ be the node adjacent to $a\in e$. Let $C_e$ be the component of $\Lambda\setminus \{c\}$ that contains $e$. Note that each component of $C_e\setminus (E\cap C_e)$ is contained in a component of $\Lambda\setminus E$ (though the containment could be proper). By Lemma~\ref{cor:inherit}, $C_e$ satisfies the assumption of the proposition. As $C_e$ has fewer edges compared to $\Lambda$, by induction $e$ is $6$-solid in $C_e$. Thus by Lemma~\ref{lem:enlarge1}, Lemma~\ref{lem:inherit} and the previous paragraph, $\Delta_{\Lambda,\Upsilon}$ is bowtie free. The case of $e\cap \Upsilon$ is an interior point of $\Upsilon$, and the case of $e\cap \Upsilon=\emptyset$ can be treated similarly, and we use Lemma~\ref{lem:enlarge3} and Lemma~\ref{lem:enlarge2} instead.
	
	It remains to show for any edge $e\in E$ is 6-solid in $\Lambda$. Let $e'$ be an edge of $\Lambda$ such that $e\cap e'$ is a node, denoted $b$. Let the other endpoint of $e$ and $e'$ to be $a$ and $c$ respectively. Let $X=\Delta_{\Lambda,e\cup e'}$. Then $X$ is simply connected. We metrize triangles in $X$ as flat Euclidean triangle such that it has angle $\pi/3$ at vertex of type $\hat a$, angle $\pi/6$ at vertex of type $\hat c$, and angle $\pi/2$ at vertex of type $\hat b$. Take a vertex $x\in X$. By Lemma~\ref{lem:link}, if $x$ has type $\hat c$, then $\lk(x,X)$ admits a type-preserving isomorphism to $\Delta_{C_1,e}$ where $C_1$ is the component of $\Lambda\setminus\{c\}$ that contains $e$. As before, we know $C_1$ and $C_1\cap E$ satisfies the assumption of the proposition. Thus the induction assumption implies that $e$ is 6-solid in $C_1$, thus $\lk(x,X)$ has girth $\ge 12$. If $x$ has type $\hat a$, then $\lk(x,X)$ admits a type-preserving isomorphism to $\Delta_{C_2,e'}$ where $C_2$ is the component of $\Lambda\setminus\{a\}$ that contains $e'$. Note that $C_2$ and $C_2\cap E$ satisfy the assumptions of the proposition. Thus if $e'\in E$, then $e'$ is 6-solid in $C_2$. If $e'\notin E$, then $\Delta_{C_2,C_{e'}}$ satisfies labeled 4-cycle condition, where $C_{e'}$ is the component of $C_2\setminus (C_2\cap E)$ containing $e'$. Thus $\Delta_{C_2,e'}$ satisfies labeled 4-cycle condition, which implies that $e'$ is 3-solid in $C_2$. Thus in either case $\lk(x,X)$ has girth $\ge 6$. If $x$ has type $\hat b$, then $\lk(x,X)$ is a complete bipartite graph. Thus $X$ is locally $\CAT(0)$, hence $\CAT(0)$. By Lemma~\ref{lem:girth} (1), $e$ is 6-solid in $\Lambda$. This proves Assertion 1. Assertion 2 follows from the contractibility of $X$.
\end{proof}
%it suffice to show there does not exists an embedded edge loop $\omega$ in $X$ such that $\omega$ does not contain any vertex of type $\hat c$ and $\omega$ has $\le 11$ edges. We view such edge loop as a geodesic polygon in a $\CAT(0)$ space. Note that the angle between two adjacent edges of $\omega$ meeting at a vertex of type $\hat a$ (resp. $\hat b$) is $\ge 2\pi/3$ (resp. $\pi$). As vertices of $\omega$ alternates between type $\hat a$ vertices and type $\hat b$ vertices, by applying \cite[Chapter II.2.12 (1)]{BridsonHaefliger1999}, we know $\omega$ has $\ge$ 12 edges, as desired. 

\begin{prop}
	\label{prop:tree contractible}
	Suppose $\Lambda$ is tree Dynkin diagram. Suppose there exists a collection $E$ of open edges with label $\ge 6$ such that for each component $\Lambda'$ of $\Lambda\setminus E$ the Artin complex $\Delta_{\Lambda'}$ satisfies the labeled 4-cycle condition. Then
	\begin{enumerate}
		\item if $\Lambda$ contains at least two open edges, with one of them in $E$, then the Artin complex $\Delta_{\Lambda}$ of $\Lambda$ is contractible;
		\item if each component of $\Lambda\setminus E$ satisfies the $K(\pi,1)$-conjecture, then $A_\Lambda$ satisfies the $K(\pi,1)$-conjecture.
	\end{enumerate}
\end{prop}

\begin{proof}
Let $\mathcal C_2$ be the collection of Dynkin diagrams satisfying the assumption of this proposition. Let $\mathcal C_1$ be the collection of Dynkin diagrams which are the union of two edges with one edge labeled by a number $\ge 6$. Now we verify $\mathcal C_1$ and $\mathcal C_2$ satisfy the assumptions Proposition~\ref{prop:contractible}. Assumption 1 of Proposition~\ref{prop:contractible} follows from Proposition~\ref{prop:tree}. Assumption 2 of Proposition~\ref{prop:contractible} follows from Proposition~\ref{prop:inherit}, Proposition~\ref{prop:4-cycle} and Lemma~\ref{lem:connect}. Thus Assertion 1 of the proposition follows by Proposition~\ref{prop:contractible}. Let $\Lambda$ and $\Lambda'$ be as in Assumption 3 of Proposition~\ref{prop:contractible}. Then each component of $\Lambda'$ is either a single edge, hence satisfies the $K(\pi,1)$-conjecture \cite{deligne}; or is contained in a component of $\Lambda\setminus E$, hence satisfies the $K(\pi,1)$-conjecture by \cite[Corollary 2.4]{godelle2012k} and our assumption. Thus $A_{\Lambda'}$ satisfies the $K(\pi,1)$-conjecture. This finishes the proof.
\end{proof}

\subsection{More Artin complexes with the labeled 4-cycle condition}
\label{subsec:application}
In order to apply Proposition~\ref{prop:tree contractible} to obtain concrete class of Artin groups satisfying the $K(\pi,1)$-conjecture, we need examples of Artin groups with tree Dynkin diagrams whose Artin complexes satisfy the labeled 4-cycle condition. We already know this is satisfied for spherical Artin groups, see Corollary~\ref{cor:wheel}. Now we consider another class. Recall that an Artin group is \emph{locally reducible} if any irreducible spherical standard parabolic subgroup has rank $\le 2$. A Dynkin diagram is \emph{locally reducible} if the associated Artin group is locally reducible.

\begin{lem}
	\label{lem:locally reducible}
Suppose $\Lambda$ is locally reducible. If an edge $e\subset\Lambda$ has label $m$, then it is $m$-solid.
\end{lem}

\begin{proof}
We prove this by induction on the number of edges in $\Lambda$. The case when $\Lambda$ has only one edge follows from \cite[Lemma 6]{appel1983artin}. For the general case, let $e_1\in \Lambda$ be an edge labeled by $m_1$. If $e_1$ is an isolated edge, then we are done by \cite[Lemma 6]{appel1983artin}. If $e_1$ is not an isolated edge, let $e_2$ be a different edge intersecting $e_1$ in a node. Suppose $e_1=\overline{s_1s_2}$ and $e_2=\overline{s_2s_3}$. Let $X=\Delta_{\Lambda,e_1\cup e_2}$. We metrize triangles with metrics of constant curvature $1,0$ or $-1$ such that each triangle has angle $\pi/m_1$ at the vertex of type $\hat s_3$, angle $\pi/m_2$ at the vertex of type $\hat s_1$ and angle $\pi/m_3$ at the vertex of type $\hat s_2$ -- here $m_2$ is the label of $e_2$, and $m_3=2$ (if $s_1$ and $s_3$ are not adjacent) or $m_3$ is the label of $\overline{s_1s_3}$. As $\Lambda$ is locally reducible, we have $\frac{1}{m_1}+\frac{1}{m_2}+\frac{1}{m_3}\le 1$, thus each triangle is either flat or hyperbolic. For a vertex $x\in X$ of type $\hat s_1$, $\lk(x,X)$ admits a type-preserving isomorphism to $\Delta_{\Lambda\setminus\{s_1\},e_2}$. By induction assumption, $\lk(x,X)$ has girth $\ge 2m_2$, thus it is $\CAT(1)$ with the angular metric induced from $X$. Similarly the link of each vertex of $X$, with the angular metric, is $\CAT(1)$. As $X$ is simply connected by Lemma~\ref{lem:relative sc}, $X$ is either $\CAT(0)$ or $\CAT(-1)$. We need to show $\Delta_{\Lambda,e_1}$ has girth $\ge 2 m_1$. Note that $\Delta_{\Lambda,e_1}$ can be identify with the induced subgraph of $X$ spanned by vertices of type $\hat s_1$ and $\hat s_2$. Thus we are done by Assertion 1 of Lemma~\ref{lem:girth}.
\end{proof}

\begin{cor}
	\label{cor:locally reducible}
Suppose $\Lambda$ is a locally reducible tree. Then $\Delta_\Lambda$ satisfies the labeled 4-cycle condition. 
\end{cor}

\begin{proof}
We first claim that for linear subgraph $\Lambda'\subset \Lambda$ with three nodes and any node $s\in \Lambda\setminus\Lambda'$ such that $\{s,\Lambda'\}$ is contained in a linear subgraph of $\Lambda$, the induced subcomplex $Y$ of $X:=\Delta_{\Lambda,\Lambda'}$ spanned by vertices which are adjacent to a vertex in $y_0\in \Delta_\Lambda$ of type $\hat s$ is a convex subcomplex of $X$. Here we view $\Delta_{\Lambda,\Lambda'}$ as a subcomplex of $\Delta_\Lambda$.

We prove the claim by induction on the number of nodes in $\Lambda$. The base case when $\Lambda$ has three nodes is trivial. For the induction step, let consecutive nodes of $\Lambda'$ be $\{s_1,s_2,s_3\}$ and we assume without loss of generality that the shortest path in $\Lambda$ from $s_1$ to $s$ passes through $s_2$ and $s_3$. We endow $X$ with the metric as in Lemma~\ref{lem:locally reducible}. Then $X$ is $\CAT(0)$ by Lemma~\ref{lem:locally reducible}. It suffices to show $Y$ is locally convex in $X$ around each vertex of $Y$. By an identical argument as in the proof of Lemma~\ref{lem:convex}, we know $\lk(x,Y)=\lk(x,X)$ when $x$ is a vertex of type $\hat s_3$ in $Y$. Clearly $Y$ is locally convex in $X$ around such $x$. The local convexity of $Y$ around a vertex of type $\hat s_2$ is also identical to the proof of Lemma~\ref{lem:convex}. 

It remains to consider $x\in Y$ is of type $\hat s_1$. Let $m$ be the label of $\overline{s_2s_3}$. It suffices to show that if two vertices in $\lk(x,Y)$ are joined by edge path in $\lk(x,X)$ with $\le m-1$ edges, then this edge path is contained in $\lk(x,Y)$.

Let $\Lambda_1$ be the component of $\Lambda\setminus\{s_1\}$ containing the node $s_2$. Then Lemma~\ref{lem:link} implies that $\lk(x,X)$ admits a type-preserving isomorphism to $\Delta_{\Lambda_1,\overline{s_2s_3}}$, moreover, as $y_0\in \lk(x,\Delta_\Lambda)$, we can also view $y_0$ as a vertex in $\Delta_{\Lambda_1}$.
Let $s_4$ be the node adjacent to $s_3$ in the shortest path of $\Lambda$ from $s_3$ to $s$ (it is possible that $s_4=s$). Let $\Lambda_2$ be the subgraph of $\Lambda_1$ spanned by $\{s_2,s_3,s_4\}$. Then we view $\Delta_{\Lambda_1,\overline{s_2s_3}}$, hence $\lk(x,X)$, as the induced subcomplex of $\Delta_{\Lambda_1,\Lambda_2}$ spanned by vertices of type $\hat s_2$ and $\hat s_3$. Let $X'=\Delta_{\Lambda_1,\Lambda_2}$, viewed as a subcomplex of $\Delta_{\Lambda_1}$. 
Let $Y'$ be the induced subcomplex of $X'$ spanned by vertices of $X'$ that are adjacent to $y_0$. We metrize $X'$ as a $\CAT(0)$ space as in Lemma~\ref{lem:locally reducible}.
By induction assumption, $Y'$ is a convex subcomplex of $X'$. Moreover, by the proof of Lemma~\ref{lem:convex}, $\lk(x',Y')=\lk(x',X')$ for any vertex $x'\in Y'$ of type $\hat s_4$. As we identified $\lk(x,X)$ with the subcomplex of $X'$ spanned by vertices of type $\hat s_2$ and $\hat s_3$, under such identification, $\lk(x,Y)$ is identified with the induced subcomplex of $Y'$ spanned by vertices of type $\hat s_2$ and $\hat s_3$. Now by Lemma~\ref{lem:girth}, the statement in the previous sentence holds true. This finishes the proof of the claim.

For the corollary, the proof is by induction on the number of nodes in $\Lambda$. The base case when $\Lambda$ has two nodes follow from \cite[Lemma 6]{appel1983artin}. For the induction step, 
by Lemma~\ref{lem:connect} and Proposition~\ref{prop:4-cycle}, we need to show $\Delta_{\Lambda,\Lambda'}$ is bowtie free for each maximal linear subgraph $\Lambda'$ of $\Lambda$. For this we will show the two assumptions of Corollary~\ref{cor:connected intersection} are satisfied. The first assumption follows from the induction assumption, Lemma~\ref{lem:link}, Lemma~\ref{lem:connect} and Proposition~\ref{prop:4-cycle}. For the second assumption, let $\{s_i\}_{i=1}^n$ be consecutive nodes in $\Lambda'$ and let $\Lambda''$ be the linear subgraph spanned by $\{s_1,s_2,s_3\}$. Let $y_1$ and $y_2$ be two vertices of type $\hat s_n$ in $\Delta_{\Lambda,\Lambda'}$.
We view $\lk(y_1,\Delta_{\Lambda,\Lambda'})$ as a subcomplex of 
$\Delta_{\Lambda,\Lambda'}$.
Let 
$$
C_i=\lk(y_i,\Delta_{\Lambda,\Lambda'})\cap \Delta_{\Lambda,\Lambda''}
$$
for $i=1,2$ (we view $\Delta_{\Lambda,\Lambda''}$ as a subcomplex of $\Delta_{\Lambda,\Lambda'}$). By Lemma~\ref{lem:locally reducible}, we metrize $\Delta_{\Lambda,\Lambda''}$ as a $\CAT(0)$ complex, and the above claim implies that $C_i$ is a convex subcomplex of $\Delta_{\Lambda,\Lambda''}$ for $i=1,2$. Then $C_1\cap C_2$ is connected if it is not empty. Thus Assumption 2 of Corollary~\ref{cor:connected intersection} follows.
\end{proof}

We have yet another class of Artin groups whose Artin complexes satisfy the bowtie free condition, which can be deduced from work of Haettel \cite{haettel2021lattices}.
\begin{prop}
	\label{prop:widetilde C}
Suppose $A_{\Lambda}$ is an affine Artin group of type $\widetilde C_n$. Then $\Delta_{\Lambda}$ is bowtie free.
\end{prop}

\begin{proof}
We verify the assumptions of Lemma~\ref{lem:bowtie free criterion}. Assumption 1 is a direct consequence of \cite[Proposition 6.6]{haettel2021lattices}. Let $x_1y_1x_2y_2$ be a 4-cycle in $\Delta_{\Lambda}$ as in Assumption 2 of Lemma~\ref{lem:bowtie free criterion}. Let $Y$ be the graph defined as follows. The vertices of $Y$ are vertices of $\Delta_{\Lambda}$. Two vertices $x_1,x_2\in \Delta_{\Lambda}$ are adjacent in $Y$ if there exist $x\in \Delta_{\Lambda}$ of type $\hat s_1$ and $y\in \Delta_{\Lambda}$ of type $\hat s_n$ such that $x\le \{x_1,x_2\}\le y$. It is shown in \cite[Proposition 7.5]{haettel2021lattices} that $Y$ is a Helly graph. As any Helly satisfies the 4-wheel condition (see e.g. \cite[Proposition 3.25]{chalopin2020weakly}), there exists $z\in Y$ such that $z$ is adjacent to each of $\{x_1,y_1,x_2,y_2\}$ in $Y$. As $x_1,x_2$ are of type $\hat s_1$ and $y_1,y_2$ are of type $\hat s_n$, by the definition of edges in $Y$ we know $z$ is also adjacent to each of $\{x_1,y_1,x_2,y_2\}$ in $\Delta_\Lambda$, as required.
\end{proof}

Locally reducible Artin groups are known to satisfy the $K(\pi,1)$-conjecture \cite{charney2000tits}, and Artin groups of type $\widetilde C_n$ are known to satisfy the $K(\pi,1)$-conjecture \cite{okonek1979k}.
The following is a combination of Corollary~\ref{cor:wheel}, Corollary~\ref{cor:locally reducible}, Proposition~\ref{prop:widetilde C} and Proposition~\ref{prop:tree contractible}.

\begin{thm}
	\label{thm:combine1}
	Suppose $\Lambda$ is a tree Dynkin diagram. Suppose there exists a collection $E$ of edges with label $\ge 6$ such that each component of $\Lambda\setminus E$ is either spherical, or locally reducible, or of type $\widetilde C_n$. Then $A_\Lambda$ satisfies the $K(\pi,1)$-conjecture.
\end{thm}

The following follows from Corollary~\ref{cor:algebraic}, Corollary~\ref{cor:wheel}, Corollary~\ref{cor:locally reducible}, Proposition~\ref{prop:tree contractible}, Proposition~\ref{prop:widetilde C} and the main theorem of \cite{paolini2021proof}.

\begin{thm}
	\label{thm:combine2}
Suppose $\Lambda$ is a tree Dynkin diagram. Suppose there exists a collection $E$ of open edges with label $\ge 6$ such that each component of $\Lambda\setminus E$ is either spherical, or affine, or locally reducible. Suppose  the two assumptions of Corollary~\ref{cor:algebraic} hold for affine Artin groups of type $\widetilde B$, $\widetilde D$, $\widetilde E$ and $\widetilde F$.
Then $A_\Lambda$ satisfies the $K(\pi,1)$-conjecture.
\end{thm}

We end this section with the following conjecture.

\begin{conj}
Suppose $\Lambda$ is a tree Dynkin diagram. Then the Artin complex $\Delta_{\Lambda}$ satisfies the labeled 4-cycle condition.
\end{conj}

\section{Folded Artin complexes and Artin groups of generalized cyclic type}
\label{sec:cycle}

In this section we discuss classes of Artin groups which are not exactly of type $\widetilde A_n$, but whose geometry are related to the geometry of type $\widetilde A_n$ in an appropriate sense. 
We provide two criterions, namely Proposition~\ref{prop:single cycle} and Theorem~\ref{thm:folded} for proving new classes of Artin groups satisfying $K(\pi,1)$-conjecture. These two criterions are motivated from the geometry of Artin groups associated with quasi-Lannner diagrams, where the $K(\pi,1)$-conjecture is widely open for them. More precisely, Proposition~\ref{prop:single cycle} treats and generalizes quasi-Lann\'er diagrams of ``lollipop'' shape, and Theorem~\ref{thm:folded} treats and generalizes the more challenging example of quasi-Lann\'er diagram made of two 4-gons glued along two consecutive edges. 

In Section~\ref{subsec:single cycle} we prove Proposition~\ref{prop:single cycle}. In Section~\ref{subsec:folded} we introduce a new tool, called the folded Artin complex, and use it to prove Theorem~\ref{thm:folded}. In Section~\ref{subsec:some concrete classes}, we discuss applications of Proposition~\ref{prop:single cycle} and Theorem~\ref{thm:folded}.

\subsection{Artin groups with cycles in Dynkin diagrams}
\label{subsec:single cycle}

\begin{lem}
	\label{lem:cyclic contractible}
	Suppose $\Lambda$ is a Dynkin diagram with an induced cyclic subgraph $C$. For a node $s\in C$, let $C_s$ be the  component of  $C\setminus\{s\}$, and let $\Lambda_s$ be  the component of $\Lambda\setminus\{s\}$ that contains $C_s$.
	Suppose  for each node $s\in \Lambda$, $C_s$ is admissible in $\Lambda_s$ and the $(\Lambda_s,C_s)$-relative Artin complex is bowtie free. Then $\Delta_{\Lambda,C}$ is contractible.
\end{lem}

\begin{proof}
We will verify that $\Delta_{\Lambda,C}$ satisfies the three requirements of Theorem~\ref{thm:contractible}. The first requirement follows from Lemma~\ref{lem:relative sc}. The second requirement follows Lemma~\ref{lem:poset} and the assumption that $C_s$ is admissible in $\Lambda_s$. The third requirement follows from Lemma~\ref{lem:link} and the assumption that $\Delta_{\Lambda_s,C_s}$ being bowtie free.
\end{proof}

\begin{prop}
	\label{prop:single cycle}
Suppose $\Lambda$ is a connected Dynkin diagram with an induced cyclic subgraph $C$. For a node $s\in C$, let $C_s$ be the  component of  $C\setminus\{s\}$, and let $\Lambda_s$ be  the component of $\Lambda\setminus\{s\}$ that contains $C_s$.
Suppose  for each node $s\in \Lambda$, $C_s$ is admissible in $\Lambda_s$ and the $(\Lambda_s,C_s)$-relative Artin complex is bowtie free. Then the following assertions hold.
\begin{enumerate}
	\item $\Delta_\Lambda$ is contractible. 
	\item If, in addition, $A_{\Lambda_s}$ satisfies the $K(\pi,1)$-conjecture for each node $s\in C$, then so does $A_\Lambda$.
\end{enumerate}
\end{prop}

\begin{proof}
Let $\mathcal C_1=\{C\}$, and let $\mathcal C_2$ be the class of Dynkin diagrams which are connected induced subgraphs of $\Lambda$ that contain $C$. It suffices to verify the Assumptions of Corollary~\ref{cor:contractible} hold true for such choice of $\mathcal C_1$ and $\mathcal C_2$. Note that for any $\Theta\in \mathcal C_2$, it follows from Proposition~\ref{prop:inherit} that the pair $C\subset \Theta$ still satisfies the assumptions of Lemma~\ref{lem:cyclic contractible}. Thus $\Delta_{\Lambda',C}$ is contractible, and Assumption 1 of Corollary~\ref{cor:contractible}. Assumptions 2 and 3 are clear. For Assumption 4 of Corollary~\ref{cor:contractible}, take $s\in C$, and $\Theta\in \mathcal C_2$. 
Let $\Phi$ be a component of $\Theta\setminus\{s\}$. If $\Phi$ contains $C_s$, then  $\Phi\subset \Lambda_s$. If $\Phi$ does not contain $C_s$, then $\Phi\cap C_2=\emptyset$. As $\Lambda$ is connected, we know $\Phi\subset \Lambda_{s'}$ for $s'\in C$ with $s'\neq s$. By Assumption 2 of Proposition~\ref{prop:single cycle} and \cite[Corollary 2.4]{godelle2012k}, each component of $\Theta\setminus\{s\}$ corresponds to an Artin group which satisfies the $K(\pi,1)$-conjecture.
\end{proof}

%Assumption 1 follows from Lemma~\ref{lem:cyclic contractible}. Assumption 2 follows from Proposition~\ref{prop:inherit}. Assumptions 3 and 4 are clear.
\subsection{Folded Artin complexes and folded cycles in Dynkin diagrams}
\label{subsec:folded}
Let $\Lambda,\Lambda'$ be a pair of Dynkin diagrams. A \emph{special folding} is a surjective graph morphism $f:\Lambda\to \Lambda'$ (i.e. it map nodes to nodes and edges to edges) such that if nodes $x$ and $y$ are adjacent in $\Lambda'$, then each node in $f^{-1}(x)$ is adjacent in $\Lambda$ to every node in $f^{-1}(y)$. Note that a special folding in our sense is different from another well-known notion of folding defined by Crisp in \cite[Section 4]{crisp2011injective}.
An induced subgraph $\Lambda_1$ of $\Lambda$ is $f$-folded if the restriction of $f$ to $\Lambda_1$ is not injective. Now we define a ``folded'' version of (relative) Artin complexes.

\begin{definition}
Given a special folding $f:\Lambda\to \Lambda'$, we also use $f:S\to S'$ to denote the induced map on the sets of nodes.
We define the \emph{folded Artin complex} $\Delta_{\Lambda,f}$ as follows. Vertices of $\Delta_{\Lambda,f}$ are in 1-1 correspondence with left cosets of $A_{S\setminus \{f^{-1}(s')\}}$ in $A_S$. A collection of vertices span a simplex if the intersection of the associated collection of left cosets is non-empty. By Lemma~\ref{lem:flag}, $\Delta_{\Lambda,f}$ is a flag complex.
For any induced subgraph $\Lambda''\subset \Lambda'$, the associated \emph{relative folded Artin complex} $\Delta_{\Lambda,f,\Lambda''}$ is defined to be a simplicial complex whose vertices are in 1-1 correspondence with left cosets of $A_{S\setminus \{f^{-1}(s'')\}}$ in $A_S$ with $s''\in \Lambda''$, and simplices in $\Delta_{\Lambda,f,\Lambda''}$ corresponds to non-empty intersection of associated left cosets. 
If a vertex of $\Delta_{\Lambda,f,\Lambda''}$ corresponds to  $gA_{S\setminus \{f^{-1}(s'')\}}$ with node $s''\in \Lambda''$, then this vertex is defined to be of \emph{type $\hat s''$}.
\end{definition}

Let $S''=f^{-1}(S')$. Then there is a natural piecewise linear embedding $$i:\Delta_{\Lambda,f,\Lambda''}\to \Delta_{S,S''}$$ as follows. Given a vertex $x\in \Delta_{\Lambda,f,\Lambda''}$ corresponds to a coset of form $gA_{S\setminus \{f^{-1}(s'')\}}$ with $s''\in \Lambda''$, $i$ sends $x$ to the barycenter of the simplex of $\Delta_{S,S''}$ spanned by vertices of form $gA_{S\setminus \{s\}}$ with $s\in f^{-1}(s'')$. Note that $i$ sends vertices inside a simplex to points inside a simplex. Thus we can extend $i$ linearly.

\begin{lem}
	\label{lem:order fold}
Let $f:\Lambda\to\Lambda'$ be a special folding with $\Lambda'$ connected. Suppose $\Lambda''\subset \Lambda'$ is an induced linear subgraph with consecutive nodes $\{s'_i\}_{i=1}^n$. For simplicity we will say a vertex of $Y=\Delta_{\Lambda,f,\Lambda''}$ has type $i$ if it has type $\hat s'_i$. We define a relation in the vertex set $V$ of $Y$ as follows. For two vertices $x,y\in Y$, we put $x<y$ if they are adjacent in $Y$ and $\type(x)<\type(y)$. 

Suppose $\Lambda''$ is an admissible subgraph of $\Lambda'$. Then $(V,\le)$ is a partially ordered set.
\end{lem}

\begin{proof}
It remains to check the transitivity of the relation $<$. Take $x,y,z\in Y$ such that $x,z\in \lk(y,Y)$ and $\type(x)<\type(y)<\type(z)$. Suppose $x,y,z$ have type $i_1,i_2$ and $i_3$ respectively. Let $\Lambda'_y$ be the union of components of $\Lambda\setminus \{s'_{i_2}\}$. Similarly we define $\Lambda'_x$ and $\Lambda'_z$. Suppose $\Lambda'_y=\Lambda'_{y,1}\sqcup\Lambda'_{y,2}\sqcup\Lambda'_{y,3}$, where $\Lambda'_{y,k}$ is the component of $\Lambda\setminus\{s'_{i_2}\}$ containing $s'_{i_k}$ for $k=1,3$, and $\Lambda'_{y,2}$ is the union of the rest components of $\Lambda\setminus\{s'_{i_2}\}$. Let $\Lambda_y=f^{-1}(\Lambda'_y)$. Similarly we define $\Lambda_{y,i}$, $\Lambda_x$ and $\Lambda_z$. Note that $(\Lambda'_{y,3}\cup\Lambda'_{y,2})\subset \Lambda'_x\cap\Lambda'_y$ and $(\Lambda'_{y,1}\cup\Lambda'_{y,2})\subset \Lambda'_z\cap\Lambda'_y$.

Up to a left translation, we assume $y$ corresponds to the identity coset $A_{\Lambda_y}$ in $A_{\Lambda}$. Suppose $x$ (resp. $z$) corresponds to a coset of form $g_x A_{\Lambda_x}$ (resp. $g_zA_{\Lambda_y}$). Note that $A_{\Lambda_y}=A_{\Lambda_{y,1}}\oplus A_{\Lambda_{y,2}}\oplus A_{\Lambda_{y,3}}$. As $A_{\Lambda_y}\cap g_x A_{\Lambda_x}$ contains a left $A_{\Lambda_{y,3}}\oplus A_{\Lambda_{y,2}}$ coset, and $A_{\Lambda_y}\cap g_z A_{\Lambda_z}$ contains a left $A_{\Lambda_{y,1}}\oplus A_{\Lambda_{y,2}}$ coset; we know $g_x A_{\Lambda_x}\cap g_z A_{\Lambda_z}\neq\emptyset$. Thus $x<z$ and the transitivity is justified.
\end{proof}

%Let $\Theta_y$ be the union of components of $\Lambda\setminus \{f^{-1}(\type(y))\}$, i.e. $\Lambda\setminus \{f^{-1}(s'_{i_2})\}$. Similarly we define $\Theta_x$ and $\Theta_z$.
%Then for any pair of nodes $t_1\in f^{-1}(s'_{i_1})$ and $t_3\in f^{-1}(s'_{i_3})$, $t_1$ and $t_3$ are in different components of $\Theta_y$, as $f(t_1)$ and $f(t_3)$ are in different components of $f(\Lambda)\setminus\{s'_{i'_2}\}$ by the assumption that $\Lambda''$ is an admissible subgraph of $\Lambda'$. 
%Let $\Theta_{y,1}$ (resp. $\Theta_{y,3}$) be the union of components of $\Theta_y$ that contains at least one node in $f^{-1}(s'_{i_1})$ (resp. $f^{-1}(s'_{i_3})$). Then the previous discussion implies that $\Theta_{y,1}\cap \Theta_{y,3}=\emptyset$. We assume $\Theta_y=\Theta_{y,1}\sqcup\Theta_{y,3}\sqcup\Theta'_y$. 

The proof of the following is identical to that of Lemma~\ref{lem:relative sc}, Lemma~\ref{lem:link}, Lemma~\ref{lem:bowtie free criterion}, Lemma~\ref{lem:connect} and Proposition~\ref{prop:4-cycle}.
\begin{lem}
	\label{lem:folded}
Let $f:\Lambda\to\Lambda'$ be a special folding, and let $\Lambda''\subset\Lambda'$ be an induced subgraph.
Then the following are true.
\begin{enumerate}
	     \item If $\Lambda''$ has $\ge 3$ nodes, then $\Delta_{\Lambda,f,\Lambda''}$ is simply-connected.
		\item Let $v\in \Delta_{\Lambda,f,\Lambda''}$ be a vertex of type $\hat s'$ with $s'\in \Lambda''$. Let $\Lambda_{s'}$ (resp. $\Lambda''_{s'}$) be the induced subgraph of $\Lambda$ (resp. $\Lambda''$) spanned by nodes in $S\setminus\{f^{-1}(s')\}$ (resp. $\Lambda''\setminus\{s'\}$). Then 
		there is a type-preserving isomorphism between $\lk(v,\Delta_{\Lambda,f,\Lambda''})$ and $\Delta_{\Lambda_{s'},f,\Lambda''_{s'}}$. Moreover, if $I_{s'}$ is the union of components of $\Lambda_{s'}$ that contain at least one component of $f^{-1}(\Lambda''_{s'})$, then $f^{-1}(\Lambda''_{s'})\subset I_{s'}$ and there is a type-preserving isomorphic between $\lk(v,\Delta_{\Lambda,f,\Lambda''})$ and $\Delta_{I_{s'},f,\Lambda''_{s'}}$.
		\item If $\Lambda''$ is admissible and linear, then  $\Delta_{\Lambda,f,\Lambda''}$ is bowtie free if the analogues two assumptions of Lemma~\ref{lem:bowtie free criterion} hold true for $\Delta_{\Lambda,f,\Lambda''}$.
		\item If $\Lambda''$ is admissible and linear, then $\Delta_{\Lambda,f,\Lambda''}$ is bowtie free if and only if it satisfies the labeled 4-cycle condition.
		\item If $\Lambda''$ is an admissible tree subgraph of $\Lambda'$, then $\Delta_{\Lambda,f,\Lambda''}$ satisfies the bowtie free condition if and only if $\Delta_{\Lambda,f,\Lambda'''}$ for any maximal linear subgraph $\Lambda'''$ of $\Lambda''$.
\end{enumerate}
\end{lem}

%We define a relation in the vertex set $V$ of $Y$ as follows. For two vertices $x,y\in Y$, we put $x<y$ if they are adjacent in $Y$ and $\type(x)<\type(y)$. We now check $(V,\le)$ is a partially ordered set. 

\begin{lem}
	\label{lem:tree fold}
Suppose $f:\Lambda\to \Lambda'$ be a special folding between two trees. If $\Delta_{\Lambda}$ satisfies the labeled 4-cycle condition, then $\Delta_{\Lambda,f}$ satisfies the labeled 4-cycle condition.
\end{lem}

\begin{proof}
We induct on the number of nodes of $\Lambda$. Take a maximal linear subgraph $\Theta'\subset \Lambda'$ with consecutive nodes $\{s'_i\}_{i=1}^n$. By Lemma~\ref{lem:folded} (4) and (5), it suffices to show $X=\Delta_{\Lambda,f,\Theta'}$ is bowtie free.
Now we show the assumptions of Lemma~\ref{lem:folded} (3) (see also Lemma~\ref{lem:bowtie free criterion}) are met for $X$. 
For Assumption 1 of Lemma~\ref{lem:bowtie free criterion}, take a vertex of $X$ of type $\hat s'_1$. Let $\Lambda'_1$ be the component of $\Lambda'\setminus\{s'_1\}$ (note that $s'_1$ is a leaf node of $\Lambda$) and $\Lambda_1=f^{-1}(\Lambda'_1)$. By induction, $\Delta_{\Lambda_1,f}$ satisfies the labeled 4-cycle condition.  By Lemma~\ref{lem:folded} (4) and (5), $\Delta_{\Lambda_1,f,\Lambda'_1\cap \Theta'}$ is bowtie free. Thus  Assumption 1 of Lemma~\ref{lem:bowtie free criterion} follows.

Now we verify Assumption 2 of Lemma~\ref{lem:bowtie free criterion}. Take an embedded 4-cycle $x_1y_1x_2y_2$ in $X$ such that $x_1,x_2$ have type $\hat s'_1$ and $y_1,y_2$ have type $\hat s'_n$. Let $$f^{-1}(s_1)=\{\theta_i\}_{i=1}^{k_1}\ \textrm{and}\ f^{-1}(s_n)=\{\eta_i\}_{i=1}^{k_n}.$$ By the previous discussion, we view $x_1$ as the barycenter of a simplex $X_1$ in $\Delta_{\Lambda,f^{-1}(\Theta')}$ whose vertices are of type $\{\hat \theta_i\}_{i=1}^{k_1}$. Similarly, we define simplices $X_2,Y_1,Y_2$ in $\Delta_{\Lambda,f^{-1}(\Theta')}$. Note that $X_i$ and $Y_j$ span a simplex in $\Delta_\Lambda$ for $1\le i,j\le 2$. As $\Lambda$ is a tree, by the definition of special folding, $f^{-1}(s'_i)$ is a single point for $1<i<n$. Let $s_i=f^{-1}(s'_i)$ for $1<i<n$. Then Assumption 2 of Lemma~\ref{lem:bowtie free criterion} for $X$ follows from the following.
\begin{claim}
	\label{claim1}
There exist vertex $z\in \Delta_{\Lambda,f^{-1}(\Theta')}$ of type $\hat s_i$ for some $i$ satisfying $1<i<n$ such that $z$ is adjacent to each vertex in $X_1\cup X_2\cup Y_1\cup Y_2$.
\end{claim}
Before we prove Claim~\ref{claim1}, we need another preparatory result. Let $x_{1,\theta_1}$ and $x_{2,\theta_1}$ be the vertex of type $\hat \theta_1$ in $X_1$ respectively $X_2$. Similarly, we define $y_{1,\eta_1}\in Y_1$ and $y_{2,\eta_1}\in Y_2$.

\begin{claim}
The following are true:
\begin{enumerate}
	\item If there is a vertex $w\in \Delta_{\Lambda,f^{-1}(\Theta')}$ of type $\hat s_i$ for some $i$ satisfying $1<i<n$ such that $w$ is adjacent to each vertex in $Y_1\cup Y_2$, then there is a vertex $w'\in \Delta_{\Lambda,f^{-1}(\Theta')}$ of type $\hat s_{i'}$ for some $i'$ with $1<i\le i'<n$ such that $w'$ is adjacent to each vertex of $Y_1\cup Y_2\cup\{x_{1,\theta_1},x_{2,\theta_1}\}$ and $w'$ is either adjacent or equal to $w$. 
	\item If there is a vertex $w\in \Delta_{\Lambda,f^{-1}(\Theta')}$ of type $\hat s_i$ for some $i$ satisfying $1<i<n$ such that $w$ is adjacent to each vertex in $X_1\cup X_2$, then there is a vertex $w'\in \Delta_{\Lambda,f^{-1}(\Theta')}$ of type $\hat s_{i'}$ for some $i'$ with $1<i'\le i<n$ such that $w'$ is adjacent to each vertex of $X_1\cup X_2\cup\{y_{1,\eta_1},x_{2,\eta_1}\}$ and $w'$ is either adjacent or equal to $w$. 
\end{enumerate}
\end{claim}

Now we prove Claim 2. We only prove the first part, as the second part is similar. Let $\Lambda_{\theta_1,\eta_i}$ denotes the linear subgraph of $\Lambda$ from $\theta_1$ to $\eta_i$. Let $\Lambda_{\theta_1}$ be the component of $\Lambda_{\theta_1,\eta_i}\setminus\{\eta_i\}$ (this definition does not depend on $i$).
 If $w$ is not adjacent to $x_{1,\theta_1}$, then the bowtie free condition on $\Delta_{\Lambda,\Lambda_{\theta_1,\eta_1}}$ implies that there is a vertex $w'\in \Delta_{\Lambda,\Lambda_{\theta_1,\eta_1}}$ such that $$\{x_{1,\theta_1},w\}\le w'\le \{y_{1,\eta_1},y_{2,\eta_1}\}.$$ Then $w'$ is adjacent to each of $\{x_{1,\theta_1},w\}$ and it is of type $\hat s_{i'}$ with $1<i\le i'<n$. In particular, $w'\in \Delta_{\Lambda,\Lambda_{\theta_1}}$, hence $x_{1,\theta_1}$ and $w$ has an upper bound in $\Delta_{\Lambda,\Lambda_{\theta_1}}$. As $\Delta_{\Lambda,\Lambda_{\theta_1}}$ is also bowtie free, by Lemma~\ref{lem:posets}, $x_{1,\theta_1}$ and $w$ admit a join, denoted by $w''$, in $\Delta_{\Lambda,\Lambda_{\theta_1}}$. The bowtie free property for each $\Delta_{\Lambda,\Lambda_{\theta_1,\eta_i}}$ implies that 
 $$w''\le y_{1,\eta_i}\ \textrm{and}\ w''\le y_{2,\eta_i}\  \mathrm{in}\ \Delta_{\Lambda,\Lambda_{\theta_1,\eta_i}}.$$ Thus $w''$ is adjacent to each vertex in $Y_1\cup Y_2$. Note that by construction $w''$ is adjacent to $x_{1,\theta_1}$. If $w''$ is adjacent to $x_{2,\theta_1}$ as well, then we are done, otherwise we can repeat this argument and consider the join $w'''$ of $w''$ and $x_{2,\theta_1}$ in $\Delta_{\Lambda,\Lambda_{\theta_1}}$. In particular $w\le w'''$ and $w'''$ satisfies all the requirements of Claim 2.

Finally we prove Claim 1. First we prove Claim 1 under the additional assumption that there is vertex $z$ of type $\hat s_i$ with $1<i<n$ such that either $z$ is adjacent to each vertex in $X_1\cup X_2\cup Y_1\cup Y_2$ except $x_{1,\theta_1}$ and $x_{2,\theta_1}$, or $z$ is adjacent to each vertex in $X_1\cup X_2\cup Y_1\cup Y_2$ except $y_{1,\eta_1}$ and $y_{2,\eta_1}$. We only treat the former case, as the latter is similar.
Part 1 of Claim 2 implies that there is a vertex $w'$ adjacent to each vertex of  $Y_1\cup Y_2\cup\{x_{1,\theta_1},x_{2,\theta_1}\}$ such that $w'\ge z$ in $\Delta_{\Lambda,\Lambda_{\theta_1}}$. Thus $w'\ge z$ in $\Delta_{\Lambda,\Lambda_{\theta_i}}$ for each $i$. As $z\ge x_{1,\theta_i}$ and $z\ge x_{2,\theta_i}$ in $\Delta_{\Lambda,\Lambda_{\theta_i}}$ for each $i\ge 2$, we know $$w'\ge x_{1,\theta_i}\ \textrm{and}\ w'\ge x_{2,\theta_i}$$ for each $i\ge 2$. Then $w'$ is adjacent to each vertex of $Y_1\cup Y_2\cup X_1\cup X_2$, as desired. This proves Claim 1 under our additional assumption. However, repeating this argument finitely many times proves Claim 1 in full generality.
\end{proof}

We say a Dynkin diagram $\Lambda$ is \emph{$\widetilde A_n$-like}, if it is a cycle, and for each node $s\in \Lambda$, $\Delta_{\Lambda_s}$ is bowtie free, where $\Lambda_s$ denotes the unique component of $\Lambda\setminus \{s\}$. In particular, if $\Lambda$ is the Dynkin diagram of a Coxeter group (or Artin group) of type $\widetilde A_n$, then $\Lambda$ is $\widetilde A_n$-like.

\begin{lem}
	\label{lem:cycle fold}
	Suppose $f:\Lambda\to \Lambda'$ be a special folding from a $\widetilde A_n$-like cycle to a linear graph. 
	Then $\Delta_{\Lambda'}$ is bowtie free.
\end{lem}

\begin{proof}
	By the definition of special folding, $\Lambda$ must be a 4-cycle and $\Lambda'$ is a linear graph with three nodes. Let consecutive nodes of $\Lambda'$ be $\{s'_i\}_{i=1}^3$. Let consecutive nodes of $\Lambda$ be $\{s_i\}_{i=1}^4$. We assume without loss of generality that $f(s_i)=s'_i$ for $1\le i\le 3$ and $f(s_4)=s'_2$. Now we verify that the assumptions of Lemma~\ref{lem:bowtie free criterion} hold for $\Delta_{\Lambda,f}$. 
	
	If we take a vertex $x$ of type $\hat s'_1$ in $X=\Delta_{\Lambda,f}$, then $\lk(x,X)$ admits a type-preserving isomorphism to $\Delta_{\Lambda_{s_1},f}$, which is bowtie free by Lemma~\ref{lem:tree fold}. Thus Assumption 1 of Lemma~\ref{lem:bowtie free criterion} follows. To verify Assumption 2, let $x_1y_1x_2y_2$ be a 4-cycle in $X$ with $x_1,x_2$ of type $\hat s'_1$ and $y_1,y_2$ of type $\hat s'_3$. We can also view these vertices as in $\Delta_\Lambda$, with $x_1,x_2$ of type $\hat s_1$ and $y_1,y_2$ of type $\hat s_3$. It suffices to find an edge $\overline{z_1z_2}\subset X$ connecting a vertex of type $\hat s_2$ and a vertex of type $\hat s_4$ such that each of $\{z_1,z_2\}$ is adjacent to each of $\{x_1,x_2,y_1,y_2\}$. By Lemma~\ref{lem:4-cycle}, there is a vertex $z_1$ adjacent to each of $\{x_1,x_2,y_1,y_2\}$. If $z_1$ is of type $\hat s_2$, then this 4-cycle is contained in $\lk(z_1,X)$, which admits a type-preserving isomorphism to $\Delta_{\Lambda_{s_2}}$. As $\Delta_{\Lambda_{s_2}}$ is bowtie free, there is a vertex $z_2\in \lk(z_1,X)$ of type $\hat s_4$ such that $z_2$ is adjacent to each of $\{x_1,x_2,y_1,y_2\}$. This finishes the proof.
\end{proof}

\begin{thm}
	\label{thm:folded}
Let $\Lambda,\Lambda'$ be connected Dynkin diagrams with a special folding $f:\Lambda\to \Lambda'$.
Suppose $\Lambda'$ is a Dynkin diagram with an induced cyclic subgraph $C'$. 
For a node $s'\in C'$, let $\Lambda_{s'}$ be the component of $\Lambda\setminus\{f^{-1}(s')\}$ that contains all the other nodes of $f^{-1}(C')$ (the definition of special folding ensures that all nodes of $f^{-1}(C'\setminus\{s'\})$ is in the same component of $\Lambda\setminus\{f^{-1}(s')\}$). Suppose for each node $s'\in C'$, $\Lambda_{s'}$ is one of the following:
\begin{enumerate}
	\item a tree whose Artin complex satisfies the labeled 4-cycle condition;
	\item an $f$-folded cycle which is $\widetilde A_n$-like.
\end{enumerate}
Then the following holds.
	\begin{enumerate}
		\item $\Delta_{\Lambda,f,C'}$ is contractile.
		\item $\Delta_\Lambda$ is contractible. 
		\item If, in addition, $A_{\Lambda_{s'}}$ satisfies the $K(\pi,1)$-conjecture for each node $s'\in C'$, then so is $A_\Lambda$.
	\end{enumerate}
\end{thm}

\begin{proof}
Let $X=\Delta_{\Lambda,f,C'}$. Take a vertex $x_0\in X$ of type $\hat s'$ with $s'\in C'$. Let $\Lambda'_{s'}$ be the component of $\Lambda'\setminus\{s'\}$ that contains other nodes of $C'$. By the definition of special folding, $\Lambda_{s'}=f^{-1}(\Lambda'_{s'})$.
Then $\lk(x_0,X)$ admits a type-preserving isomorphism to $Y=\Delta_{\Lambda_{s'},f,C'\setminus\{s'\}}$ by Lemma~\ref{lem:folded} (2). As $\Lambda_{s'}$ is either a tree or a folded cycle, we know $f(\Lambda_{s'})$ is either a tree or a line segment. Then Lemma~\ref{lem:order fold}, Lemma~\ref{lem:tree fold} and Lemma~\ref{lem:cycle fold} imply that $\Delta_{\Lambda_{s'},f,C'\setminus\{s'\}}$ satisfies the labeled 4-cycle condition. By Lemma~\ref{lem:folded} (4) and (5), $Y$ is bowtie free. By Lemma~\ref{lem:folded} (1) and Theorem~\ref{thm:contractible}, $X$ is contractible, thus Assertion 1 follows.

It is also true that $\Delta_{\Lambda_0,f,C'}$ is contractible for any $\Lambda_0=f^{-1}(\Lambda'_0)$ with $\Lambda'_0$ being a connected induced subgraph of $\Lambda'$ containing $C'$. The reason is that $f:\Lambda_0\to \Lambda'_0$ still satisfies the two assumptions of the theorem by Corollary~\ref{cor:inherit}. Hence we can repeat the proof of Assertion 1 to deduce that $\Delta_{\Lambda_0,f,C'}$ is contractible.

%For Assertion 2, note that the pair $(\Lambda',C')$ satisfies the assumption of Proposition~\ref{prop:single cycle}. By the same argument as in the proof of Proposition~\ref{prop:contractible} (1) (using Lemma~\ref{lem:folded}), we know that for any connected induced subgraph $\Lambda''$ of $\Lambda'$ satisfying $C'\subset \Lambda''$, $\Delta_{\Lambda,f,\Lambda''}$ is contractible. In particular, $\Delta_{\Lambda,f}$ is contractible. 

For Assertion 2, we first claim that for any connected induced subgraphs $\Phi'$ and $\Psi'$ of $\Lambda'$ satisfying $C'\subset \Psi'\subset \Phi'$, $\Delta_{f^{-1}(\Phi'),f,\Psi'}$ is contractible. We prove the claim by induction on number of vertices in $\Phi'\setminus C'$. Let $\Phi=f^{-1}(\Phi')$ and $\Psi=f^{-1}(\Psi')$.
The base case when $\Phi'=C'$ follows from the previous paragraph. Now we assume the claim is true when $\Phi'\setminus C'$ has $\le n-1$ vertices. Now take $\Phi'$ such that $\Phi'\setminus C'$ has $n$ vertices. 
We apply Lemma~\ref{lem:subgraph} to the pair $C'\subset \Psi'$ to obtain $\{\Lambda'_i\}_{i=1}^{k}$ and $\{s'_i\}_{i=1}^{k-1}$ as in Lemma~\ref{lem:subgraph}, such that $\Lambda'_1=\Psi'$ and $\Lambda'_{k}=C'$. Let $\Lambda_i=f^{-1}(\Lambda'_i)$.
As $\Delta_{\Phi,f,C'}$ is contractible by the previous paragraph, it suffices to show that for each $1\le i\le k$, $\Delta_{\Phi,f,\Lambda'_i}$ deformation retracts onto $\Delta_{\Phi,f,\Lambda'_{i+1}}$. By the analogue of Lemma~\ref{lem:dr} for folded relative Artin complexes, we need to show $\lk(v,\Delta_{\Phi,f,\Lambda'_i})$ is contractible for any vertex  $v\in \Delta_{\Lambda,f,\Lambda'_i}$ of type $\hat s'_i$. Let $\Theta'_i$ be the component of $\Phi'\setminus \{s'_i\}$ that contains $\Lambda'_{i+1}$. Then $\Theta_i=f^{-1}(\Theta'_i)$ is connected by the definition of special folding.
By Lemma~\ref{lem:folded} (2), $\lk(v,\Delta_{\Phi,f,\Lambda'_i})$ admits a type-preserving isomorphism to $\Delta_{\Theta_i,f,\Lambda'_{i+1}}$.
Note that $C'\subset \Theta'_i$, and $\Theta'_i$ has fewer vertices compared to $\Phi'$. Thus by induction $\Delta_{\Theta_i,f,\Lambda'_{i+1}}$ is contractible, as desired. This finishes the proof of the claim. In particular, it follows from the claim that $\Delta_{\Lambda,f}$ is contractible.

 %As $C'\subset \Lambda_{s_i}\subset\Lambda$, we know $\Lambda_{s_i}\in \mathcal C_2$ by Assumption 2. Note that $\Lambda'\subset \Lambda_{i+1}\subset\Lambda_{s_i}$. As $\Lambda_{i+1}\subsetneq \Lambda''$, $\Lambda_{i+1}\setminus \Lambda'$ has less nodes compared to $\Lambda''\setminus \Lambda'$. By our induction assumption, $\Delta_{\Lambda_{s_i},\Lambda_{i+1}}$ is contractible. The claim follows.

We claim whenever $f$ is a composition of two special folding $f':\Lambda\to \Lambda_m$ and $f'':\Lambda_m\to \Lambda'$, then $\Delta_{\Lambda,f'}$ is contractible. Assertion 2 follows immediately from this claim (by taking $f'$ to be the identity map). We now the prove the claim by induction on   $|\Lambda|-|\Lambda'|$, where $|\Lambda|$ is the number of nodes in $\Lambda$. For the base case when this quantity is $0$, we know $f'$, $f''$ and $f$ are identity maps. Then $\Delta_{\Lambda,f'}=\Delta_{\Lambda,f}$ is contractible by the previous paragraph.

For the induction step, note that there is a finite sequence of special foldings 
$$\{f_i:\Lambda_i\to \Lambda_{i+1}\}_{i=1}^k$$ 
such that
\begin{itemize}
	\item $\Lambda_1=\Lambda$, $\Lambda_{k+1}=\Lambda'$ and $\Lambda_m$ is one of $\{\Lambda_i\}_{i=1}^{k+1}$;
	\item  $f=f_k\circ\cdots\circ f_1$ and $|\Lambda_{i+1}|=|\Lambda_i|-1$ for $1\le i\le k$. 
\end{itemize}
Let $g_i=f_i\circ\cdots\circ f_1$, and let $g_0$ be the identity map from $\Lambda$ to itself. As we already know $\Delta_{\Lambda,f}$ is contractible, it suffices to show $\Delta_{\Lambda,g_i}$ deformation retracts onto $\Delta_{\Lambda,g_{i+1}}$ for $0\le i\le k$.

Let $\{s_{i1},s_{i2}\}$ be the unique pair of distinct nodes in $\Lambda_i$ with the same $f_i$-image.  Now let $\Delta'_{\Lambda,g_i}$ be a subdivision of $\Delta_{\Lambda,g_i}$ by cutting each maximal simplex in $\Delta_{\Lambda,g_i}$ along a codimension 1 sub-simplex spanned by the midpoint of the edge connecting the vertex of type $\hat s_{i1}$ and the vertex of type $\hat s_{i2}$, as well as all the other vertices of the maximal simplex. Then $\Delta_{\Lambda,g_{i+1}}$ is obtained from $\Delta'_{\Lambda,g_i}$ by removing all the open starts (in $\Delta'_{\Lambda,g_i}$) of vertices of type $\hat s_{i1}$ or type $\hat s_{i2}$. To show $\Delta_{\Lambda,g_i}$ deformation retracts onto $\Delta_{\Lambda,g_{i+1}}$, it suffices to prove $\lk(x,\Delta_{\Lambda,g_i})$ is contractible whenever $x$ is of type $\hat s_{i1}$ or $\hat s_{i2}$. We only treat the case when $x$ is of type $\hat s_{i1}$ as the other case is similar. 
Then Lemma~\ref{lem:folded} (2) implies that $\lk(x,\Delta_{\Lambda,g_i})$ admits a type-preserving isomorphism to $\Delta_{\Lambda_{i1},g_i}$ where $\Lambda_{i1}$ is the only component in $\Lambda\setminus\{g^{-1}_i(s_{i1})\}$. As $s_{i1}\neq s_{i2}$, the map $f:\Lambda_{i1}\to \Lambda'$ is a special folding. Note that $|\Lambda_{i1}|-|\Lambda'|<|\Lambda|-|\Lambda'|$, and by Proposition~\ref{prop:inherit},  the special folding $f:\Lambda_{i1}\to \Lambda'$ still satisfies the assumptions of Theorem~\ref{thm:folded}. Now the induction assumption implies that $\lk(x,\Delta_{\Lambda,g_i})$ is contractible, as desired.

For Assertion 3, we induct on $|\Lambda|-|C'|$. The base case where this quantity is $0$ follows from Proposition~\ref{prop:single cycle}. By Assertion 2, it suffices to show the claim that for any node $s\in \Lambda$, each component of $\Lambda\setminus\{s\}$ satisfies the $K(\pi,1)$-conjecture. If $s$ is not the only node in $f^{-1}(f(s))$, then $f$ induces a special folding from the component in $\Lambda\setminus\{s\}$ to $\Lambda'$ which still satisfies the assumptions of Theorem~\ref{thm:folded} (by Proposition~\ref{prop:inherit} and \cite[Corollary 2.4]{godelle2012k}), and the claim follows from the induction assumption. Now suppose $s$ is the only node in $f^{-1}(f(s))$. 

First we consider the case $f(s)\in C'$. Let $\Phi$ be a component of $\Lambda\setminus \{s\}$. If $f^{-1}(C'\setminus\{f(s)\})$ is contained in $\Phi$, then $\Phi=\Lambda_{f(s)}$, hence $A_{\Phi}$ satisfies the $K(\pi,1)$-conjecture by the assumption in Assertion 3. If $\Phi$ does not contain all the nodes of $f^{-1}(C'\setminus\{f(s)\})$, then $\Phi\cap\Lambda_{f(s)}=\emptyset$. As $\Lambda$ is connected, we know $\Phi\subset \Lambda_{s'}$ for $s'\in C'$ and $s'\neq f(s)$. Hence $A_{\Phi}$ satisfies the $K(\pi,1)$-conjecture by assumption in Assertion 3 and \cite[Corollary 2.4]{godelle2012k}. Thus $\Lambda\setminus \{s\}$ corresponds to an Artin group satisfying the $K(\pi,1)$-conjecture, as each of its components does so.

Second we consider the case $f(s)\notin C'$. Let $\Phi'$ be a component of $C'\setminus\{f(s)\}$. Let $\Phi=f^{-1}(\Phi')$. If $C'\subset \Phi'$, then by Proposition~\ref{prop:inherit} and \cite[Corollary 2.4]{godelle2012k}, $f:\Phi\to \Phi'$ is a special folding satisfying the assumptions of Theorem~\ref{thm:folded}. As $\Phi$ has fewer vertices than $\Lambda$, by induction, $A_{\Phi}$ satisfies the $K(\pi,1)$-conjecture. If $C'$ is not contained in $\Phi'$, then $\Phi'\cap C'=\emptyset$. As $\Lambda$ is connected, we know $\Phi\subset \Lambda_{s'}$ for some $s'\in C'$. By the assumption in Assertion 3 and \cite[Corollary 2.4]{godelle2012k}, $A_{\Phi}$ satisfies the $K(\pi,1)$-conjecture. This finishes the proof.
\end{proof}

%is similar to the proof of Corollary~\ref{cor:contractible} (i.e. by an extra layer of induction on the number of nodes in $\Lambda'\setminus C'$). Now we are done by Theorem~\ref{thm:combine}.

\subsection{Some concrete classes}
\label{subsec:some concrete classes}
Here we discuss several concrete new classes of Artin groups satisfying $K(\pi,1)$, as a consequence of Proposition~\ref{prop:single cycle} and Theorem~\ref{thm:folded}.

The following is a consequence of  Theorem~\ref{thm:folded}, Corollary~\ref{cor:wheel}, Corollary~\ref{cor:algebraic}, Corollary~\ref{cor:locally reducible}, Proposition~\ref{prop:widetilde C} and the main theorem of \cite{paolini2021proof}.
\begin{thm}
	\label{thm:folded1}
	Let $f:\Lambda\to \Lambda'$  be a special folding between connected Dynkin diagrams. Assume $\Lambda'$ contains an induced cyclic subgraph $C'$.
	For a node $s'\in C'$, let $\Lambda_{s'}$ be the component of $\Lambda\setminus\{f^{-1}(s')\}$ that contains all the other nodes of $f^{-1}(C')$. Suppose 
	\begin{enumerate}
		\item the two assumptions of Corollary~\ref{cor:algebraic} holds for affine Artin groups of type $\widetilde B$, $\widetilde D$, $\widetilde E$ and $\widetilde F$;
		\item for each node $s'\in C'$, $\Lambda_{s'}$ is either of a tree of spherical type or affine type or locally reducible type, or an $f$-folded $\widetilde A_n$-type cycle.
	\end{enumerate}
	Then the $K(\pi,1)$-conjecture holds for $A_\Lambda$.
\end{thm}

Now the following is immediate.
\begin{thm}
	\label{thm:folded2}
	Let $f:\Lambda\to \Lambda'$  be a special folding between connected Dynkin diagrams. Assume $\Lambda'$ contains an induced cyclic subgraph $C'$.
	For a node $s'\in C'$, let $\Lambda_{s'}$ be the component of $\Lambda\setminus\{f^{-1}(s')\}$ that contains all the other nodes of $f^{-1}(C')$. Suppose for each node $s'\in C'$, $\Lambda_{s'}$ is either of spherical type, or type $\widetilde A_n$ which is $f$-folded, or type $\widetilde C_n$, or locally reducible tree.
	Then the $K(\pi,1)$-conjecture holds for $A_\Lambda$.
\end{thm}

\begin{cor}
	\label{cor:extended lanner}
All Artin groups whose Dynkin diagrams belongs to Figure~\ref{fig:9} satisfy the $K(\pi,1)$-conjecture.
\end{cor}

\begin{proof}
Note that Theorem~\ref{thm:folded2} applies directly to all the classes in Figure~\ref{fig:9} (the family with two 4-gons inside requires a special folding), except the two horseshoe crab shaped families and the family on the top right corner. 

Let $\Lambda$ be a diagram belong to one of the horseshoe crab families. Let $\Lambda_1,\Lambda_2$ be the $n$-cycle and  3-cycle inside $\Lambda$. It suffices to show $\Delta_{\Lambda,\Lambda_2}$ is contractible, then an argument similar to Proposition~\ref{prop:single cycle} will imply the $K(\pi,1)$-conjecture for this family. Let the nodes of $\Lambda'$ be $\{s_i\}_{i=1}^3$ with $s_1,s_2\in \Gamma_1\cap \Gamma_2$. We metrize each triangle in $X=\Delta_{\Lambda,\Lambda_2}$ as a flat equilateral triangle. By Lemma~\ref{lem:4-cycle}, the edge $\overline{s_1s_2}$ is 3-solid in $\Lambda_2$. For any vertex $x\in X$ of type $\hat s_3$, $\lk(x,X)$ admits a type-preserving isomorphism to $\Delta_{\Lambda_1,\overline{s_1s_2}}$ by Lemma~\ref{lem:link}. Thus $X$ is locally $\CAT(0)$ around vertices of type $\hat s_3$. Note that $\Lambda\setminus\{s_i\}$ is spherical for $i=1,2$ (either type $A_n$ or $B_n$ or $D_n$). Then Theorem~\ref{thm:bowtie free} and Lemma~\ref{lem:link} imply that $\lk(x,X)$ has girth $\ge 6$ when $x$ is of type $\hat s_1$ or $\hat s_2$. As $X$ is simply-connected (Lemma~\ref{lem:relative sc}), we know $X$ is $\CAT(0)$, hence it is contractible.

Now let $\Lambda$ be a diagram belonging to the family of the top right corner. Let $\Lambda'$ be the triangle subgraph of $\Lambda$ with its nodes $\{s_1,s_2,s_3\}$. Suppose $s_3$ is the node not in the edge labeled by $n$. Again it suffices to show $\Delta_{\Lambda,\Lambda'}$ is contractible. We metrize triangles in $X=\Delta_{\Lambda,\Lambda'}$ as flat equilateral triangles, and it suffices to check the girth of link at each vertex. When $x$ is a vertex of type $\hat s_3$, then $\Lambda_1$ is of type $A_4$ or $B_4$ or $H_4$ where $\Lambda_1$ is the component of $\Lambda\setminus\{s_3\}$ containing $\overline{s_1s_2}$. Note that $\overline{s_1s_2}$ is $3$-solid in $\Lambda_1$ -- indeed, this follows from Theorem~\ref{thm:bowtie free} if $n\le 5$ and Lemma~\ref{lem:enlarge1} when $n\ge 6$. As $\lk(x,X)$ admits a type-preserving isomorphism to $\Delta_{\Lambda_1,\overline{s_1s_2}}$, we know $\lk(x,X)$ has girth $\ge 6$. The other cases are similar and simpler.
\end{proof}

\begin{cor}
	\label{cor:lanner}
Suppose $\Lambda$ is either a Lann\'er diagram or a quasi-Lann\'er diagram such that $\Lambda$ contains a cycle. If the Artin complex of Artin group of type $\widetilde B_4$, $\widetilde F_4$ and $\widetilde E_7$ satisfies the labeled 4-cycle condition (or equivalently the condition in Corollary~\ref{cor:wheel}), then $A_\Lambda$ satisfies the $K(\pi,1)$-conjecture.
\end{cor}

\begin{proof}
See \cite{chein1969recherche} for the list of Lann\'er and quasi-Lann\'er diagrams. Theorem~\ref{thm:folded}, Corollary~\ref{cor:wheel} and Corollary~\ref{cor:locally reducible} already imply the $K(\pi,1)$-conjecture holds true for all the diagrams satisfying the assumption of Corollary~\ref{cor:lanner} except the four diagrams in Figure~\ref{f:lanner}, which contain sub-diagrams of type $\widetilde B_4$, $\widetilde F_4$ or $\widetilde E_7$. Thus Theorem~\ref{thm:folded} reduces the $K(\pi,1)$-conjecture for these four diagrams to the labeled 4-cycle condition of affine Artin complexes of type $\widetilde B_4$, $\widetilde F_4$ and $\widetilde E_7$.
\end{proof}

\begin{figure}[h!]
	\centering
	\includegraphics[scale=1]{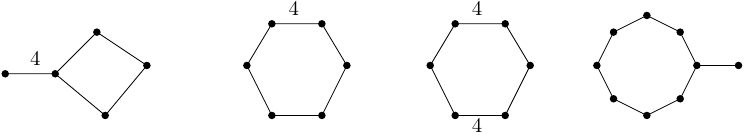}
	\caption{Remaining cases.}
	\label{f:lanner}
\end{figure}

An Artin group is of \emph{hyperbolic cyclic} type if its Dynkin diagram is a circle, and the associated reflection group acts \emph{cocompactly} on $\mathbb H^n$. 

\begin{cor}
	\label{cor:group1}
	Suppose $A_S$ is an Artin group of hyperbolic cyclic type. Then
	\begin{enumerate}
		\item the intersection of any parabolic subgroups of $A$ is a parabolic subgroup;
		\item $A_S$ satisfies Properties $(\ast),(\ast\ast)$ and $(\ast\ast\ast)$ from \cite{godelle2007artin}, in particular, for any $X\subset S$, the commensurator of $A_X$ in $A_S$ coincides with the normalizer of $A_X$ in $A_S$, which is equal to $A_X$ times its quasi-center $QZ_{A_S}(X)$ (recall that $QZ_{A_S}(X)=\{g\in A_S\mid g\cdot X=X\}$);
		\item for any group $G$ of symmetries of the Artin systems $A_S$, the fix point subgroup $A^G_S$ is isomorphic to an Artin group.
	\end{enumerate}
\end{cor}

\begin{proof}
If $A_S$ is a hyperbolic cyclic type Artin group, then by Theorem~\ref{thm:bowtie free}, $\Delta_S$ satisfies the assumption of Theorem~\ref{thm:contractible}. In particular, $\Delta_S$ can be equipped with an $A_S$-invariant metric with consistent convex geodesic bicombing $\sigma$ such that each simplex is $\sigma$-convex. Also note that the stabilizer of each vertex in $\Delta_S$ is a spherical parabolic subgroup of $A_S$. Now the rest follows by repeating the discussion in \cite[Section 8]{haettel2021lattices}, using works \cite{morris2021parabolic,godelle2007artin,crisp2000symmetrical}.
\end{proof}
%
%\section{Tree of triangles}
%
%\begin{lem}
%Let $\Lambda$ be a tree of triangles with a triangle $\Lambda_1$ and connector $\Lambda_2$ inside such that $\Lambda_1\cap \Lambda_2=\{s_1\}$. Suppose the vertices of $\Lambda_1$ are $\{r_1,s_1,t_1\}$ and the consecutive vertices of $\Lambda_2$ are 
%
%\end{lem}
\bibliographystyle{alpha}
\bibliography{mybib}
\end{document}